\documentclass[10pt, reqno]{amsart}
\usepackage{amssymb}
\usepackage{graphicx}
\usepackage{xcolor} % A package to add color.
\usepackage{tensor}

\usepackage[english]{babel}
\usepackage[utf8]{inputenc}
\usepackage[colorinlistoftodos]{todonotes}

\usepackage{amsfonts}
\usepackage{amssymb}
\usepackage{amsmath}
\usepackage{amsthm}

\usepackage{graphics}
\usepackage{pdfsync}

\usepackage{graphicx}
\usepackage{amssymb}
\usepackage{epstopdf}
\usepackage{hyperref}
\usepackage{color}
\usepackage{tikz}

\usepackage{cases}

\newtheorem{theorem}{Theorem}[section]
\newtheorem{corollary}[theorem]{Corollary}
\newtheorem{lemma}[theorem]{Lemma}
\newtheorem{proposition}[theorem]{Proposition}
\theoremstyle{definition}

\theoremstyle{remark}
\newtheorem{remark}[theorem]{Remark}
\numberwithin{equation}{section}

\newcommand{\nrm}[1]{\Vert#1\Vert}
\newcommand{\abs}[1]{\vert#1\vert}

\newcommand{\set}[1]{\{#1\}}

\newcommand{\aleq}{\lesssim}

\newcommand{\lap}{\Dlt}

\newcommand{\ud}{\mathrm{d}}
\newcommand{\rd}{\partial}

\newcommand{\bb}{\Big}

\newcommand{\alp}{\alpha}

\newcommand{\dlt}{\delta}
\newcommand{\Dlt}{\Delta}
\newcommand{\eps}{\epsilon}

\newcommand{\tht}{\theta}

\newcommand{\omg}{\omega}

\newcommand{\bfj}{{\bf j}}
\newcommand{\bfk}{{\bf k}}

\newcommand{\bfA}{{\bf A}}

\newcommand{\bbR}{\mathbb R}
\newcommand{\bbS}{\mathbb S}

\newcommand{\bbZ}{\mathbb Z}

\newcommand{\calC}{\mathcal C}

\newcommand{\calF}{\mathcal F}

\newcommand{\calH}{\mathcal H}

\newcommand{\calL}{\mathcal L}
\newcommand{\calM}{\mathcal M}
\newcommand{\calN}{\mathcal N}

\newcommand{\calP}{\mathcal P}

\newcommand{\calR}{\mathcal R}
\newcommand{\calS}{\mathcal S}

\newcommand{\mb}{\mathbb}
\newcommand{\be}{\begin{equation}}
\newcommand{\ee}{ \end{equation}}

\newcommand{\ep}{\varepsilon}

\newcommand{\sg}{\sigma}
\newcommand{\al}{\alpha}

\newcommand{\dd}{\,\mathrm{d}}

\newcommand{\vp}{\varphi}

\newcommand{\vn}[1]{\|#1\|}
\newcommand{\vm}[1]{\left|#1\right|}

\newcommand{\lpr}{ \left( }
\newcommand{\rpr}{ \right) }

\newcommand{\lng}{\langle}
\newcommand{\rng}{\rangle}

\newcommand{\lmd}{\lambda}
\newcommand{\lpp}{\left[}
\newcommand{\rpp}{\right]}

\newcommand{\la}{\Delta}

\newcommand{\ls}{\lesssim}

\newcommand{\pfstep}[1]{\vskip.5em \noindent {\bf #1.}}

\newcommand{\med}{\mathrm{med}}

\newcommand*\jb[1]{\left\langle #1 \right\rangle} % Japanese Bracket

\newcommand{\pt}{\partial}

\usepackage{mathtools}
\newcommand{\defeq}{\vcentcolon=}
\usepackage{enumitem}

\begin{document}

\title[Global well-posedness for massive MKG]{Global well-posedness for the massive Maxwell-Klein-Gordon equation with small critical Sobolev data}%: Title of the article
\author{Cristian Gavrus}%
\address{Department of Mathematics, UC Berkeley, Berkeley, CA, 94720}%
\email{cristian@berkeley.edu}%

\thanks{The author thanks Daniel Tataru for many fruitful conversations and thanks Sung-Jin Oh for suggesting this problem. The author was supported in part by the NSF grant DMS-1266182.}

%\thanks{}%
%\subjclass{}%
%\keywords{}%

%\date{\today}%
%\dedicatory{}%
%\commby{}%
% ----------------------------------------------------------------
\begin{abstract}
In this paper we prove global well-posedness and modified scattering for the massive Maxwell-Klein-Gordon equation in the Coulomb gauge on $\bbR^{1+d}$ $(d \geq 4)$ for data with small critical Sobolev norm. This extends to the general case $ m^2 > 0 $ the results of Krieger-Sterbenz-Tataru ($d=4,5 $) and Rodnianski-Tao ($ d \geq 6 $), who considered the case $ m=0$.

We proceed by generalizing the global parametrix construction for the covariant wave operator and the functional framework from the massless case to the Klein-Gordon setting.  The equation exhibits a trilinear cancelation structure identified by Machedon-Sterbenz. To treat it one needs sharp $ L^2 $ null form bounds, which we prove by estimating renormalized solutions in null frames spaces similar to the ones considered by Bejenaru-Herr.   
To overcome logarithmic divergences we rely on an embedding property of $ \Box^{-1} $ in conjunction with endpoint Strichartz estimates in Lorentz spaces. 
\end{abstract}
\maketitle
% ----------------------------------------------------------------

\setcounter{tocdepth}{1}
\tableofcontents

\newpage

\section{Introduction} In this paper we consider the \emph{massive Maxwell-Klein-Gordon} (MKG) equation on the Minkowski space $ \mb{R}^{d+1} $ endowed with the metric $ g=\text{diag}(1,-1,-1,-1,-1) $. Throughout the paper we assume $ d \geq 4 $.

This equation arises as the Euler-Lagrange equations for the Lagrangian

$$ \calS[A_{\mu},\phi]= \iint_{\mb{R}^{d+1}} \frac{1}{2} D_{\al} \phi  \overline{D^{\al} \phi} + \frac{1}{4}F_{\al \beta} F^{\al \beta} - \frac{1}{2} m^2 \vm{\phi}^2   \dd x \dd t	$$

Here $ \phi : \mb{R}^{d+1} \to \mb{C} $ is a complex field, while $ A_{\al} : \mb{R}^{d+1} \to \mb{R} $ is a real $1$-form with curvature
$$ F_{\al \beta} \defeq \pt_{\al} A_{\beta}- \pt_{\beta} A_{\al}. $$
One defines the \emph{covariant derivatives}  and the \emph{covariant Klein-Gordon operator} by
$$ D_{\al} 	\phi \defeq (\pt_{\al}+i A_{\al}) \phi, \qquad   \Box_m^A \defeq D^{\al} D_{\al}+m^2
$$
A brief computation shows that the Euler-Lagrange equations take the form

\begin{equation} \label{MKG}
\left\{
\begin{aligned}
	& \pt^{\beta} F_{\al \beta} = \mathfrak{I}(\phi \overline{D_{\al}\phi}), \\
	& (D^{\al} D_{\al}+m^2) \phi = 0, \\
\end{aligned}
\right.
\end{equation}

The MKG system is considered the simplest classical field theory enjoying a nontrivial \emph{gauge invariance}. Indeed, for any potential function $ \chi $, replacing 
\be \label{gauge:inv} \phi \mapsto e^{i \chi} \phi, \quad A_{\al} \mapsto A_{\al}-\pt_{\al} \chi, \quad D_{\al} \mapsto e^{i \chi} D_{\al} e^{-i \chi}
\ee
one obtains another solution to \eqref{MKG}. To remove this gauge ambiguity we will work with the \emph{Coulomb gauge} 
\be \label{Coulomb}
 \text{div}_x A=\pt^j A_j=0
\ee
where Roman indices are used in sums over the spatial components. The system is Lorentz invariant and admits a \emph{conserved energy}, which we will not use here.

Under \eqref{Coulomb}, denoting $ J_{\al}=-\mathfrak{I}  (\phi \overline{D_{\al} \phi}) $, the MKG system \eqref{MKG} becomes
\begin{equation} \label{MKGCG}
\left\{
\begin{aligned}
	& \Box_m^A \phi = 0 \\
	& \Box A_i = \calP_i J_x \\
	& \Delta A_0=J_0, \ \Delta \pt_t A_0=\pt^i J_i
\end{aligned}
\right.
\end{equation}
where $ \calP $ denotes the Leray projection onto divergence-free vector fields
\be \label{Leray} \calP_j A \defeq \Delta^{-1} \pt^k (\pt_k A_j-\pt_j A_k ).  
\ee

When $ m=0 $ the equation is invariant under the \emph{scaling}
$$ \phi \mapsto \lmd \phi(\lmd t, \lmd x); \qquad A_{\al} \mapsto \lmd A_{\al}(\lmd t, \lmd x) $$
which implies that $ \sg=\frac{d}{2}-1 $ is the \emph{critical regularity}.  We shall refer to $ H^{\sg}\times H^{\sg-1} \times \dot{H}^{\sg} \times \dot{H}^{\sg-1} $ as the critical space for $ (A,\phi)[0] $ also when $ m\neq 0 $.

At this regularity, globally in time, the mass term $ m^2 \phi $ is not perturbative and must be considered as part of the operator $ \Box_m^A $.

The main result of this paper consists in extending the results in \cite{KST}, \cite{RT} to the case $ m\neq 0 $. For a more detailed statement, see Theorem \ref{thm:main-iter}.

\begin{theorem}[Critical small data global well-posedness and scattering] \label{thm:main}
Let $ d \geq 4 $ and $ \sg=\frac{d}{2}-1$.
The MKG equation \eqref{MKGCG} is well-posed for small initial data on $\bbR^{1+d}$ with $m^2 > 0$, in the following sense:  there exists a universal constant $\ep > 0$ such that: 

\begin{enumerate} [leftmargin=*]
\item Let $(\phi[0], A_{x}[0])$ be a smooth initial data set satisfying the Coulomb condition \eqref{Coulomb} and the smallness condition
\begin{equation} \label{eq:main:smalldata}
	\nrm{\phi[0]}_{H^{\sg}\times H^{\sg-1}} 
	+ \nrm{A_{x}[0]}_{ \dot{H}^{\sg} \times \dot{H}^{\sg-1}  } < \ep.
\end{equation}
Then there exists a unique global smooth solution $(\phi, A)$ to the system \eqref{MKGCG} under the Coulomb gauge condition \eqref{Coulomb} on $\bbR^{1+d}$ with these data. 
\item For any $T > 0$, the data-to-solution map $ (\phi[0], A_{x}[0]) \mapsto (\phi,\pt_t \phi, A_{x}, \rd_{t} A_{x})$ extends continuously to
\begin{align*}
	H^{\sg}\times H^{\sg-1} \times \dot{H}^{\sg} \times \dot{H}^{\sg-1} (\bbR^{d}) \cap \set{\hbox{\eqref{eq:main:smalldata}}} \to  C([-T, T]; H^{\sg}\times H^{\sg-1} \times \dot{H}^{\sg} \times \dot{H}^{\sg-1} (\bbR^{d})).
\end{align*}
\item The solution $(\phi, A)$ exhibits modified scattering as $t \to \pm \infty$: there exist a solution $(\phi^{\pm \infty}, A^{\pm \infty}_{j})$ to the linear system
\footnote{Here, $ A^{free} $ is the free solution $ \Box A^{free}=0 $ with $A^{free}_{x}[0] = A_{x}[0]$ and $ A^{free}_{0} = 0$.}
$$ \Box A_{j}^{\pm \infty} =0, \qquad \Box_m^{A^{free}} \phi=0, \qquad \text{such that}		$$
\begin{equation*}
	\nrm{(\phi - \phi^{\pm \infty})[t]}_{H^{\sg}\times H^{\sg-1}}
	+ \nrm{(A_{j} - A^{\pm \infty}_{j})[t]}_{\dot{H}^{\sg} \times \dot{H}^{\sg-1}  } 
	 \to 0 \quad \hbox{ as } t \to \pm \infty.
\end{equation*}

\end{enumerate}
\end{theorem}

\

The case $ d=4 $ is the most difficult. When $ d \geq 5 $ the argument simplifies, in particular the spaces $  NE_C^{\pm}, \ PW_C^{\pm}, \ L^2 L^{4,2} $  discussed below are not needed. To fix notation, the reader is advised to set $ d=4, \ \sg=1 $. 

% In the general case $d \geq 4$, the theorem holds with the spaces $ H^1 \times L^2(\bbR^{4}) $ % and $\dot{H}^{1} \times L^{2}(\bbR^{4})$ replaced by $ H^{\frac{d-2}{2}} \times H^{\frac{d-4}{2}}% (\bbR^{d})$ and $\dot{H}^{\frac{d-2}{2}} \times \dot{H}^{\frac{d-4}{2}}(\bbR^{d})$, respectively. 
% In fact, when $d \geq 5$ many simplifications can be made (for instance the spaces $  NE_C^{\pm}, \ PW_C^{\pm} $ are not needed). In an updated version of this paper we will elaborate on the simplifications that can be made for $ d \geq 5 $.

\begin{remark} The theorem is stated for Coulomb initial data. However, it can be applied to an arbitrary initial data satisfying \eqref{eq:main:smalldata} by performing a gauge transformation.
Given a real 1-form $A_{j}(0)$ on $\bbR^{d}$, one solves the Poisson equation
$$	\lap \chi = \mathrm{div}_{x} A_{j}(0), \qquad \qquad \chi \in \dot{H}^{\frac{d}{2}} \cap \dot{W}^{1,d}(\bbR^{d}).	$$
Then $\tilde{A}(0) = A(0) - \ud \chi$ obeys the Coulomb condition \eqref{Coulomb}. For small $\ep$, the small data condition \eqref{eq:main:smalldata} is preserved up to multiplication by a constant.
\end{remark}
In what follows we set $ m=1 $, noting that by rescaling, the statements for any $ m \neq 0 $ can be obtained. Notation-wise, we will write $ \Box_m $ rather than $ \Box_1 $.

\

\subsection{Previous work} Progress on the Maxwell-Klein-Gordon equation has occurred in conjunction with the Yang--Mills(-Higgs) equations. An early result was obtained by Eardley and Moncrief \cite{eardley1982}.

On $ \mb{R}^{2+1} $ and $  \mb{R}^{3+1} $ the MKG system is energy subcritical. Klainerman-Machedon \cite{KlMc} and Selberg-Tesfahun \cite{Selb} (in the Lorenz gauge) have proved global regularity as a consequence of local well-posedness at the energy regularity. Further results were obtained by Cuccagna \cite{Cucc}. Machedon-Sterbenz \cite{MachedonSterbenz} proved an essentially optimal local well-posedness result. In \cite{Keel2011573} in $  \mb{R}^{3+1} $, global well-posedness below the energy norm was considered. 

On $  \mb{R}^{4+1} $, an almost optimal local well-posedness result was proved by Klainerman-Tataru \cite{KlaTat} for a model problem closely related to MKG and Yang-Mills. This was refined by Selberg \cite{SSel} for MKG and Sterbenz \cite{sterbenz2007global}.

At critical regularity all the existing results are for the massless case $ m=0 $. Rodnianski-Tao \cite{RT} proved global regularity for smooth and small critical Sobolev data in dimensions $ 6+1 $ and higher. This result was extended by Krieger-Sterbenz-Tataru \cite{KST} to  $ \mb{R}^{4+1} $.

The small data $ \mb{R}^{4+1} $ energy critical massless result in \cite{KST} has been extended to large data global well-posedness by Oh-Tataru (\cite{OT2},\cite{OT3},\cite{OT1})  and independently by Krieger-L\"uhrmann \cite{KL}. Proving a similar large data result for the massive case remains an open problem.

Recent works on Yang-Mills include: \cite{KS}, \cite{KT}, \cite{OhYM}, \cite{oh2015}. We also mention the related $ \mb{R}^{4+1} $ Maxwell-Dirac result \cite{MD} at critical regularity. 

\

\subsection{Main ideas}

\subsubsection*{Null structures in the Coulomb gauge.} Null structures arise in equations from mathematical physics where it manifests through the vanishing of resonant interactions. A classical null form refers to a linear combination of 
\be \label{cl:nf}
\calN_{ij}(\phi,\psi) =\pt_{i} \phi \pt_{j} \psi- \pt_{j}  \phi \pt_{i} \psi, \qquad \calN_{0}(\phi,\psi) =\pt_{\al} \phi \cdot \pt^{\al} \psi.  \ee
A key observation of Klainerman and Machedon in \cite{KlMc} was that quadratic nonlinearities of MKG consist of null forms of the type $ \calN_{ij} $ (see \eqref{phi:nf:identity}, \eqref{ax:nf:identity}). They used this to prove global well-posedness at energy regularity on $\bbR^{1+3}$.

Furthermore, in the proof of essentially optimal local well-posedness of MKG on $\bbR^{1+3}$ by Machedon and Sterbenz \cite{MachedonSterbenz}, secondary trilinear null structures involving $ \calN_{0}$ were identified in MKG after one iteration, see \eqref{Q:dec}, \eqref{Q:dec2}.

Both of these structures played an important role in \cite{KST}, and they also do so here. However, special care must be taken since the null form $ \calN_{0}$ is adapted to the wave equation while we work with Klein-Gordon waves; see subsections \ref{Mform} \ref{M0form}.

\subsubsection*{Presence of non-perturbative nonlinearity.}
As usual in works on low regularity well-posedness, we take a paradifferential approach in treating the nonlinear terms, exploiting the fact that the high-high to low interactions are weaker and that terms where the derivative falls on low frequencies are weaker as well. 

From this point of view, the worst interaction in MKG occurs in 
$$ \sum_{k} A^{\al}_{<k-C} \pt_{\al} \bar{P}_k \phi. $$
At critical regularity this term is \emph{nonperturbative}, in the sense that even after utilizing all available null structure, it cannot be treated with multilinear estimates for the usual wave and Klein-Gordon equations. Instead, following the model set in \cite{RT} and \cite{KST}, this term must be viewed as a part of the underlying linear operator.

The presence of a nonperturbative term is characteristic of geometric wave equations; examples include Wave Maps, MKG, Yang--Mills, Maxwell--Dirac.

\subsubsection*{Parametrix construction for paradifferential covariant wave equation}\

A key breakthrough of Rodnianski and Tao \cite{RT} was proving Strichartz estimates for the covariant wave equation by introducing a microlocal parametrix construction, motivated by the gauge covariance of $\Box_{A} =D^{\al} D_{\al} $ under $ \eqref{gauge:inv} $, i.e., $e^{-i \Psi} \Box_{A'} (e^{i \Psi} \phi)= \Box_A \phi$.
The idea was to approximately conjugate (or renormalize) the modified d'Alembertian $ \Box+ 2i A_{<k-c} \cdot \nabla_x P_k $ to $ \Box $ by means of a carefully constructed pseudodifferential gauge transformation
$$ \Box_A^p \approx e^{i \Psi_{\pm}}(t,x,D) \Box e^{-i \Psi_{\pm}} (D,s,y). $$
These Strichartz estimates were sufficient to prove global regularity of the Maxwell-Klein-Gordon equation at small critical Sobolev data in dimensions $ d \geq 6 $.
A summary of the construction in \cite{KST} can be found in Section~\ref{Constr:phase} below.

A renormalization procedure has been also applied to the Yang-Mills equation at critical regularity \cite{KS}, \cite{KT}.

The above idea was also implemented for covariant Dirac operators in \cite{MD} in the context of the Maxwell-Dirac equation.

\subsubsection*{Function spaces}

In \cite{KST}, Krieger, Sterbenz and Tataru further advanced the parametrix idea in $ d=4 $, showing that it interacts well with the function spaces previously introduced for critical wave maps in \cite{Tat}, \cite{Tao2}. In particular, the resulting solution obeys similar bounds as ordinary waves, yielding control of an $X^{s, b}$ norm, null-frame norms and square summed angular-localized Strichartz norms. 

In this paper we will follow their strategy, showing that both the spaces and the renormalization bounds generalize to the Klein-Gordon ($ m^2 > 0$) context. 

Critical $ X^{s,\pm \frac{1}{2}} $ spaces, 'null' energy $ L^{\infty}_{t_{\omega,\lmd}} L^2_{x_{\omega,\lmd}} $ and Strichartz $ L^{2}_{t_{\omega,\lmd}} L^{\infty}_{x_{\omega,\lmd}} $ spaces in adapted frames, were already developed for the Klein-Gordon operators by Bejenaru and Herr \cite{BH1} and we will use some of their ideas. The difficulty here consists in proving the bounds for renormalized solutions
$$ \vn{e^{-i \psi}(t,x,D)\phi}_{\bar{S}^1_k} \ls \vn{\phi[0]}_{H^1 \times L^2}+ \vn{\Box_m \phi}_{\bar{N}_k}. $$

We shall rely on $ TT^{*} $ and stationary phase arguments for both $ L^{\infty}_{t_{\omega,\lmd}} L^2_{x_{\omega,\lmd}} $ and $ L^{2}_{t_{\omega,\lmd}} L^{\infty}_{x_{\omega,\lmd}} $ bounds, as well for $ P_C L^2 L^{\infty} $, see Corollaries \ref{Cornullframe}, \ref{corPW} and \ref{Cor:L2Linf}.

However, at low frequency or at high frequencies with very small angle interactions, the adapted frame spaces do not work and we are confronted with logarithmic divergences. To overcome this we rely on Strichartz estimates in Lorentz spaces $ L^2 L^{4,2} $ and an embedding property of  $ \Box^{-1} $ into $ L^1 L^{\infty} $. 

Here we have been inspired by the paper \cite{ShSt} of Shatah and Struwe. The difference is that instead of inverting $ \Delta $ by a type of Sobolev embedding $ \vm{D_x}^{-1}: L^{d,1}_x \to L^{\infty}_x $, we have to invert $ \Box $ by 
$$  2^{\frac{1}{2}l} \sum_{k'} P_l^{\omega} Q_{k'+2l} P_{k'} \frac{1}{\Box} : L^1 L^{2,1} \to L^1 L^{\infty} $$
See Proposition \ref{Box:Embedding}.

\section{Preliminaries}

In what follows we normalize the Lebesgue measure so that factors of $ \sqrt{2 \pi} $ do not appear in the Fourier transform formulas. We summarize the conventions that we use in the following table.
\bgroup
\def\arraystretch{1.5}%  1 is the default, change whatever you need

\begin{tabular}{ |p{3cm}||p{3cm}|p{3cm}|p{3cm}|  }
 \hline

  & Klein-Gordon & Wave & Laplace\\
 \hline
 Functions   &  $ \qquad \phi  $   & $ \qquad A_x $ &  $ \qquad  A_0 $ \\ \hline
Operator &   $ \quad  \Box_m=\Box+I $ & $ \qquad \Box $   & $ \qquad \Delta $  \\ \hline
Frequency  &  $   \bar{P}_k , \     k \geq 0       $   & $ P_{k'} ,\ k' \in \mb{Z} $ & $ P_{k'} , \ k' \in \mb{Z}  $   \\
localization & $  \{ \jb{\xi} \simeq 2^k \} $ & $ \{ \vm{\xi} \simeq 2^k \} $ &  $ \{ \vm{\xi} \simeq 2^k \} $
\\ \hline
 Modulations    & \qquad $ \bar{Q}^{\pm}_j $    &  \qquad $ Q ^{\pm}_j $&   \\ 
 &  $ \{ \tau \mp \jb{\xi} \simeq 2^j \} $    & $ \{ \tau \mp \vm{\xi} \simeq 2^j \} $ & 
 \\ \hline
Spaces &   $ \bar{S}_k, \bar{N}_k, \quad \bar{S}^{\sg},  \bar{N}^{\sg-1}  $ & $ S_{k'}, N_{k'}, \quad S^{\sg}, N^{\sg-1} $   & \qquad  $ Y^{\sg} $ \\
 \hline
\end{tabular}
\egroup

\

\subsection{Notation}
We denote 
$$ \jb{\xi}_k=(2^{-2k}+\vm{\xi}^2)^{\frac{1}{2}}, \qquad \jb{\xi}=(1+\vm{\xi}^2)^{\frac{1}{2}}. $$
We define $ A \prec B $ by $ A \leq B-C $, $ A \ls B $ by $ A \leq C B $ and $ A=B+O(1) $ by $ \vm{A-B} \leq C $, for some absolute constant $ C $. We say $ A \ll B $ when $ A \leq \eta B $ for a small constant $ 0<\eta<1 $ and $ A \simeq B $ when he have both $ A \ls B $ and $ B \ls A $. 

\subsection{Frequency projections}

Let $ \chi $ be a smooth non-negative bump function supported on $ [2^{-2},2^2] $ which satisfies the partition of unity property
$$ \sum_{k' \in \mathbb{Z}} \chi \big( \vm{\xi}/2^{k'} \big)=1 $$
for $ \xi \neq 0 $. For $ k' \in \mb{Z},\ k \geq 0 $ we define the Littlewood-Paley operators $ P_{k'}, \bar{P}_k $ by 
$$ \widehat{P_{k'} f} (\xi)=\chi \big( \vm{\xi}/2^{k'} \big) \hat{f} (\xi), \quad \bar{P}_0=\sum_{k' \leq 0} P_{k'}, \quad \bar{P}_k=P_{k}, \ \text{for} \ k \geq 1.  $$
The modulation operators $ Q_j, Q_j^{\pm}, \ \bar{Q}_j, \bar{Q}_j^{\pm} $ are defined by 
$$  \mathcal{F} (\bar{Q}_j^{\pm}f) (\tau,\xi)= \chi \big( \frac{\vm{\pm \tau-\jb{\xi}} }{2^j} \big)  \mathcal{F}f (\tau,\xi), \quad  \mathcal{F} (Q_j^{\pm}f) (\tau,\xi)= \chi \big( \frac{\vm{\pm \tau-\vm{\xi}} }{2^j} \big)  \mathcal{F}f (\tau,\xi). $$ 
and $ Q_j=Q_j^{+}+Q_j^{-}, \ \bar{Q}_j=\bar{Q}_j^{+}+\bar{Q}_j^{-} $ for $ j \in \mathbb{Z} $, where $ \mathcal{F} $ denotes the space-time Fourier transform. 

Given  $ \ell \leq 0 $ we consider a collection of directions $ \omega $ on the unit sphere which is maximally $ 2^\ell $-separated. To each $ \omega $ we associate a smooth cutoff function $ m_{\omega} $ supported on a cap $ \kappa \subset \bbS^{d-1} $ of radius $ \simeq 2^\ell $ around $ \omega $, with the property that $ \sum_{\omega}  m_{\omega}=1 $. We define $ P_\ell^{\omega} $ (or $ P_{\kappa} $ )to be the spatial Fourier multiplier with symbol $ m_{\omega}(\xi/\vm{\xi}) $. In a similar vein, we consider rectangular boxes $ \mathcal{C}_{k'}(\ell') $ of dimensions $ 2^{k'} \times (2^{k'+\ell'})^{d-1} $, where the $ 2^{k'} $ side lies in the radial direction, which cover $\bbR^{d}$ and have finite overlap with each other. We then define $P_{\mathcal{C}_{k'}(\ell')}$ to be the associated smooth spatial frequency localization to $\calC_{k'}(\ell')$. For convenience, we choose the blocks so that $P_{k} P_{\ell}^{\omg} = P_{\calC_{k}(\ell)}$.

We will often abbreviate $ A_{k'} = P_{k'} f$ or $ \phi_k=\bar{P}_k \phi $. We will sometimes use the operators  $ \tilde{P}_k,  \tilde{Q}_{j/<j},  \tilde{P}^{\omg}_{\ell} $ with symbols given by bump functions which equal $ 1 $ on the support of the multipliers $ P_k, Q_{j/<j} $ and $ P^{\omg}_{\ell} $ respectively and which are adapted to an enlargement of these supports. 

We call multiplier disposable when it's convolution kernel is a function (or measure) with bounded mass. Minkowski's inequality insures that disposable operators are bounded on translation-invariant normed spaces. Examples include $ P_k, P_{\ell}^{\omega}, P_{\calC} $.

When $ j \geq k+2\ell-C $ the operator $ P_k P_{\ell}^{\omega} Q_{j/<j} $ is disposable \cite[Lemma 6]{Tao2}. Similar considerations apply to $ Q^{\pm}_j \bar{Q}_j, \bar{P}_k $ etc.

\subsection{Sector projections} For $ \omega \in \mb{S}^{d-1} $ and $ 0<\tht \ls 1 $ we define the sector projections $ \Pi^{\omega}_{>\tht} $ by
\be \label{sect:proj1} \widehat{\Pi^{\omega}_{>\tht} u} (\xi)=\big(1- \eta(\angle(\xi,\omega) \tht^{-1}) \big) \big(1- \eta(\angle(\xi,-\omega) \tht^{-1}) \big) \hat{u}(\xi) \ee
where $ \eta $ is a bump function on the unit scale. We define
\be \label{sect:proj2}  \Pi^{\omega}_{<\tht}=1-\Pi^{\omega}_{>\tht}, \qquad \Pi^{\omega}_{\tht}=\Pi^{\omega}_{>\tht/2}-\Pi^{\omega}_{>\tht}.
\ee

\subsection{Adapted frames} Following \cite{BH1}, for $ \lmd >0 $ and $ \omega \in \mb{S}^{d-1} $ we define the frame
\be \omega^{\lmd}=\frac{1}{\sqrt{1+\lmd^2}} (\pm \lmd,\omega), \quad \bar{\omega}^{\lmd}=\frac{1}{\sqrt{1+\lmd^2}} (\pm 1,-\lmd \omega), \quad \omega_i^{\perp} \in (0, \omega^{\perp})   \label{frame} \ee
and the coordinates in this frame
\be t_{\omega}=(t,x) \cdot \omega^{\lmd}, \quad  x^1_{\omega}=(t,x) \cdot \bar{\omega}^{\lmd}, \quad x_{\omega,i}'=x \cdot  \omega_i^{\perp}  \label{frame2} \ee
When $ \lmd=1 $ one obtains the null coordinates as in \cite{Tat}, \cite{Tao2}. 

For these frames we define the spaces $ L^{\infty}_{t_{\omega}} L^2_{x_{\omega}^1, x_{\omega}' } , L^{2}_{t_{\omega}} L^{\infty}_{x_{\omega}^1, x_{\omega}' } $ in the usual way, which we denote $ L^{\infty}_{t_{\omega,\lmd}} L^2_{x_{\omega},\lmd} , L^{2}_{t_{\omega,\lmd}} L^{\infty}_{x_{\omega,\lmd}} $ to emphasize the dependence on $ \lmd $.

\subsection{Pseudodifferential operators}
To implement the renormalization we will use pseudodifferential operators. For symbols $ a(x,\xi) : \mb{R}^d \times \mb{R}^d \to \mb{C} $ one defines the left quantization $ a(x,D) $ by
\be \label{left:quant}
a(x,D)u=\int_{\mb{R}^d} e^{i x \cdot \xi} a(x,\xi) \hat{u}(\xi) \dd \xi
\ee
while the right quantization $ a(D,y) $ is defined by 
\be \label{right:quant}
a(D,y)u=\iint_{\mb{R}^d \times \mb{R}^d} e^{i (x-y) \cdot \xi} a(y,\xi) u(y)  \dd y \dd \xi.
\ee
Observe that $ a(x,D)^*=\bar{a}(D,y) $. We will only work with symbols which are compactly supported in $ \xi $. 

\subsection{Bilinear forms}

We say that the translation-invariant bilinear form $ \calM(\phi^1,\phi^2) $ has symbol $ m(\xi_1,\xi_2) $ if 
$$ \calM(\phi^1,\phi^2) (x) = \int_{\mb{R}^d \times \mb{R}^d }  e^{i x \cdot (\xi_1+\xi_2)} m(\xi_1,\xi_2) \hat{\phi}^1(\xi_1) \hat{\phi}^2(\xi_2) \dd \xi_1 \dd \xi_2. $$
We make the analogous definition for functions defined on $ \mb{R}^{1+d} $ and symbols $ m(\Xi^1,\Xi^2) $ where $ \Xi^i=(\tau_i,\xi_i) $.

\subsection{Stationary/non-stationary phase}
We will bound oscillatory integrals using the stationary and non-stationary phase methods. For proofs of these two propositions as stated here see \cite{hormander2003introduction}.
\begin{proposition}
\label{nonstationary}
Suppose $ K \subset \mb{R}^n $ is a compact set, $ X $ is an open neighborhood of $ K $ and $ N \geq 0 $. If $ u \in C_0^N(K), f \in C^{N+1}(X) $ and $ f $ is real valued, then
\be \vm{ \int e^{i \lmd f(x)} u(x) \dd x } \leq C \frac{1}{\lmd^N}  \sup_{\vm{\al} \leq N}  \vm{D^{\al} u} \vm{f'} ^{\vm{\al}-2N} , \qquad \lmd >0 \ee 
where $ C $ is bounded when $ f $ stays in a bounded set in $ C^{N+1}(X) $.

\end{proposition}

\begin{proposition}[Stationary phase]
\label{stationary}
Suppose $ K \subset \mb{R}^n $ is a compact set, $ X $ is an open neighborhood of $ K $ and $ k \geq 1 $. If $ u \in C_0^{2k}(K), f \in C^{3k+1}(X) $ and $ f $ is real valued, $ f'(x_0)=0, \ det f''(x_0)\neq 0, \ f \neq 0 $ in $ K \setminus \{ x_0 \} $, then for $ \lmd >0 $ we have
\be \label{expansion} \vm{ \int e^{i \lmd f(x)} u(x) \dd x - e^{i \lmd f(x_0)} \lpr \frac{\det(\lmd f''(x_0))}{2 \pi i} \rpr^{-\frac{1}{2}} \sum_{j<k} \frac{1}{\lmd^j} L_j u } \leq C \frac{1}{\lmd ^k}  \sum_{\vm{\al} \leq 2k}  \sup \vm{D^{\al} u}  \ee 
where $ C $ is bounded when $ f $ stays in a bounded set in $ C^{3k+1}(X) $ and $ \vm{x-x_0}/ \vm{f'(x)} $ has a uniform bound. $ L_j $ are differential operators of order $ 2j $ acting on $ u $ at $ x_0 $.
\end{proposition}

Moreover, one controls derivatives in $ \lambda $ (see \cite[Lemma 2.35]{NaSch}):
\be \label{st-phase:est}
\vm{ \pt_{\lmd}^j  \int e^{i \lmd [f(x)-f(x_0)]} u(x) \dd x} \leq C \frac{1}{\lmd^{\frac{n}{2}+j}}, \qquad j \geq 1.
\ee

\subsection{$ L^p $ estimates} We will frequently use Bernstein's inequality, which states that 
$$ \vn{u}_{L_x^q} \ls \vm{V}^{\frac{1}{p}-\frac{1}{q}} \vn{u}_{L_x^p} $$
when $ \hat{u} $ is supported in a box of volume $ V $ and $ 1 \leq p \leq q \leq \infty $. In particular, 
$$ \vn{P_k u} _{L_x^q} \ls 2^{dk(\frac{1}{p}-\frac{1}{q})} \vn{P_k u}_{L_x^p}, \qquad  \vn{P_{\mathcal{C}_{k'}(\ell')}  u} _{L_x^q} \ls 2^{(dk'+(d-1) \ell')(\frac{1}{p}-\frac{1}{q})}  \vn{P_{\mathcal{C}_{k'}(\ell')} u}_{L_x^p}
$$

For $ L^2 $ estimates we will rely on 
\begin{lemma}[Schur's test] Let $ K: \mb{R}^n \times \mb{R}^n \to \mb{C} $ and the operator $ T $ defined by
$$ Tf(x)=\int_{\mb{R}^n} K(x,y) f(y) \dd y $$
which satisfies
$$ \sup_{x} \int \vm{K(x,y)} \dd y \leq M, \qquad  \sup_{y} \int \vm{K(x,y)} \dd x \leq M.   $$
Then
$$ \vn{T}_{L^2(\mb{R}^n) \to L^2(\mb{R}^n)} \leq M $$
\end{lemma}

\subsection{Frequency envelopes} Given $ 0<\delta_1 < 1 $, an admissible frequency envelope $ (c_k)_{k \geq 0} $ is defined to be a sequence such that $ c_{p}/c_k \leq C 2^{\delta_1 \vm{p-k}} $ for any $ k,p \geq 0 $. Given spaces $ \bar{X},\ X $ and sequences $ (c_k)_{k \geq 0} $,  $ (\tilde{c}_k)_{k \in \mb{Z}} $  we define the $ \bar{X}_c, \ X_{\tilde{c}} $ norms by
\be \label{fe:X}	 \vn{f}_{\bar{X}_c}=\sup_{k \geq 0} \frac{\vn{\bar{P}_k f}_{\bar{X}}}{c_k}, \qquad \vn{A}_{X_{\tilde{c}}}=\sup_{k \in \mb{Z}} \frac{\vn{P_k A}_{\bar{X}}}{\tilde{c}_k}.   \ee

\

\section{Function spaces}

All the spaces that we will use are translation invariant.

\subsection{The Strichartz and $ X^{b} $-type spaces.}
 
We first define the admissible Strichartz norms for the $ d+1 $ dimensional wave equation. For any $ d \geq 4 $ and any $ k $ we set
$$ S_k^{Str,W} = \bigcap_{\frac{2}{q}+\frac{d-1}{r} \leq \frac{d-1}{2} } 2^{(\frac{d}{2}-\frac{1}{q}-\frac{d}{r})k} L^q L^r  $$
with norm
\be \label{Str:wave}
\vn{f}_{S_k^{Str,W}}=\sup_{\frac{2}{q}+\frac{d-1}{r} \leq \frac{d-1}{2} } 2^{-(\frac{d}{2}-\frac{1}{q}-\frac{d}{r})k} \vn{f}_{L^q L^r} 
\ee
Next we define the $ X^{\frac{1}{2}}_{\infty}, X^{-\frac{1}{2}}_{1} $ and $ \bar{X}^{\frac{1}{2}}_{\infty}, \bar{X}^{-\frac{1}{2}}_{1}$ spaces, which are logarithmic of refinements of the usual $X^{s, b}$ space. Their dyadic norms are
\begin{align*} 
& \nrm{F}_{X^{-\frac{1}{2}}_{1}} = \sum_{j \in \bbZ} 2^{-\frac{1}{2} j} \nrm{Q_{j} F}_{L^{2}_{t,x}}, \quad
& \nrm{A}_{X^{\frac{1}{2}}_{\infty}} = \sup_{j \in \bbZ} 2^{\frac{1}{2} j} \nrm{Q_{j} A}_{L^{2}_{t,x}} \\
& \nrm{F}_{\bar{X}^{\pm \frac{1}{2}}_{1}}  = \sum_{j \in \bbZ} 2^{\pm \frac{1}{2} j} \nrm{\bar{Q}_{j} F}_{L^{2}_{t,x}}, \quad
& \nrm{\phi}_{\bar{X}^{\frac{1}{2}}_{\infty}} = \sup_{j \in \bbZ} 2^{\frac{1}{2} j} \nrm{\bar{Q}_{j} \phi}_{L^{2}_{t,x}}
\end{align*}

\subsection{The spaces for the nonlinearity}
For the nonlinearity, we define for $ k \geq 0 $ and $ k' \in \mb{Z} $
\be \bar{N}_k=L^1L^2 + \bar{X}_1^{-\frac{1}{2}}, \qquad N_{k'}=L^1L^2 + X_1^{-\frac{1}{2}} \ee
with norms
$$ \vn{F}_{\bar{N}_k}=\inf_{F=F_1+F_2} \vn{F_1}_{L^1L^2 }+\vn{F_2}_{\bar{X}_1^{-\frac{1}{2}}}, \qquad \vn{F}_{N_{k'}}=\inf_{F=F_1+F_2} \vn{F_1}_{L^1L^2 }+\vn{F_2}_{X_1^{-\frac{1}{2}}} $$
By duality we can identify $ \bar{N}_k^* $ with $ L^{\infty} L^2 \cap \bar{X}^{\frac{1}{2}}_{\infty} $.
For the scalar equation nonlinearity, respectively the $ A_i $ equation, for $ s\in \mb{R} $  we define
$$ \vn{F}_{\bar{N}^s}^2 =\sum_{k \geq 0} 2^{2sk} \vn{\bar{P}_k F }_{\bar{N}_k}^2 $$
$$ \vn{F}_{\ell^1 N^s}=\sum_{k' \in \mb{Z}}  2^{sk'} \vn{P_{k'} F}_{N_k'}   ,\qquad \vn{F}_{N^s}^2=\sum_{k' \in \mb{Z} } 2^{2sk'}  \vn{P_{k'} F}_{N_k'}^2. $$

\subsection{The iteration space for $ \phi $} 

The solution of the scalar equation will be placed in the space $ \bar{S}^{\sg} $ for $ \sg=\frac{d-2}{2} $ where, for any $ s $ we define
$$ \vn{\phi}_{\bar{S}^s}^2=\vn{\bar{P}_0 (\phi,\pt_t\phi)}_{\bar{S}_0}^2 + \sum_{k \geq 1} 2^{2(s-1)k}  \vn{\nabla_{x,t} \bar{P}_k \phi}_{\bar{S}_k}^2 + \vn{\Box_m \phi}_{L^2 H^{s-\frac{3}{2}}}^2 $$ 
where $ \bar{S}_k $ are defined below. 
% One may think of $ \bar{S}^1 $ as an intermediate space between
% $ \ell^2 \bar{X}_{1}^{1,\frac{1}{2}} \cap \Box^{-1} L^2 H^{-\frac{1}{2}} $ and $ C_t H^1 \cap \ell^2 % \bar{X}_{\infty}^{1,\frac{1}{2}} $.

When $ d=4 $, in addition to \eqref{Str:wave}, we will also use the Klein-Gordon Strichartz norms below. In general, using these K-G Strichartz norms at high frequencies does not lead to optimal estimates. Therefore, we will only rely on them for low frequencies or when there is enough additional dyadic gain coming from null structures. We set
\be  \label{Str:KG}
\begin{aligned}
& \text{For } d=4: \qquad     \bar{S}_k^{Str}=S_k^{Str,W} \cap 2^{\frac{3}{8}k} L^4 L^{\frac{8}{3}} \cap 2^{\frac{3}{4}k} L^2 L^4 \cap 2^{\frac{3}{4}k} L^2 L^{4,2}  \\
& \text{For } d \geq 5: \qquad  \bar{S}_k^{Str}=S_k^{Str,W} 
\end{aligned} \ee
Notice that we incorporate the Lorentz norms $ L^{4,2} $. See section \ref{sec:Emb} for more information.

For low frequencies $ \{ \vm{\xi} \ls 1 \} $ we define
\be \label{Sbar0}   \vn{\phi}_{\bar{S}_0}=\vn{\phi}_{\bar{S}^{Str}_0}+ \vn{\phi}_{\bar{X}_{\infty}^{\frac{1}{2}}}+ \sup_{\pm,k'<0} \vn{ \bar{Q}_{<k'}^{\pm} \phi}_{S_{box(k')}} \qquad (d \geq 4) \ee
where
$$ \vn{ \phi}_{S_{box(k')}}^2=2^{-2\sg k'}  \sum_{\calC=\calC_{k'}} \vn{P_{\calC} \phi}_{L^2 L^{\infty}}^2  $$
where $ (\calC_{k'})_{k'} $ is a finitely overlapping collection of cubes of sides $ \simeq 2^{k'} $.

For higher frequencies we define as follows. Let $ d \geq 4 $,  $ k \geq 1 $ and
\be \label{highfreqSp}   \vn{\phi}_{\bar{S}_k}^2=\vn{\phi}_{\bar{S}^{Str}_k}^2+ \vn{\phi}_{\bar{X}_{\infty}^{\frac{1}{2}}}^2+ \sup_{\pm} \sup_{l<0} \sum_{\omega}\vn{P_l^{\omega} \bar{Q}^{\pm}_{<k+2l} \phi}_{\bar{S}_k^{\omega \pm}(l)}^2 \ee
where, for $ d\geq 5 $ we define 
$$
\vn{ \phi}_{\bar{S}_k^{\omega \pm}(l)}^2 =\vn{\phi}_{S_k^{Str}}^2  + \sup_{\substack{k' \leq k;-k \leq l' \leq 0 \\ k+2l \leq k'+l' \leq k+l }} \sum_{\calC=\calC_{k'}(l')} 2^{-(d-2)k'-(d-3)l'-k} \vn{P_{\calC} \phi}_{L^2L^{\infty}}^2
$$
while for $ d=4 $ we set 
% $$ \red{ S_k^{str}= \cap_{\frac{1}{q}+\frac{3}{2r} \leq \frac{3}{4}} 2^{-(\frac{1}{q}+\frac{4}{r}-2)k} L_{t}^q L_{x}^r } % $$ 
\begin{equation*}
 \begin{aligned} 
\vn{ \phi}_{\bar{S}_k^{\omega \pm}(l)}^2 =\vn{\phi}_{S_k^{Str}}^2  &+  \sup_{\substack{k' \leq k;-k \leq l' \leq 0 \\ k+2l \leq k'+l' \leq k+l }} \sum_{\calC=\calC_{k'}(l')} \big( 2^{-2k'-k-l'} \vn{P_{\calC} \phi}_{L^2L^{\infty}}^2 +\\
% \sup_{\pm} \sup_{\red{\substack{  2l  \leq l' \leq l \\ k+c \leq l'  }}} \sum_{\omega'} \vn{P_{l'}^{\omega'} \bar{Q}^{\pm} \phi}_{NE_{\pm \omega'}^{l'}}^2 + 
&\ + 2^{-3(k'+l')} \vn{P_{\calC} \phi}_{PW_{\calC}^{\pm}}^2 +  \vn{ P_{\calC} \phi}_{NE_{C}^{\pm}}^2    \big).
\end{aligned} \end{equation*}
where, for any $ \calC=\calC_{k'}(l') $  
\be \vn{\phi}_{NE_{\calC}^{\pm}}= \sup_{\substack{\bar{\omega},\lmd=\lmd(p) \\   \angle(\bar{\omega},\pm \calC)  \gg 2^{-p},2^{-k}, 2^{l'+k'-k}}} \angle(\bar{\omega},\pm \calC)\  \vn{\phi}_{L^{\infty}_{t_{\bar{\omega},\lmd}} L^2_{x_{\bar{\omega},\lmd}}}, \quad \lmd(p) \defeq \frac{1}{\sqrt{1+2^{-2p}}}
\label{NE:norm} \ee
\be \vn{\phi}_{PW_{\calC}^{\pm}}= \inf_{\phi=\sum_i \phi^i} \sum_i \vn{\phi^i}_{L^{2}_{t_{\omega_i,\lmd}} L^{\infty}_{x_{\omega_i,\lmd}}}, \quad  \pm \omega_i \in \calC, \ \lmd=\frac{\vm{\xi_0}}{\jb{\xi_0}},\ \xi_0=\text{center}(\calC)
\label{PW:norm}  \ee
The norms $ L^{\infty}_{t_{\bar{\omega},\lmd}} L^2_{x_{\bar{\omega},\lmd}} $ and $ L^{2}_{t_{\omega_i,\lmd}} L^{\infty}_{x_{\omega_i,\lmd}} $ are taken in the frames \eqref{frame}, \eqref{frame2}.

In other words, $ PW_{\calC}^{\pm} $ is an atomic space whose atoms are functions $ \phi $ with $ \vn{\phi}_{L^{2}_{t_{\omega,\lmd}} L^{\infty}_{x_{\omega,\lmd}}} \leq 1 $ for some $ \omega \in \pm \calC $, where $ \lmd $ depends on the location of $ \calC=\calC_{k'}(l') $.

\	

The purpose of controlling the $ NE_{C}^{\pm} $ and $ PW_C^{\pm} $ norms lies in using the following type of bilinear $ L^2_{t,x} $ estimate, which was introduced in \cite{Tat} for the wave equation (see also  \cite{Tao2}). A Klein-Gordon analogue was first developed in \cite{BH1}, which served as inspiration for our implementation. 
\begin{proposition} \label{L2:nullFrames}
Let $ k, k_2 \geq 1 $, $ k'+C \leq k,k_2 $; $ l \in [-\min(k,k_2),C] $, and let $ \pm_1,\pm_2 $ be two signs.  Let $ \calC, \calC' $ be boxes of size $ 2^{k'} \times (2^{k'+l})^3 $ located in $ \{ \vm{\xi} \simeq 2^{k} \} \subset \mb{R}^4 $ , resp. $ \{ \vm{\xi} \simeq 2^{k_2} \} \subset \mb{R}^4 $ such that
\be \angle(\pm_1 \calC, \pm_2 \calC') \simeq 2^{l'} \gg \max(2^{-\min(k,k_2)}, 2^{l+k'-\min(k,k_2)})   \label{angSep} \ee
Then we have
\be \label{L2:nullFrames:est}
\vn{  \phi_k \cdot  \varphi_{k_2}  }_{L^2_{t,x}(\mb{R}^{4+1}) } \ls 2^{-l'} \vn{ \phi_k}_{NE_\calC^{\pm_1}} \vn{ \varphi_{k_2}}_{PW_{\calC'}^{\pm_2}}
\ee
\end{proposition}

\begin{proof}
The condition \eqref{angSep} insures that $ \pm_1 \calC $ and $ \pm_2 \calC' $ are angularly separated and the angle between them is well-defined. Since $ PW $ is an atomic space, we may assume the second factor is an atom with $ \vn{\varphi_{k_2}}_{L^{2}_{t_{\omega,\lmd}} L^{\infty}_{x_{\omega,\lmd}}} \leq 1 $ for some $ \omega \in \pm_2 \calC' $ and  $ \lmd $ given by \eqref{PW:norm}.
 We choose $ 2^p=\vm{\xi_0} \simeq 2^{k_2} $, so that $ \lmd=\lmd(p) $ from \eqref{NE:norm} so that together with \eqref{angSep} we have 
 $$ \vn{ \phi_k}_{L^{\infty}_{t_{\omega,\lmd}} L^2_{x_{\omega,\lmd}}} \ls 2^{-l'} \vn{P_{\calC} \bar{Q}^{\pm}_{<j} \phi_k}_{NE_\calC^{\pm_1}}. $$
 Now \eqref{L2:nullFrames:est} follows from H\" older's inequality $ L^{\infty}_{t_{\omega,\lmd}} L^2_{x_{\omega,\lmd}} \times L^{2}_{t_{\omega,\lmd}} L^{\infty}_{x_{\omega,\lmd}} \to L^2_{t,x} $.
\end{proof}

\begin{remark} When $ \Box_m \phi_k=\Box_m \varphi_{k_2}=0 $ and $ \phi_k, \varphi_{k_2} $ have Fourier support in $ \calC $, respectively $ \calC'$ then one has
\be \label{free:nf}
\vn{  \phi_k \cdot  \varphi_{k_2}  }_{L^2_{t,x}(\mb{R}^{4+1}) } \ls 2^{-l'} 2^{\frac{3}{2}(k'+l')} \vn{ \phi_k[0]}_{L^2 \times H^{-1}} \vn{ \varphi_{k_2}[0]}_{L^2 \times H^{-1}}  \ee
by convolution estimates (see eg. \cite{Foschi2000}, \cite{TaoMultilinear}). Thus \eqref{L2:nullFrames:est} is meant as a more general substitute for \eqref{free:nf}.
\end{remark}

\subsection{The iteration space for $ A $} For any $ d \geq 4$ and $ k' \in \mb{Z} $ we define
$$ \vn{A}_{S_{k'}}^2=\vn{A}_{S^{Str,W}_{k'}}^2+ \vn{A}_{X_{\infty}^{\frac{1}{2}}}^2+ \sup_{\pm} \sup_{l<0} \sum_{\omega}\vn{P_l^{\omega} Q^{\pm}_{<k'+2l} A}_{S_{k'}^{\omega}(l)}^2 $$
where
$$
\vn{A}_{S_{k'}^{\omega}(l)}^2= 2^{-(d-1)k'-(d-3)l} \vn{A}_{L^2 L^{\infty}}^2+ \sup_{\substack{k'' \in [0,k'], l' \leq 0  \\ k''+l' \leq k'+l}}   \sum_{C_{k''}(l')} 2^{-(d-2)k''-(d-3)l'-k'} \vn{P_{C_{k''}(l')} A}_{L^2L^{\infty}}^2.
$$
Now we define $  \qquad \vn{A}_{\ell^1 S^{\sg}}=\sum_{k' \in \mb{Z}}  2^{(\sg-1)k'} \big( \vn{\nabla_{t,x} P_{k'} A}_{S_{k'}}+ 2^{-\frac{1}{2}k'}  \vn{\Box P_{k'} A}_{ L^2_{t,x} } \big), $ 
$$ \qquad \qquad \vn{A}_{S^{\sg}}^2=\sum_{k' \in \mb{Z}} 2^{2(\sg-1)k'} \vn{\nabla_{t,x} P_{k'} A}_{S_{k'}}^2+\vn{\Box A}_{ L^2 \dot{H}^{\sg-\frac{3}{2}}}^2  $$ 
For the elliptic variable we set
\be
 \vn{A_0}_{Y^{\sg}} =\sum_{k' \in \mb{Z}} \vn{ \nabla_{x,t} P_{k'} A_0}_{L^{\infty} \dot{H}^{\sg-1} \cap L^2 \dot{H}^{\sg-\frac{1}{2}} }
\ee 

\subsection{The $ L^1 L^{\infty} $-type norms}

We set
\be \label{Z:norm:def}
\vn{A}_Z =\sum_{k' \in \mb{Z}} \vn{P_{k'} A}_{Z_{k'}}, \qquad \vn{A}_{Z_{k'}}^2=\sup_{\pm} \sup_{\ell<C} \sum_{\omega} 2^{\ell}  \vn{ P_{\ell}^{ \omega} Q_{k'+2\ell}^{\pm} A}_{L^1 L^{\infty}}^2 
\ee

and define $ Z^{ell}=\Box^{\frac{1}{2}} \Delta^{-\frac{1}{2}} Z $, or equivalently,
\be \label{Z-ell:norm:def}
\vn{A_0}_{Z^{ell}} =\sum_{k' \in \mb{Z}} \vn{P_{k'} A_0}_{Z^{ell}_{k'}}, \quad \vn{A_0}_{Z_{k'}^{ell}}^2 \simeq\sup_{\pm, \ell<C} \sum_{\omega} 2^{-\ell}  \vn{ P_{\ell}^{ \omega} Q_{k'+2\ell}^{\pm} A_0}_{L^1 L^{\infty}}^2 
\ee

We will use the embedding
\be \label{ZZ:emb}
\big( \Box^{-1} \times \Delta^{-1}  \big) P_{k'} : L^1 L^2 \times L^1 L^2 \to 2^{(\sg-1)k'} Z_{k'} \times Z^{ell}_{k'} 
\ee
To see this, we note that both $ (2^{2k'}/\Delta) \tilde{P}_{k'} $ and  $ (2^{2k'+2\ell}/\Box) \tilde{P}_{\ell}^{ \omega} \tilde{Q}_{k'+2\ell}^{\pm} \tilde{P}_{k'} $ are bounded. The latter one has symbol obeying the same bump function estimates as the symbol of $ P_{\ell}^{ \omega} Q_{k'+2\ell}^{\pm} P_{k'} $ on the rectangular region of size $ (2^{k'+\ell})^{d-1}\times 2^{k'+2\ell} \times 2^{k'} $ where it is supported. Then one uses Bernstein's inequality and an orthogonality argument.

Similarly one obtains,
\be \label{Zell:emb}
\vn{\Delta^{-1} P_{k'}A_0}_{Z^{ell}_{k'}} \ls \sup_{\pm, \ell<C} 2^{\ell} 2^{(\sg-1)k'} \vn{Q_{k'+2\ell}^{\pm} P_{k'}A_0}_{L^1 L^2}
\ee

\subsection{Extra derivatives}

For $ X=S, N, Y, \dot{H} $ and $ \bar{X}= \bar{S}, \bar{N}, H $, for any $ s, \rho \in \mb{R} $  we have
$$ \vn{A}_{X^{s+\rho}} \simeq \vn{ \nabla_x^{\rho} A}_{X^s}, \qquad \vn{f}_{\bar{X}^{s+\rho}} \simeq\vn{ \jb{\nabla_x}^{\rho} f}_{\bar{X}^s}
$$ 
Similar definitions are made for their dyadic pieces, for instance 
$$ \vn{\phi_k}_{\bar{S}^s_k} \simeq 2^{(s-1)k}  \vn{(\jb{D_x},\pt_t)  \phi_k}_{\bar{S}_k} .$$

\section{Embeddings} \label{sec:Emb}

\subsection{Lorentz spaces and $ \Box^{-1} $ embeddings} \

\

For functions $ f $ in the Lorentz space $ L^{p,q} $, by decomposing 
$$ f=\sum f_m, \quad \text{where} \quad  f_m(x) \defeq f(x) 1_{\{ \vm{f(x)}\in [2^m,2^{m+1}] \}} $$
 we have the following equivalent norm (see \cite{Grafakos})
\be \label{atomic:Lorentz}
\vn{f}_{L^{p,q}} \simeq \vn{ \vn{f_m}_{L^p(\mb{R}^d)}}_{\ell^q_m(\mb{Z})}.
\ee

The Lorentz spaces also enjoy a H\"older-type inequality which is due to O'Neil \cite{Oneil}. We will need the following case
\be \label{Lorentz:Holder}
\vn{\phi \psi}_{L^{2,1}} \ls \vn{\phi}_{L^{4,2}} \vn{\psi}_{L^{4,2}}
\ee
For $ M \in \mb{Z} $ and $ l \leq 0 $ let 
\be \label{T:op}
T_l^{\omega}=\sum_{k' \leq M} P_l^{\omega} Q_{k'+2l}^{\pm} P_{k'} \frac{1}{\Box} \ee

\begin{remark} We will use the $ T_l^{\omega} $ operators on $ \mb{R}^{4+1} $ to estimate parts of the potential $ A $ in $ L^1 L^{\infty} $, using the embedding \eqref{L1Linf:emb} together with Lorentz space Strichartz estimates $ L^2 L^{4,2} $ for $ \phi $ and \eqref{Lorentz:Holder}. We have been motivated by \cite{ShSt}, where $ A \approx \frac{\pt}{\Delta} (d u)^2 $, and where essentially a Sobolev-type emdedding $ \frac{1}{\vm{D_x}}: L^{d,1}_x \to L^{\infty}_x(\mb{R}^d) $   is used.

When $ l=0 $ the symbol of the operator $ T_l^{\omega} $ makes it resemble $ \Delta^{-1}_x $.

The main point here will be that it is crucial to keep the $ k' $ summation inside the norm in order to overcome logarithmic divergences in \eqref{SmallAngleSmallMod}. 
\end{remark}

\begin{proposition} \label{Box:Embedding}
On $ \mb{R}^{4+1} $ the following embeddings hold uniformly in $ l\leq 0 $ and $ M $:
\begin{align}
2^{\frac{1}{2}l} T_l^{\omega} &: L^2 L^{\frac{4}{3}} \to L^2 L^4, \label{L2L4:emb} \\
2^{\frac{1}{2}l} T_l^{\omega} &: L^1 L^{2,1} \to L^1 L^{\infty}.  \label{L1Linf:emb}
\end{align} 
\end{proposition}

\begin{proof}
\pfstep{Step~1}{\it Proof of \eqref{L2L4:emb}.} Apply an angular projection such that $ \tilde{P}_l^{\omega} P_l^{\omega}=P_l^{\omega}  $. Now \eqref{L2L4:emb} follows by composing the following embeddings
\begin{align}
& 2^{-\frac{3}{4}l} P_l^{\omega} \vm{D_x}^{-1}  :L^2 L^{\frac{4}{3}} \to L^2_{t,x} \label{sob:emb1}, \quad 2^{2l} \frac{\vm{D_x}^2}{\Box} \sum_{k' \leq M}  Q_{k'+2l}^{\pm} P_{k'}  :L^2_{t,x} \to L^2_{t,x} \\
& 2^{-\frac{3}{4}l} \tilde{P}_l^{\omega} \vm{D_x}^{-1}  :L^2_{t,x} \to L^2 L^4. \label{sob:emb2}
\end{align} 
When $ l=0 $, the first and third mappings follow from Sobolev embedding. For smaller $ l $ we make a change of variable that maps an angular cap of angle $ \simeq 2^l $ into one of angle $ \simeq 2^0 $, which reduces the bound to the case $ l=0 $.

The second mapping holds because the operator has a bounded multiplier.

\pfstep{Step~2}{\it Proof of \eqref{L1Linf:emb}.} Let $ k(t,x) $ be the kernel of $ 2^{\frac{1}{2}l} T_l^{\omega} $. It suffices to show 
\be   \label{red:delta}
2^{\frac{1}{2}l} T_l^{\omega}[ \delta_0(t) \otimes \cdot \ ] : L^{2,1}_x \to L^1 L^{\infty}, \ \text{i.e.} \
   \vn{ \int f(y) k(t, x-y) \dd y}_{L^1_t L^{\infty}_x} \ls \vn{f}_{L^{2,1}_x}         
\ee 
Indeed, assuming \eqref{red:delta}, denoting $ \phi_s(\cdot)=\phi(s,\cdot) $, we have
$$  \vn{2^{\frac{1}{2}l} T_l^{\omega} \phi}_{ L^1 L^{\infty}}  \leq \int \vn{ \int \phi(s,y) k(t-s,x-y) \dd y }_{L^1_t L^{\infty}_x} \dd s \ls \int \vn{\phi_s}_{L^{2,1}} \dd s  $$
using the time translation-invariance in \eqref{red:delta}.

To prove \eqref{red:delta}, since $ q=1 $, by \eqref{atomic:Lorentz} we may assume that $f=f_m $, i.e.  $ \vm{f(x)} \simeq 2^{m} $ for $ x \in E $ and $ f(x)=0 $ for $ x \notin E $. We normalize $ \vn{f}_{L^{2,1}} \simeq \vn{f}_{L^2_x}=1 $, which implies $ \vm{E} \simeq 2^{-2m} $. We have
\be \label{set:kernel} \vn{ \int f(x-y) k(t,y) \dd y}_{L^{\infty}_x} \ls 2^m \sup_{\vm{F} \simeq 2^{-2m}} \int_{F} \vm{k(t,y)} \dd y 
\ee

For $ x_{\omega}=x \cdot \omega, \ x'_{\omega,i}=x \cdot \omega^{\perp}_i $, we will show
\be \label{Emb:kernel} \vm{k(t,x)} \ls 2^{\frac{1}{2}l} \frac{2^{3l}}{(2^{2l} \vm{t}+\vm{x_{\omega}}+2^l \vm{x'_{\omega}})^3}.
\ee 
Assuming this, we integrate it on $ F $ and since the fraction is decreasing in $ \vm{x_{\omega}}, \vm{x'_{\omega}} $,
\begin{align*}
 \text{RHS}  \  \eqref{set:kernel} &  \ls 2^m 2^{\frac{1}{2}l} \int_{[-R,R] \times (2^{-l}[-R,R])^3} \frac{2^{3l}}{(2^{2l} \vm{t}+\vm{x_{\omega}}+2^l \vm{x'_{\omega}})^3} \dd x_{\omega} \dd x'_{\omega} \\
& \ls 2^m 2^{\frac{1}{2}l} \int_{[-R,R]^4} \frac{1}{(2^{2l} \vm{t}+\vm{(x_{\omega},x'_{\omega})})^3} \dd x_{\omega} \dd x'_{\omega} \ls 2^m 2^{\frac{1}{2}l} \frac{R^4}{(2^{2l} \vm{t})^3+R^3}
\end{align*} 
for $ R^4 \simeq 2^{3l} 2^{-2m} $. Integrating this bound in $ t $ we obtain  \eqref{red:delta}.

\pfstep{Step~3}{\it Proof of \eqref{Emb:kernel}.} Let $ k_0(t,x) $ be the kernel of $  P_{0} Q_{2l}^{\pm} P_l^{\omega} \frac{1}{\Box} $. Then 
\be k(t,x)=2^{\frac{1}{2}l} \sum_{k' \leq M} 2^{3k'} k_0 \big(2^{k'}(t,x)\big).  \label{emb:kernel:sum}
\ee
Let $ (t_{\omega}, x^1_{\omega},x'_{\omega}) $ be the coordinates in the frame \eqref{frame}, \eqref{frame2} for $ \lmd=1 $. Then $ 2^{-3l} k_0(t_{\omega},2^{-2l} x^1_{\omega}, 2^{-l} x'_{\omega}) $ is a Schwartz function, being the Fourier transform of a bump function. Thus, 
$$  \vm{k_0(t,x) } \ls \frac{2^{3l}}{\jb{ \vm{t_{\omega}}+2^{2l} \vm{x^1_{\omega}}+2^l \vm{x'_{\omega}}}^N} \ls  \frac{2^{3l}}{\jb{2^{2l} \vm{t}+ \vm{x_{\omega}}+2^l \vm{x'_{\omega}}}^N}.$$
Using this and \eqref{emb:kernel:sum}, denoting $ S=2^{2l} \vm{t}+ \vm{x_{\omega}}+2^l \vm{x'_{\omega}} $, we have
$$  \vm{k(t,x)} \ls 2^{\frac{1}{2}l} 2^{3l} \big( \sum_{2^{k'} \leq S^{-1}} 2^{3k'} + \sum_{S^{-1} < 2^{k'} } 2^{-(N-3) k'} S^{-N} \big) \ls 2^{\frac{1}{2}l} 2^{3l} S^{-3} $$
obtaining \eqref{Emb:kernel}. \end{proof}

\subsection{Further properties}

\begin{lemma}[Sobolev-type embedding]  \label{Sobolev_lemma}
Let $ p \geq q $. For any sign $ \pm $ we have
$$ \vn{\bar{Q}^{\pm}_j u}_{L^p L^2} \ls 2^{(\frac{1}{q}-\frac{1}{p})j} \vn{\bar{Q}^{\pm}_j u}_{L^q L^2} \ls 2^{(\frac{1}{q}-\frac{1}{p})j} \vn{u}_{L^q L^2}. $$ 
The same statements holds for $ Q_j^{\pm} $. 
\end{lemma}

\begin{proof} We conjugate by the operator $ U $ defined by
$$ \mathcal{F}(U u) (\tau, \xi)=\mathcal{F} u (\tau \pm \jb{\xi},\xi), $$
which acts at each $ t $ as the unitary multiplier $ e^{\mp i t \jb{D}} $. Thus we have
$$  Q^{\pm}_j u=U^{-1}  \chi(\frac{D_t}{2^j}) U u.  $$
This clearly implies the second inequality. For the first one we write
$$
 \vn{Q^{\pm}_j u}_{L^p L^2}  \lesssim \vn{ \chi(\frac{D_t}{2^j}) U u}_{L^p L^2} \lesssim 2^{(\frac{1}{q}-\frac{1}{p})j} \vn{\chi(\frac{D_t}{2^j}) U f}_{L^q L^2} \ls 2^{(\frac{1}{q}-\frac{1}{p})j} \vn{\bar{Q}^{\pm}_j u}_{L^q L^2}.
$$
The same argument works for $ Q_j^{\pm} $, conjugating by $ e^{\mp i t \vm{D}} $ instead.
\end{proof}

Next we prove the embedding $ \bar{X}_1^{\frac{1}{2}} \subset \bar{S}_k $.

\begin{proposition} \label{Xembedding} For $ k \geq 0 $ and $ \phi $ with Fourier support in $ \{ \jb{\xi} \simeq 2^k \} $ we have
$$ \vn{\phi}_{\bar{S}_k} \ls \vn{\phi}_{\bar{X}_1^{\frac{1}{2}}} $$
\end{proposition}
\begin{proof} We may assume that $ \phi $ has Fourier support in $ \{ \vm{\tau-\jb{\xi}} \simeq 2^j,\ \tau \geq 0 \}$.
The bound clearly holds for the $ \bar{X}_{\infty}^{\frac{1}{2}} $ component of $ \bar{S}_k $. For the other norms we claim $ \vn{e^{it \jb{D}} u}_{\bar{S}_k} \ls \vn{u}_{L^2_x} $. Assuming this, we write $ \tau=\rho+\jb{\xi}$ in the inversion formula
$$ \phi(t)=\int e^{i t \tau+ ix \xi} \mathcal{F} \phi (\tau, \xi) \dd \xi \dd \tau=\int_{\vm{\rho} \simeq 2^j} e^{i t \rho} e^{i t \jb{D}} \phi_{\rho} \dd \rho $$
for $ \hat{\phi_{\rho}}(\xi)=\mathcal{F} \phi (\rho+\jb{\xi}, \xi) $. Then by Minkowski and Cauchy-Schwarz inequalities 
$$ \vn{\phi}_{\bar{S}_k} \ls \int_{\vm{\rho} \simeq 2^j} \vn{e^{i t \jb{D}} \phi_{\rho} }_{\bar{S}_k}  \dd \rho \ls \int_{\vm{\rho} \simeq 2^j}  \vn{\phi_{\rho} }_{L^2_x} \dd \rho \ls  2^{\frac{j}{2}}  \vn{\phi}_{L^2_{t,x}} \simeq \vn{\phi}_{\bar{X}_1^{\frac{1}{2}}}. $$
By an orthogonality argument, for any $ l<0 $ it remains to establish
$$ e^{it \jb{D}} \bar{P}_k: L^2_x \to \bar{S}^{Str}_k, \qquad e^{it \jb{D}} \bar{P}_k P_l^{\omega}: L^2_x \to \bar{S}_k^{\omega \pm}(l) $$
The first mapping follows by taking $ \psi_{k,\pm}=0 $ in \eqref{reducedStr}. The second one follows similarly, by orthogonality and \eqref{red:L2Linf} for $ L^2 L^{\infty} $, \eqref{PWwaves} for $ PW_C^{\pm} $ and Corollary \ref{CorNE} for $ NE_C^{\pm} $. For $ k=0 $, the $ S_{box(k')} $ component follows similarly. \end{proof}

For iterating Maxwell's equation we will use the following proposition.

\begin{proposition} \label{A:solv}
For any $ A $ such that $ A[0]=0 $ one has
\be
\vn{A}_{\ell^1 S^{\sg}} \ls \vn{\Box A}_{\ell^1(N^{\sg-1} \cap L^2 \dot{H}^{\sg-\frac{3}{2}})}
\ee
For any free solution $ A^{free} $, i.e. $ \Box A^{free} $=0, one has $ \vn{A^{free}}_{S^{\sg}} \simeq \vn{A[0]}_{\dot{H}^{\sg} \times \dot{H}^{\sg-1}}$. Thus, for any $ A $,
\be
\vn{A}_{S^{\sg}} \ls \vn{A[0]}_{\dot{H}^{\sg} \times \dot{H}^{\sg-1}} + \vn{\Box A}_{N^{\sg-1} \cap L^2 \dot{H}^{\sg-\frac{3}{2}}}
\ee
In addition, for any $ A_0 $ one has
\be
\vn{A_0}_{Y^{\sg}} \ls \vn{\Delta A_0}_{\ell^1(L^{\infty} \dot{H}^{\sg-2}\cap L^2 \dot{H}^{\sg-\frac{3}{2}})}+\vn{\Delta \pt_t A_0}_{{\ell^1(L^{\infty} \dot{H}^{\sg-3}\cap L^2 \dot{H}^{\sg-\frac{5}{2}})}}
\ee
\end{proposition}

\begin{proof} The $ A_0 $ bound follows easily from the definition of $ Y^{\sg} $. The $ A $ bounds are reduced to 
$$  \vn{\nabla_{t,x}  A_{k'}}_{S_{k'}} \ls \vn{A_{k'}[0]}_{\dot{H}^1 \times L^2} + \vn{\Box A_{k'}}_{N_{k'}} $$
The $ X_{\infty}^{\frac{1}{2}} $ part follows easily from Lemma \ref{Sobolev_lemma}. Using the argument of Lemma \ref{waves} (with $ \psi=0 $), we reduce to showing
\be e^{\pm it \vm{D}} P_{k'} : L^2_x \to S^{Str,W}, \qquad e^{\pm it \vm{D}} P_{k'} P_l^{\omega}: L^2_x \to S_{k'}^{\omega }(l)
\ee 
The first mapping represents well-known Strichartz estimates. By orthogonality, the second one follows from
$$ 2^{-\frac{d-1}{2}k'-\frac{d-3}{2}l} e^{\pm it \vm{D}} P_{k'} P_l^{\omega}: L^2_x \to L^2 L^{\infty}, $$
$$ 2^{-\frac{d-2}{2} k''-\frac{1}{2}k'-\frac{d-3}{2}l'} e^{\pm it \vm{D}} P_{k'} P_{C_{k''}(l')}: L^2_x \to L^2 L^{\infty}
$$
By a $ TT^* $ argument, these are reduced to the dispersive estimate \eqref{dispestt2}, like in Cor. \ref{Cor:L2Linf} (with $ \psi=0 $ and $ \vm{D} $ instead of $ \jb{D} $, which does not affect the proof).
\end{proof}

Finally, we have
\begin{proposition} \label{Nk:orthog}
Let $ k \geq 0 $ and $ \calC_{k'}(l') $ be a finitely overlapping collection of boxes. We have
$$ \sum_{\calC_{k'}(l')} \vn{P_{\calC_{k'}(l')} F}_{\bar{N}_k}^2 \ls \vn{F}_{\bar{N}_k}^2   
$$
\end{proposition}
\begin{proof} Since $ \bar{N}_k $ is an atomic space the property reduces to the corresponding inequalities for $ L^1 L^2 $ and $ L^2_{t,x} $, which are standard inequalities.
\end{proof}

\section{The parametrix theorem} \label{Sec_parametrix} 

We define the paradifferential covariant Klein-Gordon operator 
\be
\Box_m^{p,A}=  \Box+I-2i \sum_{k \geq 0} A^j_{<k-C} \pt_j \bar{P}_k
\ee
where $ A=A^{free}=(A_1,\dots,A_d,0) $ is a real-valued 1-form defined on $ \mb{R}^{1+d} $, assumed to solve the free wave equation and to obey the Coulomb gauge condition
\be  \label{A:cond}
 \Box A=0, \qquad \pt^j A_j=0. 
\ee

By the argument in Prop. \ref{A:solv} one may show
$$ \vn{\phi}_{\bar{S}^{\sg}} \ls \vn{\phi[0]}_{H^{\sg} \times H^{\sg-1}} + \vn{\Box_m \phi}_{\bar{N}^{\sg-1}}
$$
Following \cite{KST}, the goal is to generalize this inequality, showing that $ \Box_m $ can be replaced by $  \Box_m^{p,A} $.

Consider the problem
\begin{equation} \label{problem}
\left\{ 
\begin{array}{l}
 \Box_m^{p,A} \phi=F  \\
 \phi[0]=(f,g)
\end{array} 
\right.
\end{equation}
 
\begin{theorem} \label{main:parametrix} Let $ A $ be a real 1-form obeying \eqref{A:cond} on $ \mb{R}^{d+1} $ for $ d \geq 4 $. If $ \vn{A[0]}_{\dot{H}^{\sg} \times \dot{H}^{\sg-1}} $ is sufficiently small, then for any $ F \in \bar{N}^{\sg-1} \cap L^2 H^{\sg-\frac{3}{2}} $ and $ (f,g) \in H^{\sg} \times H^{\sg-1} $, the solution of \eqref{problem} exists globally in time and it satisfies 
\be \label{en:est}
\vn{\phi}_{\bar{S}^{\sg}} \ls \vn{(f,g)}_{H^{\sg} \times H^{\sg-1}} + \vn{F}_{\bar{N}^{\sg-1} \cap L^2 H^{\sg-\frac{3}{2}}}
\ee
\end{theorem}

The proof of this theorem will reduce to its frequency localized approximate version. 

\begin{theorem} \label{corethm} 
Let $ A $ be a real 1-form obeying \eqref{A:cond} on $ \mb{R}^{d+1} $ for $ d \geq 4 $ and let $ k \geq 0 $. If $ \vn{A[0]}_{\dot{H}^{\sg} \times \dot{H}^{\sg-1}} $ is sufficiently small, then for any $ (f_k,g_k) $ with Fourier support in $ \{ \jb{\xi} \simeq 2^k \} $ and any $ F_k $ with Fourier support in $ \{ \jb{\xi} \simeq 2^k, \ \vm{\vm{\tau}- \jb{\xi} } \ll 2^k  \} $ there exists a function $ \phi_k $ with Fourier support in $ \{ \jb{\xi} \simeq 2^k, \ \vm{\vm{\tau}- \jb{\xi} } \ll 2^k  \} $ such that
\begin{align} 
& \vn{ (\jb{D_x},\pt_t) \phi_k}_{\bar{S}_k} \lesssim  \vn{(f_k,g_k)}_{H^1 \times L^2}+\vn{F_k}_{\bar{N}_k} \label{core1} =: M_k  \\
& \vn{(\Box_m - 2i A^j_{<k-c} \partial_j) \phi_k -F_k}_{\bar{N}_k} \lesssim \ep^{\frac{1}{2}} M_k  \label{core2} \\
& \vn{(\phi_k(0)-f_k, \partial_t \phi_k(0)-g_k)}_{H^1 \times L^2}   \lesssim \ep^{\frac{1}{2}} M_k. \label{core3} 
\end{align}
\end{theorem}
\

The approximate solution will be defined by $ 2 \phi_k=T^{+}+T^{-}+ S^{+}+S^{-} $ where 
\be \label{eq:renorm}
\begin{aligned}
& T^{\pm} \defeq e^{-i \psi^k_{\pm}}_{<k}(t,x,D)  \frac{e^{\pm i t \jb{D}}}{i \jb{D}} e^{i \psi^k_{\pm}}_{<k}(D,y,0) (i \jb{D} f_k \pm  g_k) \\
& S^{\pm} \defeq \pm e^{-i \psi^k_{\pm}}_{<k}(t,x,D) \frac{K^{\pm}}{i \jb{D}}  e^{i \psi^k_{\pm}}_{<k}(D,y,s) F_k,
\end{aligned}
\ee
The phase $ \psi^k_{\pm}(t,x,\xi) $ is defined in Section \ref{Constr:phase} and $ K^{\pm} F $ are the Duhamel terms 
$$ K^{\pm} F(t)=u(t)=\int_0^t e^{\pm i (t-s) \jb{D}} F(s) \dd s, \qquad (\partial_t \mp i \jb{D})u=F, \quad u(0)=0. $$

To implement this one needs estimates for the operators $ e^{-i \psi^k_{\pm}}_{<k}(t,x,D) $ and their adjoints, adapted to the function spaces used in the iteration.

\begin{theorem} \label{Renormalization:thm} For any $ k \geq 0 $, the frequency localized renormalization operators have the following properties for any $ X \in \{ \bar{N}_k,L^2_x,\bar{N}^{*}_k \} $:
\begin{align}
\label{renbd} e_{<k}^{\pm' i \psi^k_{\pm}} (t,x,D) & : X \to X \\
\label{renbdt}  2^{-k} \pt_{t,x} e_{<k}^{\pm' i \psi^k_{\pm}} (t,x,D) & : X \to  \ep X \\
\label{renbd2} e_{<k}^{-i \psi^k_{\pm}} (t,x,D) e_{<k}^{i \psi^k_{\pm}} (D,y,s)-I & : X \to \ep^{\frac{1}{2}} X 
\end{align}
as well as 
\be \label{renbd3}
 2^k \vn{e_{<k}^{-i \psi^k_{\pm}} (t,x,D) u_k}_{\bar{S}_k} \lesssim \vn{u_k}_{L^{\infty}(H^1 \times L^2)}+ \vn{\Box_m u_k}_{\bar{N}_k} 
\ee
\be \label{conj} 
\begin{aligned}
& \vn{e_{<k}^{-i \psi^k_{\pm}} (t,x,D) \Box_m u_k - \Box_{m}^{A_{<k}} e_{<k}^{-i \psi^k_{\pm}} (t,x,D) u_k}_{\bar{N}_{k}} \ls \\ 
& \qquad \qquad \qquad \qquad \qquad \qquad \qquad \ep \vn{ u_k}_{L^{\infty} H^1}+ \ep 2^k \vn{(i\pt_t\pm \jb{D})u_k}_{\bar{N}_k} 
\end{aligned} \ee
\end{theorem}

Moreover, by \eqref{renbd2} and \eqref{spec_decomp2} one obtains
\be \label{renbd4}
e^{-i \psi^k_{\pm}}_{<k}(t,x,D) \frac{1}{\jb{D}} e^{i \psi^k_{\pm}}_{<k}(D,y,s) - \frac{1}{\jb{D}} : X \to \ep^{\frac{1}{2}} 2^{-k} X
\ee

The proof of Theorem \ref{Renormalization:thm} is given in section \ref{sec:pf:thm:ren}, relying on the contents of sections \ref{Constr:phase}, \ref{sec:Osc-int}. Now we show how these mappings imply Theorems \ref{main:parametrix}, \ref{corethm}.

\begin{proof}[Proof of Theorem \ref{main:parametrix}]
\pfstep{Step~1} We first look to define an approximate solution $ \phi^a=\phi^a[f,g,F] $ satisfying, for some $ \delta \in (0,1) $:
\be \label{aprr1} \vn{\Box_m^{p,A} \phi^a-F}_{\bar{N}^{\sg-1} \cap L^2 H^{\sg-\frac{3}{2}}}+\vn{\phi^a[0]-(f,g)}_{H^{\sg} \times H^{\sg-1}} \leq \delta \big[ \vn{F}_{\bar{N}^{\sg-1} \cap L^2 H^{\sg-\frac{3}{2}}}+\vn{(f,g)}_{H^{\sg} \times H^{\sg-1}} \big]
\ee
and
\be \label{aprr2}
\vn{\phi^a}_{\bar{S}^{\sg}} \ls \vn{F}_{\bar{N}^{\sg-1} \cap L^2 H^{\sg-\frac{3}{2}}}+\vn{(f,g)}_{H^{\sg} \times H^{\sg-1}}.
\ee
We define $ \phi^a $ from its frequency-localized versions 
$$ \phi^a \defeq \sum_{k \geq 0} \phi^a_k, \qquad  \phi^a_k=\phi^1_k+\phi^2_k $$
which remain to be defined. We decompose $ \bar{P}_k F=\bar{Q}_{<k-6} \bar{P}_k F+ \bar{Q}_{>k-6}\bar{P}_k F $ and first define $ \phi^2_k $ by
$$
\calF \phi^2_k(\tau,\xi) \defeq \frac{1}{-\tau^2+\vm{\xi}^2+1} \calF(\bar{Q}_{>k-6}\bar{P}_k F)  (\tau,\xi)
$$
so that $ \Box_m \phi^2_k=\bar{Q}_{>k-6}\bar{P}_k F $. We have
$$ \vn{ (\jb{D_x},\pt_t) \phi^2_k}_{\bar{S}_k } \ls \vn{\phi^2_k}_{L^{\infty}(H^1 \times L^2)}+\vn{\bar{Q}_{>k-6}\bar{P}_k F}_{\bar{N}_k} \ls \vn{\bar{P}_k F}_{\bar{N}_k}. $$

Then we apply Theorem \ref{corethm} to $ \bar{Q}_{<k-6} \bar{P}_k F $ and $ \bar{P}_k(f,g)-\phi^2_k[0] $ which defines the function $ \phi^1_k $. We are left with estimating 
$$ \vn{A^j_{<k-C} \pt_j  \phi^2_k}_{L^1 L^2 \cap L^2 H^{-\frac{1}{2}}} \ls \vn{A^j_{<k-C}}_{L^2 L^{\infty}} \vn{\nabla  \phi^2_k}_{L^2_{t,x}\cap L^{\infty} H^{-\frac{1}{2}}  } \ls  \ep \vn{\bar{P}_k F}_{\bar{N}_k}   $$
and similarly, using also Lemma \ref{Sobolev_lemma},
\begin{align*}
 2^{-\frac{1}{2}k} \vn{\Box_m \phi^1_k}_{L^2_{t,x}} & \ls \vn{\Box_{m}^{A_{<k}} \phi^1_k-\bar{Q}_{<k-6} \bar{P}_k F}_{\bar{N}_k}+\vn{\bar{Q}_{<k-6} \bar{P}_k F}_{\bar{N}_k}+ \vn{A^j_{<k-C} \pt_j  \phi^1_k}_{L^2 H^{-\frac{1}{2}}}  \\
 & \ls \vn{\bar{P}_k F}_{\bar{N}_k} + \vn{\bar{P}_k(f,g)}_{H^1 \times L^2}
\end{align*}
The following error term, for $ k',k''=k \pm O(1) $, follows from \eqref{est:phi1:freqAx}, \eqref{est:phi6}
$$
\vn{A^j_{k'} \pt_j \bar{P}_{k''} \phi^a_k}_{\bar{N}_k \cap L^2 H^{-\frac{1}{2}}} \ls \ep \vn{\phi^a_k}_{\bar{S}^1_k}
$$

\pfstep{Step~2} Now we iterate the approximate solutions from Step 1 to construct an exact solution. We define $ \phi \defeq \lim \phi^{\leq n} $ where 
$$ \phi^{\leq n} \defeq \phi^1 + \dots + \phi^n $$
and the $ \phi^n $ are defined inductively by $ \phi^1 \defeq \phi^a[f,g,F] $ and 
$$ \phi^n \defeq \phi^a[(f,g)-\phi^{\leq n-1}[0],F-\Box_m^{p,A} \phi^{\leq n-1} ] $$
Normalizing $\vn{F}_{\bar{N}^{\sg-1} \cap L^2 H^{\sg-\frac{3}{2}}}+\vn{(f,g)}_{H^{\sg} \times H^{\sg-1}} =1 $ it follows by induction using \eqref{aprr1}, \eqref{aprr2} that  
\be \label{aprr3} \vn{\Box_m^{p,A} \phi^{\leq n}-F}_{\bar{N}^{\sg-1} \cap L^2 H^{\sg-\frac{3}{2}}}+\vn{\phi^{\leq n}[0]-(f,g)}_{H^{\sg} \times H^{\sg-1}} \leq \delta^n \ee
and
\be \label{aprr4}
\vn{\phi^{n}}_{\bar{S}^{\sg}} \ls \delta^{n-1}.
\ee
Thus $ \phi^{\leq n} $ is a Cauchy sequence in $ \bar{S}^{\sg} $ and $ \phi $ is well-defined, satisfying \eqref{en:est}. Passing to the limit in \eqref{aprr3} we see that $ \phi $ solves \eqref{problem}. 
\end{proof}

\begin{remark} \label{fe:par} The argument above also implies a frequency envelope version of \eqref{en:est}, which will be useful in proving continuous dependence on the initial data   :
\be
\vn{\phi}_{\bar{S}^{\sg}_c} \ls \vn{(f,g)}_{H^{\sg}_c \times H^{\sg-1}_c} + \vn{F}_{(\bar{N}^{\sg-1} \cap L^2 H^{\sg-\frac{3}{2}})_c}
\ee
\end{remark}

\begin{proof}[Proof of Theorem \ref{corethm}]
We define $ \phi_k $ by
$$  \phi_k=\frac{1}{2} \big( T^{+}+T^{-}+ S^{+}+S^{-} \big) $$
where $ T^{\pm}, S^{\pm} $ are defined by \eqref{eq:renorm}.

The bound \eqref{core1} follows from \eqref{renbd3} and \eqref{renbd}, where for $ \pt_t \phi_k $ we use the low modulation support of $ \phi_k $. 
We turn to \eqref{core3} and write
$$ \phi_k(0)-f_k =\frac{1}{2i} \sum_{\pm} [ e^{-i \psi^k_{\pm}}_{<k}(0,x,D) \frac{1}{\jb{D}} e^{i \psi^k_{\pm}}_{<k}(D,y,0) - \frac{1}{\jb{D}} ] (i \jb{D} f_k \pm  g_k)   $$
\begin{align*}  \partial_t \phi_k(0)-g_k = \frac{1}{2} \sum_{\pm} \biggr[ & 
[ e^{-i \psi^k_{\pm}}_{<k}(0,x,D)  e^{i \psi^k_{\pm}}_{<k}(D,y,0)-I ] (\pm i \jb{D} f_k +  g_k) \\
&+ [\pt_t e^{-i \psi^k_{\pm}}_{<k}] (0,x,D) \frac{1}{i \jb{D}} e^{i \psi^k_{\pm}}_{<k}(D,y,0)  (i \jb{D} f_k \pm  g_k)\\
& \pm [ e^{-i \psi^k_{\pm}}_{<k}(0,x,D) \frac{1}{i \jb{D}} e^{i \psi^k_{\pm}}_{<k}(D,y,0)- \frac{1}{i \jb{D}}  ] F_k(0)   \biggr]
\end{align*}

These are estimated using \eqref{renbd4}, \eqref{renbd2},  \eqref{renbd}, respectively \eqref{renbd4}, together with 
$$ \vn{F_k(0)}_{L^2_x} \ls \vn{F_k}_{L^{\infty}L^2} \ls 2^k \vn{F_k}_{\bar{N}_k} $$
which follows from Lemma \ref{Sobolev_lemma} considering the modulation assumption on $ F_k $.

Now we prove \eqref{core2}. We write 
\begin{align} 
 \Box_{m}^{A_{<k}} \phi_k -F_k= & \sum_{\pm} \big[ [ \Box_{m}^{A_{<k}} e_{<k}^{-i \psi^k_{\pm}} (t,x,D) - e_{<k}^{-i \psi^k_{\pm}} (t,x,D) \Box_m ] \phi_{\pm} \label{cj:ln1}  \\
&  \pm \frac{1}{2} e^{-i \psi^k_{\pm}}_{<k}(t,x,D) \frac{\pt_t \pm i \jb{D}}{i \jb{D}}  e^{i \psi^k_{\pm}}_{<k}(D,y,s) F_k \big] -F_k. \label{cj:ln2}
\end{align}
where
$$
\phi_{\pm} \defeq \frac{1}{2i \jb{D}} \big[ e^{\pm i t \jb{D}} e^{i \psi^k_{\pm}}_{<k}(D,y,0)  (i \jb{D} f_k \pm  g_k)   \pm K^{\pm} e^{i \psi^k_{\pm}}_{<k}(D,y,s) F_k  \big]
$$
Using \eqref{conj} we estimate
$$ \vn{\eqref{cj:ln1}}_{\bar{N}_k} \ls \sum_{\pm} \ep [ \vn{e^{i \psi^k_{\pm}}_{<k}(D,y,0)  (i \jb{D} f_k \pm  g_k) }_{L^2}+ \vn{e^{i \psi^k_{\pm}}_{<k}(D,y,s) F_k }_{\bar{N}_k} ]
$$ 
and then we use \eqref{renbd}. Now we turn to \eqref{cj:ln2} and write
\begin{align}
\eqref{cj:ln2}= \sum_{\pm} & \frac{1}{2} \biggr[  [e^{-i \psi^k_{\pm}}_{<k}(t,x,D) e^{i \psi^k_{\pm}}_{<k}(D,y,s)-I] F_k \label{cj:ln3} \\ 
& \pm i^{-1} [ e^{-i \psi^k_{\pm}}_{<k}(t,x,D) \frac{1}{ \jb{D}}  e^{i \psi^k_{\pm}}_{<k}(D,y,s)-\frac{1}{ \jb{D}}] \pt_t F_k \label{cj:ln4} \\
& \pm e^{-i \psi^k_{\pm}}_{<k}(t,x,D) \frac{1}{i \jb{D}} [\pt_t e^{i \psi^k_{\pm}}_{<k}](D,y,s) F_k \biggr]. \label{cj:ln5}
\end{align}
For \eqref{cj:ln3} we use \eqref{renbd2}, for \eqref{cj:ln4} we use  \eqref{renbd4}, and for \eqref{cj:ln5} we use \eqref{renbd}, \eqref{renbdt}, all with $ X=\bar{N}_k $.
\end{proof}

\section{Statements of the main estimates} \label{Sec_statements}

To analyze the equation for $ A $ we introduce the main terms
\be  \label{A:bil:op}
\begin{aligned} 
	\bfA_{x} ( \phi^{1},  \phi^{2}) :=& - \Box^{-1} \calP_{j} \mathfrak{I}  (\phi^1 \nabla_x \bar{\phi^2}) , \\
	\bfA_{0} ( \phi^{1}, \phi^{2})  :=& - \Delta^{-1}  \mathfrak{I}  (\phi^1 \partial_t \bar{\phi^2}).
\end{aligned}
\ee
where here $ \Box^{-1} f$  denotes the solution $\phi$ to the inhomogeneous wave equation $\Box \phi = f$ with $\phi[0] = 0$. Using the formula for $ \calP_j $ one identifies the null structure (see \eqref{cl:nf})
\be \label{ax:nf:identity}
\calP_{j}  (\phi^1 \nabla_x \phi^2)=\Delta^{-1} \nabla^i \calN_{ij} (\phi^1,\phi^2).
\ee

\begin{remark} \label{ax:skew-adj}
Note that \eqref{ax:nf:identity} shows that $ \calP_{j}  (\phi^1 \nabla_x \phi^2) $ is a skew adjoint bilinear form.
\end{remark}

\

\begin{proposition} \label{prop:ax:est} One has the following estimates:
\begin{align} 
&  \vn{\calP_{j}  (\phi^1 \nabla_x \phi^2)}_{\ell^1 N^{\sg-1}}   \ls \vn{\phi^1}_{\bar{S}^{\sg}} \vn{\phi^2}_{\bar{S}^{\sg}} \label{est:ax1}   \\
& \vn{\phi^1 \nabla_{t,x} \phi^2}_{\ell^1  ( L^2 \dot{H}^{\sg-\frac{3}{2}} \cap L^{\infty} \dot{H}^{\sg-2} ) }   \ls \vn{\phi^1}_{\bar{S}^{\sg}} \vn{\phi^2}_{\bar{S}^{\sg}} \label{est:a01}   \\
& \vn{\phi^1 \phi^2 A}_{\ell^1 ( L^1 \dot{H}^{\sg-1} \cap L^2 \dot{H}^{\sg-\frac{3}{2}} \cap L^{\infty} \dot{H}^{\sg-2} )}   \ls   \vn{\phi^1}_{\bar{S}^{\sg}} \vn{\phi^2}_{\bar{S}^{\sg}} \vn{A}_{S^{\sg} \times Y^{\sg}} \label{est:ax2} 
\end{align}
\end{proposition}

\

Moving on to the $ \phi $ nonlinearity, when $ A_x $ is divergence free, we can write $ A_j=\calP_j A $, which implies
\be  \label{phi:nf:identity}
A^i \pt_i \phi= \sum \calN_{ij} \big( \nabla_i \Delta^{-1} A_j,\phi  \big).
\ee

As discussed in the introduction, the most difficult interaction occurs when $A_{0}$ and $A_{x}$ have frequencies lower than $\phi$.
To isolate this part, we introduce the low-high paradifferential operators 
\begin{align} \label{pi:op}
	\pi[A] \phi &:= \sum_{k \geq 0} P_{< k-C} A_{\al} \, \pt^{\al} \bar{P}_{k} \phi,
\end{align}
Moreover, we define
\begin{align*} \calH^{\ast}_{k'} L(A, \phi)
&= \sum_{j < k' + C_{2}} \bar{Q}_{< j} L(P_{k'} Q_{j} A, \bar{Q}_{< j} \phi), \\
\calH^{\ast} L(A, \phi)
&=  \sum_{\substack{ k' < k - C_{2} - 10 \\  k' \in \mb{Z}, \ k,\tilde{k} \geq 0 }} \bar{P}_{\tilde{k}}  \calH^{\ast}_{k'} L(A, \phi_{k}).
\end{align*}
With these notations, we have

\begin{proposition} \label{prop:phi:est}  \

\begin{enumerate} [leftmargin=*]

\item For all $ \phi $ and $ A=(A_x,A_0) $ such that $ \pt_j A_j=0 $ one has the null form estimates:
\begin{align} 
 \vn{A_{\al} \pt^{\al} \phi- \pi[A] \phi }_{\bar{N}^{\sg-1} }  & \ls \vn{A}_{S^{\sg} \times Y^{\sg}} \vn{\phi}_{\bar{S}^{\sg}} \label{est:phi1}   \\
 \vn{(I-\calH^*) \pi[A] \phi}_{\bar{N}^{\sg-1}}  & \ls \vn{A}_{\ell^1 S^{\sg} \times Y^{\sg} } \vn{\phi}_{\bar{S}^{\sg}} \label{est:phi2} \\
 \vn{ \calH^* \pi[A] \phi}_{\bar{N}^{\sg-1}}  & \ls \vn{A}_{Z \times Z_{ell}} \vn{\phi}_{\bar{S}^{\sg}} \label{est:phi3}  
\end{align}
\item For all $ \phi $ and $ A=(A_x,A_0) $ one has
\begin{align}
\vn{A^{\al} \pt_{\al} \phi}_{L^2 H^{\sg-\frac{3}{2}}} & \ls \vn{A}_{S^{\sg} \times Y^{\sg}} \vn{\phi}_{\bar{S}^{\sg}} \label{est:phi6} 	 \\
 \vn{\pt_t A_0  \phi }_{\bar{N}^{\sg-1} \cap L^2 H^{\sg-\frac{3}{2}}}  & \ls \vn{A_0}_{Y^{\sg}}  \vn{\phi}_{\bar{S}^{\sg}} \label{est:phi4}  \\
 \vn{A^1_{\al}A^2_{\al} \phi }_{\bar{N}^{\sg-1} \cap L^2 H^{\sg-\frac{3}{2}} }  & \ls \vn{A^1}_{S^{\sg} \times Y^{\sg}} \vn{A^2}_{S^{\sg}  \times Y^{\sg}}  \vn{\phi}_{\bar{S}^{\sg}} \label{est:phi5}.
\end{align}
\end{enumerate}
\end{proposition}  

\

The following trilinear bound contains the more delicate estimates occurring in our system. It relies crucially on the cancelation discovered in \cite{MachedonSterbenz} and to handle it we will need the norms $ L^{\infty}_{t_{\omega,\lmd}} L^2_{x_{\omega,\lmd}},\ L^{2}_{t_{\omega,\lmd}} L^{\infty}_{x_{\omega,\lmd}} $, the Lorentz norms $ L^1 L^{2,1}, \ L^2 L^{4,2} $ as well as the bilinear forms from section \ref{bil:forms:sec}. The proof is in Section \ref{Trilinear:section}.

\

\begin{proposition} \label{trilinear} For $ \bfA $ and $ \pi $ defined by \eqref{A:bil:op} and \eqref{pi:op} one has:
\be
\vn{\pi[\bfA( \phi^{1},  \phi^{2}) ] \phi}_{\bar{N}^{\sg-1} }  \ls \vn{\phi^1}_{\bar{S}^{\sg}} \vn{\phi^2}_{\bar{S}^{\sg}} \vn{\phi}_{\bar{S}^{\sg}}  \label{est:trilin}
\ee
\end{proposition}

\

The null form in \eqref{est:trilin} can be seen as follows (\cite{KST}, \cite{MachedonSterbenz}). Plugging in the Hodge projection $ \calP=I- \nabla \Delta^{-1} \nabla $ and doing some computations (see the appendix of \cite{KST} for details) we may write

\be \label{Q:dec}
\bfA^{\al}(\phi^1,\phi^2) \pt_{\al} \phi=(\mathcal{Q}_1+\mathcal{Q}_2+\mathcal{Q}_3)(\phi^1,\phi^2,\phi) \ee
where
\be \label{Q:dec2}
\begin{aligned}
	\mathcal{Q}_1(\phi^1,\phi^2,\phi) :=& - \Box^{-1}  \mathfrak{I}  (\phi^1 \pt_{\al} \bar{\phi^2})\cdot \partial^{\al}\phi , \\
	\mathcal{Q}_2(\phi^1,\phi^2,\phi) :=& \Delta^{-1} \Box^{-1} \pt_t \pt_{\al} \mathfrak{I}  (\phi^1 \pt_{\al} \bar{\phi^2})\cdot \partial_{t}\phi , \\
	\mathcal{Q}_3(\phi^1,\phi^2,\phi)  :=&  \Delta^{-1} \Box^{-1} \pt_{\al} \pt^i  \mathfrak{I}  (\phi^1 \pt_{i} \bar{\phi^2})\cdot \partial^{\al}\phi .
\end{aligned}
\ee

We also define
\begin{align*}
\calH_{k'} L (\phi, \psi)
= & \sum_{j < k' + C_{2}}  P_{k'} Q_{j} L(\bar{Q}_{< j} \phi, \bar{Q}_{< j} \psi), \\
\calH L(\phi, \psi)
= & \sum_{\substack{ k' < k_2 - C_{2} - 10 \\  k' \in \mb{Z}, \ k_1,k_2 \geq 0 }} \calH_{k'} L( \bar{P}_{k_{1}} \phi, \bar{P}_{k_{2}} \psi), \\ 
\end{align*}

\

Before solving the system MKG we give an example of using the estimates above together with Theorem \ref{main:parametrix} to solve the Cauchy problem for $ \Box_m^A \phi=F $, in the particular case $ A=A^{free} $, which will be useful below.

\begin{proposition} \label{cov:Afree}
Let $ A=A^{free} $ be a real 1-form obeying $  \Box A=0, \ \pt^j A_j=0,\ A_0=0 $. If $ \vn{A[0]}_{\dot{H}^{\sg} \times \dot{H}^{\sg-1}} $ is sufficiently small, then for any $ \phi[t_0] \in H^{\sg} \times H^{\sg-1} $ and any $ F \in \bar{N}^{\sg-1} \cap L^2 H^{\sg-\frac{3}{2}} $, the solution of $ \Box_m^A \phi=F $ with data $ \phi[t_0] $ satisfies:
\be \label{en:est:free}
\vn{\phi}_{\bar{S}^{\sg}} \ls \vn{\phi[t_0]}_{H^{\sg} \times H^{\sg-1}}+ \vn{F}_{\bar{N}^{\sg-1} \cap L^2 H^{\sg-\frac{3}{2}}}
\ee
\end{proposition}

\begin{proof} We show that the mapping $ \psi \mapsto \phi $ given by $ \Box_m^{p,A} \phi=F+ \bar{\calM}(A,\psi) $ with data $ \phi[t_0] $ at $ t=t_0 $ is a contraction on $ \Bar{S}^{\sg} $, where
$$ \bar{\calM}(A,\psi)=2i ( A_{\al} \pt^{\al} \psi - \pi[A] \psi )- A^{\al} A_{\al} \psi 
$$ 
is chosen so that $ \bar{\calM}(A,\psi)=\Box_m^{p,A} \psi-\Box_m^A \psi $. Using \eqref{est:phi1}, \eqref{est:phi6}, \eqref{est:phi5}, noting that $ \vn{A}_{S^{\sg} \times Y^{\sg}} \ls \vn{A[0]}_{\dot{H}^{\sg} \times \dot{H}^{\sg-1}} \leq \ep \ll 1 $ (since $ A_0=0 $) we obtain
$$ \vn{ \bar{\calM}(A,\psi)}_{ {\bar{N}}^{\sg-1} \cap L^2 H^{\sg-\frac{3}{2}} } \ls \ep \vn{\psi}_{\bar{S}^{\sg} } $$
which together with Theorem \ref{main:parametrix} proves the existence of $ \phi $ for $ \ep $ small enough. The same estimates imply \eqref{en:est:free}.
\end{proof}

\section{Proof of the main theorem}

Assuming the estimates in sections \ref{Sec_parametrix} and \ref{Sec_statements} we prove Theorem \ref{thm:main}.

For $ J_{\al}=-\mathfrak{I}(\phi \overline{D_{\al} \phi}) $, the MKG system is written as

\be \label{MKG:CG} \tag{MKG}
\left\{
\begin{aligned}
\Box_{m}^{A} \phi & =0 \\
\Box A_i & =\mathcal{P}_i J_x \\
\Delta A_0 &=J_0
\end{aligned}
\right.
\ee

We begin with a more detailed formulation of the main part of Theorem \ref{thm:main}. After proving it we proceed to the proofs of statements (2) and (3) of Theorem \ref{thm:main}.

\

\begin{theorem} \label{thm:main-iter}
There exists a universal constant $\ep > 0$ such that 
\begin{enumerate}[leftmargin=*]
\item For any initial data $\phi[0] \in H^{\sigma} \times H^{\sigma-1} $, $A_{x}[0] \in \dot{H}^{\sg} \times \dot{H}^{\sg-1} $ for \emph{MKG}   satisfying the smallness condition \eqref{eq:main:smalldata} and \eqref{Coulomb}, there exists a unique global solution $(\phi, A_x,A_0) \in \bar{S}^{\sg} \times S^{\sg} \times Y^{\sg} $ to \emph{MKG} with this data. 

\item
For any admissible frequency envelope $ (c_k)_{k \geq 0} $ such that  $ \vn{\bar{P}_k \phi[0]}_{H^{\sigma} \times H^{\sigma-1}} \leq c_k $, we have
\begin{equation} \label{eq:main-fe}
	\vn{\bar{P}_k \phi}_{\bar{S}^{\sg}} \ls c_k ,\quad  \nrm{P_{k'} [ A_{x} - A_{x}^{free}]}_{S^{\sg}} + \nrm{P_{k'} A_{0}}_{Y^{\sg}} \ls  \begin{cases} c_{k'}^2, &  k' \geq 0 \\ 2^{\frac{k'}{2}} c_0^2, & k' \leq 0 \end{cases}.	
\end{equation}
\item(Weak Lipschitz dependence) Let $(\phi',A') \in \bar{S}^{\sg} \times S^{\sg} \times Y^{\sg} $ be another solution to \emph{MKG} with small initial data. Then, for \footnote{ $\dlt_{1} $ is the admissible frequency envelope constant.} 
$\delta \in (0, \dlt_{1})$ we have
\be \label{eq:weak-lip} \vn{\phi-\phi'}_{\bar{S}^{\sg-\delta}}+ \vn{A-A'}_{S^{\sg-\delta} \times Y^{\sg-\delta}} \ls \vn{(\phi-\phi')[0]}_{H^{\sg-\delta} \times H^{\sg-\delta-1}} + \vn{(A_x-A_x')[0]}_{\dot{H}^{\sg-\delta} \times \dot{H}^{\sg-\delta-1}}
\ee
\item (Persistence of regularity) If $\phi[0] \in H^{N} \times H^{N-1} $, $A_{x}[0] \in \dot{H}^{N} \times \dot{H}^{N-1}$ $(N \geq \sg)$, then $ (\phi,\pt_t \phi) \in C_{t}(\bbR; H^{N}\times H^{N-1})$, $\nabla_{t, x} A_{x} \in C_{t}(\bbR; \dot{H}^{N-1})$. In particular, if the data $(\phi[0], A_{x}[0])$ are smooth, then so is the solution $(\phi,A)$.
\end{enumerate}
\end{theorem}

Theorem \ref{thm:main-iter} is proved by an iteration argument as in \cite{KST}. The presence of the non-perturbative interaction with $ A^{free} $ precludes both the usual iteration procedure based on inverting $ \Box $ and the possibility of proving Lipschitz dependence in the full space $ \bar{S}^{\sg} \times S^{\sg} \times Y^{\sg} $. Instead, we will rely on Theorem \ref{main:parametrix} which provides linear estimates for $ \Box_m^{p,A^{free}} $.

\begin{remark} \label{cov:eq-currents} When $ \phi $ solves a covariant equation $ \Box_m^{A} \phi=0 $ for some real 1-form $ A$, denoting the currents $ J_{\al}=-\mathfrak{I}  (\phi \overline{D_{\al}^A \phi}) $, a simple computation shows 
$ \partial^{\al} J_{\al}=0. $
\end{remark}

\

\subsection{Existence and uniqueness}
We first prove Statement~(1) of Theorem~\ref{thm:main-iter}. 

\pfstep{Step~1} We set up a Picard iteration. For the zeroth iterate, we take  $(\phi^{0},A_j^{0},A_0^0) = (0,A_j^{free},0) $ and for any $n \geq 0$ define $ J_{\al}^n=-\mathfrak{I}(\phi^n \overline{D_{\al}^{A_n} \phi^n}) $ and, recursively, 
\begin{align}
& \Box_{m}^{A^n} \phi^{n+1}  =0  \label{cov:iterate} \\
& \Box A_j^{n+1}  =\mathcal{P}_j J_x^n  \label{ax:iterate} \\
& \Delta A_0^{n+1} =J_0^n   \label{a0:iterate}
\end{align}
with initial data $ (\phi[0], A_x[0]) $. Differentiating \eqref{a0:iterate} and using Remark \ref{cov:eq-currents}, we get
\be \label{a0t:iterate}
\Delta \pt_t A_0^{n+1}=\pt^i J_i^n.
\ee
 Note that $ A_0^1=0 $. We claim that
\be  \label{first:iterate}
 \vn{A_x^1}_{S^{\sg}}=\vn{A_x^{free}}_{S^{\sg}} \leq C_0 \vn{A_x[0]}_{\dot{H}^{\sg} \times \dot{H}^{\sg-1}} \leq C_0 \ep, \qquad \vn{\phi^1}_{\bar{S}^{\sg}} \leq C_0 \ep
\ee
where $A_{j}^{free}$ denotes the free wave development of $A_{j}[0] = (A_{j}, \rd_{t} A_{j})(0)$. 

For $n \geq 1$, denoting $ A^m=(A_x^m,A_0^m) $ we make the induction hypothesis 
\begin{equation} \label{eq:main-iter-ind}
	\vn{\phi^m-\phi^{m-1}}_{\bar{S}^{\sg}}+\vn{A^m-A^{m-1}}_{\ell^1 S^{\sg} \times  Y^{\sg}}	\leq (C_{\ast} \ep)^{m} \qquad  m=2,n.
\end{equation}
for a universal constant $C_{\ast} > 0$. By summing this up and adding \eqref{first:iterate}  we get 
\be  \label{eq:sec-iter-ind}
\vn{\phi^m}_{\bar{S}^{\sg}}+\vn{A^m_x-A^{free}_x}_{\ell^1 S^{\sg}}+ \vn{A^m_x}_{S^{\sg}}  +\vn{A_0^m}_{Y^{\sg}}	\leq 2 C_0 \ep \quad  m=1,n.\ee

These estimates imply convergence of $ (\phi^n, A^n_x, A_0^n) $ in the topology of $ \bar{S}^{\sg} \times S^{\sg} \times Y^{\sg} $ to a solution of MKG.

\pfstep{Step~2} Notice that we can decompose
\begin{equation*}
\begin{aligned}
	 & A_{0}^{n+1} = \bfA_{0}(\phi^{n}, \phi^{n})+A_0^{R,n+1},  \qquad & A_0^{R,n+1}  & \defeq - \Delta^{-1} (\vm{\phi^n}^2 A_0^n ) \\
	& A_{j}^{n+1} = A_{j}^{free} + \bfA_{j}(\phi^{n}, \phi^{n}) + A_j^{R,n+1},  \qquad & A_j^{R,n+1} & \defeq -\Box^{-1} \mathcal{P}_j (\vm{\phi^n}^2 A_x^n )
\end{aligned}
\end{equation*}
for $ \bfA=( \bfA_{0}, \bfA_{j} )  $ defined in \eqref{A:bil:op}, and set $ A^{R,n}=(A_x^{R,n},A_0^{R,n}) $. To estimate $ A^{n+1}-A^n $ we write
\be \label{eqA:dif}
\begin{aligned}
A^{n+1}-A^n &=\bfA(\phi^{n}-\phi^{n-1}, \phi^{n})+\bfA(\phi^{n-1}, \phi^{n}-\phi^{n-1})+\big( A^{R,n+1}-A^{R,n}  \big) \\
\pt_t A_{0}^{n+1}-\pt_t A_{0}^{n} &=\Delta^{-1} \nabla_x \mathfrak{I} \big( \phi^{n-1} \nabla_x \overline{\phi^{n-1}} -\phi^n \nabla_x \bar{\phi^{n}}+ i \vm{\phi^{n-1}}^2 A_x^{n-1}- i \vm{\phi^n}^2 A_x^n    \big) 
\end{aligned}
\ee
The difference $ A^{n+1}-A^n $ is estimated in $ \ell^1 S^{\sg} \times  Y^{\sg} $ using Proposition \ref{A:solv} and \eqref{est:ax1}-\eqref{est:ax2}, together with \eqref{eq:main-iter-ind}, \eqref{eq:sec-iter-ind}. With an appropriate choice of $C_{\ast}$ and $\ep$, this insures the induction hypothesis \eqref{eq:main-iter-ind} for $ A $ remains valid with $ m=n+1 $.

Moreover, using  \eqref{ZZ:emb} and \eqref{est:ax2} with  \eqref{eq:main-iter-ind}, \eqref{eq:sec-iter-ind} we obtain
\be   \label{eq:aux:z}
\vn{A^{R,n}}_{(Z \cap \ell^1 S^{\sg}) \times (Z_{ell} \cap Y^{\sg})} \ls \ep ,\quad \vn{A^{R,n}-A^{R,n-1} }_{(Z \cap \ell^1 S^{\sg}) \times (Z_{ell} \cap Y^{\sg})} \ls ( C_{\ast} \ep )^{n+1}
\ee

\pfstep{Step~3} In order to solve \eqref{cov:iterate}, we rewrite it as 
$$ \Box_m^{p,A^{free}} \phi^{n+1}=\calM(A^n,\phi^{n+1})$$
where 
\begin{align*} 
(2i)^{-1} \calM(A^n,\phi)=\  &  \big( A^n_{\al} \cdot \pt^{\al} \phi - \pi[A^n] \phi \big) +    \pi[A^{R,n}]  \phi  \\
+ &  \pi[\bfA(\phi^{n-1}, \phi^{n-1})] \phi -(2i)^{-1} \big( \pt_t A_0^n \phi+ A^{n,\al} A^n_{\al} \phi \big)
\end{align*}
We prove that the map $ \phi \mapsto \psi $ defined by $ \Box_m^{p,A^{free}} \psi=\calM(A^n,\phi) $ is a contraction on $ \bar{S}^{\sg} $. This follows from Theorem \ref{main:parametrix} together with 
\be  \label{phi:contr}
\vn{\calM(A^n,\phi) }_{\bar{N}^{\sg-1} \cap L^2 H^{\sg-\frac{3}{2}}} \ls  \ep \vn{ \phi}_{\bar{S}^{\sg}}.
\ee
which holds due to \eqref{est:phi1}-\eqref{est:phi5}, \eqref{est:trilin} since we have \eqref{eq:sec-iter-ind} and \eqref{eq:aux:z}.

Moreover, this argument also establishes \eqref{first:iterate} for $ \phi^1 $ since we are assuming $ A^{R,0}=\bfA(\phi^{-1}, \phi^{-1})=0 $.

\pfstep{Step~4} To estimate $ \phi^{n+1}-\phi^n $ using Theorem \ref{main:parametrix} in addition to applying $ \eqref{phi:contr} $ with $ \phi=\phi^{n+1}-\phi^n $  we also need
$$
\vn{\calM(A^n,\phi^n)-\calM(A^{n-1},\phi^n) }_{\bar{N}^{\sg-1} \cap L^2 H^{\sg-\frac{3}{2}}} \ls ( C_{\ast} \ep )^{n}  \vn{\phi^n}_{\bar{S}^{\sg}} 
$$  
This follows by applying \eqref{est:phi1}, \eqref{est:phi4}, \eqref{est:phi5} with $ A=A^n-A^{n-1} $, then \eqref{est:phi2}, \eqref{est:phi3} with $ A=A^{R,n}-A^{R,n-1} $, and finally \eqref{est:trilin} with $ \bfA(\phi^{n-1}, \phi^{n-1}- \phi^{n-2}) $  and $ \bfA(\phi^{n-1}- \phi^{n-2},\phi^{n-2}) $. We use these together with \eqref{eq:main-iter-ind} and \eqref{eq:aux:z}. We conclude that, with appropriate $C_{\ast}$ and $\ep$, the induction hypothesis \eqref{eq:main-iter-ind} remains valid with $ m=n+1 $ for $ \phi $ as well.

\pfstep{Step~5} To prove uniqueness, assume that $ (\phi,A) $ and $ (\phi',A') $ are two solutions with the same initial data. Then the same $ A^{free} $ is used in $ \Box_m^{p,A^{free}} $ for both $ \phi, \phi' $ and using the same estimates as above one obtains
$$ \vn{A-A'}_{\ell^1 S^{\sg} \times Y^{\sg}}+\vn{\phi-\phi'}_{\bar{S}^{\sg}} \ls \ep \big(  \vn{A-A'}_{\ell^1 S^{\sg} \times Y^{\sg}}+\vn{\phi-\phi'}_{\bar{S}^{\sg}}  \big).
$$
Choosing $ \ep $ small enough the uniqueness statement follows.

\subsection{The frequency envelope bounds \eqref{eq:main-fe}} The main observation here is that all estimates used in the proof of existence have a frequency envelope version. Using Remark \ref{fe:par} and $ \Box_m^{p,A^{free}} \phi=\calM(A,\phi)$ we have
\be \label{fee}
\vn{\phi}_{\bar{S}^{\sg}_c} \ls \vn{\phi[0]}_{(H^{\sg} \times H^{\sg-1})_c} + \vn{\calM(A,\phi)}_{(\bar{N}^{\sg-1} \cap L^2 H^{\sg-\frac{3}{2}})_c}    
\ee
By \eqref{est:phi1:freqA0}, \eqref{est:phi1:freqAx}, \eqref{est:phi2:freqA0}, \eqref{est:phi2:freqAx}, \eqref{est:phi3:freq}, \eqref{Q1:trilinear} , \eqref{Q2:bilest}, \eqref{Q3:est}, Lemma \ref{lemma:additional} and the proof of \eqref{est:phi6}-\eqref{est:phi5} we have
\be \label{nonl:fe}  \vn{\calM(A,\phi)}_{(\bar{N}^{\sg-1} \cap L^2 H^{\sg-\frac{3}{2}})_c} \ls \big( \vn{A}_{S^{\sg} \times Y^{\sg}}+\vn{A^R}_{(Z \cap \ell^1 S^{\sg}) \times (Z_{ell} \cap Y^{\sg})} + \vn{\phi}_{\bar{S}^{\sg}}^2  \big)  \vn{\phi}_{\bar{S}^{\sg}_c}
\ee
The term in the bracket is $ \ls \ep $, thus from \eqref{fee} we obtain $ \vn{\phi}_{\bar{S}^{\sg}_c} \ls \vn{\phi[0]}_{(H^{\sg} \times H^{\sg-1})_c} $ which implies $ \vn{\bar{P}_k \phi}_{\bar{S}^{\sg}} \ls c_k $.

Now we turn to $ A $. We define $ \tilde{c}_{k'}=c_{k'}^2 $ for $ k' \geq 0 $ and $  \tilde{c}_{k'}=2^{\frac{k'}{2}} c_0^2 $ for $ k' \leq 0 $. One has
$$ \vn{A_x-A^{free}_x }_{S^{\sg}_{\tilde{c}}}+ \vn{A_0}_{Y^{\sg}_{\tilde{c}}} \ls \vn{\Box A_x}_{N^{\sg-1}_{\tilde{c}} \cap L^2 \dot{H}^{\sg-\frac{3}{2}}_{\tilde{c}}}+ \vn{\Delta A_0}_{\Delta Y^{\sg}_{\tilde{c}}} \ls  \vn{\phi}_{\bar{S}^{\sg}_c}^2 \ls 1  $$
using \eqref{est:ax1:freq} and the proofs of \eqref{est:a01}, \eqref{est:ax2}. This concludes the proof of \eqref{eq:main-fe}.

\begin{remark} A consequence of \eqref{eq:main-fe} is that if we additionally assume $ (\phi[0],A_x[0]) \in H^s \times H^{s-1} \times \dot{H}^s \times \dot{H}^{s-1} $ for $ s\in (\sg,\sg+\delta_1) $ then we can deduce
\be \label{hs:fe}
\vn{\phi}_{L^{\infty} (H^s \times H^{s-1})} + \vn{A}_{L^{\infty} (\dot{H}^s \times \dot{H}^{s-1})} \ls \vn{\phi[0]}_{H^s \times H^{s-1}} + \vn{A_x[0]}_{\dot{H}^s \times \dot{H}^{s-1}}
\ee
Indeed, choosing the frequency envelope
\be \label{ck:fe}
c_k=\sum_{k_1 \geq 0} 2^{-\delta_1 \vm{k-k_1}} \vn{\bar{P}_{k_1} \phi[0]}_{H^{\sg} \times H^{\sg-1}}, \qquad \vn{c_k}_{\ell^2(\mb{Z}_{+})} \simeq  \vn{\phi[0]}_{H^{\sg} \times H^{\sg-1}}
\ee
from \eqref{eq:main-fe} we obtain
$$ \vn{\phi}_{L^{\infty} (H^s \times H^{s-1})} \ls \vn{\jb{D}^{s-\sg} \phi}_{\bar{S}^{\sg}} \ls \vn{2^{k(s-\sg)} c_k}_{\ell^2(\mb{Z}_{+})} \ls  \vn{\phi[0]}_{H^s \times H^{s-1}}  $$
and similarly with $ (A_x-A_x^{free},A_0) $; meanwhile $ \vn{A_x^{free}}_{L^{\infty} (\dot{H}^s \times \dot{H}^{s-1})} \ls \vn{A_x[0]}_{\dot{H}^s \times \dot{H}^{s-1}} $.  
\end{remark}

\subsection{Weak Lipschitz dependence \eqref{eq:weak-lip}} Let $ \delta \phi=\phi-\phi' $ and $ \delta A=A-A' $. Similarly to the equations in \eqref{eqA:dif} we write
$$ \delta A =\bfA(\delta \phi, \phi)+\bfA(\phi', \delta \phi)+\big( A^{R}-A'^{R} \big) 
$$
and similarly for $ \delta \pt_t A_0 $. Applying \eqref{est:ax1:freq} and the estimates in the proofs of \eqref{est:a01}, \eqref{est:ax2} we get
$$ \vn{\delta A}_{S^{\sg-\delta} \times Y^{\sg-\delta} } \ls  \vn{\delta A_x [0]}_{\dot{H}^{\sg-\delta} \times \dot{H}^{\sg-\delta-1}}+ \ep \vn{\delta \phi}_{\bar{S}^{\sg-\delta}} + \ep \vn{\delta A}_{S^{\sg-\delta}\times Y^{\sg-\delta} }.
$$
By Remark \ref{fe:par} we have
$$ \vn{ \delta \phi}_{\bar{S}^{\sg-\delta}} \ls \vn{\delta \phi[0]}_{H^{\sg-\delta} \times H^{\sg-\delta-1}}+  \vn{\Box_m^{p,A^{free}} \delta \phi}_{\bar{N}^{\sg-\delta-1}}.
$$
The equation for $ \delta \phi $ is 
$$ \Box_m^{p,A^{free}} \delta \phi=\calM(A, \delta \phi)+ \big( \calM(A, \phi')-\calM(A', \phi') \big) + 2i \sum_{k \geq 0} \delta A^{free}_{<k-C} \cdot \nabla_x \phi_k'
$$ 
By applying \eqref{nonl:fe} with an appropriate frequency envelope $ c $ we get
$$  \vn{\calM(A, \delta \phi)}_{\bar{N}^{\sg-\delta-1} \cap L^2 H^{\sg-\frac{3}{2}-\delta}} \ls \ep \vn{\delta \phi}_{\bar{S}^{\sg-\delta}} $$
Similarly we obtain
$$ \vn{  \calM(A, \phi')-\calM(A', \phi') }_{\bar{N}^{\sg-\delta-1} \cap L^2 H^{\sg-\frac{3}{2}-\delta}} \ls \ep \big( \vn{\delta A}_{S^{\sg-\delta}\times Y^{\sg-\delta} } + \vn{\delta \phi}_{\bar{S}^{\sg-\delta}} 	\big) $$
Using \eqref{est:phi2:freqAx} (note that the $ \calH^{\ast} $ term is $ 0 $ for $ A^{free} $) we get
$$  \vn{\sum_{k \geq 0} \delta A^{free}_{<k-C} \cdot \nabla_x \phi_k'}_{\bar{N}^{\sg-\delta-1} \cap L^2 H^{\sg-\frac{3}{2}-\delta}} \ls \vn{\delta A^{free}}_{S^{\sg-\delta}} \vn{\phi'}_{\bar{S}^{\sg}} \ls \ep  \vn{\delta A_x [0]}_{\dot{H}^{\sg-\delta} \times \dot{H}^{\sg-\delta-1}}
$$
At this point is where $ \delta>0 $ was used, to do the $ k'<k\ $ $ \ell^2 $-summation of $ \delta A^{free} $. Putting the above together we obtain
\begin{align*} 
\vn{ \delta \phi}_{\bar{S}^{\sg-\delta}}+\vn{\delta A}_{S^{\sg-\delta} \times Y^{\sg-\delta} } & \ls \vn{\delta \phi[0]}_{H^{\sg-\delta} \times H^{\sg-\delta-1}}+ \vn{\delta A_x [0]}_{\dot{H}^{\sg-\delta} \times \dot{H}^{\sg-\delta-1}} \\
& + \ep \big(  \vn{\delta \phi}_{\bar{S}^{\sg-\delta}} + \vn{\delta A}_{S^{\sg-\delta}\times Y^{\sg-\delta} } \big). 
\end{align*}
For $ \ep $ small enough we obtain \eqref{eq:weak-lip}.

\subsection{Subcritical local well-posedness} \label{subcritical:} Here we review some local wellposedness facts that will be used in the proofs below. 

Given $ s>\sg $ we introduce the shorthand $ \calH^{\sg,s}=(\dot{H}^s \times \dot{H}^{s-1}) \cap (\dot{H}^{\sg} \times \dot{H}^{\sg-1}) $.
Note that for $ s>\frac{d}{2}+1 $, $ H^{s-1} $ becomes a Banach Algebra of functions on $ \mb{R}^d $.

\begin{proposition} \label{subcritical:prop}
Let $ s>\frac{d}{2}+1 $. For any initial data $ \phi[0] \in H^s \times H^{s-1} $ and $ A_x[0] \in \calH^{\sg,s} $ there exists a unique local solution $ (\phi,A) $ to \emph{MKG} with these data in the space $ (\phi, \pt_t \phi; A,\pt_t A) \in C_t([0,T], H^s \times H^{s-1}; \calH^{\sg,s}) $ where $ T>0 $ depends continuously on $ \vn{\phi[0]}_{H^s \times H^{s-1}} $ and $ \vn{A_x[0]}_{\calH^{\sg,s}} $. The data-to-solution map in these spaces is Lipschitz continuous. Moreover, additional Sobolev regularity of the initial data is preserved by the solution.
\end{proposition}

We omit the proof, which proceeds by usual Picard iteration (based on the d'Alembertian $ \Box $) and the algebra and multiplication properties of the spaces above. Here, the massive term $ \phi $ can be treated perturbatively. 

We remark that a stronger subcritical result - almost optimal local well-posedness (i.e. initial data in $ H^{1+\ep}( \mb{R}^4) $) was proved in \cite{SSel}. 

\subsection{Persistence of regularity} Now we sketch the proof of Statement (4) of Theorem \ref{thm:main-iter}. In view of Prop. \ref{subcritical:prop} it remains to check 
\be \label{higher:reg}
\vn{\nabla^N \phi}_{\bar{S}^{\sg}}+ \vn{\nabla^N (A-A_x^{free})}_{\ell^1 S^{\sg} \times Y^{\sg}} \ls \vn{\phi[0]}_{H^{\sg+N}\times H^{\sg+N-1}} +\vn{A_x[0]}_{\dot{H}^{\sg+N} \times \dot{H}^{\sg+N-1}}.
\ee
for $ N=1,2 $ whenever the RHS is finite. For brevity we will only consider $ N=1 $; the case $ N=2 $ can be treated similarly. We will already assume that $ \nabla \phi \in \bar{S}^{\sg},\  \nabla A \in S^{\sg} \times Y^{\sg} $. This assumption can be bypassed by repeating the proof of \eqref{higher:reg} for each iterate in the proof of existence. We write
$$  \nabla (A_{x}-A_x^{free}) = \bfA_{x}( \nabla \phi, \phi) + \bfA_{x}(\phi, \nabla \phi)+\nabla A_x^{R},  \quad  A_{x}^{R} = -\Box^{-1} \mathcal{P}_x (\vm{\phi}^2 A_x ) $$
Using the product rule we distribute the derivative on the terms inside $ A_x^{R} $. We also write the similar formula for $ \nabla A_0 $. From Prop. \ref{prop:ax:est} we get
\be \label{hr1}  \vn{\nabla (A-A_x^{free})}_{\ell^1 S^{\sg} \times Y^{\sg}} \ls \ep ( \vn{\nabla \phi}_{\bar{S}^{\sg}}+ \vn{\nabla A}_{S^{\sg} \times Y^{\sg}} 	)
\ee
The equation for $ \nabla \phi $ is 
$$ \Box_m^{p,A^{free}} \nabla \phi=\nabla \calM(A, \phi) + 2i \sum_{k \geq 0} \nabla A^{free}_{<k-C} \cdot \nabla_x \phi_k
$$ 
Using the product rule on $ \nabla \calM(A, \phi) $ and Prop. \ref{prop:ax:est}, \ref{prop:phi:est}, \ref{trilinear} we obtain 
$$ \vn{ \nabla \calM(A, \phi)}_{\bar{N}^{\sg-1} \cap L^2 H^{\sg-\frac{3}{2}}} \ls \ep ( \vn{\nabla \phi}_{\bar{S}^{\sg}}+ \vn{\nabla A}_{S^{\sg} \times Y^{\sg}} 	)
$$ 
Using \eqref{est:phi2:freqAx} (note that the $ \calH^{\ast} $ term is $ 0 $ for $ A^{free} $) we get
$$  \vn{\sum_{k \geq 0} \nabla A^{free}_{<k-C} \cdot \nabla_x \phi_k}_{\bar{N}^{\sg-1} \cap L^2 H^{\sg-\frac{3}{2}}}  \ls \ep  \vn{\nabla A_x [0]}_{\dot{H}^{\sg} \times \dot{H}^{\sg-1}}
$$
We bound $ \nabla \phi $ using Theorem \ref{main:parametrix} so that together with \eqref{hr1} we have
\begin{align*} \vn{\nabla \phi}_{\bar{S}^{\sg}}+ \vn{\nabla (A-A_x^{free})}_{\ell^1 S^{\sg} \times Y^{\sg}} \ls &\vn{\nabla \phi[0]}_{H^{\sg}\times H^{\sg-1}} +\ep  \vn{\nabla A_x [0]}_{\dot{H}^{\sg} \times \dot{H}^{\sg-1}} \\
 +&  \ep ( \vn{\nabla \phi}_{\bar{S}^{\sg}}+ \vn{\nabla A}_{S^{\sg} \times Y^{\sg}} 	)
\end{align*}
Choosing $ \ep $ small enough gives \eqref{higher:reg}.

\begin{remark} An alternative approach would be to use \eqref{hs:fe} for $ s\in (\sg,\sg+\delta_1) $ together with the almost optimal local well-posedness result in \cite{SSel} and its higher dimensional analogue.
\end{remark}

\subsection{Proof of continuous dependence on data} Now we prove statement (2) of Theorem \ref{thm:main} and that every solution obtained by Theorem \ref{thm:main-iter} can be approximated by smooth solutions. 

Let $ \phi[0] \in H^{\sg} \times H^{\sg-1} , \ A_x[0] \in \dot{H}^{\sg} \times \dot{H}^{\sg-1} $ be an initial data set for MKG. For large $ m $ we denote $ \phi^{(m)}[0]=P_{\leq m} \phi[0],\ A_x^{(m)}[0]=P_{\leq m} A_x[0] $. Let $ \phi, A $ (resp. $\phi^{(m)}, A^{(m)} $) be the solutions with initial data $ \phi[0], A_x[0] $ (resp. $ \phi^{(m)}[0], A_x^{(m)}[0] $) given by Theorem \ref{thm:main-iter}.

\begin{lemma}[Approximation by smooth solutions] \label{apr:smoth:lm} Let $ (c_k) $ be a $ H^{\sg} \times H^{\sg-1} $ admissible frequency envelope for $ \phi[0] $, i.e. $ \vn{\bar{P}_k \phi[0] }_{H^{\sg} \times H^{\sg-1}} \leq c_k $. Then
$$ \vn{\phi-\phi^{(m)}}_{\bar{S}^{\sg}}+\vn{A-A^{(m)}}_{S^{\sg} \times Y^{\sg}} \ls \big( \sum_{k>m} c_k^2 \big)^{1/2}+ \vn{P_{> m} A_x[0]}_{\dot{H}^{\sg} \times \dot{H}^{\sg-1}}.
$$
\end{lemma}
\begin{proof}
Clearly, $ c $ is also a frequency envelope for $ \phi^{(m)}[0] $. Applying the bound \eqref{eq:main-fe} to $ (\phi,A) $ and $ (\phi^{(m)}, A^{(m)}) $ separately, we obtain the estimate above for $ P_{>m}(\phi-\phi^{(m)}) $ and $ P_{>m}(A-A^{(m)}) $. For the terms $ P_{\leq m}(\phi-\phi^{(m)}) $ and $ P_{\leq m}(A-A^{(m)}) $ we use the weak Lipschitz dependence bound \eqref{eq:weak-lip}:
\begin{align*}
\vn{P_{\leq m}(\phi&-\phi^{(m)})}_{\bar{S}^{\sg}}   + \vn{P_{\leq m}(A-A^{(m)})}_{S^{\sg} \times Y^{\sg}}  \ls \\ 
& \ls 2^{\delta m} \vn{P_{\leq m}(\phi-\phi^{(m)})}_{\bar{S}^{\sg-\delta}}+ 2^{\delta m} \vn{P_{\leq m}(A-A^{(m)})}_{S^{\sg-\delta} \times Y^{\sg-\delta}}   \\
& \ls 2^{\delta m} \vn{P_{>m} \phi[0]}_{H^{\sg-\delta} \times H^{\sg-\delta-1}}+2^{\delta m}  \vn{P_{>m} A_x [0]}_{\dot{H}^{\sg-\delta} \times \dot{H}^{\sg-\delta-1}}
\end{align*}
which concludes the proof of the lemma.
\end{proof}

We continue with the proof of statement (2) of Theorem \ref{thm:main}. Let $ (\phi^n[0],A_x^n[0]) $ be a sequence of initial data sets converging to $ (\phi[0],A_x[0]) $ in $ H^{\sg} \times H^{\sg-1} \times \dot{H}^{\sg} \times \dot{H}^{\sg-1} $. For large $ n $ we denote $ (\phi^n,A^n) $ the corresponding solutions given by Theorem \ref{thm:main-iter}. We also define the approximations $ (\phi^{n(m)},A^{n(m)}) $ as above.

For $ T>0 $ and $ \eps>0 $ we prove
\be \vn{\phi-\phi^n}_{C_t([0,T]; H^{\sg} \times H^{\sg-1})}+\vn{A-A^n}_{C_t([0,T]; \dot{H}^{\sg} \times \dot{H}^{\sg-1})} <\eps
\ee
for large $ n $. We apply Lemma \ref{apr:smoth:lm} with the frequency envelopes \eqref{ck:fe} and 
$$ c_k^n=\sum_{k_1 \geq 0} 2^{-\delta_1 \vm{k-k_1}} \vn{\bar{P}_{k_1} \phi^n[0]}_{H^{\sg} \times H^{\sg-1}}. $$
Since $ \vn{\phi^n[0]-\phi[0]}_{H^{\sg} \times H^{\sg-1}}+\vn{A_x^n[0]-A_x[0]}_{\dot{H}^{\sg} \times \dot{H}^{\sg-1}} \to 0 $, there exists $ m $ such that 
$$ \sum_{k >m} c_k^2 < \eps^6, \quad  \sum_{k >m} (c_k^n)^2 < \eps^6, \quad \vn{P_{> m} A_x[0]}_{\dot{H}^{\sg} \times \dot{H}^{\sg-1}}<\eps^3, \quad \vn{P_{> m} A_x^n[0]}_{\dot{H}^{\sg} \times \dot{H}^{\sg-1}} < \eps^3
$$ 
for all $ n \geq n_{\eps} $. By Lemma \ref{apr:smoth:lm} we obtain
$$ \vn{\phi-\phi^{(m)}}_{\bar{S}^{\sg}}+\vn{A-A^{(m)}}_{S^{\sg} \times Y^{\sg}} <\eps^2, \quad \vn{\phi^n-\phi^{n(m)}}_{\bar{S}^{\sg}}+\vn{A^n-A^{n(m)}}_{S^{\sg} \times Y^{\sg}} < \eps^2
$$
and it remains to prove
$$
\vn{\phi^{(m)}-\phi^{n(m)}}_{C_t([0,T]; H^{\sg} \times H^{\sg-1})}+\vn{A^{(m)}-A^{n(m)}}_{C_t([0,T]; \dot{H}^{\sg} \times \dot{H}^{\sg-1})} <\frac{1}{2} \eps
$$
Now the solutions are smooth enough and we may apply Prop. \ref{subcritical:prop}. It is a simple matter to note that for large $ n $ the $ H^s $ norms of the differences stay finite for all $ t \in [0,T] $. This concludes the proof.

\subsection{Proof of scattering} Here we discuss the proof of statement (3) of Theorem \ref{thm:main}. Without loss of generality we set $ \pm=+ $. 

Let $ (\phi, A) $ be the solutions with initial data $ (\phi[0], A_x[0]) $ given by Theorem \ref{thm:main-iter} and let $ A^{free} $ be the free wave development of $ A_x[0] $. We denote by $ S^{A^{free}}(t',t) $ the propagator from time $ t $ to $ t' $ for the covariant equation $ \Box_m^{A^{free}} \phi=0 $, given by Prop. \ref{cov:Afree}, which implies, for any $ t<t' $
$$ \vn{\phi[t']- S^{A^{free}}(t',t) \phi[t]}_{H^{\sg} \times H^{\sg-1}} \ls \vn{\Box_m^{A^{free}} \phi}_{(\bar{N}^{\sg-1} \cap L^2 H^{\sg-\frac{3}{2}})[t,\infty)}
$$ 
the last one being the time interval localized norm (see \cite[Proposition~3.3]{OT2}). Using the estimates from Prop. \ref{prop:phi:est} like in the proof of existence shows that the RHS is finite for, say $ t=0 $, and the RHS vanishes as $ t \to \infty $. By the uniform boundedness of $ S^{A^{free}}(0,t) $ on $ H^{\sg} \times H^{\sg-1} $ (Prop. \ref{cov:Afree}) and the formula $ S^{A^{free}}(t'',t)=S^{A^{free}}(t'',t') S^{A^{free}}(t',t) $ it follows that, as $ t \to \infty $
$$ \vn{S^{A^{free}}(0,t') \phi[t']- S^{A^{free}}(0,t) \phi[t]}_{H^{\sg} \times H^{\sg-1}}  \ls  \vn{\phi[t']- S^{A^{free}}(t',t) \phi[t]}_{H^{\sg} \times H^{\sg-1}} \to 0 
$$
Therefore the limit $ \lim_{t \to \infty} S^{A^{free}}(0,t) \phi[t]=:\phi^{\infty}[0] $ exists in $ H^{\sg} \times H^{\sg-1} $ and $ \phi^{\infty}[0] $ is taken as the initial data for $ \phi^{\infty} $ in Theorem \ref{thm:main}.

The proof of scattering for $ A_x $ is similar, we omit the details.

\section{Core bilinear forms} \label{bil:forms:sec}

This section is devoted to the analysis of translation-invariant bilinear forms.

\subsection{The $ \calM $ form} \label{Mform}
During the proof of the trilinear estimate, we will need to consider terms like 
$$ P_{k'} Q_j \calM( \bar{Q}_{<j} \phi^1_{k_1}, \bar{Q}_{<j} \phi^2_{k_2} ) $$
where 
\be \label{M:form} \calM(\phi^1,\phi^2) \defeq \pt_{\al}( \phi^1\cdot \pt^{\al} \phi^2) \ee
is a null-form adapted to the wave equation, while $ \phi^1_{k_1}, \phi^2_{k_2} $ are assumed to be high-frequency Klein-Gordon waves of low $ \bar{Q} $-modulation, with low frequency output.

To obtain effective bounds, we need to split 
\be \label{M:form:decom}  \calM=\calR_0^{\pm}+\calM_0-\calN_0  \ee

where, denoting $ \Xi^i=(\tau_i,\xi_i) $, the symbols of $ \calM, \calR_0^{\pm}, \calM_0, \calN_0 $ are 
\be m(\Xi^1,\Xi^2) =(\tau_1+\tau_2) \tau_2-(\xi_1+\xi_2) \cdot \xi_2, \label{m:form:symb} \ee
and, respectively,
 \begin{align}
 r_0^{\pm}(\Xi^1,\Xi^2) & \defeq \tau_1(\tau_2 \pm \jb{\xi_2})+ (\jb{\xi_1} \mp \tau_1) \jb{\xi_2}+(\tau_2^2-\jb{\xi_2}^2),\label{r0:form:symb} \\
 m_0(\Xi^1,\Xi^2) & \defeq1+\vm{\xi_1} \vm{\xi_2}-\jb{\xi_1} \jb{\xi_2},\label{m0:form:symb} \\
 n_0(\Xi^1,\Xi^2) & \defeq \vm{\xi_1} \vm{\xi_2}+\xi_1 \cdot \xi_2.  \label{n0:form:symb}
 \end{align}
 
\subsection{The $ \calM_0$ form} \label{M0form}

Let $ \calM_0 (\phi^1,\phi^2) $ be the bilinear form with symbol
$$ m_0(\xi_1,\xi_2)=1+\vm{\xi_1} \vm{\xi_2}-\jb{\xi_1} \jb{\xi_2}. $$
Notice that this multiplier is a radial function in $ \xi_1 $ and $ \xi_2 $. 

The following two statements are aimed at obtaining an exponential gain for $  \calM_0 $ in the high $ \times $ high $ \to $ low frequency interactions.

\begin{lemma} \label{lemma:m:bound}
The following bounds hold:
\begin{align*}
\vm{m_0(\xi_1,\xi_2)} \leq & \frac{\vm{\xi_1+\xi_2}^2}{\jb{\xi_1}\jb{\xi_2}} \\
\vm{\pt_{\xi_i} m_0(\xi_1,\xi_2)} \leq & \frac{\vm{\xi_1+\xi_2}}{\jb{\xi_i}} \Big( \frac{1}{\jb{\xi_1}}+\frac{1}{\jb{\xi_2}} \Big), \qquad i=1,2 \\
\vm{\pt_{\xi_i}^{\al} m_0(\xi_1,\xi_2)}  \ls & \frac{\jb{\xi_1}\jb{\xi_2}}{\jb{\xi_i}^{\vm{\al}+2}}, \qquad \qquad \quad \vm{\al} \geq 2, \ i=1,2.
\end{align*}
\end{lemma}

We return to the proof of this lemma after the following proposition which provides an exponential gain needed for estimate \eqref{Q3:est}.

\begin{proposition} \label{M0:form}
Let $ k \geq 0 $, $ k' \leq k-C $ and $ 1 \leq p,q_1,q_2 \leq \infty $ with $ p^{-1}=q_1^{-1}+q_2^{-1} $. Let $ \calC_1, \calC_2 $ be boxes of size $ \simeq (2^{k'})^d $ located to $ \calC_i \subset \{ \jb{\xi_i} \simeq 2^k \} $ so that 
$$  \calC_1 + \calC_2 \subset \{  \vm{\xi} \leq 2^{k'+2} \} $$

Then, for all functions $ \phi_1, \phi_2 $ with Fourier support in $ \calC_1, \calC_2 $ we have
\be \label{eq:m-basic}
\vn{\calM_0 (\phi^1,\phi^2)}_{L^p} \ls 2^{2(k'-k)}  \vn{\phi_1}_{L^{q_1}} \vn{\phi_2}_{L^{q_2}}.
\ee
\end{proposition}

\begin{proof}
We expand $m_0(\xi_1,\xi_2)$ as a rapidly decreasing sum of tensor products
\begin{equation} \label{eq:m-basic:dcmp}
	m_0(\xi_1,\xi_2) = \sum_{\bfj, \bfk \in \bbZ^{d}} c_{\bfj, \bfk} \, a^1_{\bfj}(\xi_1) a^2_{\bfk}(\xi_2) \quad \hbox{ for } (\xi_1,\xi_2) \in \calC_1 \times \calC_2
\end{equation}
where, denoting $ \mu=2^{2(k'-k)} $, for any $n \geq 0$, $c_{\bfj, \bfk}$ obeys 
\begin{equation}  \label{eq:m-basic:c}
	\abs{c_{\bfj, \bfk}} \aleq_{n} \mu (1+\abs{\bfj} + \abs{\bfk})^{-n},
\end{equation}
and for some universal constant $n_{0} > 0$, the  $a_{\bfj}^i$  satisfy
\begin{equation} \label{eq:m-basic:ab}
	\nrm{a_{\bfj}^i(D)}_{L^{q} \to L^{q}} \aleq (1+\abs{\bfj})^{n_{0}}, \qquad i=1,2. \end{equation}

Assuming \eqref{eq:m-basic:dcmp}--\eqref{eq:m-basic:ab}, the desired estimate \eqref{eq:m-basic} follows immediately. Indeed, \eqref{eq:m-basic:dcmp} implies that 
\begin{equation*}
	\calM_0(\phi_{1}, \phi_{2}) = \sum_{\bfj, \bfk \in \bbZ^{d}} c_{\bfj, \bfk} \cdot a_{\bfj}^1(D) \phi_{1} \cdot a_{\bfk}^2(D) \phi_{2},
\end{equation*}
so \eqref{eq:m-basic} follows by applying H\"older's inequality and \eqref{eq:m-basic:ab}, then using \eqref{eq:m-basic:c} to sum up in $\bfj, \bfk \in \bbZ^{d}$.

Let the boxes $ \tilde{\calC}_1, \tilde{\calC}_2 $ be enlargements of $ \calC_1, \calC_2 $ of size $ \simeq (2^{k'})^d $ and let $ \chi_1, \chi_2 $ be bump functions adapted to these sets which are equal to $ 1 $ on $ \calC_1 $, respectively $\calC_2 $.

Then for $ (\xi_1,\xi_2) \in \calC_1 \times \calC_2 $, we have $ m_0(\xi_1,\xi_2)=m_0(\xi_1,\xi_2) \chi_1(\xi_1) \chi_2(\xi_2) $. Performing a Fourier series expansion of $m_0(\xi_1,\xi_2) \chi_1(\xi_1) \chi_2(\xi_2) $ by viewing $ \tilde{\calC}_1\times \tilde{\calC}_2$ as a torus, we may write

\begin{equation} \label{eq:m-basic:fs}
	m_0(\xi_1,\xi_2) = \sum_{\bfj, \bfk \in \bbZ^{d}} c_{\bfj, \bfk} \, e^{2 \pi i \bfj \cdot  \xi_1'/2^{k'+c}} e^{2 \pi i \bfk \cdot  \xi_2'/2^{k'+c}} \quad \hbox{ for } (\xi_1,\xi_2) \in \calC_1 \times \calC_2.
\end{equation}
for $ \xi_i'=\xi_i-\xi_i^0 $ where $ \xi_i^0 $ is the center of $ \calC_i $. Defining
\begin{equation*}
	a_{\bfj}^i(\xi) = \chi_i(\xi_i) e^{2 \pi i \bfj \cdot  \xi_1'/2^{k'+c}},  \qquad i=1,2,
\end{equation*}
we obtain the desired decomposition \eqref{eq:m-basic:dcmp} from \eqref{eq:m-basic:fs}.

To prove \eqref{eq:m-basic:c}, we use the Fourier inversion formula
\begin{equation*}
	c_{\bfj, \bfk} =  \frac{1}{\hbox{ Vol }( \tilde{\calC}_1\times \tilde{\calC}_2)}  \int_{ \tilde{\calC}_1\times \tilde{\calC}_2} m_0(\xi_1^0+\xi_1',\xi_2^0+\xi_2') \chi_1 \chi_2 e^{-2 \pi i( \bfj \cdot  \xi_1'+ \bfk \cdot  \xi_2')/2^{k'+c}} \, \ud \xi_1' \, \ud \xi_2'.
\end{equation*}

By Lemma \ref{lemma:m:bound}, for $ (\xi_1,\xi_2) \in \calC_1 \times \calC_2 $, since $ \vm{\xi_1+\xi_2} \ls 2^{k'} $, for any $ \vm{\al} \geq 0 $ we have
$$ \vm{ (2^{k'} \pt_{\xi_i})^{\al} m_0(\xi_1,\xi_2)} \ls \mu, \qquad i=1,2 $$
Thus, integrating by parts in $ \xi_{1}' $ [resp. in $\xi_{2}' $], we obtain 
\begin{equation*}
	\abs{c_{\bfj, \bfk}} \aleq_{n} \mu (1+\abs{\bfj})^{-n}, \qquad  \abs{c_{\bfj, \bfk}} \aleq_{n} \mu (1+\abs{\bfk})^{-n}, \quad n \geq 0.
\end{equation*}
These bounds imply \eqref{eq:m-basic:c}. Next, we have
$$ \vm{(2^{k'} \pt_{\xi_i})^{\al} a_{\bfj}^i(\xi_i) }\ls (1+\abs{\bfj})^{\vm{\al}}, \qquad \vm{\al} \geq 0, \ i=1,2 $$
This implies that the convolution kernel of $ a_{\bfj}^i(D_i) $ satisfies  $\nrm{\check{a}^i_{\bfj}}_{L^{1}} \aleq (1+\abs{\bfj})^{n_{0}} $ for $ n_0=d+1 $, which gives \eqref{eq:m-basic:ab}
\end{proof}

\begin{proof}[Proof of Lemma \ref{lemma:m:bound}] The bounds follow from elementary computations. Indeed, 
$$ -m_0(\xi_1,\xi_2)=\frac{(\vm{\xi_1}-\vm{\xi_2} )^2}{1+\vm{\xi_1} \vm{\xi_2}+\jb{\xi_1} \jb{\xi_2}}\leq\frac{\vm{\xi_1+\xi_2}^2}{\jb{\xi_1}\jb{\xi_2}}.  $$
Next, wlog assume $ i=1 $. Since $ m_0 $ is radial in $ \xi_1 $ it suffices to compute
$$ \pt_{\vm{\xi_1}} m_0(\xi_1,\xi_2)=\frac{1}{\jb{\xi_1}} \big( \jb{\xi_1} \vm{\xi_2}- \jb{\xi_2} \vm{\xi_1} \big)=\frac{1}{\jb{\xi_1}} \frac{\vm{\xi_2}^2-\vm{\xi_1}^2}{\jb{\xi_1} \vm{\xi_2}+\jb{\xi_2} \vm{\xi_1}}   $$
which gives the desires bound.

Finally, the estimate for higher derivatives follows from $ \vm{ \pt_r^n \jb{r}} \ls \jb{r}^{-n-1} $ for $ n \geq 2 $, which is straightforward to prove by induction.
\end{proof}

\subsection{The $ \calN_0$ and $ \tilde{\calN}_0 $ forms}
We consider the bilinear forms $ \tilde{\calN}_0(\phi^1,\phi^2) $ on $ \mb{R}^{d+1} $ with symbol 
\be \label{nn0:eq} \tilde{n}(\Xi^1,\Xi^2)= \frac{1}{\vm{(\tau_1,\xi_1)}}\frac{1}{\vm{(\tau_2,\xi_2)}}(\tau_1 \tau_2- \xi_1 \cdot \xi_2) \ee
and $ \calN_0 (\phi^1,\phi^2) $ on $ \mb{R}^d $ with symbol
\be \label{n0:eq} n_0(\xi_1,\xi_2)=\vm{\xi_1} \vm{\xi_2}+\xi_1 \cdot \xi_2. \ee

\begin{proposition} \label{N0:form}
Let $ k_1, k_2 \in \mathbb{Z} $, $ l' \leq 0 $, and signs $ \pm_1, \pm_2 $. Let $ \kappa_1, \kappa_2 $ be spherical caps of angle $ \simeq 2^{l'} $ centered at $ \omega_1, \omega_2 $ such that $ \angle(\pm_1 \omega_1,\pm_2 \omega_2) \ls 2^{l'} $. Let $ X_1, X_2 $ be translation-invariant spaces and $ L $ be a translation-invariant bilinear operator. Suppose that
$$ \vn{L(  \phi^1,  \phi^2)}_X \ls C_{S_1,S_2} \vn{\phi^1}_{X_1} \vn{\phi^2}_{X_2} $$
holds for all $ \phi_1, \phi_2 $ which are Fourier-supported, respectively, in some subsets
$$ S_i \subset E_i \defeq \{ \vm{\xi_i} \simeq 2^{k_i}, \ \vm{\tau_i \mp_i \vm{\xi_i}} \ls 2^{k_i+2l'}, \ \frac{\xi_i}{\vm{\xi_i}} \in \kappa_i   \}   ,\qquad i=1,2. $$
Then one also has
\be \label{eq:nn0:basic} \vn{L(\pt_{\al} \phi^1, \pt^{\al} \phi^2) }_X \ls 2^{2l'} C_{S_1,S_2} \vn{\nabla_{t,x} \phi^1}_{X_1} \vn{\nabla_{t,x} \phi^2}_{X_2} \ee
for all such $ \phi_1, \phi_2 $ .
\end{proposition}

\begin{corollary}  \label{L2:NFnullFrames:cor} Under the conditions from Proposition \ref{L2:nullFrames}, for $ j \leq \min(k,k_2)+2l'-C$ one has
$$
 \vn{ \pt^{\al} P_{\calC} \bar{Q}^{\pm_1}_{<j} \phi_k \cdot \pt_{\al} P_{\calC'} \bar{Q}^{\pm_2}_{<j} \varphi_{k_2}  }_{L^2_{t,x}} \ls 2^{l'} \vn{ P_{\calC} \bar{Q}^{\pm_1}_{<j}\nabla \phi_k}_{NE_\calC^{\pm_1}}  \vn{ P_{\calC'} \bar{Q}^{\pm_2}_{<j}\nabla \varphi_{k_2}}_{PW_{\calC'}^{\pm_2}}   
$$  
\end{corollary}
\begin{remark} \label{NF:remark}
One may of course formulate analogues of Prop. \ref{N0:form} also for multilinear forms, such as the trilinear expressions $ L(\phi^1, \pt_{\al} \phi^2, \pt^{\al} \phi^3) $ that occur in the proofs of \eqref{SmallAnglesLargeMod}, \eqref{SmallAngleSmallMod}, \eqref{SmallAngleRemainders}. Checking that the same argument applies for them is straightforward and is left to the reader.
\end{remark}

\begin{proof}[Proof of Prop. \ref{N0:form}] \pfstep{Step~1} Let $ \ell(\Xi^1,\Xi^2) $ be the multiplier symbol of $ L $. In \eqref{eq:nn0:basic} we have the operator with symbol $ \ell(\Xi^1,\Xi^2) \tilde{n}(\Xi^1,\Xi^2) $ applied to $ \vm{D_{t,x}} \phi^1, \vm{D_{t,x}} \phi^2$. 

The idea is to perform a separation of variables in the form
\begin{equation}  \label{eq:nf-basic:dcmp}
	\tilde{n}(\Xi^1,\Xi^2) = \sum_{\bfj, \bfk \in \bbZ^{d}} c_{\bfj, \bfk} \, a_{\bfj}(\Xi^1) b_{\bfk}(\Xi^2) \quad \hbox{ for } (\Xi^1,\Xi^2) \in E_{1} \times E_{2}
\end{equation}
where for each $n \geq 0$ the coefficients obey 
\begin{equation}  \label{eq:nf-basic:c}
	\abs{c_{\bfj, \bfk}} \aleq_{n} 2^{2l'} (1+\abs{\bfj} + \abs{\bfk})^{-n},
\end{equation}
and for some universal constant $n_{0} > 0$, the operators $a_{\bfj}$ and $b_{\bfk}$ satisfy
\begin{equation} \label{eq:nf-basic:ab}
	\nrm{a_{\bfj}(D_{t,x})}_{X_1 \to X_1} \aleq (1+\abs{\bfj})^{n_{0}}, \quad \nrm{b_{\bfk}(D_{t,x}) }_{X_2 \to X_2} \aleq (1+\abs{\bfk})^{n_{0}},
\end{equation}
From these, \eqref{eq:nn0:basic} follows immediately.
\newline

We do a change of variables such that $ \tau_i^{\omega} $ is the (essentially null vector) radial coordinate, $  \tau_i^{\omega^{\perp}} $ is orthogonal to it, and $ \xi_i' $ are angular type coordinates in the $ \xi $ hyperplane, so that $ \vm{\xi_i'} \simeq 2^{k_i} \tht_i  $ where $ \tht_i $ are the angles between $ \xi_i $ and the center of $ \kappa_i $. We denote $ \tilde{\Xi_i}=( \tau_i^{\omega},\tau_i^{\omega^{\perp}}, \xi_i') $.

Denote by $ \tilde{E}_i $ an enlargement of $ E_i $, chosen be a rectangular region of size $  \simeq  2^{k_i} \times 2^{k_i+2l'} \times (2^{k_i+l'})^{d-1} $ (consistently with the coordinates $ ( \tau_i^{\omega},\tau_i^{\omega^{\perp}}, \xi_i') $). Let $ \chi_i $ be a bump function adapted to $ \tilde{E}_i $, which is equal to $ 1 $ on $ E_i $.

\pfstep{Step~2} We claim the following bounds for $ (\Xi^1,\Xi^2) \in E_{1} \times E_{2} $:
\begin{align} 
\vm{\tilde{n}(\Xi^1,\Xi^2)} & \ls 2^{2l'}  \label{n:formbd:1} \\
\vm{ \pt_{\xi_i'}  \tilde{n}(\Xi^1,\Xi^2)} & \ls 2^{-k_i} 2^{l'}, \qquad \quad  \  i=1,2; \label{n:formbd:2}  \\
\vm{ \pt_{\Xi_i}^{\al}  \tilde{n}(\Xi^1,\Xi^2) }& \ls \vm{\Xi^i}^{-\vm{\al}} , \qquad \ i=1,2. \label{n:formbd:3} 
\end{align}
Recall \eqref{nn0:eq}. We write 
$$ \tau_1 \tau_2- \xi_1 \cdot \xi_2=(\tau_1 \mp_1 \vm{\xi_1} )\tau_2 \pm_1 \vm{\xi_1} (\tau_2 \mp_2 \vm{\xi_2}) \pm_1 \pm_2 \vm{\xi_1} \vm{\xi_2} \big(1-\cos\angle(\pm_1 \xi_1,\pm_2 \xi_2 )  \big) $$
which clearly implies \eqref{n:formbd:1}. It is easy to see that
$$ \vm{ \pt_{\xi_i'}  \tilde{n}(\Xi^1,\Xi^2)} \ls 2^{-k_i} \sin \angle(\xi_1,\xi_2)  $$
which implies \eqref{n:formbd:2}, while \eqref{n:formbd:3} follows from the fact that $ \tilde{n} $ is homogeneous in both $ \Xi^1,\Xi^2 $.

\pfstep{Step~3} Performing a Fourier series expansion of $ \tilde{n}(\tilde{\Xi_1},\tilde{\Xi_2}) \chi_1(\tilde{\Xi_1})  \chi_2(\tilde{\Xi_2}) $ by viewing $ \tilde{E}_1 \times \tilde{E}_2 $ as a torus, we may write
\begin{equation} \label{eq:nf-basic:fs}
	\tilde{n}(\tilde{\Xi_1},\tilde{\Xi_2})  = \sum_{\bfj, \bfk \in \bbZ^{d}} c_{\bfj, \bfk} \, e^{2 \pi i \bfj \cdot D_{1} \tilde{\Xi_1}} e^{2 \pi i \bfk \cdot D_{2} \tilde{\Xi_2}} \quad \hbox{ for } (\tilde{\Xi_1},\tilde{\Xi_2}) \in E_{1} \times E_{2},
\end{equation}
where $D_{1}, D_{2}$ are diagonal matrices of the form
\be	D_{i} = \mathrm{diag} \, (O(2^{-k_i}), O(2^{-k_i-2l'}), O(2^{-k_i-l'}), \ldots, O(2^{-k_i-l'})).
\ee
	Defining
\begin{equation*}
	a_{\bfj}(\Xi_1) = ( \chi_1(\tilde{\Xi_2}) e^{2 \pi i \bfj \cdot D_{1} \tilde{\Xi_1}})(\Xi_1), \quad 
	b_{\bfk}(\Xi_2) = (\chi_2(\tilde{\Xi_2}) e^{2 \pi i \bfk \cdot D_{2} \tilde{\Xi_2}})(\Xi_2),
\end{equation*}
we obtain the desired decomposition \eqref{eq:nf-basic:dcmp} from \eqref{eq:nf-basic:fs}.

To prove \eqref{eq:nf-basic:c}, by the Fourier inversion formula
\begin{equation*}
	c_{\bfj, \bfk} =  \frac{1}{\hbox{Vol}(\tilde{E}_1 \times \tilde{E}_2)}  \int_{\tilde{E}_1 \times \tilde{E}_2} n(\tilde{\Xi_1},\tilde{\Xi_2}) \chi_1(\tilde{\Xi_1})  \chi_2(\tilde{\Xi_2})e^{- 2 \pi i \bfj \cdot D_{1} \tilde{\Xi_1}} e^{- 2 \pi i \bfk \cdot D_{2} \tilde{\Xi_2}} \, \ud \tilde{\Xi_1}  \, \ud \tilde{\Xi_2}.
\end{equation*}
Integrating by parts w.r.t. to $ \tau_i^{\omega} $ by the homogeneity of $ \tilde{n} $ and \eqref{n:formbd:1} we obtain 
\begin{equation*}
	\abs{c_{\bfj, \bfk}} \aleq_{n} 2^{2l'} (1+\abs{\bfj_{1}})^{-n} \quad [\hbox{resp. } \abs{c_{\bfj, \bfk}} \aleq_{n} 2^{2l'} (1+\abs{\bfk_{1}})^{-n}],
\end{equation*}
for any $n \geq 0$.  On the other hand, for any $j = 2, \ldots, d+1$, integration by parts in $ \tau_i^{\omega^{\perp}} $ or in $ \xi_i' $ and using \eqref{n:formbd:1}-\eqref{n:formbd:3} yields
\begin{equation*}
	\abs{c_{\bfj, \bfk}} \aleq_{n} 2^{2l'} \abs{\bfj_{j}}^{-n} \quad [\hbox{resp. } \abs{c_{\bfj, \bfk}} \aleq_{n} 2^{2l'} \abs{\bfk_{j}}^{-n}].
\end{equation*}
The preceding bounds imply \eqref{eq:nf-basic:c} as desired.

Finally, we need to establish \eqref{eq:nf-basic:ab}. We will describe the case of $a_{\bfj}(D)$. Consider the differential operators 
$$ D_{\omega_1}=( 2^{k_1} \pt_{\tau_i^{\omega}}, 2^{k_1+2l'} \pt_{\tau_i^{\omega^{\perp}}}, 2^{k_1+l'} \pt_{\xi_1'} ) $$
For any multi-index $\alp$, observe that
$$
\abs{D_{\omega_1}^{\alp} (\chi_1(\tilde{\Xi_1}) e^{2 \pi i \bfj \cdot D_{1} \tilde{\Xi_1}})} \aleq_{\alp} (1+\abs{\bfj})^{\abs{\alp}}.
$$
From this bound, it is straightforward to check that the convolution kernel of $a_{\bfj}(D)$ obeys $\nrm{\check{a}_{\bfj}}_{L^{1}} \aleq (1+\abs{\bfj})^{n_{0}}$ for some universal constant $n_{0}$, which implies the bound \eqref{eq:nf-basic:ab} for $a_{\bfj}(D)$.
\end{proof}

\begin{proof}[Proof of Corollary \ref{L2:NFnullFrames:cor}] The corollary follows from Prop. \ref{N0:form}  and Prop. \ref{L2:nullFrames}. Indeed, with $ k=k_1, k_2=k_2 $ we take $ C_{S_1,S_2}=2^{-l'} $ with
$$ S_1=\{ (\tau_1,\xi_1) \ | \ \xi_1 \in \calC, \ \vm{\xi_1} \simeq 2^k, \ \vm{\tau_1 \mp_1 \jb{\xi_1}} \ls 2^j \} $$
and $ S_2 $ defined analogously. We check that $ S_i \subset E_i $. The condition \eqref{angSep} insures that we can define $ \kappa_1,\kappa_2 $ appropriately. It remains to verify
$$ \vm{\tau_i \mp_i \vm{\xi_i}} \leq \vm{\tau_i \mp_i \jb{\xi_i}}+ \jb{\xi_i}-\vm{\xi_i} \ls 2^j+2^{-k_i} \ls 2^{k_i+2l'} $$
by the condition on $ j $ and \eqref{angSep}.
\end{proof}

If we replace $ \tau_i $ by $ \pm \vm{\xi_i} $ in \eqref{nn0:eq} we remove the time dependence in Prop. \ref{N0:form} and may formulate a spatial analogue for the bilinear form defined by $ \vm{\xi_1} \vm{\xi_2}\pm \xi_1 \cdot \xi_2 $. We consider the $ + $ case for $ \calN_0(\phi_1,\phi_2) $ in \eqref{n0:eq}, which will be useful for high $ \times $ high $ \to $ low frequency interactions.

\begin{proposition} \label{n0:form:prop}
Let $ k \in \mb{Z},\ l \leq 0 $ and $ 1 \leq p,q_1,q_2 \leq \infty $ with $ p^{-1}=q_1^{-1}+q_2^{-1} $. Let $ \kappa_1, \kappa_2 $ be spherical caps of angle $ \simeq 2^l $ such that $ \angle( \kappa_1, -\kappa_2) \ls 2^l $. 

Then, for all functions $ \phi_1, \phi_2 $ with Fourier support, respectively, in $ \{ \vm{\xi_i} \simeq 2^k, \ \xi_i/\vm{\xi_i} \in \kappa_i \}, \ i=1,2, $ we have 
$$  \vn{\calN_0 (\phi^1,\phi^2)}_{L^p} \ls 2^{2l+2k}  \vn{\phi_1}_{L^{q_1}} \vn{\phi_2}_{L^{q_2}}.
$$
\end{proposition}

\begin{proof}
The proof is very similar to the proof of Prop. \ref{N0:form} and is omitted. The basic difference is that here one performs the Fourier series expansion on a $ \big( 2^{k} \times (2^{k+l})^{d-1} \big)^2 $-sized region in $ \mb{R}_{\xi}^{d} \times \mb{R}_{\xi}^{d} $ instead of $ \mb{R}_{\tau,\xi}^{d+1} \times \mb{R}_{\tau,\xi}^{d+1} $.
\end{proof}

\subsection{The $ \calN_{ij} $ forms}

In this subsection we consider the null forms
\be \label{Nijforms}
 \calN_{ij}(\phi,\varphi)=\pt_i \phi \pt_j \varphi- \pt_j \phi \pt_i \varphi. \ee
We begin with a general result
\begin{proposition} \label{ntht:form:prop}
Let $ N $ be a bilinear form with symbol $ n(\xi,\eta) $ assumed to be homogeneous of degree $ 0 $ in $ \xi,\eta $ and to obey
$$
\vm{n(\xi,\eta)} \leq A \vm{ \angle(\xi,\eta) }.
$$
Let $ \omega_1, \omega_2 \subset \mb{S}^{d-1} $ be angular caps of radius $ \vm{r_i}\leq 2^{-10}, \ i=1,2 $ and define  $ \tht \defeq \max \{ \angle(  \vm{ \omega_1, \omega_2 )}, r_1,r_2 \} $. Let $ 1 \leq p, q_1, q_2 \leq \infty $ be such that $ p^{-1}=q_1^{-1}+q_2^{-1} $.  Let the functions $ f_1, f_2 $ be defined on $ \mb{R}^{d} $ with Fourier support in 
$$ \{  \vm{\xi} \simeq 2^{k_i}, \ \frac{\xi}{\vm{\xi}} \in \omega_i \}, \qquad i=1,2. $$
Then we have
\be \label{ntht:est}
\vn{N(f_1,f_2)}_{L^p} \ls \tht \vn{f_1}_{L^{q_1}} \vn{f_2}_{L^{q_2}}.
\ee
\end{proposition}

\begin{proof}
This is a known proposition, whose proof is similar to the one of Prop. \ref{N0:form}. For a complete proof we refer to \cite[Prop. 7.8 and appendix of Section 7]{MD}.
\end{proof}

\begin{corollary} \label{Nij:form:prop}
Under the conditions of Prop. \ref{ntht:form:prop} we have
\be \label{Nij:est}
 \vn{\calN_{ij}(f_1,f_2)}_{L^p} \ls \tht \vn{\nabla_x f_1}_{L^{q_1}} \vn{\nabla_x f_2}_{L^{q_2}}.
\ee
\end{corollary}
\begin{proof}
This follows by writing $ \calN_{ij}(f_1,f_2)=N(\vm{D}f_1,\vm{D} f_2) $ for $ N $ with symbol $ n(\xi,\eta)=\frac{\xi_i}{\vm{\xi}} \frac{\eta_j}{\vm{\eta}}- \frac{\xi_j}{\vm{\xi}} \frac{\eta_i}{\vm{\eta}} $ and applying \eqref{ntht:est}.
\end{proof}

\subsection{The geometry of frequency interactions} \

Before we state the core $ \calN_{ij} $ estimates that will be used to estimate the nonlinearity, we pause here to analyze the geometry of the frequencies of two hyperboloids interacting with a cone at low modulations. 
The methods of doing this is well-known, see \cite[sec. 13]{Tao2}, \cite[Lemma 6.5]{BH1}.

In what follows we denote by $ k_{\min}, k_{\max} $ the minimum and maximum of $ k_0,k_1,k_2 $, and similarly we consider $ j_{\min}, j_{\med}, j_{\max} $ for $ j_0, j_1, j_2 $.

\begin{lemma} \label{geom:cone}  Let $ (k_0,k_1,k_2) \in \mb{Z} \times \mb{Z}_+ \times \mb{Z}_+ $, $ j_i \in \mb{Z} \ $ for $ i=0,1,2 $ and let $ \omega_i \subset \mb{S}^{d-1} $ be angular caps of radius $ r_i \ll 1 $. Let $ \phi^1, \phi^2 $ have Fourier support, respectively, in 
$$ S_i=\{  \jb{\xi} \simeq 2^{k_i}, \ \frac{\xi}{\vm{\xi}} \in \omega_i, \ \vm{\tau-s_i  \jb{\xi}} \simeq 2^{j_i} \}, \qquad i=1,2 $$
and let $ A $ have Fourier support in 
$$ S_0=\{ \vm{\xi} \simeq 2^{k_0}, \ \frac{\xi}{\vm{\xi}} \in \omega_0, \ \vm{\tau-s_0 \vm{\xi} } \simeq 2^{j_0} \}, $$
for some signs $ s_0,s_1,s_2 $. Let $ L $ be translation-invariant and consider
\be \label{expr}
\int A \cdot L(\phi^1,\phi^2) \dd x \dd t.
\ee
\begin{enumerate} [leftmargin=*]
\item Suppose $ j_{\max} \leq k_{\min}+C_0 $. Then \eqref{expr} vanishes unless $$ j_{\max} \geq k_{\min}-2 \min(k_1,k_2)-C. $$
\item Suppose $ j_{\max} \leq k_{\min}+C_0 $ and define $ \ell \defeq \frac{1}{2}(j_{\max}-k_{\min})_{-} $. 

Then \eqref{expr} vanishes unless $ 2^{\ell} \gtrsim 2^{-\min(k_1,k_2)} $ and 
\be \angle(s_i \omega_i, s_{i'} \omega_{i'} ) \ls 2^{\ell} 2^{k_{\min}-\min(k_i,k_{i'})} + \max (r_i,r_{i'}) \label{geom:ang}  \ee
for every pair $ i,i' \in \{ 0,1,2\} $.
\item If in addition we assume $ j_{\med} \leq j_{\max}-5 $, then in \eqref{geom:ang} we have $ \simeq $ instead of $ \ls $.    \\
\item If $ j_{\med} \leq j_{\max}-5 $ then \eqref{expr} vanishes unless either $ j_{\max}=k_{\max}+O(1) $ or $ j_{\max} \leq k_{\min}+\frac{1}{2} C_0 $.
\end{enumerate}
\end{lemma}

% \begin{cases} 2^{\ell} +2^{-\min(k_1,k_2)}+r,  &  \qquad \qquad k_{\min} \in \{k_i,k_{i'}\}  \\    2^{k_{\min}-k_i} \big( 2^{\ell} +2^{-\min(k_1,k_2)}  \big) +r, &  \ \{ k_{\med}, k_{\max} \}=\{k_i,k_{i'}\} \end{cases}

\begin{proof} If  \eqref{expr} does not vanish, there exist $ (\tau^i,\xi^i) \in S_i $, ($ i=0,1,2 $) such that $ \sum_i (\tau^i,\xi^i) =0 $. Consider
$$ H \defeq s_0 \vm{\xi_0} + s_1 \jb{\xi_1}+s_2 \jb{\xi_2}. $$
Using $ \sum_i \tau^i=0 $, note that
\be  \label{H:jmax}
\vm{H} = \vm{(s_0 \vm{\xi_0}-\tau^0) + (s_1 \jb{\xi_1}-\tau^1)+(s_2 \jb{\xi_2}-\tau^2) } \ls 2^{j_{\max}}. \ee

When the signs $ s_i $ of the two highest frequencies are the same, we have $ \vm{H} \simeq 2^{k_{\max}} $. This implies $ j_{\max} \geq k_{\max}-C $ and with the assumption $ j_{\max} \leq k_{\min}+C_0 $ we deduce $ \vm{k_{\max}-k_{\min}} \leq C $ and $ \ell=O(1) $, in which case the statements are obvious. 

Now suppose the high frequencies have opposite signs. By conjugation symmetry we may assume $ s_0=+ $. By swapping $ \phi^1 $ with $ \phi^2 $ if needed, we may assume $ s_2=- $ and that $ k_2 \neq k_{\min} $. We write
\begin{align*} H & =\vm{\xi_0} + s_1 \jb{\xi_1}- \jb{\xi_2}=\frac{(\vm{\xi_0} +s_1 \jb{\xi_1})^2-(1+\vm{\xi_0+\xi_1}^2)}{\vm{\xi_0} + s_1 \jb{\xi_1}+ \jb{\xi_2}} = \\
& =\frac{2 s_1 \vm{\xi_0} \jb{\xi_1}-2 \xi_0 \cdot \xi_1}{\vm{\xi_0} + s_1 \jb{\xi_1}+ \jb{\xi_2}}=\frac{2 s_1 \vm{\xi_0} \vm{\xi_1}-2 \xi_0 \cdot \xi_1}{\vm{\xi_0} + s_1 \jb{\xi_1}+ \jb{\xi_2}}+ \frac{2 s_1 \vm{\xi_0}}{\jb{\xi_1}+\vm{\xi_1}} \frac{1}{{\vm{\xi_0} + s_1 \jb{\xi_1}+ \jb{\xi_2}}}    .
\end{align*}
where we have used $ \jb{\xi_1}-\vm{\xi_1}=(\jb{\xi_1}+\vm{\xi_1})^{-1} $.

If $ k_0=k_{\min} $ we are in the case $ (s_0,s_1,s_2)=(+,+,-) $. If $ k_0=k_{\max}+O(1) $, we are in the case $ k_1=k_{\min} $. Either way, we deduce  
$$ \vm{H} \simeq 2^{k_{\min}} \angle(\xi^0,s_1 \xi^1)^2+ 2^{k_0-k_1-k_2} .$$
This and \eqref{H:jmax} proves Statement (1) and (2) for $ (i,i')=(0,1) $. The other pairs $ (i,i') $ are reduced to this case. Indeed, denote by $ \xi^l $ and $ \xi^h $ the low and high frequencies among $ \xi_0, \xi_1 $. By the law of sines we have
$$ \sin \angle(\xi^h, -\xi_2) =\frac{\vm{\xi^l}}{\vm{\xi_2}} \sin \angle(\xi^l, \xi^h) \ls 2^{\ell} 2^{k_{\min}-k_2} $$
which implies \eqref{geom:ang} in the high-high case. The remaining low-high case now follows from the previous two cases and the triangle inequality.

Statement (3) follows by noting that in the case $ j_{\med} \leq j_{\max}-5 $ we have $ \vm{H} \simeq 2^{j_{\max}} $. Similarly, for statement (4), since either $ \vm{H} \simeq 2^{k_{\max}} $ or $ \vm{H} \ls 2^{k_{\min}} $, the statement follows by choosing $ C_0 $ large enough. 
\end{proof}

\begin{remark} \label{rk:geom:cone}
In the case $ k_{\min} \in \{k_i, k_{i'} \} $, Statement (3) can be rephrased as follows. Denoting $ 2^{\ell_0}= \angle(s_i \omega_i, s_{i'} \omega_{i'} ) $ and choosing $ r_i,r_{i'} \ll 2^{\ell_0} $, then \eqref{expr} vanishes unless
$$ j_{\max}=k_{\min}+2 \ell_0+O(1).
$$
\end{remark}

\subsection{Core $ \calN_{ij} $ and $ \calL $ bilinear estimates} We now state the main bilinear estimates for $ \calL $ and $ \calN_{ij} $ when the inputs and the output have low modulation (i.e. less than the minimum frequency).

\begin{lemma} \label{lem:ellip}
Let $ (k_{0}, k_{1}, k_{2}) \in \mb{Z}\times \mb{Z}_+\times \mb{Z}_+$ be such that $\abs{k_{\med} - k_{\max}} \leq 5$. Let $\calL$ be a translation invariant bilinear operator on $\bbR^{d}$ with bounded mass kernel. Then we have
\begin{align}
	\nrm{P_{k_{0}} \calL(\bar{P}_{k_{1}} f, \bar{P}_{k_{2}}g)}_{L^{2} L^{2}} 
	\aleq & \nrm{f_{k_{1}}}_{L^{\infty} L^{2}} \bb( \sum_{\calC_{k_{\min}}} \nrm{P_{\calC_{k_{\min}}} g_{k_{2}}}_{L^{2} L^{\infty}}^{2} \bb)^{1/2}, \label{eq:ellip-0} \\
	\nrm{P_{k_{0}} \calL(\bar{P}_{k_{1}} f, \bar{P}_{k_{2}}g)}_{L^{1} L^{2}} 
	\aleq & \nrm{f_{k_{1}}}_{L^{2} L^{2}} \bb( \sum_{\calC_{k_{\min}}} \nrm{P_{\calC_{k_{\min}}} g_{k_{2}}}_{L^{2} L^{\infty}}^{2} \bb)^{1/2}.		\label{eq:ellip-1}
\end{align}
The same statement holds when $ (k_{0}, k_{1}, k_{2}) \in \mb{Z}_+ \times \mb{Z} \times \mb{Z}_+ $ or $ (k_{0}, k_{1}, k_{2}) \in \mb{Z}_+ \times \mb{Z}_+ \times \mb{Z} $ when we replace the LHS by $ \bar{P}_{k_{0}} \calL(P_{k_{1}} f, \bar{P}_{k_{2}}g) $, respectively  $ \bar{P}_{k_{0}} \calL(\bar{P}_{k_{1}} f, P_{k_{2}}g) $.
\end{lemma}

\begin{proposition} \label{prop:no-nf}
Let $ (k_{0}, k_{1}, k_{2}) \in \bbZ_{+} \times \bbZ_{+} \times \bbZ $ be such that $\abs{k_{\max} - k_{\med}} \leq 5$ and $j \leq k_{\min} + C_{0}$. Define $\ell := \frac{1}{2} (j-k_{\min})_{-}$. Then, the following estimates hold:
\begin{equation} \label{eq:no-nf:est0}
\nrm{\bar{P}_{k_{0}} \bar{Q}_{j} ( \bar{Q}_{<j} f_{k_{1}} \cdot Q_{<j} g_{k_{2}})}_{L^{2}_{t,x}} 
\aleq \nrm{f_{k_{1}}}_{L^{\infty} L^{2}} \big( \sup_{\pm} \sum_{\calC_{k_{\min}}(\ell)}\nrm{P_{\calC_{k_{\min}}(\ell)} Q_{<j}^{\pm} g_{k_{2}}}_{L^{2} L^{\infty}}^{2} \big)^{1/2}
\end{equation}
\begin{equation} \label{eq:no-nf:est1}
\nrm{\bar{P}_{k_{0}} \bar{Q}_{<j} (\bar{Q}_{j} f_{k_{1}}, Q_{<j} g_{k_{2}})}_{L^{1} L^{2}} 
\aleq \nrm{\bar{Q}_{j} f_{k_{1}}}_{L^{2}_{t,x}} \bb( \sup_{\pm} \sum_{\calC_{k_{\min}}(\ell)}\nrm{P_{\calC_{k_{\min}}(\ell)} Q_{<j}^{\pm} g_{k_{2}}}_{L^{2} L^{\infty}}^{2} \bb)^{1/2}
\end{equation}
The same statement holds when we replace $ (\bar{Q}_j,\bar{Q}_{<j},Q_{<j}) $ by $ (Q_j,\bar{Q}_{<j},\bar{Q}_{<j}) $ and all the similar variations. 
\end{proposition}

\begin{proposition}  \label{prop:nf}
Let $k_{0} \in \bbZ, \ k_{1}, k_{2} \geq 0, \  j \in \bbZ $ be such that $\abs{k_{\max} - k_{\med}} \leq 5$ and $j \leq k_{\min} + C_{0}$.  Define $\ell := \frac{1}{2} (j-k_{\min})_{-}$ and let $ \calN $ be any of the null forms $ \calN_{m,r} $ from \eqref{Nijforms}. Then, the following estimates hold:
\begin{equation} \label{eq:nf:est0}
\begin{aligned} 
& \hskip-2em
\nrm{P_{k_0} Q_{j} \calN (\bar{Q}_{<j} f_{k_{1}}, \bar{Q}_{<j} g_{k_{2}})}_{L^{2} L^{2}} \\
\aleq &  2^{\ell} 2^{k_{\min} +k_{\max}} \nrm{f_{k_{1}}}_{L^{\infty} L^{2}} \bb( \sup_{\pm}\sum_{\calC_{k_{\min}}(\ell)}\nrm{P_{\calC_{k_{\min}}(\ell)} \bar{Q}_{<j}^{\pm} g_{k_{2}}}_{L^{2} L^{\infty}}^{2} \bb)^{1/2}
\end{aligned}
\end{equation}
\begin{equation} \label{eq:nf:est1}
\begin{aligned} 
& \hskip-2em
\nrm{P_{k_0} Q_{<j} \calN (\bar{Q}_{j} f_{k_{1}}, \bar{Q}_{<j} g_{k_{2}})}_{L^{1} L^{2}} \\
\aleq &  2^{\ell} 2^{k_{\min}+k_{\max} } \nrm{\bar{Q}_{j} f_{k_{1}}}_{L^{2} L^{2}} \bb(\sup_{\pm} \sum_{\calC_{k_{\min}}(\ell)}\nrm{P_{\calC_{k_{\min}}(\ell)} \bar{Q}_{<j}^{\pm} g_{k_{2}}}_{L^{2} L^{\infty}}^{2} \bb)^{1/2}
\end{aligned}
\end{equation}
The same statement holds in the case $ (k_0,k_1,k_2) \in \mb{Z}_{+} \times \mb{Z} \times  \mb{Z}_{+} $ when we replace the LHS of \eqref{eq:nf:est0}, \eqref{eq:nf:est1}  by $ \bar{P}_{k_0} \bar{Q}_{j} \calN (Q_{<j} f_{k_{1}}, \bar{Q}_{<j} g_{k_{2}}) $ and $ \bar{P}_{k_0} \bar{Q}_{<j} \calN ( Q_{j} f_{k_{1}}, \bar{Q}_{<j} g_{k_{2}}) $ respectively; or in the case $ (k_0,k_1,k_2) \in \mb{Z}_{+} \times \mb{Z}_{+} \times  \mb{Z}  $ when we replace the LHS of \eqref{eq:nf:est0}, \eqref{eq:nf:est1}  by $ \bar{P}_{k_0} \bar{Q}_{j} \calN (\bar{Q}
_{<j} f_{k_{1}}, Q_{<j} g_{k_{2}}) $ and $ \bar{P}_{k_0} \bar{Q}_{<j} \calN ( \bar{Q}_{j} f_{k_{1}}, Q_{<j} g_{k_{2}}) $ respectively.
\end{proposition}

\begin{proof}[Proof of Lemma \ref{lem:ellip}, Prop. \ref{prop:no-nf}, Prop. \ref{prop:nf}] The idea of these estimates is taken from \cite{KST}. For a proof of these bounds as stated here, see \cite[Section 7.5]{MD}. We can invoke that proof since we have Corollary \ref{Nij:form:prop} and Lemma \ref{geom:cone}.
\end{proof}

Finally, we record the following identity.
\begin{lemma}[Commutator identity] \label{comm_id} We can write
$$ P_{<k}(fg)=f P_{<k}g+L(\nabla_x f, 2^{-k} g) $$
where $ L $ is a translation-invariant bilinear operator with integrable kernel.

\end{lemma}
\begin{proof}
See \cite[lemma 2]{Tao2}.
\end{proof}

\section{Bilinear estimates}

The proofs in this section and the next one are based on the Littlewood-Paley trichotomy which states that $ P_{k_0}(P_{k_1} f_1 P_{k_2} f_2) $ vanishes unless $ \vm{k_{\med}-k_{\max}} \leq 5 $, where $ k_{\med}, k_{\max} $ are the the median and the maximum of $ \{ k_0,k_1, k_2 \} $.

Most of the arguments in this section originate in \cite{KST}. However, we have tried to give a thorough exposition in order to justify that the arguments carry over when two of the inputs/output correspond to Klein-Gordon waves.

\subsection{Additional bilinear estimates}

Before we begin the proofs we state some additional bilinear estimates that will be used in the proof of the trilinear estimate in the next section.

We separate the high-high and low-high parts of $ \bfA_{0} $
\be \label{a0:decomp}
\begin{aligned}
 \bfA_{0} ( \phi^{1},  \phi^{2}) & =\bfA_{0}^{LH} ( \phi^{1},  \phi^{2})+  \bfA_{0}^{HH} ( \phi^{1},  \phi^{2})  \\ 
\text{where} \qquad  \qquad \qquad \bfA_{0}^{HH} ( \phi^{1},  \phi^{2}) & =  \sum_{\substack{k_{0}, k_{1}, k_{2} \\ k_{0} < k_{2} - C_{2} - 5}} P_{k_{0}}  \bfA_{0}( \bar{P}_{k_{1}} \phi^1, \bar{P}_{k_{2}} \phi^2). 
\end{aligned}
\ee

\begin{lemma} \label{lemma:additional} With the decomposition above, one has: 
\begin{align}
\vn{ \pi[(0,A_0)] \phi}_{\bar{N}^{\sg-1} } & \ls \vn{A_0}_{\ell^1 L^1 L^{\infty}} \vn{\phi}_{\bar{S}^{\sg}}. \label{est:phi7} \\
 \vn{\bfA_{0}^{LH} ( \phi^{1},  \phi^{2})}_{\ell^1 L^1 L^{\infty}} & \ls \vn{\phi^1}_{\bar{S}^{\sg}} \vn{\phi^2}_{\bar{S}^{\sg}} \label{A0:lh}   \\
 \vn{(\bfA_{x},\bfA_{0}^{HH}) ( \phi^{1},  \phi^{2})}_{\ell^1 S^{\sg} \times Y^{\sg}} & \ls  \vn{\phi^1}_{\bar{S}^{\sg}} \vn{\phi^2}_{\bar{S}^{\sg}} \label{Ax:A0:hh} \\
 \vn{(I - \calH) (\bfA_{x},\bfA_{0}^{HH})( \phi^{1},  \phi^{2})}_{Z \times Z_{ell}}  
	& \aleq  \nrm{\phi^{1}}_{\bar{S}^{\sg}} \nrm{\phi^{2}}_{\bar{S}^{\sg}}  \label{eq:axr-Z}.
\end{align}
For $ d \geq 5 $ one also has:
\be \vn{ (\bfA_{x},\bfA_{0}^{HH})( \phi^{1},  \phi^{2})}_{Z \times Z_{ell}}  
	 \aleq  \nrm{\phi^{1}}_{\bar{S}^{\sg}} \nrm{\phi^{2}}_{\bar{S}^{\sg}} 
\label{AHH:highdim}  \ee
\end{lemma}
\begin{proof}
By doing dyadic decompositions, \eqref{est:phi7} follows trivially from H\" older's inequality $ L^1 L^{\infty} \times L^{\infty} L^2 \to L^1 L^2 $. The bound \eqref{A0:lh} follows from 
$$ \vn{P_{k'} ( \phi^{1}_{k_1} \pt_t  \phi^{2}_{k_2})}_{ L^1 L^{\infty}}  
	\aleq   \nrm{\phi^{1}_{k_1}}_{L^2 L^{\infty}} \nrm{\phi^{2}_{k_2}}_{L^2 L^{\infty}}. $$
The bound \eqref{Ax:A0:hh} follows from Prop. \ref{prop:ax:est} and from the proof of \eqref{est:a01}.

The proofs of estimates \eqref{eq:axr-Z}, \eqref{AHH:highdim} are longer and are deferred to the end of this section.
\end{proof}

\subsection{Dyadic norms} For easy referencing in the arguments below, here we collect the norms that we control. Recall that we denote
$$  \vn{A_x}_{S^s_{k'}} = 2^{(s-1)k'} \vn{\nabla_{t,x} A}_{S_{k'}},  \qquad \vn{\phi_k}_{\bar{S}^s_k} = 2^{(s-1)k}  \vn{ (\jb{D_x},\pt_t)  \phi_k}_{\bar{S}_k}  $$

For $ k' \in \mb{Z} $ and $ k \geq 0 $ we have:
\begin{align}
\label{fe-LinfL2}
& \vn{\nabla_{t,x} P_{k'} A_x}_{L^{\infty}L^2}  \ls \vn{P_{k'} A_x}_{S^1_{k'}} ,	 	
& \vn{(\jb{D_x},\pt_t) \phi_k}_{L^{\infty}L^2}  \ls \vn{\phi_{k}}_{\bar{S}^1_{k}}	
\\
\label{fe-L2}
& \nrm{Q_{j} P_{k'} A_x}_{L^{2}_{t,x}}
	 \aleq  2^{-\frac{1}{2} j}  \vn{P_{k'} A_x}_{S_{k'}} ,	 
& \nrm{\bar{Q}_{j} \phi_{k}}_{L^{2}_{t,x}}
	\aleq  2^{-\frac{1}{2} j}  \vn{\phi_{k}}_{\bar{S}_{k}}
\\
\label{fe-L2Linf}
& \vn{P_{k'} A_x}_{L^2 L^{\infty}}  \ls 2^{\frac{1}{2}k'} \vn{P_{k'} A_x}_{S^{\sg}_{k'}}, 	 
& \vn{\phi_k}_{L^2 L^{\infty}}  \ls 2^{\frac{1}{2}k} \vn{\phi_{k}}_{\bar{S}^{\sg}_{k}}
% \\
% &  \vn{P_{k'} A_x}_{L^2 L^6}  \ls 2^{-\frac{1}{6}k'} \vn{P_{k'} A_x}_{S^1_{k'}}, 	 
% & \vn{\phi_k}_{L^2 L^6}  \ls 2^{-\frac{1}{2}k} \vn{\phi_{k}}_{\bar{S}^1_{k}}
\end{align}

For any $k' \leq k$ and $ l' \in [-k,C] $, $ j=k'+2l' $ and any $ \pm $:
\begin{equation} \label{fe-L2Linfty}
\begin{aligned}
\bb(\sum_{\calC= \calC_{k'}(l')} \nrm{P_{\calC} \bar{Q}^{\pm}_{<j} \phi_{k}}_{L^{2} L^{\infty}}^{2} \bb)^{1/2}
	\aleq & 2^{\frac{1}{2} l'} 2^{\sg(k'-k)}  2^{\frac{1}{2} k} \vn{\phi_{k}}_{\bar{S}^{\sg}_{k}}, \\
\bb(\sum_{\calC=\calC_{k'}(0)} \nrm{P_{\calC} \phi_{k}}_{L^{2} L^{\infty}}^{2} \bb)^{1/2}
	\aleq &  2^{\sg(k'-k)}  2^{\frac{1}{2} k} \vn{\phi_{k}}_{\bar{S}^{\sg}_{k}}.
\end{aligned} 
\end{equation}

The former follows by choosing $ k+2l=k'+2l' $ in \eqref{highfreqSp}. When $ k=0 $ it suffices to consider $ l'=0 $.  The latter inequality holds for $ \bar{Q}_{<k'} \phi_k $, while for $ \bar{Q}_{\geq k'} \phi_k $ it follows from \eqref{fe-L2}, orthogonality and Bernstein's inequality (with $ l'=0 $)
\be \label{Brns} P_{C_{k'}(l')} L^2_x \subset 2^{\frac{d}{2}k'+\frac{d-1}{2} l'} L^{\infty}_x \ee
Using \eqref{Brns} we also obtain, when $ d=4, \ \sg=1$,
\be
\begin{aligned} \label{fe-sqX}
\big(\sum_{C_{k'}(l')} \vn{P_{C_{k'}(l')} (\pt_t \mp i \jb{D}) \bar{Q}_{<j}^{\pm} \phi_k}_{L^2 L^{\infty}}^2 \big)^{\frac{1}{2}} & \ls 2^{\frac{3}{2} l'} 2^{2k'} 2^{\frac{1}{2}j} \vn{\phi_k}_{\bar{X}^{\frac{1}{2}}_{\infty}} \\
 & \ls 2^{\frac{3}{2} l'} 2^{2k'} 2^{\frac{1}{2}j} 2^{-k} \vn{\phi_{k}}_{\bar{S}^1_{k}}.
\end{aligned}
\ee
For any $k' \leq k''$ and $ l' \leq 0 $, $ j=k'+2l' $ and any $ \pm $ we have

\begin{equation} \label{A:fe-L2Linfty}
\begin{aligned}
\bb(\sum_{\calC=\calC_{k'}(l')} \nrm{P_{\calC} Q^{\pm}_{<j}  A_{k''}}_{L^{2} L^{\infty}}^{2} \bb)^{1/2}
	\aleq & 2^{\frac{1}{2} l'} 2^{\sg(k'-k'')}  2^{\frac{1}{2} k''} \vn{P_{k''} A_x}_{S^{\sg}_{k''}}, \\
\bb(\sum_{\calC=\calC_{k'}(0)} \nrm{P_{\calC}  A_{k''}}_{L^{2} L^{\infty}}^{2} \bb)^{1/2}
	\aleq & 2^{\sg(k'-k'')}  2^{\frac{1}{2} k''} \vn{P_{k''} A_x}_{S^{\sg}_{k''}}.
\end{aligned} 
\end{equation}

For $ A_0 $ we have the following bounds
\be \label{A0:fe-LinfL2}
\vn{\nabla_{t,x} P_{k'} A_0}_{L^{\infty} L^2} \ls  \vn{P_{k'}A_0}_{Y^1}.
\ee
Since we control $ \pt_t A_0 $, for $ j \geq k' $ we have both
\be \label{A0:fe-L2}
\vn{P_{k'} A_0}_{L^2_{t,x}} \ls 2^{-(\sg+\frac{1}{2})k'} \vn{P_{k'}A_0}_{Y^{\sg}}, \qquad \vn{Q_j P_{k'} A_0}_{L^2_{t,x}} \ls 2^{-j} 2^{-(\sg-\frac{1}{2})k'} \vn{P_{k'}A_0}_{Y^{\sg}}
\ee
and for $ j=k'+2l'$, using \eqref{Brns} and orthogonality, we have
\be  \label{A0:fe-L2Linfty}
\bb(\sum_{\calC=\calC_{k'}(l')} \nrm{P_{\calC} ( Q^{\pm}_{<j}) A^0_{k'}}_{L^{2} L^{\infty}}^{2} \bb)^{1/2}
	\aleq  2^{\frac{d}{2}k'+\frac{d-1}{2} l'} \vn{A^0_{k'}}_{L^2_{t,x}} \ls 2^{\frac{1}{2} k'+\frac{3}{2}l'} \vn{P_{k'}A_0}_{Y^{\sg}}
\ee
\be   \label{A0:fe-L2Linf}  \text{In particular,} \qquad \qquad  \qquad \qquad \qquad \qquad
\vn{P_{k'} A_0}_{L^2 L^{\infty}} \ls 2^{\frac{1}{2} k'} \vn{P_{k'}A_0}_{Y^{\sg}}.
\ee

Now we turn to the proofs of Prop. \ref{prop:ax:est}, \ref{prop:phi:est}.

\

\subsection{Proof of \eqref{est:ax1}}

This follows from proving, for $ k' \in \mb{Z} $, $ k_1,k_2 \geq 0 $:
\be  \label{est:ax1:freq}
\vn{P_{k'} \calP_{j}  (\phi^1_{k_1} \nabla_x \phi^2_{k_2})}_{N_{k'}^{\sg-1}  } \ls 2^{\frac{1}{2}(k_{\min} - k_{\max})} \vn{\phi^1_{k_1}}_{\bar{S}^{\sg}_{k_1}} \vn{\phi^2_{k_2}}_{\bar{S}^{\sg}_{k_2}}. 
\ee
Note that the factor $ 2^{\frac{1}{2}(k_{\min} - k_{\max})} $ provides the $ \ell^1 $ summation in  \eqref{est:ax1}. Here $ k_{\min}, k_{\max} $ are taken from the set $ \{ k',k_1,k_2\} $.

We first treat the high modulation contribution. Since $ \calP_{j}  (\phi^1 \nabla_x \phi^2) $ is skew adjoint (see Remark \ref{ax:skew-adj}), in the low-high case ($ 2^{k'}\simeq 2^{k_{max}} $) we may assume $ k_2=k_{\min} $ (i.e. the derivative falls on the lower frequency). By Lemma~\ref{lem:ellip} we have
\begin{align}
&	\nrm{P_{k'}   \calP_{j} (\bar{Q}_{\geq k_{\min} } \phi_{k_{1}}^1 \nabla_x \phi^2_{k_{2}})}_{L^{1} L^{2}}
	\aleq \nonumber  \\
& \qquad \qquad \qquad \qquad \quad  \vn{\bar{Q}_{\geq k_{\min} } \phi_{k_{1}}^1}_{L^2_{t,x}}  \big( \sum_{\calC_{k_{\min}}} \vn{P_{\calC_{k_{\min}}} \nabla_x \phi_{k_{2}}^2}_{L^2 L^{\infty}}^2 \big)^{\frac{1}{2}} , \label{highmod:1} \\
&	\nrm{P_{k'}   \calP_{j} (\bar{Q}_{< k_{\min} } \phi_{k_{1}}^1 \nabla_x \bar{Q}_{\geq k_{\min} } \phi_{k_{2}}^2)}_{L^{1} L^{2}}
	\aleq \nonumber \\
& \qquad \qquad \qquad \qquad \quad \big( \sum_{\calC_{k_{\min}}} \vn{P_{\calC_{k_{\min}}} \bar{Q}_{< k_{\min}}  \phi_{k_{1}}^1}_{L^2 L^{\infty}}^2  \big)^{\frac{1}{2}}  \vn{\bar{Q}_{\geq k_{\min} } \nabla_x \phi_{k_{2}}^2}_{L^2_{t,x}}  \label{highmod:2} \\ 
&	\nrm{P_{k'} Q_{\geq k_{\min}}  \calP_{j} (\bar{Q}_{< k_{\min}}  \phi_{k_{1}}^1 \nabla_x \bar{Q}_{< k_{\min}}  \phi_{k_{2}}^2)}_{L^2_{t,x}}
	\aleq  \nonumber \\
& \qquad  \qquad \qquad \quad  \big( \sum_{\calC_{k_{min}}} \vn{P_{\calC_{k_{min}}} \bar{Q}_{< k_{\min}} \phi_{k_{1}}^1}_{L^2 L^{\infty}}^2  \big)^{\frac{1}{2}} \vn{ \bar{Q}_{< k_{\min}} \nabla_x \phi_{k_{2}}^2}_{L^{\infty}L^2} . \label{highmod:3}
\end{align}

Using \eqref{fe-L2}, \eqref{fe-L2Linfty} for \eqref{highmod:1}, using \eqref{fe-L2Linfty}, \eqref{fe-L2} for \eqref{highmod:2}, and using \eqref{fe-L2Linfty}, \eqref{fe-LinfL2} and the $ X_1^{-1/2} $ norm for \eqref{highmod:3}, we see that these terms are acceptable. 

We continue with the low modulation term 
$$ P_{k'} Q_{< k_{\min} }  \calP_{j} (\bar{Q}_{< k_{\min} } \phi_{k_{1}}^1 \nabla_x \bar{Q}_{< k_{\min} } \phi_{k_{2}}^2), $$
which, summing according to the highest modulation, using \eqref{ax:nf:identity}, we decompose into sums of 
\begin{align}
	I_{0} =& \sum_{j < k_{\min}} P_{k'} Q_{j} \Delta^{-1} \nabla^l  \calN_{lm} (\bar{Q}_{< j} \phi_{k_{1}}^1, \bar{Q}_{< j} \phi_{k_{2}}^2), \label{I:0} \\
	I_{1} =& \sum_{j < k_{\min}}P_{k'} Q_{\leq j}\Delta^{-1} \nabla^l  \calN_{lm}(\bar{Q}_{j} \phi_{k_{1}}^1, \bar{Q}_{< j}\phi_{k_{2}}^2),  \label{I:1} \\
	I_{2} =& \sum_{j < k_{\min}} P_{k'} Q_{\leq j}\Delta^{-1} \nabla^l  \calN_{lm} ( \bar{Q}_{\leq j} \phi_{k_{1}}^1, \bar{Q}_{ j}  \phi_{k_{2}}^2).   \label{I:2}
\end{align}
for which we have
\be  \label{low:mod:Aeq}
\vn{ \vm{D}^{\sg-1} I_0}_{X_1^{-1/2}}+\vn{I_1}_{L^1  \dot{H}^{\sg-1}} +\vn{I_2}_{L^1 \dot{H}^{\sg-1} } \ls 2^{\frac{1}{2}(k_{\min} - k_{\max})} \vn{\phi^1_{k_1}}_{\bar{S}^{\sg}_{k_1}} \vn{\phi^2_{k_2}}_{\bar{S}^{\sg}_{k_2}}.
\ee
These are estimated by Proposition \ref{prop:nf} and \eqref{fe-LinfL2}, \eqref{fe-L2}, \eqref{fe-L2Linfty}, which concludes the proof of \eqref{est:ax1:freq}.

\subsection{Proof of \eqref{est:phi1}.} 
We separate $ A_0 \pt_t \phi $ and $ A^j \pt_j \phi $. Since we subtract $ \pi[A]\phi $, this effectively eliminates low-high interactions in the Littlewood-Paley trichotomy. Thus for $ k, k_0 \geq 0 $, $ k' \geq k-C $ it suffices to prove 
\be \label{est:phi1:freqA0}
\vn{\bar{P}_{k_0} \big( A^0_{k'} \pt_t \phi_k \big)}_{L^1 H^{\sg-1}} \ls 2^{k_{\min}-k_{\max}} \vn{A_{k'}^0}_{L^2 \dot{H}^{\sg+\frac{1}{2}}} \vn{\pt_t \phi_k}_{\bar{S}^{\sg-1}_{k}},
\ee
\be \label{est:phi1:freqAx}
\vn{\bar{P}_{k_0} \big( A^j_{k'} \pt_j \phi_k \big)}_{\bar{N}_{k_0}^{\sg-1}  } \ls 2^{\frac{1}{2}(k_{\min}-k_{\max})} \vn{A_{k'}}_{S^{\sg}_{k'}} \vn{\phi_k}_{\bar{S}^{\sg}_{k}}.
\ee

The bound \eqref{est:phi1:freqA0} follows immediately from \eqref{eq:ellip-1}. Now we turn to \eqref{est:phi1:freqAx}.

We first treat the high modulation contribution. By Lemma~\ref{lem:ellip} we have

\begin{align*}
&	\nrm{\bar{P}_{k_0} \bar{Q}_{\geq k_{\min} }  \big( A^j_{k'} \pt_j \phi_k \big)}_{L^2_{t,x}}
	\aleq \vn{A_{k'}}_{L^{\infty}L^2}  \big( \sum_{\calC_{k_{\min}}} \vn{P_{\calC_{k_{\min}}} \nabla_x \phi_{k}}_{L^2 L^{\infty}}^2  \big)^{\frac{1}{2}}  , \\
&	\nrm{\bar{P}_{k_0} \bar{Q}_{< k_{\min} }  (Q_{\geq k_{\min} } A^j_{k'} \pt_j \phi_k )}_{L^{1} L^{2}}
	\aleq \\
& \qquad \qquad \qquad \qquad \qquad \qquad \vn{Q_{\geq k_{\min} } A_{k'}}_{L^2_{t,x}} \big( \sum_{\calC_{k_{\min}}} \vn{P_{\calC_{k_{\min}}} \nabla_x \phi_{k}}_{L^2 L^{\infty}}^2  \big)^{\frac{1}{2}}, \\
&	\nrm{\bar{P}_{k_0} \bar{Q}_{< k_{\min} }  (Q_{< k_{\min} }  A^j_{k'} \pt_j  \bar{Q}_{\geq k_{\min}}\phi_k)  }_{L^{1} L^{2}}
	\aleq \\
& \qquad \qquad \qquad \qquad \quad \big( \sum_{\calC_{k_{\min}}} \vn{P_{\calC_{k_{\min}}} Q_{< k_{\min}}  A_{k'}}_{L^2 L^{\infty}}^2  \big)^{\frac{1}{2}}  \vn{\bar{Q}_{\geq k_{\min} } \nabla_x \phi_{k}}_{L^2_{t,x}}.  
\end{align*}
Using \eqref{fe-LinfL2}, \eqref{fe-L2Linfty} and the $ \bar{X}_1^{-1/2} $ norm for the first term, \eqref{fe-L2}, \eqref{fe-L2Linf} for the second, and \eqref{A:fe-L2Linfty}, \eqref{fe-L2} for the third, we see that these terms are acceptable. 

We continue with the low modulation term 
$$ \bar{P}_{k_0} \bar{Q}_{< k_{\min} }  \big(Q_{< k_{\min} }  A^j_{k'} \pt_j  \bar{Q}_{<  k_{\min}}\phi_k \big) $$
which, summing according to the highest modulation, using \eqref{phi:nf:identity}, we decompose into sums of 
\begin{align}
	I_{0} =& \sum_{j < k_{\min}} \bar{P}_{k_0} \bar{Q}_{j}   \calN_{lm} (\Delta^{-1} \nabla^l Q_{< j} A^m_{k'}, \bar{Q}_{< j} \phi_{k}), \label{I:zero} \\
	I_{1} =& \sum_{j < k_{\min}}\bar{P}_{k_0} \bar{Q}_{\leq j}  \calN_{lm}(\Delta^{-1} \nabla^l Q_{j} A^m_{k'}, \bar{Q}_{< j}\phi_{k}),\\
	I_{2} =& \sum_{j < k_{\min}} \bar{P}_{k_0} \bar{Q}_{\leq j}  \calN_{lm} (\Delta^{-1} \nabla^l Q_{\leq j} A^m_{k'}, \bar{Q}_{ j}  \phi_{k}). \label{I:two}
\end{align}
These are estimated using Proposition \ref{prop:nf}. We use \eqref{eq:nf:est0} with \eqref{fe-LinfL2} and \eqref{fe-L2Linfty} to estimate $ I_0 $ in $ \bar{X}_1^{-1/2} $. For $ I_1 $ we use \eqref{eq:nf:est1} with \eqref{fe-L2} and \eqref{fe-L2Linfty}, while for $ I_2 $ we use \eqref{eq:nf:est1} with \eqref{A:fe-L2Linfty} and   \eqref{fe-L2}. This concludes the proof of \eqref{est:phi1}.

\subsection{Proof of \eqref{est:phi2}.} We separate $ A_0 \pt_t \phi $ and $ A^j \pt_j \phi $. This case corresponds to low-high interactions in the Littlewood-Paley trichotomy. Thus for $ k, k_0 \geq 0 $, $ k' \leq k-C $ (and $ \vm{k-k_0} \leq 5 $)  it suffices to prove 
\be \label{est:phi2:freqA0}
 \vn{\bar{P}_{k_0} \big( A^0_{k'} \pt_t \phi_k \big)- \bar{P}_{k_0} \calH^{\ast}_{k'} \big( A_0 \pt_t \phi_k \big) }_{\bar{N}_{k_0}} \ls \vn{P_{k'} A_{0}}_{Y^{\sg}} \vn{\phi_k}_{\bar{S}^1_{k}} 
\ee
\be \label{est:phi2:freqAx}
\vn{\bar{P}_{k_0} \big( A^j_{k'} \pt_j \phi_k \big)- \bar{P}_{k_0} \calH^{\ast}_{k'} \big( A^j \pt_j \phi_k \big) }_{\bar{N}_{k_0}} \ls \vn{P_{k'} A_{x}}_{S^{\sg}_{k'}} \vn{\phi_k}_{\bar{S}^1_{k}}.
\ee

Notice that the lack of an exponential gain of type $ 2^{\frac{1}{2}(k_{\min}-k_{\max})} $ (as in \eqref{est:phi1:freqA0}, \eqref{est:phi1:freqAx}) is responsible for the need of $ \ell^1 $ summation in the norm on the RHS of \eqref{est:phi2}.

We first treat the high modulation contribution, where we denote $ A $ for either $ A^0 $ or $ A^j $. For any $ j \geq k'+C_2 $, by H\" older's inequality
\be \label{aaph:eq}
\begin{aligned}
& \vn{\bar{P}_{k_0} \bar{Q}_{\geq j-5} \big( Q_j A_{k'} \pt \phi_k \big)}_{\bar{X}_1^{-1/2}} \ls 2^{-\frac{1}{2}j} \vn{A_{k'}}_{L^2 L^{\infty}} \vn{\nabla \phi_k}_{L^{\infty} L^2} \\
& \vn{\bar{P}_{k_0} \bar{Q}_{< j-5} \big( Q_j A_{k'}  \bar{Q}_{\geq j-5}\pt \phi_k \big)}_{L^1 L^2 } \ls \vn{A_{k'}}_{L^2 L^{\infty}} \vn{\bar{Q}_{\geq j-5} \nabla \phi_k}_{L^2_{t,x}}
\end{aligned}
\ee
Using \eqref {A0:fe-L2Linf}, \eqref{fe-L2Linf}, \eqref{fe-LinfL2}, \eqref{fe-L2} and summing over $  j \geq k'+C_2 $, it follows that $ \bar{P}_{k_0} \big( Q_{\geq k'+C_2} A_{k'} \pt \phi_k \big) $ is acceptable except
$$ T=\sum_{j \geq k'+C_2} \bar{P}_{k_0} \bar{Q}_{< j-5} \big( Q_j A_{k'}  \bar{Q}_{<j-5}\pt \phi_k \big) $$
By applying Lemma \ref{geom:cone} (here we choose $ C_2 > \frac{1}{2} C_0 $) we see that the summand vanishes unless $ j=k_{\max}+O(1) $. Then, by Lemma~\ref{lem:ellip} we have
$$ \vn{T}_{L^1 L^2} \ls \sum_{j=k+O(1)} \vn{ Q_j A_{k'}}_{L^2_{t,x}} \big( \sum_{\calC_{k'}} \vn{P_{\calC_{k'}} \nabla \phi_k}^2_{L^2 L^{\infty}}   \big)^{\frac{1}{2}}      $$
which is acceptable by \eqref{A0:fe-L2}, \eqref{fe-L2}, \eqref{fe-L2Linfty}. 
The terms 
$$ \bar{P}_{k_0} \bar{Q}_{\geq k'+C_2} \big( Q_{< k'+C_2} A_{k'} \pt \phi_k \big), \qquad \bar{P}_{k_0} \bar{Q}_{< k'+C_2} \big( Q_{< k'+C_2} A_{k'} \bar{Q}_{\geq k'+C_2} \pt \phi_k \big) $$ 
are treated in the same way as \eqref{aaph:eq}. We omit the details.

We continue with the low modulation terms. Since we are subtracting $ \calH^{\ast} $ we consider
$$ I= \sum_{j < k'+C_2} \bar{P}_{k_0} \bar{Q}_{j}   ( Q_{< j} A^0_{k'} \cdot \pt_t \bar{Q}_{< j} \phi_{k}) $$
$$ J=\sum_{j < k'+C_2} \bar{P}_{k_0} \bar{Q}_{\leq j} (Q_{\leq j} A^0_{k'} \cdot \pt_t \bar{Q}_{ j}   \phi_{k}) $$
and prove
$$ \vn{I}_{\bar{X}_1^{-1/2}} + \vn{J}_{L^1 L^2} \ls \vn{P_{k'} A_{0}}_{Y^{\sg}} \vn{\phi_k}_{\bar{S}^1_{k}}. $$
by using \eqref{eq:no-nf:est0} with  \eqref{fe-LinfL2} and \eqref{A0:fe-L2Linfty} for $ I $; we use \eqref{eq:no-nf:est1} with \eqref{fe-L2} and  \eqref{A0:fe-L2Linfty} for $ J $.

It remains to show that for $ I_0, I_2 $ from \eqref{I:zero} and \eqref{I:two} (with summation over $ j<k'+C_2 $) we have
$$ \vn{I_0}_{\bar{X}_1^{-1/2}} + \vn{I_{2}}_{L^1 L^2} \ls \vn{A_{k'}}_{S^{\sg}_{k'}} \vn{\phi_k}_{\bar{S}^1_{k}}. $$
These follow from \eqref{eq:nf:est0} with \eqref{A:fe-L2Linfty}, \eqref{fe-LinfL2} and from \eqref{eq:nf:est1} with \eqref{A:fe-L2Linfty}, \eqref{fe-L2}, respectively.

\subsection{Proof of \eqref{est:phi3}.} This estimate follows from the next bound, for $ k' <k-5 $ 
\be \label{est:phi3:freq}
\vn{\calH^{\ast}_{k'} \big( A^{\al} \pt_{\al} \phi_k \big) }_{L^1 L^2} \ls \vn{A_{k'}}_{Z_{k'} \times Z^{ell}_{k'}} \vn{\phi_k}_{\bar{S}^1_{k}}. 
\ee
To prove \eqref{est:phi3:freq}, let $ \ell=\frac{1}{2}(j-k')_{-} \geq -k-C $ and separate $ A_0 \pt_t \phi $ from $ A^j \pt_j \phi $. We use \eqref{phi:nf:identity} and denote by $ \calN(A,\phi)$ one of $ A^0 \pt_t \phi $ or $ \calN_{lm} (\Delta^{-1} \nabla^l A^m, \phi) $. We expand
$$  \calH^{\ast}_{k'} \calN \big( A, \phi_k \big)=\sum_{j<k'+C_2^{\ast}} \sum_{\omega_1,\omega_2} \bar{Q}_{<j} \calN \big(P_{\ell}^{\omega_1} Q_j A_{k'},P_{\ell}^{\omega_2} \bar{Q}_{<j} \phi_k \big) $$
Splitting $ \bar{Q}_{<j}=\bar{Q}_{<j}^{+}+\bar{Q}_{<j}^{-},  \ Q_j=Q_j^++Q_j^-, $ and applying Lemma \ref{geom:cone} we see that the summand vanishes unless $ \vm{\angle(\omega_1,\pm \omega_2 )}  \ls 2^{\ell}  $. 

For  $ \calN=\calN_{lm} (\Delta^{-1} \nabla^l A^m, \phi) $ and $ s_1,s \in \{+,-\} $, by Corollary \ref{Nij:form:prop} we have

\begin{equation*}  \vn{\bar{Q}_{<j} \calN \big(P_{\ell}^{\omega_1} Q_j^{s_1} A_{k'},P_{\ell}^{\omega_2} \bar{Q}_{<j}^s \phi_k \big)}_{L^1 L^2} \ls 2^{\ell} \vn{P_{\ell}^{\omega_1} Q_j^{s_1} A_{k'}}_{L^1 L^{\infty}} \vn{P_{\ell}^{\omega_2} \bar{Q}_{<j}^s \nabla \phi_k}_{L^{\infty} L^2}  \end{equation*}

For $ \calN=A^0 \pt_t \phi $ we have the same inequality but without the $ 2^{\ell} $ factor. This is compensated by the fact that the $ Z^{ell} $ norm is larger. Indeed, we have  
$$ Z^{ell}=\Box^{\frac{1}{2}} \Delta^{-\frac{1}{2}} Z  \quad \text{and} \quad
 2^{\ell}  \vn{ P_{\ell}^{ \omega_1} Q_j^{s_1} A^0_{k'}}_{L^1 L^{\infty}}  \simeq   \vn{\Box^{\frac{1}{2}} \Delta^{-\frac{1}{2}} P_{\ell}^{ \omega_1} Q_j^{s_1} A^0_{k'}}_{L^1 L^{\infty}}.  $$

Note that for fixed $ \omega_1 $ there are only (uniformly) bounded number of $ \omega_2 $ such that the product is non-vanishing. Therefore, by Cauchy-Schwarz,
$$ \vn{\calH^{\ast}_{k'} \big( A^{\al} \pt_{\al} \phi_k \big) }_{L^1 L^2} \ls \sum_{\ell \leq 0} 2^{\frac{1}{2} \ell} \vn{A_{k'}}_{Z_{k'} \times Z^{ell}_{k'}} \big( \sup_{\pm} \sum_{\omega_2} \vn{P_{\ell}^{\omega_2} \bar{Q}_{<j}^{\pm} \nabla \phi_k}_{L^{\infty} L^2}^2  \big)^{\frac{1}{2}} $$
which implies \eqref{est:phi3:freq}.

\subsection{Proof of \eqref{est:a01}, \eqref{est:phi6} and the $ L^2 H^{\sg-\frac{3}{2}} $ part of \eqref{est:phi4} } One proceeds by dyadic decompositions. The $ L^2_{t,x} $-type estimates follow easily by H\" older's inequality $ L^{\infty}L^2 \times L^2 L^{\infty} \to L^2_{t,x} $ in the low-high/high-low cases and by Lemma \ref{lem:ellip} (eq. \eqref{eq:ellip-0}) in the high-high to low case. One uses the norms \eqref{fe-LinfL2}, \eqref{A0:fe-LinfL2}, \eqref{fe-L2Linf}, \eqref{A0:fe-L2Linf}, \eqref{fe-L2Linfty},  \eqref{A:fe-L2Linfty}, \eqref{A0:fe-L2Linfty}.

The $ L^{\infty} L^2 $ estimate follows by H\" older's ( $ L^{\infty} L^{\infty} \times L^{\infty}L^2 \to L^{\infty}L^2 $ or $ L^{\infty}L^2 \times L^{\infty}L^2 \to L^{\infty}L^1 $) and Bernstein's inequalities ( $ P_k L^2_x \to 2^{\frac{d}{2}k} L^{\infty}_x $ or $ P_k L^1_x \to 2^{\frac{d}{2}k} L^2_x $), depending on which frequency (input or output) is the lowest. 

\subsection{Proof of \eqref{est:phi4} for $ \bar{N} $.} Suppose $ k,k_2 \geq 0 $, $ k_1 \in \mb{Z} $. Let $ r_0 $ be the endpoint Strichartz exponent (i.e. $ \frac{d-1}{r_0}=\sg-\frac{1}{2} $).   By H\" older's inequality and using Bernstein's inequality for the lowest frequency (input or output) we obtain
\be \label{est:cubic:auxx}
\vn{\bar{P}_k \big( P_{k_1}A \cdot \phi_{k_2} \big)}_{L^1 H^{\sg-1} } \ls 2^{\frac{d}{r_0}(k-\max{k_i})} 2^{-\frac{1}{r_0} \vm{k_1-k_2}}  \vn{P_{k_1}A}_{L^2 \dot{H}^{\sg-\frac{1}{2}}} \vn{\phi_{k_2}}_{L^2 W^{r_0,\rho }}
\ee
With $ A=\pt_t A_0 $ and  $ \vn{\phi_{k_2}}_{L^2 W^{r_0,\rho }} \ls \vn{\phi_{k_2}}_{\bar{S}^{\sg}_{k_2}} $, upon summation we obtain \eqref{est:phi4}.

\subsection{Proof of \eqref{est:phi5} and \eqref{est:ax2} } We first prove the $ L^1 L^2 $ part. For \eqref{est:phi5} we consider $ k, k_2 \geq 0 $, $ k_1,k_3,k_4 \in \mb{Z} $. We apply \eqref{est:cubic:auxx} with $ A=A^1_{\al} A^2_{\al} $ together with
 $$
\vn{P_{k_1}(P_{k_3} A^1_{\al} P_{k_4} A^2_{\al})}_{L^2 \dot{H}^{\sg- \frac{1}{2}}} \ls 2^{\frac{1}{2}(k_{\min}-k_{\max})}   \vn{A_{\max\{k_3,k_4\}}}_{L^{\infty} \dot{H}^{\sg}}  \vn{A_{\min\{k_3,k_4\}}}_{L^2 \dot{W}^{\infty,-\frac{1}{2}} }. 
$$
By summing we obtain  \eqref{est:phi5}. The same argument is used for $ L^1 L^2 $ of  \eqref{est:ax2}.

To prove the $ L^2 \dot{H}^{\sg-\frac{3}{2}} $ and $ L^2 H^{\sg-\frac{3}{2}} $ estimates we write
$$  \vn{P_{k}\big(P_{k_1} (fg) P_{k_2} h \big)}_{L^2 \dot{H}^{\sg-\frac{3}{2}} }  \ls 2^{\frac{1}{2}(k-\max{k_i})} 2^{-\frac{1}{2} \vm{k_1-k_2}}  \vn{P_{k_1} (fg)}_{L^{\infty} \dot{H}^{\sg-1 }} \vn{P_{k_2} h}_{L^2 W^{r_0,\rho }} $$
and use $ L^{\infty} \dot{H}^{\sg} \times L^{\infty} \dot{H}^{\sg} \to L^{\infty} \dot{H}^{\sg-1} $ by H\" older and Sobolev embedding.

The $ \ell^1 L^{\infty} \dot{H}^{\sg-2} $ part of  \eqref{est:ax2} is similarly a consequence of H\" older and Bernstein inequalities.

\subsection{Proof of \eqref{eq:axr-Z}} 

Recall that $ \calH $ subtracts terms only for high-high interactions. For $ k' \leq k_2-C_2-10 $ we claim
\be \label{Z:norm1}
\vn{(P_{k'}-\calH_{k'}) \bfA ( \phi^{1}_{k_1},  \phi^{2}_{k_2})}_{Z_{k'} \times Z^{ell}_{k'}}  
	\aleq  2^{\frac{1}{2}(k'-k_{2})}  \nrm{\phi^{1}_{k_1}}_{\bar{S}^{\sg}_{k_1}} \nrm{\phi^{2}_{k_2}}_{\bar{S}^{\sg}_{k_2}}. 
\ee
while the low-high interactions: for $ k' \geq k_2-C_2-10 $ 
\be  \label{Z:norm2}
\vn{P_{k'} \bfA_x ( \phi^{1}_{k_1},  \phi^{2}_{k_2})}_{Z_{k'}}  
	\aleq  2^{-\frac{1}{2} \vm{k_{1}-k_{2}}}  \nrm{\phi^{1}_{k_1}}_{\bar{S}^{\sg}_{k_1}} \nrm{\phi^{2}_{k_2}}_{\bar{S}^{\sg}_{k_2}}.
\ee 
Clearly, \eqref{Z:norm1} and \eqref{Z:norm2} imply \eqref{eq:axr-Z}.
First we recall that 
$$ (\Box \bfA_{x}  , \Delta \bfA_{0}) ( \phi^{1},  \phi^{2})=-\mathfrak{I}( \calP_{x }(\phi^1 \nabla_x \bar{\phi^2}), \phi^1 \partial_t \bar{\phi^2}) $$ 
and  the embedding from \eqref{ZZ:emb}
$$
\big( \Box^{-1} \times \Delta^{-1}  \big) P_{k'} : L^1 L^2 \times L^1 L^2 \to 2^{(\sg-1)k'}  Z_{k'} \times Z^{ell}_{k'} 
$$

\pfstep{Step~1}{\it Proof of \eqref{Z:norm1}.}  The terms 
$$ P_{k'} \bfA ( \bar{Q}_{\geq k'+C} \phi^{1}_{k_1},  \phi^{2}_{k_2}), \qquad P_{k'} \bfA ( \bar{Q}_{\leq k'+C} \phi^{1}_{k_1},  \bar{Q}_{\geq k'+C} \phi^{2}_{k_2}) $$
are estimated using \eqref{highmod:1}, \eqref{highmod:2} and \eqref{ZZ:emb}. For $ \bfA_0 $ we note that \eqref{highmod:1}, \eqref{highmod:2} still  hold with $ \calP_j $ replaced by $ \calL $ \footnote{ $ \calL $ denotes any translation invariant bilinear form with bounded mass kernel.} and $ \nabla_x $ replaced by $ \pt_t $.
Recall that the $ Z $ norms restrict modulation to $ Q_{\leq k'+C} $. Thus it remains to consider
$$   (P_{k'} Q_{\leq k'+C}-\calH_{k'}) \bfA ( \bar{Q}_{\leq k'+C} \phi^{1}_{k_1},  \bar{Q}_{\leq k'+C} \phi^{2}_{k_2})
$$
For $ \bfA_x $, using \eqref{ax:nf:identity}, we need to treat $ \Box^{-1} I_1 $,  $ \Box^{-1} I_2 $ as defined in \eqref{I:0}-\eqref{I:2} (the $ \Box^{-1} I_0 $ term is subtracted by $ \calH_{k'} $). These are estimated using \eqref{low:mod:Aeq} and \eqref{ZZ:emb}.

We turn to $ \bfA_0 $. By switching the roles of $ \phi^1, \phi^2 $ if needed, it remains to consider 
$$ J_j= P_{k'} Q_{\leq j} \bfA_0 ( \bar{Q}_{j} \phi^{1}_{k_1},  \bar{Q}_{\leq j} \phi^{2}_{k_2}), \qquad \qquad  j \leq k'+C. $$
Using \eqref{Zell:emb} we obtain
\begin{align*} \vn{J_j}_{Z_{k'}^{ell}}  & \ls \sum_{\pm; j' \leq j} 2^{\frac{1}{2}(j'-k')} \vn{P_{k'} Q^{\pm}_{j'} (\bar{Q}_{j} \phi^{1}_{k_1} \cdot \pt_t  \bar{Q}_{\leq j} \phi^{2}_{k_2})  }_{L^1 \dot{H}^{\sg-1}  }  \\
& \ls 2^{\frac{1}{2}(j-k')} 2^{\frac{1}{2} (k'-k_2)} \nrm{\phi^{1}_{k_1}}_{\bar{S}^{\sg}_{k_1}} \nrm{\phi^{2}_{k_2}}_{\bar{S}^{\sg}_{k_2}}
\end{align*}

For the last inequality we have used Prop. \ref{prop:no-nf}  together with \eqref{fe-L2} and  \eqref{fe-L2Linfty}.

Summing in $ j \leq k'+C $ completes the proof of \eqref{Z:norm1}.

\pfstep{Step~2}{\it Proof of \eqref{Z:norm2}.} Due to skew-adjointness (see Remark \ref{ax:skew-adj}), we may assume that $ k_2=k_{\min}+O(1) $. The terms 
$$ P_{k'} \bfA_x ( \bar{Q}_{\geq k_2-c} \phi^{1}_{k_1},  \phi^{2}_{k_2}), \qquad P_{k'} \bfA_x ( \bar{Q}_{\prec k_2} \phi^{1}_{k_1},  \bar{Q}_{\geq k_2-c} \phi^{2}_{k_2}) $$
are estimated using \eqref{highmod:1}, \eqref{highmod:2} and \eqref{ZZ:emb}.

Note that the $ Z $ norm restricts modulations to $ Q_{\leq k'+C} $. Thus it remains to consider 
\be \label{mod:Z:terms}
P_{k'} Q_j \bfA_x ( \bar{Q}_{\prec k_2} \phi^{1}_{k_1},  \bar{Q}_{\prec k_2} \phi^{2}_{k_2})
\ee
for $ j \leq k'+C $. When $ j \geq k_2+C $, by Lemma \ref{geom:cone} the term vanishes unless $ j=k'+O(1) $. In this case
\begin{align*} & \vn{P_{k'} Q_{j} \bfA_x (  \bar{Q}_{\prec k_2} \phi^{1}_{k_1},  \bar{Q}_{\prec k_2} \phi^{2}_{k_2})}_{Z_{k'}}   \ls 2^{-2k'} \vn{\bar{Q}_{\prec k_2} \phi^1_{k_1} \nabla_x \bar{Q}_{\prec k_2} \phi^2_{k_2}}_{L^1 L^{\infty}} \\
& \ls 2^{k_2-2k'} \vn{\bar{Q}_{\leq k_1} \phi^1_{k_1}}_{L^2 L^{\infty}} \vn{\bar{Q}_{\prec k_2} \phi^2_{k_2}}_{L^2 L^{\infty}}+ 2^{k_2} \vn{\bar{Q}_{[k_2-c,k_1]} \phi^1_{k_1}}_{L^2 H^{\sg-1} } \vn{\bar{Q}_{\prec k_2} \phi^2_{k_2}}_{L^2 L^{\infty}}
\end{align*}
which is estimated using \eqref{fe-L2Linf} and \eqref{fe-L2}.

It remains to consider \eqref{mod:Z:terms} for $ j<k_2+C $. Using \eqref{ax:nf:identity} we decompose into sums of $ \Box^{-1} I_i $, $ (i=0,2) $ as defined in \eqref{I:0}-\eqref{I:2} (for $ k_2-C<j<k_2+C $ with $ \bar{Q}$ indices slightly adjusted). Then $ \Box^{-1} I_1 $ and $ \Box^{-1} I_2 $ are estimated using \eqref{low:mod:Aeq} and \eqref{ZZ:emb}.

Now we consider $ \Box^{-1} I_0 $. Define $ \ell \defeq \frac{1}{2}(j-k_2)_{-}  \geq \ell' \defeq \frac{1}{2}(j-k')_{-} $ and for $ s=\pm $ we decompose
$$ P_{k'} Q_{j}^s   \calN_{lm} (\bar{Q}_{< j} \phi_{k_{1}}^1, \bar{Q}_{< j} \phi_{k_{2}}^2)=\sum_{s_2, \omega_i} P_{\ell'}^{\omega_0} P_{k'} Q_{j}^s   \calN_{lm} (P_{\ell'}^{\omega_1} \bar{Q}_{< j}^s \phi_{k_{1}}^1, P_{\ell}^{\omega_2}\bar{Q}_{< j}^{s_2} \phi_{k_{2}}^2)
$$
By Lemma \ref{geom:cone}, the summand on the RHS vanishes unless
\begin{align*}
&\ \ \vm{\angle(\omega_0,\omega_1)} \ \ \ls  2^{\ell} 2^{k_2-k'}+2^{\ell'} \ls 2^{\ell'} \\
& \vm{\angle(s \omega_0,s_2 \omega_2)}  \ls 2^{\ell}+\max(2^{\ell'},2^{\ell})  \ls 2^{\ell}.
\end{align*}
Note that $ P_{\ell'}^{\omega_0} P_{k'} Q_{j}^s $ and $ 2^{2\ell'+2k'} \Box^{-1} P_{\ell'}^{\omega_0} P_{k'} Q_{j}^s  $ are disposable. Corollary \ref{Nij:form:prop} implies 
\be \label{Z:nf:bd}
\vn{ \calN_{lm} (P_{\ell'}^{\omega_1} \bar{Q}_{< j}^s \phi_{k_{1}}^1, P_{\ell}^{\omega_2}\bar{Q}_{< j}^{s_2} \phi_{k_{2}}^2) }_{L^1 L^{\infty}} \ls 2^{\ell} \vn{P_{\ell'}^{\omega_1} \bar{Q}_{< j}^s \nabla \phi_{k_{1}}^1}_{L^2 L^{\infty}}  \vn{P_{\ell}^{\omega_2}\bar{Q}_{< j}^{s_2} \nabla \phi_{k_{2}}^2}_{L^2 L^{\infty}} 
\ee

For a fixed $\omg_{0}$ [resp. $\omg_{1}$], there are only (uniformly) bounded number of $\omg_{1}, \omg_{2}$ [resp. $\omg_{0}, \omg_{2}$] such that the summand is nonzero. Summing first in $\omg_{2}$ (finitely many terms), then the (essentially diagonal) summation in $\omg_{0}, \omg_{1}$, we obtain
$$ \big( \sum_{\omega_0} \text{LHS} \eqref{Z:nf:bd}^2 \big)^\frac{1}{2} \ls 
2^{\ell} \big( \sum_{\omega_1} \vn{P_{\ell'}^{\omega_1} \bar{Q}_{< j}^s \nabla \phi_{k_{1}}^1}_{L^2 L^{\infty}}^2 \big)^{\frac{1}{2}}  \sup_{\omega_2} \vn{P_{\ell}^{\omega_2}\bar{Q}_{< j}^{s_2} \nabla \phi_{k_{2}}^2}_{L^2 L^{\infty}}
$$
Keeping track of derivatives and dyadic factors, recalling \eqref{Z:norm:def} and using \eqref{fe-L2Linfty} for $ \phi^1, \phi^2 $, we obtain
$$ \vn{\Box^{-1} I_0}_{Z_{k'}} \ls \sum_{j<k_2} 2^{\frac{1}{4}(j-k_2)} 2^{k_2-k'} \nrm{\phi^{1}_{k_1}}_{\bar{S}^{\sg}_{k_1}} \nrm{\phi^{2}_{k_2}}_{\bar{S}^{\sg}_{k_2}}  $$
This completes the proof of \eqref{Z:norm2}.

\subsection{Proof of \eqref{AHH:highdim}} 
The low-high part of the estimate for $ \bfA_x ( \phi^{1},  \phi^{2}) $ follows from \eqref{Z:norm2}. For the high-high parts of both $ \bfA_x ( \phi^{1},  \phi^{2}) $ and $ \bfA_0 ( \phi^{1},  \phi^{2}) $ we fix the frequency and use \eqref{ZZ:emb}, H\" older $ L^2 L^4 \times L^2 L^4 \to L^1 L^2 $ together with $ L^2 L^4 $ Strichartz inequalities. We gain the factor $ 2^{\frac{d-4}{2} (k_{\min}-k_{\max})} $ which suffices to do the summation in the present case $ d \geq 5 $.

\section{Proof of the trilinear estimate \eqref{est:trilin}} \label{Trilinear:section}

This section is devoted to the the proof of Proposition \ref{trilinear}.

\subsection{Proof of Proposition \ref{trilinear}} Our goal is to prove

$$ \vn{\pi[\bfA( \phi^{1},  \phi^{2}) ] \phi}_{\bar{N}^{\sg-1}  }  \ls \vn{\phi^1}_{\bar{S}^{\sg} } \vn{\phi^2}_{\bar{S}^{\sg}} \vn{\phi}_{\bar{S}^{\sg}} $$
First we note that (recalling definition \eqref{a0:decomp}) \eqref{est:phi7} together with \eqref{A0:lh} implies
$$ \vn{\pi[0,\bfA_{0}^{LH}( \phi^{1},  \phi^{2}) ] \phi}_{\bar{N}^{\sg-1}  }  \ls \vn{\phi^1}_{\bar{S}^{\sg} } \vn{\phi^2}_{\bar{S}^{\sg}} \vn{\phi}_{\bar{S}^{\sg}}  	$$
Secondly, \eqref{Ax:A0:hh} and \eqref{est:phi2} imply
$$ \vn{(I-\calH^{\ast}) \pi[(\bfA_{x},\bfA_{0}^{HH})( \phi^{1},  \phi^{2}) ] \phi}_{\bar{N}^{\sg-1} }  \ls \vn{\phi^1}_{\bar{S}^{\sg} } \vn{\phi^2}_{\bar{S}^{\sg}} \vn{\phi}_{\bar{S}^{\sg}}    $$
For $ d \geq 5 $, \eqref{est:phi3} and \eqref{AHH:highdim} imply 
$$ 
\vn{\calH^{\ast} \pi[(\bfA_{x},\bfA_{0}^{HH})( \phi^{1},  \phi^{2}) ] \phi}_{\bar{N}^{\sg-1} }  \ls \vn{\phi^1}_{\bar{S}^{\sg} } \vn{\phi^2}_{\bar{S}^{\sg}} \vn{\phi}_{\bar{S}^{\sg}} 
$$
which concludes the proof in the case $ d \geq 5 $. In the remaining of this section we assume $ d=4 $, $ \sg=1 $.

Next we use \eqref{est:phi3} together with \eqref{eq:axr-Z}  and obtain
$$
\vn{\calH^{\ast} \pi[(I - \calH)(\bfA_{x},\bfA_{0}^{HH})( \phi^{1},  \phi^{2}) ] \phi}_{\bar{N}}  \ls \vn{\phi^1}_{\bar{S}^1} \vn{\phi^2}_{\bar{S}^1} \vn{\phi}_{\bar{S}^1} 
$$
Since $ \calH \bfA_{0}= \calH \bfA_{0}^{HH} $  it remains to consider $ \calH^{\ast} \pi[\calH \bfA ] \phi $ which we write using \eqref{Q:dec}, \eqref{Q:dec2} as
$$  \calH^{\ast} \pi[\calH \bfA( \phi^{1},  \phi^{2}) ] \phi =\mathcal{Q}^1+\mathcal{Q}^2+\mathcal{Q}^3 $$
where
\begin{align*}
	\mathcal{Q}^1 :=&   \calH^{\ast} ( \Box^{-1} \calH  \mathfrak{I}  (\phi^1 \pt_{\al} \bar{\phi^2})\cdot \partial^{\al}\phi) , \\
	\mathcal{Q}^2 :=& -  \calH^{\ast} (  \calH \Delta^{-1} \Box^{-1} \pt_t \pt_{\al} \mathfrak{I}  (\phi^1 \pt_{\al} \bar{\phi^2})\cdot \partial_{t}\phi ), \\
	\mathcal{Q}^3  :=& -  \calH^{\ast} ( \calH \Delta^{-1} \Box^{-1} \pt_{\al} \pt^i  \mathfrak{I}  (\phi^1 \pt_{i} \bar{\phi^2})\cdot \partial^{\al}\phi ).
\end{align*}
and it remains to prove 
\be  \label{eq:tri-Q}
\vn{\mathcal{Q}^i(\phi^1,\phi^2,\phi)}_{\bar{N}}  \ls \vn{\phi^1}_{\bar{S}^1} \vn{\phi^2}_{\bar{S}^1} \vn{\phi}_{\bar{S}^1} 
, \qquad i=1,3; \ (d=4).
\ee

\

\subsection{Proof of \eqref{eq:tri-Q} for $ \mathcal{Q}^1 $}

Fix $ k, k_1, k_2 \geq 0 $ and let $ k_{\min}=\min(k,k_1,k_2) \geq 0 $. The estimate follows from
\be \label{Q1:trilinear}
\vn{ \sum_{k' < k_{\min}-C} \sum_{j<k'+C} \mathcal{Q}^1_{j,k'} (\phi^1_{k_1},\phi^2_{k_2},\phi_k)}_{N_k} \ls \vn{\phi^1_{k_1}}_{\bar{S}^1_{k_1}} \vn{\phi^2_{k_2}}_{\bar{S}^1_{k_2}} \vn{\phi_{k}}_{\bar{S}^1_{k}} 
\ee
by summing in $ k_1=k_2+O(1) $, where
$$ \mathcal{Q}^1_{j,k'} (\phi^1_{k_1},\phi^2_{k_2},\phi_k)= \bar{Q}_{<j}[P_{k'}Q_j \Box^{-1} (\bar{Q}_{<j}\phi^1_{k_1} \pt_{\al} \bar{Q}_{<j} \phi^2_{k_2}) \cdot \pt^{\al} \bar{Q}_{<j} \phi_k].     	 $$
Define $ l \in [-k_{\min}, C] $ by $ j=k'+2l $ which implies $ \angle(\phi_k, P_{k'}A),  \angle(\phi^2_{k_2}, P_{k'}A) \ls 2^l $. When $ k_{\min}=0 $ we may set $ l=0 $ and similarly for $ l_0 $ below.

In proving \eqref{Q1:trilinear}, we make the normalization 
\be \label{normalz}
\vn{\phi^1_{k_1}}_{\bar{S}_{k_1}^1}=1, \qquad  \vn{\phi^2_{k_2}}_{\bar{S}^1_{k_2}}=1,\qquad \vn{\phi_{k}}_{\bar{S}^1_{k}}=1. 
\ee

Since we have a null form between $ \phi^2 $ and $ \phi $ we use a bilinear partition of unity based on their angular separation:

$$ \mathcal{Q}^1_{j,k'}=\sum_{l_0+C<l'<l} \sum_{\substack{\omega_1,\omega_2 \in \Gamma(l') \\ \angle(\omega_1,\omega_2) \simeq 2^{l'} }} \mathcal{Q}^1_{j,k'} (\phi^1_{k_1},P_{l'}^{\omega_2} \phi^2_{k_2},P_{l'}^{\omega_1} \phi_k)+ \sum_{\substack{\omega_1,\omega_2 \in \Gamma(l_0) \\ \angle(\omega_1,\omega_2) \ls 2^{l_0} } } \mathcal{Q}^1_{j,k'} (\phi^1_{k_1},P_{l_0}^{\omega_2} \phi^2_{k_2},P_{l_0}^{\omega_1} \phi_k)
$$
where $ l_0 \defeq \max(-k_{\min},l+k'-k_{\min},\frac{1}{2}(j-k_{\min})) $ and the angle $ \angle(\omega_1,\omega_2) $ is taken $ \mod \pi $.   Notice that the sums in $ \omega_1,\omega_2 $ are essentially diagonal. In each summand, we may insert $ P_l^{[\omega_1]} $ in front of $ P_{k'}Q_j \Box^{-1} $,  where $ P_l^{[\omega_1]} $ is uniquely (up to $ O(1)$) defined by $ \omega_1 $ (or $ \omega_2 $).

For the first sum, for $ k_{\min}>0 $, for any $ l' \in [l_0+C,l] $ we will prove
\be  \label{NullFramesEstimate}
\sum_{\omega_1,\omega_2} \vn{\mathcal{Q}^1_{j,k'} (\phi^1_{k_1},P_{l'}^{\omega_2} \phi^2_{k_2},P_{l'}^{\omega_1} \phi_k)}_{L^1L^2} \ls 2^{\frac{1}{4}(l'+l)} 2^{\frac{1}{2}(k'-k_2)} 
\ee
by employing the null-frame estimate in Corollary \ref{L2:NFnullFrames:cor}, which takes advantage of the angular separation. Summing in $ l',j,k' $ we obtain part of \eqref{Q1:trilinear}.

At small angles however, one does not control the null-frame norms for Klein-Gordon waves and the null-form gives only a limited gain. We consider two cases.

For $ j \geq -k_{\min} $ we sum the following in $ j,k' $
\be  \label{SmallAnglesLargeMod}
\sum_{\omega_1,\omega_2} \vn{ \mathcal{Q}^1_{j,k'} (\phi^1_{k_1},P_{l_0}^{\omega_2} \phi^2_{k_2},P_{l_0}^{\omega_1} \phi_k) }_{L^1L^2} \ls 2^l 2^{k'-k_{\min}}  
\ee

When $ j \leq -k_{\min} $ (thus $ k' \leq -k_{\min}-2l $ and $ l_0=-k_{\min} $) the operator $ P_{k'}Q_j \Box^{-1} $ becomes more singular and we encounter a logarithmic divergence if we try to sum $ k', j $ outside the norm in \eqref{Q1:trilinear}. We proceed as follows. We write
$$ \mathcal{Q}^1_{j,k'} (\phi^1_{k_1},P_{l_0}^{\omega_2} \phi^2_{k_2},P_{l_0}^{\omega_1} \phi_k)= \bar{Q}_{<j}[P_l^{[\omega_1]} P_{k'}Q_j \Box^{-1} (\bar{Q}_{<j}\phi^1_{k_1} \pt_{\al} \bar{Q}_{<j} P_{l_0}^{\omega_2} \phi^2_{k_2}) \cdot \pt^{\al} \bar{Q}_{<j}P_{l_0}^{\omega_1} \phi_k] $$
We define
$$ \tilde{\mathcal{Q}}^{\omega_i}_{j,k'}=P_l^{[\omega_1]} P_{k'}Q_j \Box^{-1} (\phi^1_{k_1} \pt_{\al} \bar{Q}_{<k_{2}+2l_0} P_{l_0}^{\omega_2} \phi^2_{k_2}) \cdot \pt^{\al} \bar{Q}_{<k+2l_0}P_{l_0}^{\omega_1} \phi_k   $$
and we shall prove, using the embeddings in Prop. \ref{Box:Embedding}, that for any $ l \in [-k_{\min}, C] $
\be \label{SmallAngleSmallMod}
\sum_{\omega_1,\omega_2} \vn{\sum_{k' \leq -k_{\min}-2l } \tilde{\mathcal{Q}}^{\omega_i}_{k'+2l,k'}}_{L^1L^2} \ls 2^{-\frac{1}{2}(l+k_{\min})} ,
\ee
which sums up (in $ l $) towards the rest of \eqref{Q1:trilinear} except for the remainders
$$ \tilde{\mathcal{Q}}^{\omega_i}_{j,k'}- \mathcal{Q}^1_{j,k'} (\phi^1_{k_1},P_{l_0}^{\omega_2} \phi^2_{k_2},P_{l_0}^{\omega_1} \phi_k)=\mathcal{R}^{1,\omega_i}_{j,k'}+\mathcal{R}^{2,\omega_i}_{j,k'}+\mathcal{R}^{3,\omega_i}_{j,k'}+\mathcal{R}^{4,\omega_i}_{j,k'} $$
for which we have
\be  \label{SmallAngleRemainders}
\sum_{\omega_1,\omega_2} \vn{\mathcal{R}^{i,\omega}_{j,k'}}_{N_k} \ls 2^{\frac{l}{2}} 2^{\frac{1}{2}(k'-k_2)} , \qquad i=1,4 
\ee
where
\begin{align*}
	 \mathcal{R}^{1,\omega_i}_{j,k'} & :=   \bar{Q}_{>j}[  P_l^{[\omega_1]} P_{k'}Q_j \Box^{-1} (\bar{Q}_{<j}\phi^1_{k_1}  \pt_{\al} \bar{Q}_{<j} P_{l_0}^{\omega_2} \phi^2_{k_2}) \cdot  \pt^{\al} \bar{Q}_{<j}P_{l_0}^{\omega_1} \phi_k] , \\
	\mathcal{R}^{2,\omega_i}_{j,k'} & :=    P_l^{[\omega_1]} P_{k'}Q_j \Box^{-1} ( \bar{Q}_{<j}\phi^1_{k_1}  \pt_{\al} \bar{Q}_{<j} P_{l_0}^{\omega_2} \phi^2_{k_2}) \cdot  \pt^{\al} \bar{Q}_{[j,k-2k_{\min}]}P_{l_0}^{\omega_1} \phi_k , \\
	\mathcal{R}^{3,\omega_i}_{j,k'} & :=   P_l^{[\omega_1]} P_{k'}Q_j \Box^{-1} ( \bar{Q}_{>j}\phi^1_{k_1}  \pt_{\al} \bar{Q}_{<j} P_{l_0}^{\omega_2} \phi^2_{k_2}) \cdot  \pt^{\al} \bar{Q}_{<k-2k_{\min}}P_{l_0}^{\omega_1} \phi_k , \\
	\mathcal{R}^{4,\omega_i}_{j,k'} & :=    P_l^{[\omega_1]} P_{k'}Q_j \Box^{-1} ( \phi^1_{k_1}  \pt_{\al} \bar{Q}_{[j,k_2-2k_{\min}]} P_{l_0}^{\omega_2} \phi^2_{k_2}) \cdot  \pt^{\al} \bar{Q}_{<k-2k_{\min}}P_{l_0}^{\omega_1} \phi_k. \\
\end{align*}
Summing in $ j,k' $ we obtain the rest of \eqref{Q1:trilinear}.

\subsection{Proof of \eqref{NullFramesEstimate} and \eqref{SmallAnglesLargeMod}}
We are in the case $ k_1=k_2+O(1) $, $ k=\tilde{k}+O(1) $, $ k'+C<k_{\min}=\min(k,k_1,k_2) > 0 $, $ j=k'+2l  $. We prove \footnote{Notice that this case does not occur when $ k_{\min}=0 $.}
\begin{multline*} 
 | \lng  P_l^{[\omega_1]} P_{k'}Q_j \Box^{-1} (\bar{Q}_{<j}\phi^1_{k_1}  \pt_{\al} \bar{Q}_{<j} P_{l'/l_0}^{\omega_2} \phi^2_{k_2}) , \tilde{P}_{k'} \tilde{Q}_j (  \pt^{\al} \bar{Q}_{<j} P_{l'/l_0}^{\omega_1} \phi_k \cdot \bar{Q}_{<j} \psi_{\tilde{k}} \rng | \\
\ls M_{\omega_1,\omega_2} \vn{\psi_{\tilde{k}}}_{L^{\infty} L^2}
%, \qquad \sum_{\omega_1,\omega_2}  M_{\omega_1,\omega_2} \ls  RHS \eqref{NullFramesEstimate}/  \eqref{SmallAnglesLargeMod}, 
\end{multline*}
where $ M_{\omega_1,\omega_2} $ will be defined below.

The two products above are summed over diametrically opposed boxes $ \pm \calC $ [resp. $ \pm \calC'  $ ] of size $ \simeq 2^{k'} \times (2^{k'+l})^3 $ included in the angular caps $ \calC_{l'/l_0}^{\omega_2} $ [resp. $\calC_{l'/l_0}^{\omega_1}$] where $ P_{l'/l_0}^{\omega_2} $ [resp. $ P_{l'/l_0}^{\omega_1} $] are supported (Lemma \ref{geom:cone}).

Note that $ 2^{j+k'}  P_l^{[\omega_1]} P_{k'}Q_j \Box^{-1} $ acts by convolution with an integrable kernel. By a simple argument based on translation-invariance we may dispose of this operator (after first making the the $ \calC, \calC'  $ summation).

{\bf Step~1:~Proof of \eqref{NullFramesEstimate}}

In this case the null form gains $ 2^{2l'} $. It suffices to show, having normalized \eqref{normalz}
\be  \label{NullFramesEstimateDual}
 2^{-j-k'} | \lng (\bar{Q}_{<j}\phi^1_{k_1}  \pt_{\al} \bar{Q}_{<j}^{\pm \pm'} P_{l'}^{\omega_2} \phi^2_{k_2}) , (  \pt^{\al} \bar{Q}_{<j}^{\pm} P_{l'}^{\omega_1} \phi_k \cdot \bar{Q}_{<j} \psi_{\tilde{k}} \rng | \ls M_{\omega_1,\omega_2} \vn{\psi_{\tilde{k}}}_{L^{\infty} L^2} \ee
 \be \label{SumedNullFramesEstimateDual}
 \sum_{\omega_1,\omega_2}  M_{\omega_1,\omega_2} \ls  2^{\frac{1}{4}(l'+l)} 2^{\frac{1}{2}(k'-k_2)}  \ee
where $ \angle(\omega_1, \pm' \omega_2 ) \simeq 2^{l'} $. We write $ 2^{j+k'} \text{LHS} \eqref{NullFramesEstimateDual}  \ls $
\begin{align*} 
& \int \sum_{\substack{ \calC \subset \calC_{l'}^{\omega_2} \\ \calC' \subset \calC_{l'}^{\omega_1} }}  \vn{P_{-\calC} \bar{Q}_{<j}\phi^1_{k_1}}_{L^{\infty}_x}  \vn{\pt_{\al} P_{\calC} \bar{Q}_{<j}^{\pm \pm'}\phi^2_{k_2}  \pt^{\al} P_{\calC'} \bar{Q}_{<j}^{\pm} \phi_{k}}_{L^2_x} \vn{P_{-\calC'}  \bar{Q}_{<j} \psi_{\tilde{k}} }_{L^2_x} (t)  \dd t  \ls \\
& \int (\sum_{\calC \subset \calC_{l'}^{\omega_2}} \vn{P_{-\calC} \bar{Q}_{<j}\phi^1_{k_1}}_{L^{\infty}_x}^2)^{\frac{1}{2}} ( \sum_{\substack{ \calC \subset \calC_{l'}^{\omega_2} \\ \calC' \subset \calC_{l'}^{\omega_1} }}\vn{\pt_{\al} P_{\calC} \bar{Q}_{<j}^{\pm \pm'} \phi^2_{k_2}  \pt^{\al} P_{\calC'} \bar{Q}_{<j}^{\pm} \phi_{k}}_{L^2_x}^2   )^{\frac{1}{2}}  \vn{ \bar{Q}_{<j} \psi_{\tilde{k}}   }_{L^2_x} (t) \dd t \\
& \ls (\sum_{\calC \subset \calC_{l'}^{\omega_2}} \vn{P_{-\calC} \bar{Q}_{<j}\phi^1_{k_1}}_{L^2 L^{\infty}}^2)^{\frac{1}{2}} \ \mathcal{I}_{\omega_1,\omega_2}(l') \vn{ \bar{Q}_{<j} \ \psi_{\tilde{k}}   }_{L^{\infty} L^2}
\end{align*}
where, using Corollary \ref{L2:NFnullFrames:cor}, 
\begin{align*} \mathcal{I}_{\omega_1,\omega_2}(l')^2 & \defeq \sum_{\calC \subset \calC_{l'}^{\omega_2}} \sum_{\calC' \subset \calC_{l'}^{\omega_1}} \vn{\pt_{\al} P_{\calC} \bar{Q}_{<j}^{\pm \pm'} \phi^2_{k_2} \cdot  \pt^{\al} P_{\calC'} \bar{Q}_{<j}^{\pm} \phi_{k}}_{L^2_{t,x}}^2 \\
& \ls 2^{l'} (  \sum_{\calC \subset \calC_{l'}^{\omega_2}} \vn{P_{\calC} \bar{Q}_{<j}^{\pm \pm'} \nabla_{t,x}\phi^2_{k_2}}_{PW_\calC^{{\pm \pm'}}}^2 )  ( \sum_{\calC' \subset \calC_{l'}^{\omega_1}} \vn{ P_{\calC'} \bar{Q}_{<j}^{\pm} \nabla_{t,x} \phi_{k}}_{NE_{\calC'}^{\pm}}^2).
\end{align*}
Thus in \eqref{NullFramesEstimateDual} we may take
$$ M_{\omega_1,\omega_2}=2^{-j-k'} (\sum_{\calC \subset \calC_{l'}^{\omega_2}} \vn{P_{-\calC} \bar{Q}_{<j}\phi^1_{k_1}}_{L^2 L^{\infty}}^2)^{\frac{1}{2}} \ \mathcal{I}_{\omega_1,\omega_2}(l') $$
and by Summing in $ \omega_2 $ (the $ \omega_1$ sum is redundant) using C-S we have
$$  \sum_{\omega_1,\omega_2}  M_{\omega_1,\omega_2} \ls 2^{-2l-2k'} \cdot 2^{l'} \cdot 2^{\frac{1}{2}l}2^{k'}2^{-\frac{1}{2}k_1} \cdot 2^{\frac{3}{2}(k'+l)}  $$
which implies \eqref{SumedNullFramesEstimateDual}.

{\bf Step~2:~Proof of \eqref{SmallAnglesLargeMod}} Here $ j \geq-k_{\min} $.
In this case the null form gains $ 2^{j-k_{\min}} $ and
$$ 2^{l_0}=\max( 2^{-k_{\min}},2^l 2^{k'-k_{\min}}, 2^{\frac{1}{2}(j-k_{\min})}) \leq 2^l $$ 
By Prop. \ref{N0:form} and Remark \ref{NF:remark} it suffices to prove, under \eqref{normalz}
\begin{multline}  \label{SmallAnglesLargeModDual}
2^{j-k_{\min}} | \lng  P_l^{[\omega_1]} P_{k'}Q_j \Box^{-1} (\bar{Q}_{<j}\phi^1_{k_1}  \nabla_{t,x} \bar{Q}_{<j} P_{l_0}^{\omega_2} \phi^2_{k_2}) , \tilde{P}_{k'} \tilde{Q}_j (  \nabla_{t,x} \bar{Q}_{<j}P_{l_0}^{\omega_1} \phi_k \cdot \bar{Q}_{<j} \psi_{\tilde{k}} \rng | \\
\ls M_{\omega_1,\omega_2} \vn{\psi_{\tilde{k}}}_{L^{\infty} L^2}, \qquad \sum_{\omega_1,\omega_2}  M_{\omega_1,\omega_2} \ls  2^l 2^{k'-k_{\min}}.
\end{multline}
We have
\begin{align*} \text{LHS} \ \eqref{SmallAnglesLargeModDual} \ls & 2^{-k_{\min}-k'} 2^{k_2} \sum_{\calC \subset \calC_{l_0}^{\omega_2}} \vn{P_{\calC} \bar{Q}_{<j}\phi^1_{k_1}}_{L^2 L^{\infty}}  \vn{P_{-\calC} \bar{Q}_{<j}\phi^2_{k_2}}_{L^2 L^{\infty}} \\
& \times \sup_{t} \sum_{\calC' \subset \calC_{l_0}^{\omega_1}} \vn{P_{\calC'} \bar{Q}_{<j} \nabla \phi_{k}(t)}_{L^2_x} \vn{P_{-\calC'}  \bar{Q}_{<j} \psi_{\tilde{k}} (t)  }_{L^2_x}  \ls  M_{\omega_1,\omega_2} \vn{\psi_{\tilde{k}}}_{L^{\infty} L^2}
\end{align*}
where for each $ t $ we have used Cauchy-Schwarz and orthogonality, where
$$ M_{\omega_1,\omega_2}= 2^{-k_{\min}-k'} 2^{k_2} ( \sum_{\calC \subset \calC_{l_0}^{\omega_2}}  \vn{P_{\calC} \bar{Q}_{<j}\phi^1_{k_1}}_{L^2 L^{\infty}} ^2)^{\frac{1}{2}} ( \sum_{\calC \subset \calC_{l_0}^{\omega_2}}  \vn{P_{-\calC} \bar{Q}_{<j}\phi^2_{k_2}}_{L^2 L^{\infty}}^2)^{\frac{1}{2}} \vn{\nabla \phi_{k}}_{L^{\infty} L^2}
$$
Summing in $ \omega_2 $ (the $ \omega_1$ sum is redundant) using C-S and  \eqref{fe-L2Linfty}, \eqref{fe-LinfL2} we get \eqref{SmallAnglesLargeModDual}.

\subsection{Proof of \eqref{SmallAngleSmallMod} } Recall that $ l \in [-k_{\min},C], \ k_{\min}=\min(k,k_1,k_2) \geq 0 $ and $ k_1=k_2+O(1) $ are fixed. We are in the case $ k'+2l=j \leq -k_{\min} $, thus $ l_0=-k_{\min} $, i.e. $ \angle (\omega_1,\omega_2) \ls 2^{-k_{\min}} $.

By Prop. \ref{N0:form} and Remark \ref{NF:remark} the null form gains $ 2^{-2k_{\min}} $. We can apply that proposition because $ \bar{Q}_{<k_i-2k_{\min}}=Q_{<k_i-2k_{\min}+C}  \bar{Q}_{<k_i-2k_{\min}} $.  

We will apply Prop. \ref{Box:Embedding}. For $ M=-k_{\min}-2l $, we write    
\be \label{interm:sum}
\sum_{k' \leq M} \tilde{\mathcal{Q}}^{\omega_i}_{k'+2l,k'}=\sum_{\pm} T_{l}^{[\omega_1]} ( \phi^1_{k_1} \pt_{\al} \bar{Q}_{<k_{2}+2l_0} P_{l_0}^{\omega_2} \phi^2_{k_2}) \cdot \pt^{\al} \bar{Q}_{<k+2l_0}P_{l_0}^{\omega_1} \phi_k
\ee

We consider two cases.

{\bf Case~1:~ $ k_{\min}=k_2+O(1) $.} The null form gains $ 2^{-2k_2} $. We have
$$ \vn{\eqref{interm:sum}}_{L^1L^2} \ls 2^{-2k_2} \vn{T_{l}^{[\omega_1]} (\phi^1_{k_1} \bar{Q}_{<k_{2}+2l_0} P_{l_0}^{\omega_2} \nabla\phi^2_{k_2})}_{L^1 L^{\infty}}  \vn{P_{l_0}^{\omega_1} \bar{Q}_{<k+2l_0} \nabla \phi_k}_{L^{\infty} L^2}. $$
Using \eqref{L1Linf:emb} and \eqref{Lorentz:Holder} we have
$$  \vn{T_{l}^{[\omega_1]} (\phi^1_{k_1} \bar{Q}_{<k_{2}+2l_0} P_{l_0}^{\omega_2} \nabla\phi^2_{k_2})}_{L^1 L^{\infty}} \ls 2^{-\frac{1}{2}l} 2^{k_2} \vn{\phi^1_{k_1}}_{L^2 L^{4,2}} \vn{P_{l_0}^{\omega_2}\bar{Q}_{<k_{2}+2l_0}  \phi^2_{k_2}}_{L^2 L^{4,2}}  $$
Summing (diagonally) in $ \omega_1, \omega_2 $ we obtain \eqref{SmallAngleSmallMod} since $ \vn{\phi^1_{k_1}}_{L^2 L^{4,2}} \ls 2^{-\frac{1}{4}k_1} \vn{\phi^1_{k_1}}_{\bar{S}^1_{k_1}} $, 
$$
\big( \sum_{\omega} \vn{P_{l_0}^{\omega}\bar{Q}_{<k_{2}+2l_0}  \phi^2_{k_2}}_{L^2 L^{4,2}}^2 \big) ^{\frac{1}{2}} \ls 2^{-\frac{1}{4}k_2} \vn{\phi^2_{k_2}}_{\bar{S}^1_{k_2}}  $$
\be \label{fe:sum:LinfL2}
 \big( \sum_{\omega}    \vn{P_{l_0}^{\omega} \bar{Q}_{<k+2l_0} \nabla_{t,x} \phi_k}_{L^{\infty} L^2} \big) ^{\frac{1}{2}} \ls \vn{\phi_{k}}_{\bar{S}^1_{k}} .
\ee		
{\bf Case~2:~ $ k_{\min}=k $.} Now the null form gains $ 2^{-2k} $, so we can put $ \phi_k $ in $ L^2 L^4 $.
$$  \vn{\eqref{interm:sum}}_{L^1L^2} \ls 2^{-2k} \vn{T_{l}^{[\omega_1]} (\phi^1_{k_1} \bar{Q}_{<k_{2}+2l_0} P_{l_0}^{\omega_2} \nabla\phi^2_{k_2})}_{L^2 L^4}  \vn{P_{l_0}^{\omega_1} \bar{Q}_{<k+2l_0} \nabla \phi_k}_{L^2 L^4}. $$
Using \eqref{L2L4:emb} and H\"older's inequality we have
$$ \vn{T_{l}^{[\omega_1]} (\phi^1_{k_1} \bar{Q}_{<k_{2}+2l_0} P_{l_0}^{\omega_2} \nabla\phi^2_{k_2})}_{L^2 L^4} \ls 2^{-\frac{1}{2}l} 2^{k_2} \vn{\phi^1_{k_1}}_{L^4 L^{\frac{8}{3}}} \vn{P_{l_0}^{\omega_2}\bar{Q}_{<k_{2}+2l_0}  \phi^2_{k_2}}_{L^4 L^{\frac{8}{3}}} $$ 
Summing (diagonally) in $ \omega_1, \omega_2 $ we obtain \eqref{SmallAngleSmallMod} since $ \vn{\phi^1_{k_1}}_{L^4 L^{\frac{8}{3}}} \ls 2^{-\frac{5}{8}k_1} \vn{\phi^1_{k_1}}_{\bar{S}^1_{k_1}}  $, 
$$
\big( \sum_{\omega} \vn{P_{l_0}^{\omega}\bar{Q}_{<k_{2}+2l_0}  \phi^2_{k_2}}_{L^4 L^{\frac{8}{3}}}^2 \big) ^{\frac{1}{2}} \ls 2^{-\frac{5}{8}k_2} \vn{\phi^2_{k_2}}_{\bar{S}^1_{k_2}} 
$$
$$ \big( \sum_{\omega}    \vn{P_{l_0}^{\omega} \bar{Q}_{<k+2l_0} \nabla_{t,x} \phi_k}_{L^{2} L^4} \big) ^{\frac{1}{2}} \ls 2^{\frac{3}{4}k} \vn{\phi_{k}}_{\bar{S}^1_{k}}.
$$ 

\subsection{Proof of \eqref{SmallAngleRemainders}} By Prop. \ref{N0:form} and Remark \ref{NF:remark} the null form gains $ 2^{-2k_{\min}} $. 

{\bf Step~1: $\mathcal{R}^{1} $ and $\mathcal{R}^{2} $.} Denoting 
$$ h^{\omega_i}=P_l^{[\omega_1]} P_{k'}Q_j \Box^{-1} ( \bar{Q}_{<j}\phi^1_{k_1}  \pt_{\al} \bar{Q}_{<j} P_{l_0}^{\omega_2} \phi^2_{k_2}), $$
we estimate using Bernstein and Prop. \ref{prop:no-nf} 
$$ \vn{h^{\omega_i}}_{L^2 L^{\infty}} \ls 2^{2k'+\frac{3}{2}l} \vn{h^{\omega_i}}_{L^2_{t,x}} \ls 2^{-j-k'}  2^{2k'+\frac{3}{2}l} \big(\sum_{\calC} \vn{P_{\calC} \bar{Q}_{<j} \phi^1_{k_1}}_{L^2 L^{\infty}}^2 \big)^{\frac{1}{2}} \vn{P_{l_0}^{\omega_2} \bar{Q}_{<j}  \nabla \phi^2_{k_2}}_{L^{\infty} L^2}  $$
where $ \calC=C_{k'}(l) $. Using the $ \bar{X}^{-\frac{1}{2}}_1 $ norm, we have
$$ \vn{\mathcal{R}^{1,\omega_i}_{j,k'}}_{N_k} \ls 2^{-\frac{j}{2}} 2^{-2k_{\min}}  \vn{h^{\omega_i}}_{L^2 L^{\infty}} \vn{P_{l_0}^{\omega_1} \bar{Q}_{<j}\nabla \phi_k}_{L^{\infty} L^2} $$
$$ \vn{\mathcal{R}^{2,\omega_i}_{j,k'}}_{L^1 L^2} \ls  2^{-2k_{\min}}   \vn{h^{\omega_i}}_{L^2 L^{\infty}} \vn{P_{l_0}^{\omega_1} \bar{Q}_{[j,k-2k_{\min}]} \nabla \phi_k}_{L^2_{t,x}} $$
Summing in $ \omega_1,\omega_2 $, we obtain \eqref{SmallAngleRemainders} for $ \mathcal{R}^{1} $, $\mathcal{R}^{2} $ by using \eqref{fe-L2Linfty} for $ \phi^1 $ and \eqref{fe:sum:LinfL2} for $ \phi^2 $ (first introducing $ \bar{Q}_{<k_2+2l_0} $ and discarding $ \bar{Q}_{<j} $), and \eqref{fe:sum:LinfL2}, respectively \eqref{fe-L2} for $ \phi $. We also use $ 2^{-k_{\min}} \ls 2^l $.

{\bf Step~2:~ $\mathcal{R}^{3} $ and $\mathcal{R}^{4} $.} We denote
\begin{align*} h^{\omega_{1,2}}_3 &= P_l^{[\omega_1]} P_{k'}Q_j \Box^{-1} ( \bar{Q}_{>j}\phi^1_{k_1}  \pt_{\al} \bar{Q}_{<j} P_{l_0}^{\omega_2} \phi^2_{k_2}) , \\ 
h^{\omega_{1,2}}_4 &=P_l^{[\omega_1]} P_{k'}Q_j \Box^{-1} ( \phi^1_{k_1}  \pt_{\al} \bar{Q}_{[j,k_2-2k_{\min}]} P_{l_0}^{\omega_2} \phi^2_{k_2}).
\end{align*}
For $ i=3,4 $ we have
$$ \vn{\mathcal{R}^{i,\omega}_{j,k'}}_{L^1 L^2} \ls   2^{-2k_{\min}}   \vn{h^{\omega_{1,2}}_{i}}_{L^1 L^{\infty}} \vn{P_{l_0}^{\omega_1} \bar{Q}_{<k-2k_{\min}} \nabla \phi_k}_{L^{\infty} L^2}
$$
We estimate using Prop. \ref{prop:no-nf}
$$ \vn{h^{\omega_{1,2}}_3}_{L^1 L^{\infty}} \ls 2^{2k'+\frac{3}{2}l} \vn{h^{\omega_{1,2}}_3}_{L^1 L^2} \ls 2^{-j-k'}  2^{2k'+\frac{3}{2}l} \vn{\bar{Q}_{>j} \phi^1_{k_1}}_{L^2_{t,x}} \big(\sum_{\calC} \vn{P_{\calC} \bar{Q}_{<j} \nabla \phi^2_{k_2}}_{L^2 L^{\infty}}^2 \big)^{\frac{1}{2}}  $$
where $ \calC=C_{k'}(0) $. Reversing the roles of $ \phi^1, \phi^2 $, $\vn{h^{\omega_{1,2}}_4}_{L^1 L^{\infty}} $ is also estimated.

Summing in $ \omega_1,\omega_2 $, we obtain \eqref{SmallAngleRemainders} for $ \mathcal{R}^{3} $, $\mathcal{R}^{4} $ by using \eqref{fe:sum:LinfL2} for $ \phi  $ and \eqref{fe-L2} \eqref{fe-L2Linfty},  for $ \phi^1, \phi^2 $. We also use $ 2^{-k_{\min}} \ls 2^l $.

\subsection{Proof of \eqref{eq:tri-Q} for $ \mathcal{Q}^2 $} Estimating in $ L^1 L^2$ we use H\" older's inequality with $ \vn{\pt_t \phi_k}_{L^{\infty}L^2} \ls  \vn{\phi_{k}}_{\bar{S}^1_{k}} $ and 
\be \label{Q2:bilest} \vn{\Delta^{-1} \Box^{-1} \pt_t Q_{j} P_{k'} \pt_{\al}  ( \bar{Q}_{<j}\phi^1_{k_1} \cdot \pt^{\al} \bar{Q}_{<j} \phi^2_{k_2})}_{L^1 L^{\infty}} \ls 2^l 2^{\frac{1}{2}(k'-k_1)} \vn{\phi^1_{k_1}}_{\bar{S}^1_{k_1}}  \vn{\phi^2_{k_2}}_{\bar{S}^1_{k_2}} \ee 
The $ \mathcal{Q}^2 $ part of  \eqref{eq:tri-Q} follows by summing this in $ k', j$, where
\be \label{Q23:cond}
k_1 =k_2+O(1),\quad k'+C<k_1,k, \quad j <k'+C, \quad l  \defeq \frac{1}{2}(j-k')_{-} \geq -k_1,k
\ee
To prove \eqref{Q2:bilest}, first note that the product is summed over diametrically opposed boxes $ \calC_1,\calC_2 $ of size $ \simeq 2^{k'} \times (2^{k'+l})^3 $ (Lemma \ref{geom:cone}). Each term in the sum forces a localization $ P_l^{\omega} $ in front of $ Q_{j} P_{k'} $ and note that $ 2^{j+k'}  P_l^{\omega} Q_j  P_{k'}\Box^{-1} $ is disposable.

Now recall for \eqref{M:form} the decomposition \eqref{M:form:decom}-\eqref{n0:form:symb}. By Prop. \ref {n0:form:prop}, and the fact that here $ \angle(\calC_1,-\calC_2) \ls 2^{l+k'-k_1} $ we have
$$  \vn{\Delta^{-1} \Box^{-1} \pt_t Q_{j} P_{k'}  \calN_0 ( \bar{Q}_{<j}\phi^1_{k_1},\bar{Q}_{<j} \phi^2_{k_2})}_{L^1 L^{\infty}} \ls 2^{-j-2k'} \times (2^{2l+2k'} ) \times $$
$$  \times \big(\sum_{\calC_1} \vn{P_{\calC_1} \bar{Q}_{<j} \phi^1_{k_1}}_{L^2 L^{\infty}}^2 \big)^{\frac{1}{2}}  \big(\sum_{\calC_2} \vn{P_{\calC_2} \bar{Q}_{<j} \phi^2_{k_2}}_{L^2 L^{\infty}}^2 \big)^{\frac{1}{2}}
$$ 
The same holds true for $ \calM_0 $ , since now, by Prop. \ref{M0:form} we gain $ 2^{2k'-2k_1} \ls 2^{2k'+2l} $. Using \eqref{fe-L2Linfty} we obtain \eqref{Q2:bilest} for $ \calN_0, \calM_0 $. We turn to $ \calR_0^{\pm} $ and write
$$ 2^{k_2} \vn{\Delta^{-1} \Box^{-1} \pt_t Q_{j} P_{k'} \big( (\pt_t \mp i \jb{D}) \bar{Q}_{<j}^{\pm} \phi^1_{k_1} \cdot \bar{Q}_{<j}^{\mp} \phi^2_{k_2} \big)}_{L^1 L^{\infty}} \ls 2^{k_2} \cdot 2^{-j-2k'} \times $$
$$  \times \big(\sum_{\calC_1} \vn{P_{\calC_1} (\pt_t \mp i \jb{D}) \bar{Q}_{<j}^{\pm} \phi^1_{k_1}}_{L^2 L^{\infty}}^2 \big)^{\frac{1}{2}}  \big(\sum_{\calC_2} \vn{P_{\calC_2} \bar{Q}_{<j}^{\mp} \phi^2_{k_2}}_{L^2 L^{\infty}}^2 \big)^{\frac{1}{2}}
$$ 
Then we use \eqref{fe-sqX},  \eqref{fe-L2Linfty} to obtain $ \calR_0^{\pm} $-part of \eqref{Q2:bilest}. The other parts of $ \calR_0^{\pm} $ follow by reversing the roles of $ \phi^1,\phi^2 $.

\subsection{Proof of \eqref{eq:tri-Q} for $ \mathcal{Q}^3 $} Let $ k_1, k_2, k, k', j, l $ as in \eqref{Q23:cond} and $ \tilde{k}=k+O(1), \ k_{\min}  \defeq \min(k_1,k_2,k), \ k'=k''+O(1)<k_{\min} -C, \  j =j'+O(1)<k'+C. $ We prove
\be
\begin{aligned} \label{Q3:est}
| \lng \frac{\pt Q_{j'} P_{k'}}{\Delta \Box}  ( \bar{Q}_{<j'}\phi^1_{k_1} \cdot \pt \bar{Q}_{<j'} \phi^2_{k_2}) ,  Q_j P_{k''} \pt_{\al}(\pt^{\al} \bar{Q}_{<j} \phi_k \cdot \bar{Q}_{<j} \psi_{\tilde{k}} )\rng | \ls \\
\ls 2^{\frac{1}{2}l} 2^{\frac{1}{2}(k'-k_{\min})}  \vn{\phi^1_{k_1}}_{\bar{S}^1_{k_1}}   \vn{\phi^2_{k_2}}_{\bar{S}^1_{k_2}} \vn{\phi_{k}}_{\bar{S}^1_{k}}  \vn{\psi_{\tilde{k}}}_{N_k^{*}}
\end{aligned}
\ee
which, by duality, implies \eqref{eq:tri-Q} for $ \mathcal{Q}^3 $. Like for $ \mathcal{Q}^2 $, we sum over diametrically opposed boxes $ \pm \calC $ of size $ \simeq 2^{k'} \times (2^{k'+l})^3 $ and introduce $ P_l^{\omega} $ to bound $ \Box^{-1} $. 

First, using Prop. \ref{prop:no-nf} and \eqref{fe-LinfL2}, \eqref{fe-L2Linfty}, we estimate
$$ \vn{\Delta^{-1} \Box^{-1} \pt Q_{j'} P_{k'}   ( \bar{Q}_{<j'}\phi^1_{k_1} \cdot \pt \bar{Q}_{<j'} \phi^2_{k_2})}_{L^2_{t,x}} \ls 2^{-j-2k'} (2^{\frac{1}{2}l} 2^{k'}  2^{-\frac{1}{2}k_1}) \vn{\phi^1_{k_1}}_{\bar{S}^1_{k_1}}   \vn{\phi^2_{k_2}}_{\bar{S}^1_{k_2}}  $$

For the second product, we recall the decomposition \eqref{M:form}-\eqref{n0:form:symb}.
By Prop. \ref {n0:form:prop} and orthogonality, using the fact that $ \angle(\phi,\psi) \ls 2^{l+k'-k} $, we have
$$
\vn{Q_j P_{k''} \calN_0 (\bar{Q}_{<j} \phi_k,\bar{Q}_{<j} \psi_{\tilde{k}} ) }_{L^2_{t,x}} \ls 2^{2l+2k'} \big(\sum_{\calC} \vn{P_{\calC} \bar{Q}_{<j} \phi_k}_{L^2 L^{\infty}}^2 \big)^{\frac{1}{2}} \vn{\psi_{\tilde{k}}}_{L^{\infty} L^2}
$$
The same holds true for $ \calM_0 $ , since now, by Prop. \ref{M0:form} we gain $ 2^{2k'-2k} \ls 2^{2k'+2l} $.

For $ \calR_0^\pm $ we prove
\begin{align*}
& \vn{Q_j P_{k''}  \big( (\pt_t \mp i \jb{D}) \bar{Q}_{<j}^{\pm} \phi_k \cdot \bar{Q}_{<j}^{\mp} \psi_{\tilde{k}} \big) }_{L^2_{t,x}} \ls \big(\sum_{\calC} \vn{P_{\calC} (\pt_t \mp i \jb{D}) \bar{Q}_{<j}^{\pm} \phi_k}_{L^2 L^{\infty}}^2 \big)^{\frac{1}{2}} \vn{\psi_{\tilde{k}}}_{L^{\infty} L^2}, \\
& 2^k \vn{Q_j P_{k''} (\bar{Q}_{<j}^{\pm} \phi_k \cdot (\pt_t \pm i \jb{D}) \bar{Q}_{<j}^{\mp} \psi_{\tilde{k}} ) }_{L^2_{t,x}} \ls 2^k \vn{\bar{Q}_{<j}^{\pm} \phi_k}_{L^{\infty} L^2} \times \\
& \qquad \qquad \times \big(\sum_{\calC} \vn{P_{\calC} (\pt_t \pm i \jb{D}) \bar{Q}_{<j}^{\mp} \psi_{\tilde{k}}}_{L^2 L^{\infty}}^2 \big)^{\frac{1}{2}} \ls 2^k 2^{2k'+\frac{3}{2}l} 2^{\frac{1}{2}j} \vn{\phi_k}_{L^{\infty} L^2} \vn{\psi_{\tilde{k}}}_{\bar{X}^{\frac{1}{2}}_{\infty}}.
\end{align*}
where we have used Bernstein's inequality and orthogonality. 

Putting all of the above together, using \eqref{fe-L2Linfty}, \eqref{fe-sqX} and \eqref{fe-LinfL2}, we obtain \eqref{Q3:est}.

\section{The construction and properties of the phase} \label{Constr:phase}

\subsection{Motivation}

We begin by recalling some heuristic considerations motivating the construction in \cite{RT}, which also extends to the massive case. 

Suppose one is interested in solving the equation
\be \label{eq:cva}
\Box_m^A \phi=0, \qquad \Box_m^A\defeq D^{\al} D_{\al}+I
\ee
where $ D_{\al}\phi=(\pt_{\al}+iA_{\al}) \phi $ and $ \Box A=0 $. After solving \eqref{eq:cva}, one can also obtain solutions to the inhomogeneous equation $ \Box_m^A \phi=F $ by Duhamel's formula. The equation \eqref{eq:cva} enjoys the following gauge invariance. For any real function $ \psi $, replacing 
$$  \phi \mapsto e^{i \psi} \phi, \quad A_{\al} \mapsto A_{\al}-\pt_{\al} \psi, \quad D_{\al} \mapsto e^{i \psi} D_{\al} e^{-i \psi}
$$
we obtain another solution. To make use of this, one expects that by choosing $ \psi $ appropriately ($ \nabla \psi \approx A $) one could reduce closer to the free wave equation $ \Box \phi \approx 0 $.

However, this is not in general possible since $ A $ is not a conservative vector field. Instead, one makes the construction microlocally and for each dyadic frequency separately. Taking $ e^{i x \cdot \xi} $ as initial data, considering $ \phi=e^{-i \psi_{\pm}(t,x)} e^{\pm i t \jb{\xi}+i  x \cdot \xi} $ we compute
$$ \Box_m^A \phi=2 \lpr \pm \jb{\xi} \pt_t \psi_{\pm} - \xi \cdot \nabla \psi_{\pm} + A \cdot \xi \rpr \phi+ \big( -i \Box \psi_{\pm} +  (\pt_t \psi_{\pm})^2- \vm{\nabla \psi_{\pm}}^2 - A \cdot \nabla \psi_{\pm} \big) \phi $$

The second bracket is expected to be an error term, while for the first, one wants to define $ \psi_{\pm} $ so as to get as much cancelation as possible, while also avoiding to make $ \psi_{\pm} $ too singular. Defining
$$  L_{\pm}=\pm \pt_t + \frac{\xi}{\jb{\xi}} \cdot \nabla_x, \qquad \text{one has} $$
\be \label{nullvf} -L_{+} L_{-}=\Box + \la_{\omega^{\perp}}+\frac{1}{\jb{\xi}^2} (\omega \cdot \nabla_x) ^2, \quad \omega=\frac{\xi}{\vm{\xi}}.  \ee

We would like to have $ L_{\mp} \psi_{\pm}=A \cdot \xi / \jb{\xi} $ thus applying $ L_{\pm} $ and neglecting $ \Box $ in \eqref{nullvf} (since $ \Box A=0 $) one obtains, for fixed $ \xi $:
\be \psi_{\pm}(t,x)=\frac{-1}{\la_{\omega^{\perp}}+\frac{1}{\jb{\xi}^2} (\omega \cdot \nabla_x) ^2} L_{\pm} \lpr A(t,x) \cdot \frac{\xi}{\jb{\xi}} \rpr.  \label{appr:defn} \ee

Taking general initial data $ \int e^{i x \cdot \xi} h_{\pm}(\xi) \dd \xi $, using linearity, one obtains the approximate solutions
$$ \phi_{\pm}(t,x)=\int e^{-i \psi_{\pm}(t,x,\xi)} e^{\pm i t \jb{\xi}} e^{i x \cdot \xi} h_{\pm}(\xi) \dd \xi.
$$
Thus, the renormalization is done through the pseudodifferential operators $ e^{-i \psi_{\pm}}(t,x,D) $.

In what follows, $ \xi $ will be restricted to dyadic frequencies $ \vm{\xi} \simeq 2^k $ or $ \vm{\xi} \ls 1 $, while $ A(t,x) $ (and thus $ \psi $ too) will be localized to strictly lower frequencies $ \ll 2^k $. When $ \vm{\xi} \ls 1 $, the denominator in \eqref{appr:defn} is essentially $ \Delta_x $. If $ \xi $ is a high frequency then the dominant term is $ \la_{\omega^{\perp}}^{-1} $ and the construction needs to be refined to remove the singularity; see the next subsection for precise definitions.

For more details motivating the construction see \cite[sec. 7,8]{RT}.

\

The construction in \cite{KST} slightly differs from the one in \cite{RT} in that they further localize the exponentials in the $ (t,x) $-frequencies $  \big( e^{-i \psi_{\pm}(t,x,\xi)} \big)_{<k-c} $.
By Taylor expansion one can see that these constructions are essentially equivalent. Indeed, since
$$ e^{i \psi_{<k-c}(t,x,\xi)} =1+i \psi_{<k-c}(t,x,\xi)+O\big(\psi_{<k-c}^2(t,x,\xi) \big) $$
we see that they differ only by higher order terms, which are negligible due to the smallness assumption on $ A $. Here, following \cite{KST}, it will be technically convenient to do this localization. 
We denote by 
$$  e^{\pm i \psi_{\pm}^k}_{<h} (t,x,D), \qquad e^{\pm i \psi_{\pm}^k}_{<h} (D,s,y) $$
the left and right quantizations of the symbol $ e^{\pm i \psi_{\pm}^k}_{<h}(t,x,\xi) $  where the $ <h $  subscript denotes $ (t,x)$-frequency localization to frequencies $ \leq h-C $, pointwise in $ \xi $. Thus
\be \label{averaging}  e^{\pm i \psi_{\pm}^k}_{<h}(t,x,\xi)=\int_{\mb{R}^{d+1}} e^{\pm i T_{(s,y)} \psi_{\pm}^k(t,x,\xi)} m_h(s,y) \dd s \dd y  \ee
where $ T_{(s,y)} \psi(t,x,\xi)=\psi(t+s,x+y,\xi) $ and $ m_h=2^{(d+1)h} m(2^h \cdot) $ for a bump function $ m(s,y) $. By averaging arguments such as Lemmas \ref{loczsymb}, \ref{lemma:locz:symb}, estimates for $ e^{-i \psi_{\pm}}(t,x,D) $ will automatically transfer to $ e^{-i \psi_{\pm}}_{<k}(t,x,D) $.

\

\subsection{The construction of the phase}\

\

We recall that $ A $ is real-valued, it solves the free wave equation $ \Box A=0 $ and satisfies the Coulomb gauge condition $ \nabla_x \cdot A=0 $.

For $ k=0 $ we define
\be \label{defn1} \begin{aligned}
 \psi_{\pm}^0(t,x,\xi) & \defeq \sum_{j<-C} \psi^{0}_{j,\pm}(t,x,\xi), \qquad \text{where} \\
 \psi^{0}_{j,\pm}(t,x,\xi) & \defeq \frac{-L_{\pm}}{\la_{\omega^{\perp}}+\frac{1}{\jb{\xi}^2}(\omega \cdot \nabla_x) ^2} \lpr P_jA (t,x) \cdot \frac{\xi}{\jb{\xi}} \rpr 
\end{aligned} \ee
For $ k\geq 1$ we define 
\be  \label{defn2} \psi^{k} _{\pm}(t,x,\xi) \defeq   \frac{-1}{\la_{\omega^{\perp}}+\frac{1}{\jb{\xi}^2} (\omega \cdot \nabla_x) ^2} L_{\pm} \sum_{k_1 <k-c} \lpr \Pi^{\omega}_{> \delta(k_1-k)} P_{k_1} A \cdot \frac{\xi}{\jb{\xi}} \rpr  \ee
It will be convenient to rescale the angular pieces that define $ \psi^k_{\pm} $ to $ \vm{\xi} \simeq 1 $:
\be \label{phase_piece} \psi^{k}_{j,\theta,\pm} (t,x,2^k \xi) \defeq \frac{-L_{\pm,k}}{\la_{\omega^{\perp}}+2^{-2k} \frac{1}{\jb{\xi}_k^2} (\omega \cdot \nabla_x) ^2} \lpr \Pi^{\omega}_{\tht} P_{j} A \cdot \omega \frac{\vm{\xi}}{\jb{\xi}_k} \rpr \ee 
for $ 2^{\delta(j-k)}< \tht < c $ and $  j<k-c $, where 
$$ L_{\pm,k}= \pm \pt_t + \frac{\vm{\xi}}{\jb{\xi}_k} \omega \cdot \nabla_x, \qquad \omega=\frac{\xi}{\vm{\xi}}, \qquad \jb{\xi}_k=\sqrt{2^{-2k}+\vm{\xi}^2}.  $$

Note that $ \Pi^{\omega}_{\tht}, \Pi^{>\omega}_{\tht} $ defined in \eqref{sect:proj1}, \eqref{sect:proj2}  behave like Littlewood-Paley projections in the space $ \omega^{\perp} $.

\begin{remark} It will be important to keep in mind that $ \psi^k_{\pm} $ is real-valued, since it is defined by applying real and even Fourier multipliers to the real function $ A $. 
\end{remark}

\begin{remark}\label{rk:Col:nf} Due to the Coulomb condition $ \nabla_x \cdot A=0 $ the expression in \eqref{defn2} acts like a null form, leading to an angular gain. Indeed, a simple computation shows 
$$ \vm{\widehat{\Pi^{\omega}_{\tht} A} (\eta) \cdot \omega} \ls \tht  \vm{\widehat{\Pi^{\omega}_{\tht} A} (\eta)},  $$
which implies $ \vn{\Pi^{\omega}_{\tht} A \cdot \omega}_{L^2_x} \ls \tht  \vn{\Pi^{\omega}_{\tht} A}_{L^2_x} $.
\end{remark}

\

We denote by 
$$ \vp_{\xi}(\eta)=\vm{\eta}^2_{\omega^{\perp}}+ 2^{-2k} \jb{\xi}_k^{-2} (\omega \cdot \eta)^2 $$
the Fourier multiplier of the operator $ \la_{\omega^{\perp}}+2^{-2k} \frac{1}{\jb{\xi}_k^2} (\omega \cdot\nabla_x)^2 $. We have the following bounds on $ \vp_{\xi}(\eta) $:

\begin{lemma} \label{lemmadermultiplier}Let $ k \geq 1 $. For any $ \eta $ and $ \xi=\vm{\xi} \omega $  such that $ \angle (\xi, \eta) \simeq \tht $ and $ \vm{\eta} \simeq 2^j $ we have
\begin{align} 
\label{dermult}
\vm{ (\tht \nabla_{\omega})^{\al} \frac{1}{\vp_{\xi}(\eta)}} & \leq \frac{C_{\al}}{(2^j \tht)^2+2^{2j-2k}}
\\
\label{dermultiplier}
\vm{ \pt_{\vm{\xi}}^l (\tht \nabla_{\omega})^{\al} \frac{1}{\vp_{\xi}(\eta)}} & \leq \frac{C_{\al,l}}{(2^j \tht)^2+2^{2j-2k}} \cdot \frac{2^{-2k}}{ \tht^2+2^{-2k}}, \qquad l\geq 1.
\end{align}
\end{lemma}  
  
\begin{remark} \label{rkdermultiplier} Suppose we want to estimate $\pt_{\vm{\xi}}^{l} ( \tht \pt_{\omega})^{\al}  \psi^{k}_{j,\theta,\pm}(t_0, \cdot,2^k \xi) $ in $ L^2_x $.
By lemma \ref{lemmadermultiplier} and the Coulomb condition (remark \ref{rk:Col:nf}), the following multiplier applied to $ A(t_0) $
$$ \pt_{\vm{\xi}}^{\al_1} ( \tht \pt_{\omega})^{\al} \frac{-L_{\pm,k}}{\la_{\omega^{\perp}}+2^{-2k} \frac{1}{\jb{\xi}_k^2} (\omega \cdot \nabla_x) ^2} \lpr \Pi^{\omega}_{\tht} P_{j} ( \ ) \cdot \omega \frac{\vm{\xi}}{\jb{\xi}_k} \rpr   $$ 
may be replaced by 
$$ \frac{2^{-j} \tht}{\tht^2+2^{-2k}} \Pi^{\omega,\al}_{\tht} \tilde{P}_{j} \quad (\text{if}\  l=0), \qquad \frac{2^{-2k}2^{-j} \tht}{(\tht^2+2^{-2k})^2} \Pi^{\omega,\al,l}_{\tht} \tilde{P}_{j} \quad (\text{if}\  l\geq 1),  $$ 
for the purpose of obtaining an upper bound for the $ L^2_x $ norm, where  $ \Pi^{\omega,\al,l}_{\tht} $ and $ \tilde{P}_{j} $ obey the same type of localization properties and symbol estimates as $ \Pi^{\omega}_{\tht} $ and $ P_{j} $. 
\end{remark}  
  
\begin{proof}
For $ \al=0, \ l=0 $ the bound is clear since 
\be \label{indbd0} \vp_{\xi}(\eta) \simeq (2^j \tht)^2+2^{2j-2k} .\ee
 For $ N\geq 1$ we prove the lemma by induction on $ N=l+\vm{\al} $. We focus on the case $ l\geq 1 $ since the proof of \eqref{dermult} is entirely similar. Suppose the claim holds for all $ l', \al' $ such that $ 0 \leq l'+\vm{\al'} \leq N-1 $. Applying the product rule to $ 1=\vp_{\xi}(\eta) \frac{1}{\vp_{\xi}(\eta)} $ we obtain
$$ \vp_{\xi}(\eta) \cdot \pt_{\vm{\xi}}^l (\tht \nabla_{\omega})^{\al} \frac{1}{\vp_{\xi}(\eta)}=\sum C^{\al',\beta'}_{\al'',\beta''} \cdot \pt_{\vm{\xi}}^{l'} (\tht \nabla_{\omega})^{\al'} \vp_{\xi}(\eta)\cdot \pt_{\vm{\xi}}^{l''} (\tht \nabla_{\omega})^{\al''} \frac{1}{\vp_{\xi}(\eta)} $$
where we sum over $ l'+l''=l, \ \al'+\al''=\al, \ l''+\vm{\al''} \leq N-1$. Given the induction hypothesis and \eqref{indbd0}, for the terms in the sum it suffices to show 
\begin{align}
 \label{indbd1} \vm{\pt_{\vm{\xi}}^{l'} (\tht \nabla_{\omega})^{\al'} \vp_{\xi}(\eta)} & \ls 2^{2j-2k} \qquad \hbox{for} \ l' \geq 1, \\
  \label{indbd2} \vm{\pt_{\vm{\xi}}^{l'} (\tht \nabla_{\omega})^{\al'} \vp_{\xi}(\eta)} & \ls (2^j \tht)^2  \qquad \hbox{for } \ l'=0,\ \vm{\al'} \geq 1\ (l''=l\geq 1) 
 \end{align}
We write 
$$ \vp_{\xi}(\eta)=C_{\eta}- (\omega \cdot \eta)^2(1-2^{-2k} \jb{\xi}_k^{-2} ). $$
We have $ \vm{ \omega \cdot \eta} \ls 2^j $ and thus for $ l'\geq 1 $ we obtain \eqref{indbd1}.

Now suppose $ l'=0 $ and thus $ \vm{\al'} \geq 1 $. Observe that $ \pt_{\omega} (\omega \cdot \eta)^2 \simeq 2^{2j} \tht $ and thus for all $ \vm{\al'} \geq 1 $ we have $ (\tht \nabla_{\omega})^{\al'} (\omega \cdot \eta)^2 \ls 2^{2j} \tht^2 $, which implies \eqref{indbd2}.
\end{proof}

The following proposition will be used in stationary phase arguments.

\begin{proposition} \label{phasederProp} For $ k \geq 0 $, $ \vm{\xi} \simeq 1 $, denoting $ T= \vm{t-s}+\vm{x-y} $ we have:
\begin{align}
\label{phasediff} \vm{  \psi_{\pm}^k(t,x,2^k \xi) - \psi_{\pm}^k(s,y, 2^k \xi)   } & \ls \ep \log(1+ 2^k T ) \\
\label{phaseder} \vm{ \pt_{\omega}^{\al} \lpr \psi_{\pm}^k(t,x,2^k \xi) - \psi_{\pm}^k(s,y, 2^k \xi) \rpr  } & \ls \ep (1+ 2^k T )^{ (\vm{\al}- \frac{1}{2}) \delta} , \quad 1\leq \vm{\al} \leq \delta^{-1} \\
\label{phaseder2}
\vm{ \pt_{\vm{\xi}}^l \pt_{\omega}^{\al} \lpr \psi_{\pm}^k(t,x,2^k \xi) - \psi_{\pm}^k(s,y, 2^k \xi) \rpr  } & \ls \ep 2^{-2k} (1+ 2^k T )^{ (\vm{\al}+ \frac{3}{2}) \delta},  \quad l \geq 1, (\vm{\al}+ \frac{3}{2}) \delta <1 
\end{align}
\end{proposition}

\begin{proof}
 Using  $ \vn{\vm{\nabla}^{\sigma} A }_{L^{\infty}L^2} \ls \ep $, Bernstein's inequality $ P_j \Pi_{\tht}^{\omega} L^2_x \to ( 2^{dj}  \tht^{d-1})^{\frac 1 2}  L^{\infty}_x $ and the null form (Remark \ref{rk:Col:nf}), for $ k \geq 1 $ we obtain 
$$ \vm{ \psi^{k}_{j,\theta,\pm} (t,x,2^k \xi) } \ls \ep (2^{dj}  \tht^{d-1})^{\frac 1 2}  \tht   \frac{2^j 2^{-\sigma j}}{(2^j \tht)^2+2^{2j-2k}}  \ls  \ep  \tht^{\frac{1}{2}}$$
Thus, for both $ k=0 $ and $ k \geq 1 $, one has 
$$  \vm{\psi_{j,\pm}^k(t,x,2^k \xi)} \ls \ep, \qquad \vm{\nabla_{x,t}  \psi_{j,\pm}^{k} (t,x,2^k \xi) } \ls 2^{j} \ep $$
We sum the last bound for $ j \leq j_0 $ and the previous one for $ j_0 <j \leq k-c $:
$$ \vm{ \psi_{\pm}^k(t,x,2^k \xi) - \psi_{\pm}^k(s,y, 2^k \xi)   } \ls \ep  \lpr 2^{j_0} T + (k-j_0)   \rpr $$
Choosing $ k-j_0=\log_2(2^k T)-O(1) $ we obtain \eqref{phasediff}.

For the proof of \eqref{phaseder} and \eqref{phaseder2} we use Remark \ref{rkdermultiplier}. Since their proofs are similar we only write the details for \eqref{phaseder2}. First suppose $ k \geq 1 $.

 From Bernstein's inequality and Remark \ref{rkdermultiplier} we obtain
\begin{align*} & \vm{ \pt_{\vm{\xi}}^l \pt_{\omega}^{\al} \psi^{k}_{j,\theta,\pm} (t,x,\xi) } \ls (2^{dj}  \tht^{d-1})^{\frac 1 2} \tht   \frac{1}{(2^j \tht)^2+2^{2j-2k}} \frac{2^{-2k}}{\tht^2} \tht^{-\vm{\al}} \ls 2^{-2k} \ep  \tht^{-\frac{3}{2}-\vm{\al}} \\
& \vm{\nabla_{x,t} \pt_{\vm{\xi}}^l \pt_{\omega}^{\al}  \psi^{k}_{j,\theta,\pm} (t,x,\xi) } \ls 2^{j} 2^{-2k} \ep  \tht^{-\frac{3}{2}-\vm{\al}}.
\end{align*}
We sum after $ 2^{\delta(j-k)}< \tht < c $. Summing one bound for $ j \leq j_0 $ and the other one for $ j_0 <j \leq k-c $ we obtain
$$ \vm{ \pt_{\vm{\xi}}^l \pt_{\omega}^{\al} \lpr \psi_{\pm}^k(t,x,2^k \xi) - \psi_{\pm}^k(s,y, 2^k \xi) \rpr  } \ls \ep 2^{-2k} \lpr 2^{j_0} T 2^{-(\frac{3}{2}+\vm{\al})\delta(j_0-k)}+2^{-(\frac{3}{2}+\vm{\al})\delta(j_0-k)}    \rpr $$
Choosing $ 2^{-j_0} \sim T $ , we obtain \eqref{phaseder2}. When $k=0$ the same numerology, but without the $ \tht $ factors, implies \eqref{phaseder2}.
\end{proof}

\subsection{Decomposable estimates} \

The decomposable calculus was introduced in \cite{RT}. The formulation that we use here is similar to \cite{KST}, which we have modified to allow for non-homogeneous symbols.

For $ k=0 $ we define 
$$ \vn{F}_{D_0(L^q L^r)} = \sum_{\al \leq 10d} \sup_{\vm{\xi} \leq C} \vn{ \partial_{\xi}^{\al} F(\cdot,\cdot,\xi)}_{L^q L^r}. $$

For $ k \geq 1 $ we define
\be \label{decomposable_def} \vn{F^{\tht}}_{D_k^{\tht} (L^q L^r)}^2 = \sum_{\phi} \sum_{\al_1, \vm{\al} \leq 10d} \sup_{\vm{\xi} \sim 2^k} \vn{\chi^{\tht}_{\phi}(\xi) (2^k \pt_{\vm{\xi}})^{\al_1} (2^k \tht \nabla_{\omega})^{\al} F^{\tht}(\cdot,\cdot,\xi)}_{L^q L^r}^2.   \ee
where $ \chi^{\tht}_{\phi}(\xi) $ denote cutoff functions to sectors centered at $ {\phi} $ of angle $ \ls \tht $ and $ {\phi} $ is summed over a finitely overlapping collection of such sectors.

The symbol $ F(t,x,\xi) $ is in $ D_k(L^q L^r) $ if we can decompose $ F=\sum F^{\tht} $ such that 
$$ \sum_{\tht} \vn{F^{\tht}}_{D_k^{\tht} (L^q L^r)} < \infty $$
and we define $ \vn{F}_{D_k(L^q L^r)} $ to be the infimum of such sums.

\begin{lemma} \label{decomp_lemma}
Suppose that for $ i=1,N $ the symbols $ a(t,x,\xi), F_i(t,x,\xi) $  satisfy $ \vn{F_i}_{D_k(L^{q_i} L^{r_i})} \ls 1$ and 
$$ \sup_{t \in \mb{R}} \vn{A(t,x,D)}_{L^2_x \to L^2_x} \ls 1. $$ 
The symbol of $ T $ is defined to be
$$  a(t,x,\xi) \prod_{i=1}^N F_i(t,x,\xi) . $$ 
Then, whenever $ q, \tilde{q}, r, r_i \in [1,\infty], \ q_i \geq 2 $ are such that 
$$ \frac{1}{q}=\frac{1}{\tilde{q}}+ \sum_{i=1}^N \frac{1}{q_i}, \qquad \frac{1}{r}=\frac{1}{2}+\sum_{i=1}^N \frac{1}{r_i}, $$ 
we have
$$ T(t,x,D) \bar{P}_k  : L^{\tilde{q}} L^2 \to L^{q} L^{r}, $$
By duality, when $ r=2, \ r_i=\infty $, the same mapping holds for $ T(D,s,y) $.

\end{lemma}

\begin{proof}[Proof for $ k=0$] 
For each $ i=1,N $ we decompose into Fourier series
$$ F_i(t,x,\xi)=\sum_{j \in \mb{Z}^d} d_{i,j} (t,x) e_j(\xi), \qquad  e_j(\xi)=e^{i j \cdot \xi} $$
on a box $ [-C/2,C/2]^d $. 
From the Fourier inversion formula and integration by parts we obtain
$$ \jb{j}^{M} \vn{d_{i,j}}_{L^{q_i} L^{r_i}} \ls \vn{F_i}_{D_0(L^{q_i} L^{r_i})} \ls 1 $$
for some large $ M $. Then
$$ Tu(t,x)= \sum_{j_1, \dots j_N \in \mb{Z}^d} \prod_{i=1}^N d_{i,j_i}(t,x) \int e^{i x \cdot \xi} e^{i (j_1+\dots +j_N) \cdot \xi} a(t,x,\xi) \hat{u} (t,\xi) \dd \xi $$
From H\" older's inequality we obtain
\begin{align*}
 \vn{Tu}_{L^q L^r} & \leq  \sum_{j_1, \dots j_N \in \mb{Z}^d} \prod_{i=1}^N \vn{d_{i,j_i}}_{L^{q_i} L^{r_i}} \vn{A(t,x,D) e^{i (j_1+\dots +j_N)D} u}_{L^{\tilde{q}} L^2} \ls \\
& \ls \sum_{j_1, \dots j_N \in \mb{Z}^d} \prod_{i=1}^N \jb{j_i}^{-M} \vn{u}_{L^{\tilde{q}} L^2} \ls \vn{u}_{L^{\tilde{q}} L^2},
\end{align*}
which proves the claim.
\end{proof}

\begin{proof}[Proof for $ k \geq1 $] 
We present the proof for the case $ N=2 $. It is straightforward to observe that the following works for any $ N $. From the decompositions $ F_i=\sum_{\tht_i} F^{\tht_i}_i $ and definition of $ D_k(L^q L^r) $ we see that it suffices to restrict attention to the operator $ T $ with symbol
$$  a(t,x,\xi) F_1(t,x,\xi) F_2(t,x,\xi) $$
in the case $ F_i=F^{\tht_i}_i, \ i=1,2 $ and to prove
\be \label{symbmapping}
\vn{T(t,x,D)  \tilde{P}_k }_{ L^{\tilde{q}} L^2 \to L^{q} L^{r}} \ls  \vn{F_1}_{D^{\tht_1}_k(L^{q_1} L^{r_1})} \vn{F_2}_{D^{\tht_2}_k(L^{q_2} L^{r_2})}.  \ee
For $ i=1,2 $ we decompose 
$$ F_i=\sum_{T_i} F_i^{T_i}, \qquad  F_i^{T_i} (t,x,\xi) \defeq  \vp_{\tht_i}^{T_i}(\xi) F_i(t,x,\xi) $$
where $ \vp_{\tht_i}^{T_i}(\xi) $ are cutoff functions to sectors $ T_i $ of angle $ \tht_i $, where the $ T_i $ is summed over a finitely overlapping collection of such sectors. We also consider the bump functions  $ \chi_{T_i}(\xi) $ which equal $ 1 $ on the supports of $ \vp_{\tht_i}^{T_i}(\xi) $ and are adapted to some enlargements of the sectors $ T_i $. We expand each component as a Fourier series
$$  F_i^{T_i} (t,x,\xi)= \sum_{j \in \mb{Z}^d} d_{i,j}^{T_i} (t,x) e^{T_i}_{\tht_i,j}(\xi), \qquad e^{T_i}_{\tht_i,j}(\xi)=\exp ( i  2^{-k} j \cdot (\vm{\xi}, \tilde{\omega} \tht_i^{-1})/C ) $$
on the tube $ T_i=\{ \vm{\xi}\sim 2^k, \angle(\xi, \phi_i )\ls \tht_i \} $ where $ \xi=\vm{\xi} \omega $ in polar coordinates so that $ \omega $ is parametrized by $ \tilde{\omega} \in \mb{R}^{d-1} $ such that $ \vm{\tilde{\omega}} \ls \tht_i $.
 Integrating by parts in the Fourier inversion formula for $ d_{i,j}^{T_i} (t,x) $ we obtain
$$ \jb{j}^M  \vn{d_{i,j}^{T_i} (t)}_{L^{r_i}}  \ls  \sum_{\al_1, \vm{\al} \leq 10d} \sup_{\vm{\xi} \sim 2^k} \vn{ (2^k \pt_{\vm{\xi}})^{\al_1} (2^k \tht_i \nabla_{\omega})^{\al} F_i^{T_i}(t,\cdot,\xi)}_{L^{r_i}} $$
and since $ q_i \geq 2 $ we have
\be \label{decaydec}
 \vn{d_{i,j}^{T_i}}_{L^{q_i}_t l^2_{T_i} L^{r_i}}  \ls \jb{j}^{-M} \vn{F_i}_{D^{\tht_i}_k(L^{q_i} L^{r_i})} \ee
Since for $ \xi \in T_i $ we have $ F_i^{T_i}=F_i^{T_i} \chi_{T_i}(\xi) $, we can write
$$ T u (t,x)=\sum_{T_1,T_2} \sum_{j_1,j_2} d_{1,j_1}^{T_1} (t,x) d_{2,j_2}^{T_2} (t,x) \int e^{i x \xi} a(t,x,\xi) e^{T_1}_{\tht_1,j_1} e^{T_i}_{\tht_2,j_2} \chi_{T_1} \chi_{T_2} \tilde{\chi}(\xi/2^k) \hat{u} (t,\xi) \dd \xi. $$
Thus
\begin{align*}
 \vn{Tu(t)}_{L^r_x} & \ls \sum_{j_1,j_2 \in \mb{Z}^d} \sum_{T_1,T_2} \vn{d_{1,j_1}^{T_1} (t)}_{L_x^{r_1}} \vn{d_{2,j_2}^{T_2} (t)}_{L_x^{r_2}} \vn{\chi_{T_1} \chi_{T_2}  \hat{u} (t)}_{L^2} \ \\
& \ls \sum_{j_1,j_2 \in \mb{Z}^d}  \vn{d_{1,j_1}^{T_1} (t)}_{l^2_{T_1} L_x^{r_1}} \sum_{T_2}  \vn{d_{2,j_2}^{T_2} (t)}_{L_x^{r_2}} \vn{\chi_{T_2}  \hat{u} (t)}_{L^2}  \ \\
& \ls \sum_{j_1,j_2 \in \mb{Z}^d} \vn{d_{1,j_1}^{T_1} (t)}_{l^2_{T_1} L_x^{r_1}} \vn{d_{2,j_2}^{T_2} (t)}_{l^2_{T_2} L_x^{r_2 }} \vn{\hat{u} (t)}_{L^2}  
\end{align*}
Applying H\" older's inequality and \eqref{decaydec} we obtain
$$ \vn{Tu}_{L^q L^r} \ls \vn{u}_{L^{\tilde{q}} L^2}  \sum_{j_1,j_2 \in \mb{Z}^d} \jb{j_1}^{-M}\jb{j_2}^{-M}  \vn{F_1}_{D^{\tht_1}_k(L^{q_1} L^{r_1})} \vn{F_2}_{D^{\tht_2}_k(L^{q_2} L^{r_2})} $$
which sums up to \eqref{symbmapping}.
\end{proof}

\subsection{Decomposable estimates for the phase} \ 

\

Now we apply the decomposable calculus to the phases $ \psi^k_{\pm}(t,x,\xi) $. 

\begin{lemma} Let $ q \geq 2, \ \frac{2}{q}+\frac{d-1}{r} \leq \frac{d-1}{2} $. For $ k \geq 1$ we have
\begin{align}
 \label{decomp1}  \vn{(\psi^{k}_{j,\theta,\pm},2^{-j}\nabla_{t,x} \psi^{k}_{j,\theta,\pm} ) }_{D_k^{\tht}(L^q L^r)} & \ls \ep 2^{-(\frac{1}{q}+\frac{d}{r})j} \frac{\tht^{\frac{d+1}{2}-(\frac{2}{q}+\frac{d-1}{r})}}{\tht^2+2^{-2k}} . \\
 \label{decomp6} \vn{P_{\tht}^{\omega} A_{k_1} (t,x) \cdot \omega}_{D_k^{\tht}(L^2L^{\infty})} & \ls \tht^{\frac{3}{2}} 2^{\frac{k_1}{2}}  \ep. 
\end{align}
For $ k=0 $ we have
\be \label{decomp0} 
\vn{(\psi^{0}_{j,\pm},2^{-j}\nabla_{t,x} \psi^{0}_{j,\pm} ) }_{D_0(L^q L^r)} \ls \ep 2^{-(\frac{1}{q}+\frac{d}{r})j}.
\ee
\end{lemma}

\begin{proof} Suppose $ k \geq 1 $.
Without loss of generality, we will focus on $ \psi^{k}_{j,\theta,\pm} $, since exactly the same estimates hold for $ 2^{-j}\nabla_{t,x} \psi^{k}_{j,\theta,\pm} $. In light of the definition \eqref{phase_piece}, for any $ \xi=\vm{\xi} \omega $ and any $ \al_1, \vm{\al} \leq 10d $, the derivatives $  \pt_{\vm{\xi}}^{\al_1} ( \tht \pt_{\omega})^{\al}\psi^{k}_{j,\theta,\pm} $ are localized to a sector of angle $ O(\tht) $ in the $ (t,x) $-frequencies and they solve the free wave equation 
$$  \Box_{t,x} \pt_{\vm{\xi}}^{\al_1} ( \tht \pt_{\omega})^{\al} \psi^{k}_{j,\theta,\pm}(t,x,2^k \xi) =0 $$
Let $ r_0 $ be defined by $ \frac{2}{q}+\frac{d-1}{r_0}=\frac{d-1}{2} $. The Bernstein and Strichartz inequalities imply 
$$ \vn{\pt_{\vm{\xi}}^{\al_1} ( \tht \pt_{\omega})^{\al} \psi^{k}_{j,\theta,\pm}(\cdot,2^k \xi)}_{L^q L^r} \ls \tht^{(d-1)(\frac{1}{r_0}-\frac{1}{r})} 2^{d(\frac{1}{r_0}-\frac{1}{r})j} \vn{\pt_{\vm{\xi}}^{\al_1} ( \tht \pt_{\omega})^{\al} \psi^{k}_{j,\theta,\pm}(\cdot,2^k \xi)}_{L^q L^{r_0}} $$
\be \label{Ltwo_estim} \ls 2^{(1-\frac{1}{q}-\frac{d}{r})j} \tht^{(d-1)(\frac{1}{r_0}-\frac{1}{r})} \vn{\pt_{\vm{\xi}}^{\al_1} ( \tht \pt_{\omega})^{\al} \psi^{k}_{j,\theta,\pm}(2^k \xi)[0]}_{\dot{H}^{\sigma} \times \dot{H}^{\sigma-1} } \ee
By  Remark \ref{rkdermultiplier} (which uses the null form) we deduce
$$  \eqref{Ltwo_estim} \ls 2^{-(\frac{1}{q}+\frac{d}{r})j} \frac{\tht^{\frac{d+1}{2}-(\frac{2}{q}+\frac{d-1}{r})}}{\tht^2+2^{-2k}} \vn{\Pi^{\omega,\al}_{\tht} \tilde{P}_{j} A[0]}_{\dot{H}^{\sigma} \times \dot{H}^{\sigma-1}}.  $$
By putting together this estimate, definition \ref{decomposable_def}, the finite overlap of the sectors and the orthogonality property, we obtain
$$ \vn{\psi^{k}_{j,\theta,\pm}}_{D_k^{\tht}(L^q L^r)} \ls 2^{-(\frac{1}{q}+\frac{d}{r})j} \frac{\tht^{\frac{d+1}{2}-(\frac{2}{q}+\frac{d-1}{r})}}{\tht^2+2^{-2k}} \vn{\tilde{P}_{j} A[0]}_{\dot{H}^{\sigma} \times \dot{H}^{\sigma-1}}, $$
which proves the claim, since $  \vn{\tilde{P}_{j} A[0]}_{\dot{H}^{\sigma} \times \dot{H}^{\sigma-1}} \ls \ep $.

The same argument applies to \eqref{decomp6}. The only difference is that one uses the angular-localized Strichartz inequality $ \vn{P_{\tht}^{\omega} A_{j}}_{L^2 L^{\infty}} \ls \tht^{\frac{d-3}{2}} 2^{\frac{j}{2}} \vn{A_{j}}_{\dot{H}^{\sigma} \times \dot{H}^{\sigma-1}} $, which holds for free waves, in addition to the null form which gives an extra $ \tht $.

When $ k=0 $ the same argument goes through without angular projections and with no factors of $ \tht $ in \eqref{Ltwo_estim}.
\end{proof}

\begin{remark} As a consequence of the above we also obtain
\be \label{decomp5}  \vn{\nabla_{\xi} (\Pi^{\omega}_{\leq \delta(k_1-k)}  A_{k_1}(t,x) \cdot \xi) }_{D_k^1 (L^2 L^{\infty})} \ls 2^{-10d \delta (k_1-k)} 2^{\frac{k_1}{2}} \ep.
\ee
\end{remark}

\begin{corollary} For $ k \geq 0 $ we have
\be \label{decomp2} \vn{(\psi^{k}_{j,\pm}, 2^{-j} \nabla_{t,x} \psi^{k}_{j,\pm} )  }_{D_k(L^q L^{\infty})} \ls 2^{-\frac{j}{q}} \ep,  \qquad q>4
\ee
\be \label{decomp3}\vn{\nabla_{t,x} \psi^{k}_{\pm}}_{D_k(L^2 L^{\infty})} \ls 2^{\frac{k}{2}} \ep
\ee
\be \label{decomp4} \vn{\nabla_{t,x} \psi^{k}_{\pm}}_{D_k(L^{\infty} L^{\infty})} \ls 2^{k} \ep.
\ee
\end{corollary}

\begin{proof} The bound \eqref{decomp4} follows by summing over \eqref{decomp2}. For $ k=0 $, \eqref{decomp2} and \eqref{decomp3} follow from \eqref{decomp0}.

Now assume  $ k \geq 1$. The condition $ q>4 $ makes the power of $ \tht $ positive in \eqref{decomp1} for any $ d \geq 4 $. Thus \eqref{decomp2} follows by summing in $ \tht $. For \eqref{decomp3}, summing in $ \tht $ gives the factor $ 2^{\frac{\delta}{2}(k-j)} $, which is overcome by the extra factor of $ 2^{j} $ when summing in $ j<k $.
\end{proof}

\subsection{Further properties}

\begin{lemma}
  Let $a(x,\xi)$ and $b(x,\xi)$ be smooth symbols. Then one has
  \begin{equation}
    \vn{a^rb^r- (ab)^r}_{L^r(L^2)
      \to L^q(L^2)}  \lesssim  \vn{(\nabla_x a)^r}_{L^r(L^2)\to L^{p_1}(L^2)}
    \vn{\nabla_\xi b}_{D^1_k L^{p_2}(L^\infty)}  \label{spec_decomp1}
  \end{equation}
  \be
   \vn{a^l b^l- (ab)^l}_{L^r(L^2)
      \to L^q(L^2)}  \lesssim  \vn{\nabla_\xi a}_{D_k^1L^{p_2}(L^\infty)}  \vn{(\nabla_x b)^l}_{L^r(L^2)\to L^{p_1}(L^2)}  \label{spec_decomp3}
  \ee
  where $q^{-1}=\sum p_i^{-1}$. Furthermore, if $b=b(\xi)$ is a smooth multiplier supported on $ \{ \jb{\xi} \simeq 2^k \} $, then for any two translation
  invariant spaces $X,Y$ one has:
  \begin{equation}
    \vn{a^rb^r- (ab)^r}_{X\to Y}  \lesssim  2^{-k}
    \vn{(\nabla_x a)^r}_{X\to Y}
    \ . \label{spec_decomp2}
  \end{equation}
\end{lemma}

\begin{proof}
See \cite[Lemma 7.2]{KST}.
\end{proof}

\begin{lemma} \label{loczsymb} Let $ X, Y $ be translation-invariant spaces of functions on $ \mb{R}^{n+1} $ and consider the symbol $ a(t,x,\xi) $ such that 
$$ a(t,x,D) : X \to Y. $$
Then the $ (t,x)$-frequency localized symbol $ a_{<h} (t,x,\xi) $ also satisfies
$$ a_{<h} (t,x,D): X \to Y .$$
\end{lemma}

\begin{proof} We write
$$ a_{<h} (t,x,D)u=\int m(s,y) T_{(s,y)} a(t,x,D) T_{-(s,y)} u \dd s \dd y $$
where $ m(s,y) $ is a bump function and $ T_{(s,y)} $ denotes translation by $ (s,y) $. Now the claim follows from Minkowski's inequality and the $ T_{\pm(s,y)} $-invariance of $ X, Y $.
\end{proof}

\section{Oscillatory integrals estimates } \label{sec:Osc-int}	\

In this section we prove estimates for oscillatory integrals that arise as kernels of $ TT^* $ operators used in proofs of the mapping \eqref{renbd3} and \eqref{renbd}, \eqref{renbdt}, \eqref{renbd2}. These bounds are based on stationay and non-stationary phase arguments (see Prop. \ref{nonstationary} and \ref{stationary}).

\subsection{Rapid decay away from the cone}

We consider 
\be \label{kernel:a} K_k^a(t,x,s,y)= \int e^{- i \psi_{\pm}^k(t,x,\xi)}  a\big( \frac{\xi}{2^k} \big) e^{\pm i (t-s) \jb{\xi}+i (x-y) \xi} e^{+ i \psi_{\pm}^k(s,y,\xi)} \dd \xi \ee
where $ a(\xi) $ is a bump function supported on $ \{ \vm{\xi} \simeq 1 \} $ for $ k\geq 1 $ and on $ \{ \vm{\xi} \ls 1 \} $ for $ k=0 $.

\begin{proposition} \label{stat_phase_a}
For $ k \geq 0 $ and any $ N \geq 0 $, we have
\be \vm{K_k^a(t,x,s,y)} \ls 2^{dk} \frac{1}{\jb{2^k ( \vm{t-s}-\vm{x-y})}^N} \ee
whenever $ 2^k \vm{ \vm{t-s}-\vm{x-y}} \gg 2^{-k} \vm{t-s}$.

Moreover, the implicit constant is bounded when $ a(\xi) $ is bounded in $ C^N(\vm{\xi}\ls 1) $.
\end{proposition}

\begin{proof} We first assume $ k\geq 1 $. Suppose without loss of generality that $ t-s \geq 0 $, and $ \pm=+ $. Denoting $ \lmd=\vm{ \vm{t-s}-\vm{x-y}} $ it suffices to consider $ 2^k \lmd \gg 1 $.  By a change of variables we write 
$$ K_k^a= 2^{dk} \int_{\vm{\xi} \simeq 1} e^{i 2^k (\phi_k+\varphi_k)(t,x,s,y,\xi)} a(\xi) \dd \xi $$
where  
$$ \phi_k(t,x,s,y,\xi)=(t-s) \jb{\xi}_k+(x-y) \cdot \xi $$
$$ \varphi_k(t,x,s,y,\xi)=-(\psi_{\pm}^k(t,x,2^k \xi)-\psi_{\pm}^k(s,y,2^k \xi))/2^k.   $$
By Prop \ref{phasederProp} and noting that $ T=\vm{t-s}+\vm{x-y} \ls 2^{2k} \lmd $ we have
$$ \vm{ \nabla  \varphi_k} \ls \ep (2^k T)^{3\delta}/2^k \ls \ep \lmd $$
Furthermore,
$$ \nabla \phi_k=(t-s) \frac{\xi}{\jb{\xi}_k}+(x-y) $$
If $ \vm{x-y} \geq 2 \vm{t-s} $ or $ \vm{t-s} \geq 2 \vm{x-y}  $, by non-stationary phase, we easily estimate $ \vm{ K_k^a}\ls 2^{dk} \jb{2^k T}^{-N} $. 

Now we assume $ \vm{t-s} \simeq \vm{x-y}  \gg 2^{-k} $. On the region $ \angle(-\xi,x-y)> 10^{-3} $ we have $ \vm{  \nabla \phi_k} \gtrsim \vm{t-s} $, thus by a smooth cutoff and non-stationary phase, that component of the integral is $ \ls 2^{dk} \jb{2^k T}^{-N}  $. Now we can assume $ a(\xi) $ is supported on the region $  \angle(-\xi,x-y) \leq 10^{-2} $. 

If $ \vm{  \nabla \phi_k}\geq 1/4 \lmd $ on that region, we get the bound $ 2^{dk} \jb{2^k \lmd }^{-N} $. We claim this is always the case. Suppose the contrary, that there exists $ \xi $ such that  $ \vm{  \nabla \phi_k} \leq 1/4 \lmd $. Then $ (t-s) \frac{\vm{\xi'}}{\jb{\xi}_k} \leq 1/4 \lmd $ writing in coordinates $ \xi=(\xi_1,\xi') $ where $ \xi_1 $ is in the direction $ x-y $ while $ \xi' $ is orthogonal to it.
$$ \nabla \phi_k \cdot \frac{x-y}{\vm{x-y}}=(t-s)\frac{\xi_1}{\jb{\xi}_k} + \vm{x-y}=\pm \lmd+ (t-s)(1+\frac{\xi_1}{\jb{\xi}_k} ) $$ 
Thus $ \xi_1 \leq 0 $ and using that $ \lmd \gg 2^{-2k} (t-s) $ we have
$$ (t-s)(1+\frac{\xi_1}{\jb{\xi}_k})=\frac{t-s}{\jb{\xi}_k}\frac{2^{-2k}+\vm{\xi'}^2}{\jb{\xi}_k+\vm{\xi_1}}<2^{-2k}(t-s)+ \frac{1}{4} \lmd \leq \frac{1}{2} \lmd $$
which implies $ \vm{  \nabla \phi_k} \geq 1/2 \lmd $, a contradiction. This concludes the case $ k \geq 1 $.
 
When $ k=0$ we have $  \vm{x-y} \gg \vm{t-s} $. For the corresponding phase we have $ \vm{\nabla \phi_0} \geq \frac{1}{2} \vm{x-y} $ and thus we get the factor $ \jb{x-y}^{-N} $. 
\end{proof}
\

\subsection{\bf Dispersive estimates} \ 

\

Dispersive estimates for the Klein-Gordon equation are treated in places like \cite[Section 2.5]{NaSch}, \cite{BH1}. The situation here is slightly complicated by the presence of the $ e^{\pm i\psi} $ transformations. To account for this we use Prop. \ref{phasederProp}.
Let 
$$ K^k \defeq \int e^{- i \psi_{\pm}^k (t/2^k,x/2^k,2^k\xi)+i\psi_{\pm}^k(s/2^k,y/2^k,2^k\xi) } e^{\pm i t(t-s) \jb{\xi}_k} e^{i (x-y) \xi} a(\xi)   \dd \xi $$ 
where $ a(\xi) $ is a bump function supported on $ \{ \vm{\xi} \simeq 1 \} $ for $ k \geq 1 $ and on $ \{ \jb{\xi} \simeq 1 \} $ for $ k=0 $.
% \red{and let $ K^k_{<0}(t,x;s,y) $  be the kernel of the operator 
% $$ e^{-i \psi_{k,\pm}(\cdot/2^k,\cdot/2^k, 2^k \cdot)}_{<0} (t,x,D) e^{\pm i (t-s) \jb{D}_k} a(D) e^{i \psi_{k,\pm}(\cdot/2^k,\cdot/2^k, 2^k \cdot)}_{<0} (D,s,y). $$}

\begin{proposition} 
For any $ k \geq 0 $ one has the inequalities
\begin{numcases}{\vm{K^k(t,x;s,y)}  \ls }
    \frac{1}{\jb{t-s}^{\frac{d-1}{2}}}  \label{dispestt1}  \\
   \frac{2^k}{\jb{t-s}^{d/2}}  \label{dispestt12}
\end{numcases}  
  
\end{proposition}

\begin{proof}
\pfstep{Step~1} We first prove \eqref{dispestt1} for $ k \geq 1 $. We assume $ \vm{t-s} \simeq \vm{x-y} \gg 1 $ and that $ a(\xi) $ is supported on the region $  \angle(\mp \xi,x-y) \leq 10^{-2} $, since in the other cases the phase is non-stationary and we obtain the bound $ \jb{t-s}^{-N} $ from the proof of Prop. \ref{stat_phase_a}. We denote
$$ \varphi(t,x,s,y,\xi)=- \psi_{\pm}^k (t/2^k,x/2^k,2^k\xi)+\psi_{\pm}^k(s/2^k,y/2^k,2^k\xi)   $$ 
and write
$$ (x-y) \cdot \xi \pm \vm{x-y} \vm{\xi}=\pm 2 \vm{x-y} \vm{\xi} \sin^2(\tht/2) $$
where $ \tht=\angle(\mp \xi,x-y) $. We write $ \xi=(\xi_1,\xi') $ in polar coordinates, where $ \xi_1=\vm{\xi} $ is the radial component. Then
\be \qquad \quad \ K^k=\int_{\xi_1 \simeq 1} \xi_1^3e^{\pm i (t-s) \jb{\xi_1}_k \mp i \vm{x-y} \xi_1 } \Omega(\xi_1)  \dd \xi_1 \label{int_pol} \ee
$$ \text{where} \quad \Omega(\xi_1)= \int  e^{\pm i \vm{x-y} 2 \xi_1 \sin^2(\tht/2)} a(\xi_1,\xi') e^{i \varphi} \dd S(\xi')
$$
For each $ \xi_1 $ we bound 
\be \label{stph:omg1} \vm{\Omega(\xi_1)} \ls \vm{x-y}^{-\frac{d-1}{2}} \ee
as a stationary-phase estimate (see Prop. \ref{stationary}). When derivatives fall on $ e^{i \varphi} $ we get factors of $ \ep \vm{x-y}^{\delta} $ by \eqref{phaseder}; however, these are compensated by the extra factors $ \vm{x-y}^{-1} $ from the expansion \eqref{expansion}. Integrating in $ \xi_1 $ we obtain \eqref{dispestt1}.

Furthermore, using \eqref{st-phase:est} we obtain
\be \label{stph:omg2} 
\vm{\pt_{\xi_1}^j \Omega(\xi_1)} \ls \vm{x-y}^{-\frac{d-1}{2}} \lng 2^{-2k} \vm{x-y}^{4 \delta} \rng  \qquad j=1,2.
\ee
The term $ \lng 2^{-2k} \vm{x-y}^{4 \delta} \rng $ occurs by \eqref{phaseder2} when $ \pt_{\xi_1} $ derivatives  fall on $ e^{i \varphi} $. 
\pfstep{Step~2} Now we prove \eqref{dispestt12}.

First we consider $ k = 0 $ and $ \vm{t-s} \gg 1 $. When $ \vm{t-s} \leq c \vm{x-y} $ the phase is non-stationary and we obtain $ \jb{t-s}^{-N} $. Otherwise, we consider the phase $ \jb{\xi}+\frac{x-y}{\vm{t-s}}\cdot \xi $ and get the bound $ \jb{t-s}^{-d/2} $ as a stationary-phase estimate using Prop. \ref{phasederProp}.

Now we take $ k \geq 1 $ under the assumptions from Step 1. We may also assume $ \vm{t-s} \gg 2^{2k} $ (otherwise \eqref{dispestt12} follows from \eqref{dispestt1}). 

In \eqref{int_pol} we have the phase $ 2^{-2k} \vm{t-s} f(\xi_1) $ where 
$$ f(\xi_1)=2^{2k} \Big( \jb{\xi_1}_k- \frac{\vm{x-y}}{\vm{t-s}} \xi_1 \Big), \quad f'(\xi_1)=2^{2k} \Big( \frac{\xi_1}{\jb{\xi_1}_k}- \frac{\vm{x-y}}{\vm{t-s}} \Big), \quad \vm{f''(\xi_1)} \simeq 1,    
$$
and $ \vm{f^{(m)}(\xi_1)} \ls 1 $ for $ m \geq 3 $. Using stationary phase in $ \xi_1 $ (Prop. \ref{stationary}/\ref{nonstationary}) one has 
$$ \vm{K^k} \ls \frac{1}{\vm{2^{-2k} \vm{t-s}}^{\frac{1}{2}}} \sup \vm{\Omega} + \frac{1}{2^{-2k}\vm{t-s}} \sup_{j \leq 2} \vm{\pt_{\xi^1}^j \Omega},
$$
which, together with \eqref{stph:omg1}, \eqref{stph:omg2}, implies \eqref{dispestt12}.
\end{proof}

Now we consider more localized estimates. 

Let $ \calC $ be a box of size $\simeq 2^{k'} \times ( 2^{k' +l'} )^{d-1} $ located in an annulus $ \{ \jb{\xi} \simeq 2^k \} $ for $ k \geq 0 $.
Suppose $ a_{\calC} $ is a bump function adapted to $ \calC $ and define 
\be K^{k',l'}(t,x;s,y) \defeq \int e^{- i \psi_{\pm}^k (t,x,\xi)+i\psi_{\pm}^k(s,y,\xi) } e^{\pm i t(t-s) \jb{\xi}} e^{i (x-y) \xi} a_{\calC}(\xi)   \dd \xi.   \label{oscInteg} \ee
% and let $ K^{k',l'}_{<k+2l}(t,x;s,y) $ be the kernel of the $ TT^* $ operator 
% $$ e^{- i \psi_{<k,\pm}}_{<k+2l} (t,x,D)    e^{\pm i t(t-s) \jb{D}} a_C(D) e^{i \psi_{<k,\pm}}_{<k+2l} (D,s,y) $$

\begin{proposition} 
Let $ k \geq 0, \ k' \leq k   $ and $ -k \leq l' \leq 0  $. Then, we have
\be \label{dispestt2} 
\vm{K^{k',l'}(t,x;s,y) }  \ls 2^{dk'+(d-1)l'} \frac{1}{\jb{2^{2(k'+l')-k} (t-s)}^{\frac{d-1}{2}}}
% \vm{K^{k',l'}_{<k+2l}(t,x;s,y) }
\ee
\end{proposition}

\begin{proof}
We assume $ 2^{2(k'+l')-k} \vm{t-s} \gg 1 $ (otherwise we bound the integrand by absolute values on $ \calC $) and assume $ \vm{t-s} \simeq \vm{x-y} $ (otherwise the phase is non-stationary). Let $ k \geq 1 $. By a change of variable we rescale to $ \vm{\xi} \simeq 1 $ and write $ K^{k',l'} $ as $ 2^{dk} \times $  \eqref{int_pol}- applied to $ 2^k (t,x;s,y) $, with $ a(\cdot) $ supported on a box $ 2^{k'-k} \times ( 2^{k' +l'-k} )^{d-1} $. Like before, for each $ \xi_1 $ we bound the inner integral $ \Omega(\xi_1) $ by $ (2^k \vm{t-s})^{-\frac{d-1}{2}} $ by stationary-phase. Integrating in $ \xi_1 $ on a radius of size $ 2^{k'-k} $ we get $ 2^{dk} 2^{k'-k} (2^k \vm{t-s})^{-\frac{d-1}{2}} $ which gives \eqref{dispestt2}. When $ k=0, \ l'=O(1) $ the estimate is straightforward. 
\end{proof}

% \pfstep{Step~2}
% using the inequality
% $$ \int_{\mb{R}} \frac{1}{\jb{\al \vm{a-r}}^{3/2}} \frac{\beta}{\jb{\beta \vm{r} }^N} \dd r \ls \frac{1}{\jb{\al \vm{a}}^{3/2}} $$
% for $ \beta \leq \al $.

\begin{corollary} \label{Cor:L2Linf}  Let $ k \geq 0, \ k' \leq k   $ and $ -k \leq l' \leq 0  $. Then
\be  e^{-i \psi_{\pm}^k} (t,x,D) e^{\pm i t \jb{D}} P_{C_{k'}(l')} :L^2_x \to 2^{\frac{k}{2}+\frac{d-2}{2}k'+\frac{d-3}{2}l' }L^2 L^{\infty}	\label{TT*:L2Linf}
\ee
\end{corollary}

\begin{proof}
By a $ TT^* $ argument this follows from
$$ 2^{-k-(d-2)k'-(d-3)l'} e^{- i \psi_{\pm}^k} (t,x,D)    e^{\pm i t(t-s) \jb{D}} P_{C_{k'}(l')}^2 e^{i \psi_{\pm}^k} (D,s,y) : L^2L^1 \to L^2 L^{\infty} $$
We use $ \eqref{dispestt2} $ to bound the kernel of this operator, and the mapping follows since $ 2^{2k'+2l'-k} \jb{2^{2k'+2l'-k} \vm{r}}^{-(d-1)/2} $ has $ L^1_r L^{\infty}_x $ norm $ \ls 1 $.
\end{proof}
\

\subsection{\bf The PW decay bound, $ d=4 $}\

\

Let $ \calC $ be a box of size $\simeq 2^{k'} \times ( 2^{k'-k} )^3 $ with center $ \xi_0 $ located in an annulus $ \{ \vm{\xi} \sim 2^k \} \subset \mb{R}^4 $. We consider the decay of the integral $ K^{k',-k} $ defined in \eqref{oscInteg}, in the frame \eqref{frame}, \eqref{frame2}, where $ \omega $ is the direction of $ \xi_0 $ and $ \lmd=\frac{\vm{\xi_0}}{\jb{\xi_0}} $. 

This type of bound is similar to the one used by Bejenaru and Herr \cite[Prop. 2.3]{BH1} to establish null-frame $ L^2_{t_{\omega,\lmd}} L^{\infty}_{x_{\omega,\lmd}} $- Strichartz estimates, an idea we will also follow in this paper.

\begin{proposition} When $ \vm{t_{\omega}-s_{\omega}} \gg 2^{k'-3k} \vm{t-s} $, we have
\be \label{PWdecay} 
\vm{K^{k',-k}(t,x;s,y)} \ls 2^{4k'-3k} \frac{1}{\jb{2^{k'} (t_{\omega}-s_{\omega})}^2}
\ee
\end{proposition}

\begin{proof} Denoting $ T=\vm{t-s}+\vm{x-y} $, we clearly have $ \vm{t_{\omega}-s_{\omega}} \leq T $. In the cases when $ \vm{t-s} \geq 2\vm{x-y} $ or $\vm{x-y} \geq 2 \vm{t-s} $ from integrating by parts radially we obtain the decay  $ \jb{2^{k'}T}^{-N} 2^{4k'-3k} $.
Now suppose $  \vm{t-s} \simeq \vm{x-y} $, $ \pm=+ $ and let
$$ \phi(\xi)=(t-s) \jb{\xi}+(x-y) \cdot \xi, \qquad \nabla \phi=(t-s) \frac{\xi}{\jb{\xi}}+x-y.  $$
For $ \xi \in \calC $ we have $ \frac{\vm{\xi}}{\jb{\xi}}=\lmd+O(2^{k'-3k}) $ and 
$$ \frac{\xi}{\vm{\xi}}=\omega+\sum_i O(2^{k'-2k}) \omega_i + O(2^{2(k'-2k)}), \qquad \omega_i \in \omega^{\perp}. $$
Therefore
$$ \omega \cdot \nabla \phi = (t_{\omega}-s_{\omega}) \sqrt{1+\lmd^2} + O(2^{k'-3k} \vm{t-s}). $$
Due to the assumption, the phase is non-stationary $ \vm{ \omega \cdot \nabla \phi} \gtrsim \vm{t_{\omega}-s_{\omega}} $ and we obtain \eqref{PWdecay} by integrating by parts with $ \partial_{\omega}=\omega \cdot \nabla $. 

When derivatives fall on $ e^{- i \psi_{\pm}^k (t,x,\xi)+\psi_{\pm}^k(s,y,\xi) } $ we get extra factors of $ 2^{k'-k} (2^k T)^{\delta} $ from Prop \ref{phasederProp}. However, we compensate this factors by writing the integral in polar coordinates similarly to \eqref{int_pol} and using stationary-phase for the inner integral like in the proof of \eqref{dispestt2}, \eqref{dispestt1}, giving an extra $ (2^{2k'-3k}T)^{-3/2} $, which suffices. \end{proof}

\begin{corollary} \label{corPW} For $ k\geq 1 $ let $ \xi_0 $  be the center of the box $ C_{k'}(-k) $, $ \lmd=\frac{\vm{\xi_0}}{\jb{\xi_0}} $  and $ \omega=\frac{\xi_0}{\vm{\xi_0}} $. Then 
$$ 2^{-\frac{3}{2}(k'-k)} e^{- i \psi_{\pm}^k} (t,x,D) e^{i t \jb{D}} P_k P_{C_{k'}(-k)} : L^2_x \to L^2_{t_{\omega,
\lmd}} L^{\infty}_{x_{\omega,\lmd}} $$ 
\end{corollary}

\begin{proof}
By a $ TT^* $ argument this follows from the mapping
$$ 2^{-3(k'-k)}  e^{- i \psi_{\pm}^k} (t,x,D)    e^{ i t(t-s) \jb{D}} P_k^2 a_{\calC}(D) e^{i \psi_{\pm}^k} (D,s,y) :L^2_{t_{\omega,\lmd}} L^1_{x_{\omega,\lmd}} \to L^2_{t_{\omega,\lmd}} L^{\infty}_{x_{\omega,\lmd}} $$
which holds since the kernel of this operator is bounded by $ 2^{k'} \jb{2^{k'} (t_{\omega}-s_{\omega})}^{-3/2} \in L^1_{t_{\omega}-s_{\omega}} L^{\infty}  $. When $ \vm{t_{\omega}-s_{\omega}} \gg 2^{k'-3k} \vm{t-s} $ this follows from \eqref{PWdecay}, while for $ \vm{t_{\omega}-s_{\omega}} \ls 2^{k'-3k} \vm{t-s} $ it follows from \eqref{dispestt2} with $ l'=-k $.
\end{proof}
\

\subsection{\bf The null frame decay bound, $ d=4 $} \

\

Let $ \bar{\omega} \in \mb{S}^3 $ and let $ \kappa_l $ be a spherical cap of angle $ 2^l $ such that $ \angle(\kappa_l,\pm \bar{\omega}) \simeq 2^{l} $.
Let $ \lmd=\frac{1}{\sqrt{1+2^{-2p}}} $, which together with $ \bar{\omega} $ defines the frame \eqref{frame} and the coordinates in this frame
$$ t_{\bar{\omega}}=(t,x) \cdot \bar{\omega}^{\lmd}, \quad  x^1_{\bar{\omega}}=(t,x) \cdot \bar{\bar{\omega}}^{\lmd}, \quad x_{\bar{\omega},i}'=x \cdot  \bar{\omega}_i^{\perp} $$

Suppose $ a_l(\xi) $ is a smooth function adapted to $ \{ \vm{\xi} \simeq 2^k, \ \frac{\xi}{\vm{\xi}} \in \kappa_l  \}$ and  consider 
$$ K_l^{a_l}(t,x;s,y) \defeq \int e^{- i \psi_{\pm}^k (t,x,\xi)+\psi_{\pm}^k(s,y,\xi) } e^{\pm i t(t-s) \jb{\xi}} e^{i (x-y) \xi} a_l(\xi)   \dd \xi. $$

\begin{proposition}
Suppose $ \max(2^{-p},2^{-k})  \ll 2^{l} \simeq \angle(\kappa_l,\pm \bar{\omega}). $ Then, we have
\be \label{NEdecay} \vm{ K_l^{a_l}(t,x;s,y) } \ls 2^{4k+3l} \frac{1}{\jb{2^{k+2l} \vm{t-s}}^N} \frac{1}{\jb{2^{k+l} \vm{x'_{\bar{\omega}}-y'_{\bar{\omega}}}}^N} \jb{2^k \vm{t_{\bar{\omega}}-s_{\bar{\omega}}}}^{2N}  \ee
Moreover, the implicit constant depends on only $ 2N+1 $ derivatives of $ a_l $.
\end{proposition}

\begin{proof}
We prove that the phase is non-stationary due to the angular separation.
Suppose $ \pm=+ $ and let
$$ \phi(\xi)=(t-s) \jb{\xi}+(x-y) \cdot \xi, \qquad \nabla \phi=(t-s) \frac{\xi}{\jb{\xi}}+x-y.  $$ 
Choosing the right $ \bar{\omega}_i^{\perp} $ we obtain
$$  \nabla \phi \cdot \bar{\omega}_i^{\perp} \simeq 2^{l} (t-s)+\vm{x'_{\bar{\omega}}-y'_{\bar{\omega}}}. $$
When $ 2^{l}  \vm{t-s} \ll \vm{x'_{\bar{\omega}}-y'_{\bar{\omega}}} $ we obtain 
$ \vm{K_l^{a_l}} \ls  2^{4k+3l} \jb{2^{k+l} \vm{x'_{\bar{\omega}}-y'_{\bar{\omega}}}}^{-2N} $, which implies \eqref{NEdecay}. Similarly when $ \vm{x'_{\bar{\omega}}-y'_{\bar{\omega}}}  \ll 2^{l}  \vm{t-s} $ we get $  \vm{K_l^{a_l}} \ls  2^{4k+3l} \jb{2^{k+l} 2^{l} \vm{t-s}}^{-2N} $ which also suffices. 

We use Prop \ref{phasederProp} to control the contribution of $ \psi_{\pm}^k (t,x,\xi)-\psi_{\pm}^k(s,y,\xi) $.

Now assume $ 2^{l}  \vm{t-s} \simeq \vm{x'_{\bar{\omega}}-y'_{\bar{\omega}}} $. When $ ( 2^{2l}  \vm{t-s}\simeq ) \ 2^{l} \vm{x'_{\bar{\omega}}-y'_{\bar{\omega}}} \ls  \vm{t_{\bar{\omega}}-s_{\bar{\omega}}} $ estimating the integrand by absolute values we get $ \vm{ K_l^{a_l}} \ls 2^{4k+3l} $, which suffices in this case.

Now we assume $ \vm{t_{\bar{\omega}}-s_{\bar{\omega}}} \ll 2^{2l}  \vm{t-s} \simeq 2^l \vm{x'_{\bar{\omega}}-y'_{\bar{\omega}}} $. 

Since $(x-y) \cdot \bar{\omega}=- \lmd (t-s)+(t_{\bar{\omega}}-s_{\bar{\omega}}) \sqrt{1+\lmd^2} $, we have

$$  \nabla \phi \cdot \bar{\omega}=(t-s) \lpr \frac{\vm{\xi}}{\jb{\xi}} \frac{\xi}{\vm{\xi}} \cdot \bar{\omega} - \lmd \rpr+(t_{\bar{\omega}}-s_{\bar{\omega}}) \sqrt{1+\lmd^2}  $$
We estimate
$$ \frac{\vm{\xi}}{\jb{\xi}}-1 \simeq -2^{-2k}, \quad \frac{\xi}{\vm{\xi}} \cdot \bar{\omega} -1\simeq -2^{2l}, \quad \lmd-1 \simeq -2^{-2p} $$
From the hypothesis on $ 2^{l} $ we conclude that this term dominates so
$$ \vm{\nabla \phi \cdot \bar{\omega}} \gtrsim 2^{2l} \vm{t-s}, $$
which implies \eqref{NEdecay} as a non-stationary phase estimate.
\end{proof}

Now we consider frequency localized symbols and look at the $ TT^{*} $ operator
\be \label{ttstarop}
 e^{- i \psi_{\pm}^k}_{<k} (t,x,D)    e^{\pm i t(t-s) \jb{D}} a_l(D) e^{i \psi_{\pm}^k}_{<k} (D,s,y) \ee
from $ L^2(\Sigma) \to  L^2(\Sigma) $, where $ \Sigma=(\bar{\omega}^{\lmd})^{\perp} $  with kernel
$$ K_{<k}^l(t,x;s,y) \defeq \int e^{- i \psi_{\pm}^k(t,x,\xi)}_{<k}     e^{\pm i t(t-s) \jb{\xi}} e^{i (x-y) \xi} a_l(\xi) e^{i \psi_{\pm}^k(s,y,\xi)}_{<k}  \dd \xi. $$
for $ (t,x;s,y) \in \Sigma \times \Sigma $, i.e. $ t_{\bar{\omega}}=s_{\bar{\omega}}=0 $.

\begin{proposition} \label{Propnullekernel}
Suppose $ \max(2^{-p},2^{-k})  \ll 2^{l} \simeq \angle(\kappa_l,\pm \bar{\omega}). $ Then,
\be \label{nullekernelbd} \vm{ K_{<k}^l(t,x;s,y) } \ls 2^{4k+3l} \frac{1}{\jb{2^{k+2l}  \vm{x^1_{\bar{\omega}}-y^1_{\bar{\omega}}}}^N} \frac{1}{\jb{2^{k+l} \vm{x'_{\bar{\omega}}-y'_{\bar{\omega}}}}^N}  \ee
holds when $ \lmd (t-s)+(x-y) \cdot \bar{\omega}=0 $.
\end{proposition}

\begin{corollary} \label{Cornullframe}
Suppose $ \max(2^{-p},2^{-k})  \ll 2^{l} \simeq \angle(\kappa_l,\pm \bar{\omega}) $. Then 
\be 2^l  e^{- i \psi_{\pm}^k}_{<k} (t,x,D)    e^{\pm i t \jb{D}} P_k P_{\kappa_l} : L^2_x \to  L^{\infty}_{t_{\bar{\omega}, \lmd}} L^2_{x_{\bar{\omega},\lmd}}. \ee
\end{corollary}
\begin{corollary} \label{CorNE}
Let $ \calC=\calC_{k'}(l') $. Then 
\be e^{- i \psi_{\pm}^k}_{<k} (t,x,D)    e^{\pm i t \jb{D}} P_k P_{\calC} : L^2_x \to NE_{\calC}^{\pm}. \ee
\end{corollary}
\begin{proof}[Proof of Prop. \ref{Propnullekernel}] We average using \eqref{averaging} to write $ K_{<k}^l(t,x;s,y) $ as
$$  \iint e^{- i T_z \psi_{\pm}^k(t,x,\xi)}  e^{+ i T_w\psi_{\pm}^k(s,y,\xi)}  a_l(\xi) e^{\pm i (t-s) \jb{\xi}}  e^{i (x-y) \xi}  \dd \xi m_k(z) m_k(w) \dd z \dd w= $$
\be =\int T_z T_w  K^{a(z,w)}_l (t,x;s,y) m_k(z) m_k(w) \dd z \dd w  \label{intermint1}  \ee
where $ a(z,w)(\xi)=e^{-i (z-w)\cdot(\pm \jb{\xi},\xi)} a_l(\xi) $. Since  $ t_{\bar{\omega}}=s_{\bar{\omega}}=0 $ using \eqref{NEdecay} we obtain 
\begin{align*} 
& \vm{T_z T_w  K^{a(z,w)}_l (t,x;s,y)}  \ls \jb{2^k (\vm{z}+\vm{w})}^{2N+1} 2^{4k+3l} \times  \\
& \quad \times \jb{2^{k+2l} \vm{t-s+z_1-w_1}}^{-N} \jb{2^{k+l} \vm{x'_{\bar{\omega}}-y'_{\bar{\omega}}+z'_{\bar{\omega}}-w'_{\bar{\omega}} }}^{-N} \jb{2^k \vm{z-w}}^{2N}
\end{align*}
We obtain \eqref{nullekernelbd} from the integral \eqref{intermint1} using the rapid decay
$$ \jb{2^k (\vm{z}+\vm{w})}^{2N+1} \jb{2^k \vm{z-w}}^{2N}  m_k(z) m_k(w) \ls \jb{2^k \vm{z}}^{-N_2} \jb{2^k \vm{w}}^{-N_2}. $$
for any $ N_2 $, and by repeatedly applying
$$ \int_{\mb{R}} \frac{1}{\jb{\al \vm{a-r}}^{N}} \frac{2^k}{\jb{2^k \vm{r} }^{N_2}} \dd r \ls \frac{1}{\jb{\al \vm{a}}^{N}} $$
for $ \al \leq 2^k $ and $ N_2 $ large enough. Note that here $ \vm{t-s} \simeq \vm{x^1_{\bar{\omega}}-y^1_{\bar{\omega}}} $.
\end{proof}

\begin{proof}[Proof of Corollary \ref{Cornullframe}] By translation invariance, it suffices to prove that the operator is bounded from $ L^2_x \to L^2(\Sigma) $. By a $ TT^* $ argument this follows if we prove $ 2^{2l} \times $  \eqref{ttstarop} $ : L^2(\Sigma) \to  L^2(\Sigma) $, for which we use Schur's test. Indeed, the kernel of this operator is $ 2^{2l} K_{<k}^l(t,x;s,y) $ on $ \Sigma \times \Sigma $, which is integrable on $ \Sigma $ by \eqref{nullekernelbd}.
\end{proof}

\begin{proof}[Proof of Corollary \ref{CorNE}]
Recall definition \eqref{NE:norm}. For any $ \bar{\omega} $, $ \lmd=\lmd(p) $ such that $ \angle(\bar{\omega},\pm \calC)  \gg \max(2^{-p},2^{-k}, 2^{l'+k'-k}) $ we may define $ 2^l \simeq \angle(\bar{\omega},\pm \calC) $ and $ \kappa_l \supset \calC $ so that Corollary \ref{Cornullframe} applies.
\end{proof}
\

\section{Proof of Theorem \ref{Renormalization:thm}} \label{sec:pf:thm:ren}

\subsection{Proof of the fixed time $ L^2_x $ estimates \eqref{renbd}, \eqref{renbdt}, \eqref{renbd2}} \

\

The following proposition establishes the $ L^2_x $ part of \eqref{renbd}, \eqref{renbdt}.

\begin{proposition} \label{fixedtimeprop} For any $ k \geq 0 $, the mappings

\be \label{fixedtime1} e^{\pm i \psi_{\pm}^k} (t_0,x,D) \bar{P}_k : L^2_x \to L^2_x \ee
\be \label{fixedtime2} e_{<h}^{\pm i \psi_{\pm}^k} (t_0,x,D) \bar{P}_k : L^2_x \to L^2_x \ee
\be \label{fixedtime3} \nabla_{t,x} e_{<h}^{\pm i \psi_{\pm}^k} (t_0,x,D) \bar{P}_k : L^2_x \to \ep 2^k L^2_x \ee

hold for any fixed $ t_0 $, uniformly in $ h,t_0\in \mb{R} $. By duality, the same mappings hold for right quantizations.
\end{proposition}

\begin{proof}

\pfstep{Step~1}
First we prove \eqref{fixedtime1} by considering the $ TT^* $ operator
$$ e^{\pm i \psi_{\pm}^k} (t_0,x,D) \bar{P}_k^2 e^{\mp i \psi_{\pm}^k} (D,y,t_0) $$ 
with kernel $ K_k^a(t_0,x,t_0,y) $ defined by \eqref{kernel:a}. 

% $$ K(x,y)=\int e^{\pm i \psi_{<k,\pm}} (t,x,\xi) a_k(\xi) e^{i (x-y) \xi} e^{\mp i \psi_{<k,\pm}}(t,y,\xi) \dd \xi $$
% where $ a_k(\xi) $ is a smooth bump function adapted to the $ 2^k $ - annulus.
Due to the $ (x,y) $ symmetry and Schur's test it suffices to show
$$ \sup_{x} \int \vm{K_k^a(t_0,x,t_0,y)} \dd y \ls 1. $$
This follows from Prop. \ref{stat_phase_a}. 
% \be \label{kernelbound} \vm{K(x,y)} \leq C 2^{nk} \frac{1}{(1+2^k \vm{x-y})^{n+1}}  \ee
% By a change of variable we have
% $$ K(x,y)= 2^{nk} \int_{\vm{\xi} \sim 1} e^{i 2^k \vm{x-y} \phi_k(t,x,y,\xi)} \varphi_0 (\xi) \dd \xi $$
% where 
% $$ \phi_k(t,x,y,\xi)=\frac{x-y}{\vm{x-y}} \cdot \xi \pm \frac{ \psi_{<k,\pm}(t,x, 2^k\xi)- \psi_{<k,\pm} (t,y, 2^k\xi)}{2^k \vm{x-y}} $$ 
% Since \eqref{kernelbound} is obvious for $ 2^k \vm{x-y} \ls 1 $, we assume $ \lmd =2^k \vm{x-y} \gg 1 $ and use the non-stationary phase estimate in Proposition \ref{nonstationary}. From \eqref{phaseder} we can conclude that $ \vm{\phi_k'}>1/4 $ and that $ C $ is bounded since $ \phi_k $ stays in a bounded set in $ C^{n+1}( \vm{\xi} \sim 1) $.
\pfstep{Step~2} Now we prove \eqref{fixedtime2} using \eqref{averaging} and \eqref{fixedtime1}. For $ u \in L^2_x $ we write 
$$   e^{\pm i \psi_{\pm}^k}_{<h}(t_0,x,D) \bar{P}_k u  =\int_{\mb{R}^{d+1}} m_h(s,y) e^{\pm i \psi_{\pm}^k}(t_0+s,x+y,D) [ \bar{P}_k u_{y} ] \dd s \dd y  $$
where $ \hat{u}_{y}(\xi)=e^{-i y \xi} \hat{u}_0(\xi) $.
By Minkowski's inequality, \eqref{fixedtime1} for $ t_0+s $, translation invariance of $ L^2_x $,  and the bound $ \vn{u_{y}}_{L^2_x} \leq \vn{u}_{L^2_x} $ we obtain 
\begin{align*} \vn{e^{\pm i \psi_{\pm}^k}_{<h}(t_0,x,D) \bar{P}_k u }_{L^2_x}  & \ls \int_{\mb{R}^{d+1}} m_h(s,y) \vn{e^{\pm i \psi_{\pm}^k}(t_0+s,\cdot,D) [ \bar{P}_k u_{y} ]}_{L^2_x} \dd s \dd y \\
& \ls \int_{\mb{R}^{d+1}} m_h(s,y) \vn{ \bar{P}_k u_{y}}_{L^2_x} \dd s \dd y \ls \vn{u}_{L^2_x}.
\end{align*}

\pfstep{Step~3} Since we have
$$ \nabla_{t,x} e^{\pm i \psi_{\pm}^k(t,x,\xi)} =\pm i \nabla_{t,x}\psi_{\pm}^k (t,x,\xi)e^{\pm i \psi_{\pm}^k(t,x,\xi)}  $$ 
using \eqref{decomp4}, \eqref{fixedtime1} and Lemma \ref{decomp_lemma} we obtain
$$ \vn{ \nabla_{t,x} e^{\pm i \psi_{\pm}^k}(t,x,D)P_k}_{L^1L^2 \to L^1 L^2} \ls \ep 2^k $$
Applying this to $ \phi(t,x)=\delta_{t_0}(t) \otimes u(x) $ (or rather with an approximate to the identity $ \eta_{\ep} $ converging to $ \delta_{t_0} $ in $ t $) we obtain
$$ \nabla_{t,x} e^{\pm i \psi_{\pm}^k} (t_0,x,D) P_k : L^2_x \to \ep 2^k L^2_x $$
for any $ t_0 $. By averaging this estimate as in Step 2 we obtain \eqref{fixedtime3}.
\end{proof}

\begin{remark}
The same argument also shows
\be \label{fixedtime1b} e^{\pm i \psi^{k}_{<k'',\pm}} (t_0,x,D) P_k : L^2_x \to L^2_x
\ee 
\end{remark}

\

Now we turn to the proof of \eqref{renbd2}.

\begin{proposition} \label{fixedtime4}  Let $ k \geq 0 $. For any $ t_0 $ we have
$$ e_{<k}^{-i \psi_{\pm}^k} (t_0,x,D) e_{<k}^{i \psi_{\pm}^k} (D,t_0,y)-I : \bar{P}_k L^2_{x} \to \ep^{\frac{1}{2}} L^2_{x}  $$
\end{proposition}

\begin{proof}
\pfstep{Step~1} First, let us note that  
$$  e_{<k}^{-i \psi_{\pm}^k} (t_0,x,D) [ e_{<k}^{i \psi_{\pm}^k} (D,t_0,y)\bar{P}_k- \bar{P}_k e_{<k}^{i \psi_{\pm}^k} (D,t_0,y)] :  L^2_{x} \to \ep L^2_{x}  $$
This follows from \eqref{fixedtime2} and from \eqref{spec_decomp2}, \eqref{fixedtime3}.

The kernel of $  e_{<k}^{-i \psi_{\pm}^k} (t_0,x,D) \bar{P}_k e_{<k}^{i \psi_{\pm}^k} (D,t_0,y) $ is 
$$  K_{<k}(x,y)= \int e^{- i \psi_{\pm}^k}_{<k} (t_0,x,\xi) a(\xi /2^k) e^{i (x-y) \xi} e^{+ i \psi_{\pm}^k}_{<k} (t_0,y,\xi) \dd \xi  $$
while the kernel of $ \bar{P}_k $ is $ 2^{dk} \check{a}(2^k(x-y)) $. Thus, by Schur's test it remains to prove
\be \label{kerSchur} \sup_x \int \vm{K_{<k}(x,y)-2^{dk} \check{a}(2^k(x-y))} \dd y \ls \ep^{\frac{1}{2}}.  \ee
\pfstep{Step~2} For large $ \vm{x-y} $ we will use
\be  \label{krbd1} 2^{dk} \vm{\check{a}(2^k(x-y))},\ \vm{K_{<k}(x,y)} \ls \frac{2^{dk}}{(1+2^k \vm{x-y})^{2d+1}}.
\ee 
The bound for $  \check{a} $ is obvious. Recalling \eqref{averaging} we write $  K_{<k}(x,y) $ as
$$  \iint e^{- i T_z \psi_{\pm}^k(t_0,x,\xi)}  e^{+ i T_w\psi_{\pm}^k(t_0,y,\xi)}  a(\xi /2^k) e^{i (x-y) \xi}  \dd \xi m_k(z) m_k(w) \dd z \dd w= $$
\be =\int T_z T_w  K^{a(z,w)}_k (t_0,x,t_0,y) m_k(z) m_k(w) \dd z \dd w  \label{intermint}  \ee
where $ z=(t,z'),\ w=(s,w') $,  $ a(z,w)(\xi)=e^{-i2^k (z-w)\cdot(\pm \jb{\xi}_k,\xi)} a(\xi) $ and
$$ K^{a}_k (t,x,s,y)=\int e^{- i \psi_{\pm}^k} (t,x,\xi) a(\xi /2^k) e^{\pm i (t-s) \jb{\xi}+i (x-y) \xi} e^{+ i \psi_{\pm}^k}(s,y,\xi) \dd \xi $$
From Prop. \ref{stat_phase_a}, on the region $ \vm{ \vm{t-s}-\vm{x-y+z'-w'}} \gg 2^{-2k} \vm{t-s} $ we have
$$ \vm{T_z T_w  K^{a(z,w)}_k (t_0,x,t_0,y)} \ls \jb{2^k (\vm{z}+\vm{w})}^N \frac{2^{dk}}{\jb{2^k ( \vm{t-s}-\vm{x-y+z'-w'})}^N} $$
Over this region, the integral \eqref{intermint} obeys the upper bound in \eqref{krbd1}. This can be seen by repeatedly applying
$$ \int_{\mb{R}} \frac{1}{(1+2^k \vm{r-a})^N} \frac{2^k}{(1+2^k r)^{N_1}} \dd r \ls \frac{1}{(1+2^k a)^{N-1}}, \quad N_1 \geq 2 N $$
and for any $ N_2 $
$$ \jb{2^k (\vm{z}+\vm{w})}^N  m_k(z) m_k(w) \ls \jb{2^k \vm{z}}^{-N_2} \jb{2^k \vm{w}}^{-N_2}. $$
On the region $ \vm{ \vm{t-s}-\vm{x-y+z'-w'}} \ls 2^{-2k} \vm{t-s} $, we use the term $ \jb{2^k (t-s)}^{-N} $ from the rapid decay of $ m_k(z), m_k(w) $ and bound 
$$ \frac{1}{\jb{2^k (t-s)}^{N}} \ls \frac{1}{\jb{2^k(  \vm{t-s}-\vm{x-y+z'-w'})}^N}, \quad \vm{T_z T_w  K^{a(z,w)}_k} \ls 2^{dk}  $$ 
which imply the upper bound in \eqref{krbd1} as before.

\pfstep{Step~3} The kernel of $ e_{<k}^{-i \psi_{\pm}^k} (t_0,x,D) \bar{P}_k e_{<k}^{i \psi_{\pm}^k} (D,t_0,y)-\bar{P}_k $ obeys the bound
 \be \label{krbd2} \vm{K_{<k}(x,y)-2^{dk} \check{a}(2^k(x-y))} \ls \ep 2^{dk} (3+2^k \vm{x-y}). \ee
Indeed, we write $ K_{<k}(x,y)-2^{dk} \check{a}(2^k(x-y)) $ as
$$ 2^{dk} \iint ( e^{- i T_z \psi_{\pm}^k(t_0,x,2^k \xi)+ i T_w\psi_{\pm}^k(t_0,y,2^k \xi)}-1)  a(\xi) e^{i 2^k (x-y) \xi}  \dd \xi m_k(z) m_k(w) \dd z \dd w $$
and by \eqref{phasediff}, we bound
\begin{align*} \vm{ e^{- i T_z \psi_{\pm}^k(t_0,x,2^k \xi)+ i T_w\psi_{\pm}^k(t_0,y,2^k \xi)}-1} & \ls \ep \log(1+2^k (\vm{x-y}+ \vm{z}+\vm{w})) \\
& \ls \ep [1+2^k (\vm{x-y}+ \vm{z}+\vm{w})]. 
\end{align*}
Bounding by absolute values and integrating in $ z $ and $ w $ we obtain \eqref{krbd2}.

\pfstep{Step~4} Now we prove \eqref{kerSchur}. We integrate \eqref{krbd2} on $ \{ y \ | \ \vm{x-y} \leq R \} $ and integrate \eqref{krbd1} on the complement of this set, for $ (2^k R)^{d+1} \simeq \ep^{-\frac{1}{2}} $. We obtain
$$ \text{LHS} \ \eqref{kerSchur} \ls \ep (2^kR)^{d+1}+ \frac{1}{(2^k R)^{d+1}} \ls \ep^{\frac{1}{2}}.  $$
\end{proof}

\subsection{Proof of the $ \bar{N}_k $, $ \bar{N}_{k} ^* $ estimates \eqref{renbd}, \eqref{renbdt}, \eqref{renbd2}} \ 

\

In the proof we will need the following lemma.

\begin{lemma} \label{Nspacelemma}
For $ k \geq 0 $, $ k \geq k' \geq j-O(1) $ and for both quantizations, we have:
\be \label{Nlemmabd} 2^{j/2} \vn{\bar{Q}_{j} e^{\pm i \psi_{\pm}^{k}}_{k'}  \bar{P}_k G}_{L^2_{t,x}} \ls \ep 2^{\delta(j-k')} \vn{G}_{\bar{N}_k^{*}}, \ee
and thus, by duality
\be \label{Nlemmaest} 2^{\frac{j}{2}} \vn{\bar{P}_k e^{\pm i \psi_{\pm}^{k}}_{k'} \bar{Q}_{j} F_k}_{\bar{N}_k} \ls \ep 2^{\delta(j-k')}  \vn{F_k}_{L^2_{t,x}}    \ee
\end{lemma}

\begin{corollary} \label{coremb}
For $k \geq 0,\ l\leq 0 $ we have
$$  \bar{Q}_{<k+2l} ( e^{- i \psi_{\pm}^{k}}_{< k}-e^{- i \psi_{\pm}^{k}}_{< k+2l})(t,x,D) \bar{P}_k : \bar{N}_k^{*} \to \bar{X}^{1/2}_1 $$
\end{corollary}

\begin{proof}
This follows by summing over \eqref{Nlemmabd}.
\end{proof}
The proof of this lemma is a bit long and is defered to the end of this section. The following proposition completes the proofs of \eqref{renbd}, \eqref{renbdt}, \eqref{renbd2}.

\begin{proposition}\label{Nspacemapprop} For any $ k \geq 0 $, denoting $ \psi=\psi^{k}_{\pm} $, one has:

\be \label{Nspacemap1}  e_{< k}^{\pm i \psi} (t,x,D), e_{< k}^{\pm i \psi} (D,s,y) : \bar{N}_{k} \to  \bar{N}_{k} \ee

\be \label{Nspacemap2}  \pt_{t,x} e_{< k}^{\pm i \psi} (t,x,D), \pt_{t,x} e_{< k}^{\pm i \psi} (D,s,y) : \bar{N}_{k} \to  \ep 2^k \bar{N}_{k} \ee

\be \label{Nspacemap3}  e_{< k}^{-i \psi} (t,x,D) e_{< k}^{i \psi} (D,s,y)-I :  \bar{N}_{k} \to \ep^{\frac{1}{2}} \bar{N}_{k} \ee
By duality, the same mappings hold for $ \bar{N}_{k} ^* $ in place of $ \bar{N}_{k}  $.
\end{proposition}
\begin{proof}
\pfstep{Step~1} Since $ \bar{N}_{k}  $ is defined as an atomic space, it suffices to prove \eqref{Nspacemap1}, \eqref{Nspacemap2} applied to $ F $ when $ F $ is an $ L^1 L^2 $-atom ($ \vn{F}_{L^1 L^2} \leq 1 $) or to $ \bar{Q}_j F $ an $ \bar{X}_{1}^{-\frac{1}{2}} $-atom ( $ 2^{-\frac{j}{2}} \vn{ \bar{Q}_j F}_{L^2_{t,x}} \leq 1 $). The first case follows from integrating the pointwise in $ t $  bounds \eqref{fixedtime2}, \eqref{fixedtime3}
$$ \vn{ (e_{< k}^{\pm i \psi}, \ep^{-1}2^{-k} \nabla e_{< k}^{\pm i \psi})F_k(t)}_{L^2_x} \ls \vn{F_k(t)}_{L^2_x} $$
for both the left and right quantizations. 

Now consider the second case. We split 
$$ e_{< k}^{\pm i \psi}=e_{< \min(j,k)}^{\pm i \psi}+( e_{< k}^{\pm i \psi}-e_{< \min(j,k)}^{\pm i \psi}). $$ 
Note that $ e_{< \min(j,k)}^{\pm i \psi} \bar{Q}_j F= \tilde{\bar{Q}}_j e_{< \min(j,k)}^{\pm i \psi} \bar{Q}_j F $ and thus the bound
$$ \vn{e_{< \min(j,k)}^{\pm i \psi} \bar{Q}_j F}_{\bar{N}_{k} } \ls 2^{-\frac{j}{2}} \vn{e_{< \min(j,k)}^{\pm i \psi} \bar{Q}_j F}_{L^2_{t,x}} \ls 2^{-\frac{j}{2}} \vn{\bar{Q}_j F}_{L^2_{t,x}}   $$
follows from integrating \eqref{fixedtime2}. The same argument applies to $ \nabla e_{< \min(j,k)}^{\pm i \psi} $ using  \eqref{fixedtime3}.
 The remaining estimate, for $ j \leq k $
\be \label{intermbd} \vn{(e_{< k}^{\pm i \psi}-e_{< j}^{\pm i \psi})  \bar{Q}_j F_k}_{\bar{N}_{k} }  \ls \ep 2^{-\frac{j}{2}} \vn{\bar{Q}_j F_k}_{L^2_{t,x}} \ee
follows by summing \eqref{Nlemmaest} in $ k'$. Note that \eqref{intermbd} remains true with $ e^{\pm i \psi} $ replaced by $ 2^{-k} \nabla e^{\pm i \psi} $, because \eqref{Nlemmaest} remains true, which concludes \eqref{Nspacemap2}. To see this, one writes $ 2^{-k'} \nabla e^{\pm i \psi}_{k'}=L e^{\pm i \psi}_{k'} $ where $ L $ is disposable and use translation invariance and \eqref{Nlemmaest}.

\pfstep{Step~2} To prove \eqref{Nspacemap3}, since that operator is self-adjoint, we prove that it is bounded from $ \bar{N}_{k} ^* \to \ep^{\frac{1}{2}} \bar{N}_{k} ^* $, where $ \bar{N}_{k} ^* \simeq L^{\infty}L^2 \cap \bar{X}_{\infty}^{\frac{1}{2}} $. The $ L^{\infty}L^2 $ mapping follows from Prop. \ref{fixedtime4}, so it remains to prove
$$ 2^{\frac{j}{2}} \vn{\bar{Q}_j \tilde{P}_k [ e_{< k}^{-i \psi} (t,x,D) e_{< k}^{i \psi} (D,s,y)-I ] F_k}_{L^2_{t,x}} \ls \ep^{\frac{1}{2}} \vn{F_k}_{L^{\infty}L^2 \cap \bar{X}_{\infty}^{\frac{1}{2}}}  $$
For $ \bar{Q}_{>j-c} F_k $ we can discard $ \bar{Q}_j \tilde{P}_k $ and since $ \vn{\bar{Q}_{>j-c} F_k}_{L^2_{t,x}} \ls 2^{-j/2} \vn{F_k}_{\bar{X}_{\infty}^{\frac{1}{2}}} $, the bound follows for this component from Prop. \ref{fixedtime4} by integration . 

For $ \bar{Q}_{\leq j-c} F_k $ the claim follows by adding the following 
$$  2^{\frac{j}{2}} \bar{Q}_j \tilde{P}_k [ e_{< k}^{-i \psi} -e_{< j}^{-i \psi} ](t,x,D) e_{< k}^{i \psi}(D,s,y) \bar{Q}_{\leq j-c} : \bar{N}_{k} ^* \to \ep L^2_{t,x}    $$ 
$$  2^{\frac{j}{2}} \bar{Q}_j \tilde{P}_k e_{< j}^{-i \psi}(t,x,D) [ e_{< k}^{i \psi} -e_{< j}^{i \psi} ](D,s,y)  \bar{Q}_{\leq j-c} : \bar{N}_{k} ^* \to \ep L^2_{t,x}    $$
since
$$ \bar{Q}_j \tilde{P}_k I \bar{Q}_{\leq j-c}=0, \qquad  \bar{Q}_j \tilde{P}_k e_{< j}^{-i \psi} e_{< j}^{i \psi} \bar{Q}_{\leq j-c}=0. $$ 
These mappings follow from \eqref{Nlemmabd}, \eqref{Nspacemap1} for $ \bar{N}_k^* $ and Prop. \ref{fixedtime4} and writing $ \bar{Q}_j \tilde{P}_k e_{< j}^{-i \psi}=\bar{Q}_j \tilde{P}_k e_{< j}^{-i \psi} \bar{Q}_{[j-5,j+5]} $. 
\end{proof}

\

\subsection{Proof of the conjugation bound \eqref{conj}} \ 

\

In general, for pseudodifferential operators one has the composition property $ a(x,D) b(x,D)=c(x,D) $ where, in an asymptotic sense
$$ c(x,\xi) \sim \sum_{\al} \frac{1}{\al !} \pt_{\xi}^{\al} a(x,\xi)  D_x^{\al} b(x,\xi).
$$
In the present case this formula will be exact, as seen by differentiating under the integral in \eqref{left:quant}.
  
By definition \eqref{left:quant}, the symbol of $  e_{< k}^{-i \psi_{\pm}^k} (t,x,D) \Box_m $ is
\be  e_{< k}^{-i \psi_{\pm}^k(t,x,\xi)}  ( \pt_t^2+ \vm{\xi}^2+1).  \label{symb:01} \ee
By differentiating \eqref{left:quant}, we see that the symbol of  $  \Box_m e_{< k}^{-i \psi_{\pm}^k} (t,x,D) $ is
\be  e_{< k}^{-i \psi_{\pm}^k} ( \pt_t^2+ \vm{\xi}^2+1) + \Box e_{< k}^{-i \psi_{\pm}^k} + 2 \lpr \pt_t e_{< k}^{-i \psi_{\pm}^k} \pt_t - i (\nabla e_{< k}^{-i \psi_{\pm}^k}) \cdot \xi \rpr   \label{symb:02} \ee
while the symbol of the operator $ 2i ( A_{< k} \cdot \nabla) e_{< k}^{-i \psi_{\pm}^k} (t,x,D) $ is 
\be -2 e_{< k}^{-i \psi_{\pm}^k (t,x,\xi)} A_{< k}(t,x) \cdot \xi+ 2i \nabla e_{< k}^{-i \psi_{\pm}^k (t,x,\xi)} \cdot A_{< k}(t,x)   \label{symb:03} \ee

Now, the inequality \eqref{conj} follows from the following proposition.

\begin{proposition}
Denoting $ \psi= \psi_{\pm}^k $, we can decompose 
$$ e_{< k}^{-i \psi_{\pm}^k} (t,x,D) \Box_m - \Box_{m}^{A_{< k}} e_{< k}^{-i \psi_{\pm}^k} (t,x,D)=\sum_{i=0}^5 F_i(t,x,D) $$ 
where 
\begin{align*}
	F_0(t,x,\xi) :=&  2 \lpp ((\pm  \jb{\xi} \pt_t -\xi \cdot \nabla )  \psi(t,x,\xi) + A_{< k}(t,x) \cdot \xi ) e^{-i \psi(t,x,\xi)}   \rpp_{< k}  \\
	F_1(t,x,\xi) :=&  - \Box e_{< k}^{-i \psi(t,x,\xi)}\\
	F_2(t,x,\xi) :=& 2i \nabla e_{< k}^{-i \psi(t,x,\xi)} \cdot A_{< k}(t,x) \\
	F_3(t,x,\xi) :=& 2i^{-1}  \pt_t  e_{< k}^{-i \psi(t,x,\xi)} (i \pt_t \pm \jb{\xi})   \\
	F_4(t,x,\xi) :=& 2 \lpp \lpr A_{< k}(t,x) e^{-i \psi(t,x,\xi)} \rpr_{< k} - A_{< k}(t,x) e_{< k}^{-i \psi(t,x,\xi)} \rpp \cdot \xi
\end{align*}

and for all $ i=0,4 $ we have 
\be \label{cjglm} \vn{F_i(t,x,D) u_k}_{\bar{N}_{k}} \ls \ep \vn{u_k}_{L^{\infty} H^1}+ \ep 2^k \vn{(i \pt_t \pm \jb{D}) u_k}_{\bar{N}_k} \ee	
	
\begin{proof} The decomposition follows from \eqref{symb:01}-\eqref{symb:03} and basic manipulations. We proceed to the proof of \eqref{cjglm}. We will make use of the bound
$$  \vn{u_k}_{\bar{N}_{k}^*} \ls \vn{u_k}_{L^{\infty} L^2}+  \vn{(i \pt_t \pm \jb{D}) u_k}_{\bar{N}_k}, $$
for which we refer to the proof of Lemma \ref{waves}. Recall that we identify $ \bar{N}_k^* \simeq L^{\infty} L^2 \cap \bar{X}^{\frac{1}{2}}_{\infty} $.
\pfstep{Step~1}[The main term $F_0$] Recall the identity \eqref{nullvf} and the definitions \eqref{defn1}, \eqref{defn2}. For $ k=0 $ the term $ F_0(t,x,\xi) $ vanishes. Now assume $ k \geq 1 $ and write
$$  F_0(t,x,\xi)=2 ( ( \sum_{k_1<k-c}   \Pi^{\omega}_{\leq \delta(k_1-k)}  A_{k_1} \cdot \xi )e^{-i \psi(t,x,\xi)}   )_{< k} =  2 F'(t,x,\xi)_{< k}  $$
where
$$ F'(t,x,\xi)= a(t,x,\xi) e^{-i \psi(t,x,\xi)}_{< k}, \qquad a(t,x,\xi)  \defeq  \sum_{k_1<k-c}   \Pi^{\omega}_{\leq \delta(k_1-k)}  A_{k_1}(t,x) \cdot \xi  $$
By \eqref{spec_decomp3} we have

$$ \vn{ F'(t,x,D)- a(t,x,D) e^{-i \psi}_{< k}(t,x,D)}_{L^{\infty}L^2 \to L^1 L^2} \ls  \vn{\nabla_\xi a}_{D_k^1L^{2}(L^\infty)}  \vn{\nabla_x e^{-i \psi}_{< k}(t,x,D)}_{L^{\infty} (L^2)\to L^{2}(L^2)} 
$$ 
By lemma \ref{decomp_lemma}, \eqref{decomp3} and lemma \ref{loczsymb} we have
$$ \vn{(\nabla_x \psi e^{-i \psi})_{< k} (t,x,D)}_{L^{\infty} (L^2)\to L^{2}(L^2)}  \ls \vn{\nabla_{x} \psi^{<k}_{\pm}}_{D_k(L^2 L^{\infty})} \ls 2^{\frac{k}{2}} \ep$$
Summing over \eqref{decomp5}, we get $ \vn{\nabla_\xi a}_{D_k^1L^{2}(L^\infty)} \ls 2^{\frac{k}{2}} \ep $. Thus
$$ F'(t,x,D)- a(t,x,D) e^{-i \psi}_{< k}(t,x,D) : \bar{N}_{k}^* \to 2^k \ep \bar{N}_k $$
and it remains to prove
$$ a(t,x,D) e^{-i \psi}_{< k}(t,x,D) : \bar{N}_{k}^* \to 2^k \ep \bar{N}_k $$
By Proposition \ref{Nspacemapprop}, $ e^{-i \psi}_{< k}(t,x,D) $ is bounded on $ \bar{N}_{k}^* $. 

Assume $ -k \leq \delta (k_1-k) $ (the case $ \delta (k_1-k) \leq -k $ is analogous).  We decompose 
$$ a(t,x,\xi)=\sum_{k_1 <k-c} \sum_{\tht\in[2^{-k},2^{\delta (k_1-k)}]} a_{k_1}^{\tht} (t,x,\xi), $$ 
$$  a_{k_1}^{2^{-k}} (t,x,\xi) \defeq \Pi^{\omega}_{\leq 2^{-k}}  A_{k_1}(t,x) \cdot \xi, \qquad a_{k_1}^{\tht} (t,x,\xi) \defeq \Pi^{\omega}_{\tht}  A_{k_1}(t,x) \cdot \xi \quad ( \tht > 2^{-k}),  $$
and it remains to prove
$$ \vn{ a_{k_1}^{\tht} (t,x,D) v_k}_{\bar{N}_k} \ls \tht^{\frac{1}{2}} \ep 2^k \vn{v_k}_{\bar{N}_k^*}. $$
for all $ \tht=2^l,  \ l \geq -k $. First, using \eqref{decomp6} we have
$$ \vn{a_{k_1}^{\tht} (t,x,D) \bar{Q}_{>k_1+ 2l-c} v_k}_{L^1 L^2} \ls \vn{a_{k_1}^{\tht}}_{D_k L^2 L^{\infty}} \vn{\bar{Q}_{>k_1+ 2l-c} v_k}_{L^2_{t,x}} \ls \tht^{\frac{1}{2}} \ep 2^k \vn{v_k}_{\bar{N}_k^*} $$
Then, denoting $ f(t,x)=a_{k_1}^{\tht} (t,x,D) \bar{Q}_{<k_1+ 2l-c} v_k $ we have 
$$ \vn{ f}_{L^2_{t,x}} \ls \vn{a_{k_1}^{\tht}}_{D_k L^2 L^{\infty}} \vn{\bar{Q}_{<k_1+ 2l-c} v_k}_{L^{\infty} L^2} \ls 2^{\frac{3}{2}l}2^{\frac{k_1}{2}} \ep 2^k \vn{v_k}_{\bar{N}_k^*} $$
For each $ \xi $, the term $ \bar{Q}_j[ \Pi^{\omega}_{\tht}  A_{k_1}(t,x) \xi e^{i x \xi} \bar{Q}_{<k_1+ 2l-c} \hat{v_k}(t,\xi)] $ is non-zero only for $ j=k_1+2l+O(1) $ (by Remark \ref{rk:geom:cone} of Lemma \ref{geom:cone}). Thus,
$$ \vn{f}_{\bar{N}_k} \leq \vn{f}_{\bar{X}_1^{-1/2}} \ls \sum_{ j=k_1+2l+O(1)} \vn{\bar{Q}_j f}_{L^2_{t,x}} 2^{-\frac{j}{2}} \ls \tht^{\frac{1}{2}} \ep 2^k \vn{v_k}_{\bar{N}_k^*} $$ 
 
\pfstep{Step~2} [The terms $ F_1 $ and $ F_2 $] Since $ \Box_{t,x} \psi(t,x,\xi)=0 $ we have 
\begin{align*}
 F_1(t,x,\xi)= &\lpp  ( \vm{\pt_t \psi (t,x,\xi)}^2 -  \vm{\nabla \psi (t,x,\xi)}^2 )  e^{-i \psi (t,x,\xi)}   \rpp_{< k}, \\ 
 F_2(t,x,\xi)=& 2i A^j_{< k}(t,x) \lpr \pt_j \psi (t,x,\xi) e^{-i \psi (t,x,\xi)}   \rpr_{< k} 
\end{align*}

By lemma \ref{decomp_lemma} and \eqref{decomp3} we have
\begin{align*}
(\pt_{j} \psi e^{-i \psi}) (t,x,D) &  : L^{\infty} L^2 \to \ep 2^{\frac{k}{2}} L^2 L^2 \\
( \vm{\pt_{\al} \psi}^2 e^{-i \psi}) (t,x,D)  & : L^{\infty} L^2 \to \ep^2 2^k L^1 L^2 
\end{align*}
By lemma \ref{loczsymb} the same mappings hold for the $ < k $ localized symbols, which proves \eqref{cjglm} for $ F_1 $, while for $ F_2 $ we further apply H\" older's inequality together with $ \vn{A_{< k}}_{L^2 L^{\infty}} \ls 2^{k/2} \ep $. 

\pfstep{Step~3} [The term $F_3$] The bound follows by using \eqref{Nspacemap2} to dispose of 
$$ 2^{-k} \ep^{-1} \pt_t e_{< k}^{-i \psi}(t,x,D). $$
%Then, for functions $ u_k $ localized at $ \bar{Q}^{\pm} $-modulation $\leq k-c $ (i.e. to the region $ \vm{\tau \mp \jb{\xi}} \leq 2^{k-c} $) we have 
% $$ \ep \vn{ 2^k(i \pt_t \pm \jb{D}) u_k}_{N_k} \ls \ep \vn{\Box_m u_k}_{N_k} $$
% as seen by factoring $ \Box_m =-(i \pt_t - \jb{D})(i \pt_t + \jb{D}) $ and noting that the operator $ 2^k/ (i \pt_t \mp \jb{D}) \bar{Q}^{\pm}_{<k-c} $ is disposable.
\pfstep{Step~4}[The term $ F_4$] Using Lemma  \ref{comm_id} we write
$$ F_4(t,x,\xi)=2^{-k} \xi_j L( \nabla_{t,x}A^j_{< k}(t,x), e^{-i \psi (t,x,\xi)})  $$
As in lemma \ref{loczsymb}, by translation-invariance it suffices to prove
$$ 2^{-k} \vn{ \nabla_{t,x}A^j_{< k} e^{-i \psi} (t,x,D) \partial_j u_k}_{\bar{N}_k} \ls \ep 2^k \vn{e^{-i \psi} (t,x,D) u_k}_{\bar{N}_k^*} \ls \ep 2^k \vn{ u_k}_{\bar{N}_k^*}  $$
which follows from \eqref{est:phi2:freqAx} (observe that the $ \calH_{k'}^* $ term is zero when $ \Box A^j=0 $ and in this case the $ \bar{N}_{k'}^* $ norm of $ \phi $ suffices. One uses the derivative on $ A^j $ to do the $ k' $ summation).
\end{proof}

\end{proposition}

\subsection{Proof of the $ \bar{S}_k $ bound \eqref{renbd3}} \ 

\

We begin by stating a simple lemma that provides bounds for localized symbols.

\begin{lemma} \label{lemma:locz:symb}   Let $ X $ be a translation-invariant space of functions defined on $ \mb{R}^{d+1} $. Let $ P $ be a bounded Fourier multiplier. Suppose we have the bounded map
\be \label{symb:X}  e^{-i \psi}(t,x,D) e^{\pm i t \jb{D}} P:  L^2_x \to  X. \ee
Then, uniformly in $ h $, we also have the bounded map for localized symbols:
\be \label{locz:X}  e^{-i \psi}_{<h}(t,x,D) e^{\pm i t \jb{D}} P:  L^2_x \to  X. \ee
\end{lemma} 

\begin{proof} Recalling \eqref{averaging}, for $ u_0 \in L^2_x $ we write 
$$   e^{- i \psi}_{<h}(t,x,D) e^{\pm i t \jb{D}} P u_0  =\int_{\mb{R}^{d+1}} m_h(s,y) e^{-i \psi}(t+s,x+y,D) e^{\pm i(t+s) \jb{D}} P u_{s,y} \dd s \dd y  $$
where $ \hat{u}_{s,y}(\xi)=e^{\mp i s \jb{\xi}} e^{-i y \xi} \hat{u}_0(\xi) $.
By Minkowski's inequality, translation invariance of $ X $, \eqref{symb:X} and the bound $ \vn{u_{s,y}}_{L^2_x} \leq \vn{u_0}_{L^2_x} $ we obtain \eqref{locz:X}.
\end{proof}

We will apply this lemma for $ X $ taking the various norms that define $ \bar{S}_k $. 

The next lemma will be used to reduce estimates to the case of free waves.

\begin{lemma} \label{waves}
Let $ k \geq 0 $ and  $ X $ be a space of functions on $ \mb{R}^{1+n} $ with Fourier support in $ \{ \jb{\xi} \simeq 2^k \} $ (or a subset of it, such as a $ 2^{k'} \times ( 2^{k' +l'} )^3 $ box) such that
$$ \vn{e^{i t \sigma} f}_X \ls \vn{f}_X, \quad \forall \sigma \in \mb{R} $$
$$ \vn{1_{t>s} f}_X  \ls \vn{f}_X, \quad \forall s \in \mb{R} $$
\be \label{freesolren} \vn{e^{-i \psi}_{<h} (t,x,D) e^{\pm i t \jb{D}} u_0}_X \ls C_1 \vn{u_0}_{L^2} . \ee
hold for all $ f, u_0 $ and both signs $ \pm $. Then, we have 
\be \label{est:waves}
2^k \vn{e^{-i \psi}_{<h} (t,x,D) u}_X \ls C_1(  \vn{u[0]}_{H^1 \times L^2}+  \vn{\Box_m u}_{\bar{N}_k} ) \ee

If we only assume that \eqref{freesolren} holds for one of the signs $ \pm $, then \eqref{est:waves} still holds for functions $ u $ with Fourier support in $ \{\pm \tau \geq 0 \} $.
\end{lemma}

\begin{proof} We decompose $ \Box_m u=F^1+F^2 $ such that $ \vn{\Box_m u}_{\bar{N}_k} \simeq \vn{F^1}_{L^1 L^2}+\vn{F^2}_{\bar{X}_1^{-\frac{1}{2}}} $. By \eqref{freesolren} we can subtract free solutions from $ u $ and so we may assume that $ u[0]=(0,0) $. We may also assume that $ F^2 $ is modulation-localized to $ \vm{\tau-\jb{\xi}} \simeq 2^j,\ \tau \geq 0 $.   We define $ v=\frac{1}{\Box_m}F^2 $ and write $ u=u^1+u^2 $ where $ u^1 $ is the Duhamel term
$$ u^1(t)=\int_{\mb{R}} \frac{\sin( (t-s) \jb{D})}{\jb{D}} 1_{t>s} F^1(s)  \dd s - \sum_{\pm} \pm e^{\pm it \jb{D}} \int_{-\infty}^0 e^{\mp is \jb{D}} \frac{F^1(s)}{2i\jb{D}} \dd s  $$
$$ \text{and} \qquad  u^2=v-e^{it \jb{D}} w^1-e^{-i t \jb{D}} w^2 $$
so that $ \Box_m u^2 =0 $ and $ w^1,w^2 $ are chosen such that $ u^2[0]=(0,0) $.  

For the second part of $ u^1 $ we use \eqref{freesolren} together with
$$ \vn{\int_{-\infty}^0 e^{\mp is \jb{D}} \frac{F^1(s)}{2i\jb{D}} \dd s}_{L^2} \leq \int_{-\infty}^0 \vn{e^{\mp is \jb{D}} \frac{F^1(s)}{2i\jb{D}}}_{L^2}   \dd s \ls 2^{-k} \vn{F^1(s)}_{L^1 L^2}. $$
For the first part of $ u^1 $ we again write $ \sin( (t-s) \jb{D}) $ in terms of $ e^{\pm i (t-s) \jb{D}} $, and
$$ \vn{e^{-i \psi_{k,\pm}}_{<h} \int_{\mb{R}} \frac{e^{\pm i (t-s) \jb{D}}}{\jb{D}} 1_{t>s} F^1(s)  \dd s}_{X} \leq \int_{\mb{R}} \vn{1_{t>s} e^{-i \psi_{k,\pm}}_{<h} e^{\pm i (t-s) \jb{D}} \frac{F^1(s)}{\jb{D}}  }_X  \dd s $$
$$ \ls 2^{-k} C_1 \int_{\mb{R}} \vn{e^{\mp is \jb{D}} F^1(s)}_{L^2}   \dd s \leq 2^{-k} C_1 \vn{F^1(s)}_{L^1 L^2}.  $$
Now we turn to $ u^2$. For $ w^1,w^2 $ we use \eqref{freesolren} and, using Lemma \eqref{Sobolev_lemma}
$$ \vn{w^i}_{L^2} \ls \vn{(v,\frac{\pt_t}{\jb{D}}v)}_{L^{\infty}L^2} \ls 2^{\frac{j}{2}} \vn{(\frac{1}{\Box_m},\frac{i\pt_t- \jb{D}}{\jb{D} \Box_m}) F^2}_{L^2_{t,x}} \ls 2^{\frac{-j}{2}}2^{-k} \vn{F^2}_{L^2_{t,x}}   $$

Next, we write $ \tau=\rho+\jb{\xi}$ in the Fourier inversion formula
$$ v(t)=\int e^{i t \tau+ ix \xi} \mathcal{F}v (\tau, \xi) \dd \xi \dd \tau=\int_{\vm{\rho} \simeq 2^j} e^{i t \rho} e^{i t \jb{D}} \phi_{\rho} \dd \rho $$
for $ \hat{\phi_{\rho}}(\xi)=\mathcal{F}v (\rho+\jb{\xi}, \xi) $. Then
$$ \vn{e^{-i \psi_{k,\pm}}_{<h} v}_{X} \ls \int_{\vm{\rho} \simeq 2^j} \vn{ e^{i t \rho} e^{-i \psi_{k,\pm}}_{<h}(t,x,D) e^{i t \jb{D}} \phi_{\rho} }_X  \dd \rho \ls C_1 \int_{\vm{\rho} \simeq 2^j}  \vn{\phi_{\rho} }_{L^2_x} \dd \rho $$
By Cauchy-Schwarz we bound this by $ 2^{\frac{j}{2}} C_1 \vn{v}_{L^2_{t,x}} \ls 2^{\frac{-j}{2}}2^{-k} C_1 \vn{F^2}_{L^2_{t,x}} $.

If we only assume that \eqref{freesolren} holds for one of the signs $ \pm $, then we have the following variant
$$  \vn{e^{-i \psi}_{<h} (t,x,D) u}_X \ls C_1 ( \vn{u(0)}_{L^2}+\vn{(i \pt_t \pm \jb{D})u}_{\bar{N}_k})  $$
Now the Duhamel term is expressed in terms of one of the $ e^{\pm i t\jb{D}} $. For functions with Fourier support in $ \{\pm \tau \geq 0 \} $ we have $ \vn{(i \pt_t \pm \jb{D})u}_{N_k} \simeq 2^{-k} \vn{\Box_m u}_{N_k} $.  \end{proof} 

Now we are ready to begin the proof of \eqref{renbd3}. We will implicitly use Prop. \ref{Nk:orthog}.

For brevity, we drop the $ k $ and $ \pm $ subscripts and denote $ \psi=\psi^{k}_{\pm} $.

\subsubsection{\bf The Strichartz norms}

By Lemma \ref{waves}, the bound for $ \bar{S}_k^{Str} $ reduces to 

\begin{lemma}
For all $ k \geq 0 $ we have
$$ \vn{e^{-i \psi}_{<k} (t,x,D) e^{\pm i t \jb{D}} v_k}_{\bar{S}_k^{Str}} \ls \vn{v_k}_{L^2_x}  
$$
\end{lemma}
\begin{proof} Using Lemma \ref{lemma:locz:symb} this bound follows from
\be  e^{-i \psi}(t,x,D) e^{\pm i t \jb{D}} : \bar{P}_k L^2_x \to  \bar{S}_k^{Str}.  \label{reducedStr}     \ee
 We use the result of Keel-Tao on Strichartz estimates from \cite{Keel_endpointstrichartz}. As noticed in that paper (see sec. 6 and the end of sec. 5; see also \cite[sec. 5]{ShSt}, the $ L^2 L^r $ estimate also holds with $ L^r $ replaced by the Lorentz space $ L^{r,2} $. We need this only when $ d=4 $ for the $ L^2 L^{4,2} $ norm in \eqref{Str:KG}.
 
By change of variable, we rescale at frequency $ 2^0 $:
$$  U(t) \defeq  e^{- i \psi(\cdot/2^k,\cdot/2^k, 2^k \cdot)} (t,x,D) e^{\pm i t \jb{D}_k}$$

The $ L^2_x \to L^2_x $ bound follows from Prop. \ref{fixedtimeprop}. The $ L^1 \to L^{\infty} $ bound for $ U(t) U(s)^* $ follows from \eqref{dispestt1} for $ S_k^{Str,W} $ and from \eqref{dispestt12} for the other $ \bar{S}_k^{Str} $ norms in \eqref{Str:KG} when $ d=4$.
\end{proof}

\subsubsection{\bf The $ \bar{X}_{\infty}^{\frac{1}{2}} $ norms.} For any $ j \in \mb{Z} $ we show
\be
2^{\frac{1}{2}j} \vn{\bar{Q}_j e^{-i \psi}_{< k}(t,x,D) \phi_k}_{L^2_{t,x}} \ls \vn{\phi_k}_{L^{\infty}(H^1 \times L^2)} + \vn{\Box_m \phi_k}_{\bar{N}_k}. 
\ee
We separate 
$$ e^{-i \psi}_{< k}=e^{-i \psi}_{< \min(j,k)}+ \sum_{k'+C \in [j,k]} e^{-i \psi}_{k'} $$
For the first term we write 
$$ \bar{Q}_j e^{-i \psi}_{< \min(j,k)} \phi_k= \bar{Q}_j e^{-i \psi}_{< \min(j,k)} \bar{Q}_{[j-1,j+1]} \phi_k $$
Then we discard $ \bar{Q}_j e^{-i \psi}_{< \min(j,k)} $ and the estimate becomes trivial. The second term follows by summing over \eqref{Nlemmabd}.

\subsubsection{\bf The $ S_{box(k')} $ norms in \eqref{Sbar0}, $ k =0 $} For $ k'<0 $ we prove
$$ 2^{-\sg k'} ( \sum_{ C_{k'}} \vn{\bar{Q}_{<k'}^{\pm} P_{C_{k'}} e_{<0}^{-i \psi_{\pm}^0} (t,x,D) \phi}_{L^2 L^{\infty}}^2)^{1/2} \ls \vn{ (\phi,\pt_t \phi)(0)}_{L^2_x}+\vn{ \Box_m \phi}_{\bar{N}_0}  $$
We may assume $ \phi $ is Fourier supported in $ \pm \tau \geq 0 $. We split
$$  e^{-i \psi_{\pm}^0}_{<0}= (e^{-i \psi_{\pm}^0}_{<0}-e^{-i \psi_{\pm}^0}_{<k'})+ e^{-i \psi_{\pm}^0}_{<k'} $$
The estimate for the first term follows from Prop. \ref{Xembedding} and Cor. \ref{coremb}. For the second term we write 
$$ P_{C_{k'}} e^{-i \psi_{\pm}^0}_{<k'}=P_{C_{k'}} e^{-i \psi_{\pm}^0}_{<k'} \tilde{P}_{C_{k'}}. $$
Then we can discard $ \bar{Q}_{<k'}^{\pm} P_{C_{k'}} $ and prove
$$ 2^{-\sg k'} \vn{ e^{-i \psi_{\pm}^0}_{<k'} (t,x,D) \tilde{P}_{C_{k'}} \phi }_{L^2 L^{\infty}} \ls \vn{\tilde{P}_{C_{k'}} (\phi,\pt_t \phi)(0)}_{L^2_x}+\vn{\tilde{P}_{C_{k'}} \Box_m \phi}_{\bar{N}_0}  $$
By Lemma \ref{waves}, this reduces to 
$$  2^{-\sg k'} e^{-i \psi_{\pm}^0}_{<k'}(t,x,D) e^{\pm i t \jb{D}} \tilde{P}_{C_{k'}} \bar{P}_{0} :L^2_x \to L^2 L^{\infty}
$$
which follows from Corollary \ref{Cor:L2Linf} using Lemma \ref{lemma:locz:symb}.

\subsubsection{\bf The square summed $ \bar{S}_k^{\omega \pm}(l) $ norms, $ k \geq 1 $, first part} For any fixed $ l<0 $ we split
$$  e^{-i \psi}_{<k}= (e^{-i \psi}_{<k}-e^{-i \psi}_{<k+2l})+ e^{-i \psi}_{<k+2l}. $$
Here we treat the first term, while the second one is considered below. The bound
$$ 2^k ( \sum_{\omega}  \vn{P_l^{\omega} \bar{Q}_{<k+2l}^{\pm} (e^{-i \psi}_{<k}-e^{-i \psi}_{<k+2l}) \phi}_{\bar{S}_k^{\omega,\pm}(l)}^2 )^{\frac{1}{2}} \ls \vn{ \nabla_{t,x} \phi(0)}_{L^2_x}+\vn{\Box_m \phi}_{\bar{N}_k} $$
follows from Prop. \ref{Xembedding} and Cor. \ref{coremb}.

\subsubsection{\bf The square-summed $ L^2 L^{\infty} $ and $ \bar{S}_k^{Str} $ norms, $ k \geq 1 $}
Let $  l<0 $. It remains to consider $ e^{-i \psi}_{<k+2l} $. We fix $ \omega $ and the estimate we need boils down to square-summing the following over $ \omega $, after taking supremum over $ k'\leq k, \ l'<0 $, for $ k+2l \leq k'+l' \leq k+l $
$$ 2^{-\frac{k}{2}-\frac{d-2}{2}k'-\frac{d-3}{2}l' }  ( \sum_{\calC=\calC_{k'}(l')}  \vn{P_{\calC} P_l^{\omega} \bar{Q}_{<k+2l}^{\pm} e^{-i \psi}_{<k+2l}  \phi}_{L^2L^{\infty}}^2 )^{\frac{1}{2}} \ls \vn{\tilde{P}_l^{\omega} \nabla_{t,x} \phi(0)}_{L^2_x}+\vn{\tilde{P}_l^{\omega} \Box_m \phi}_{\bar{N}_k}    $$
Fix $ \calC=\calC_{k'}(l') $. Since $ k+2l \leq k'+l' $, one can write 
\be P_{\calC} e^{-i \psi}_{<k+2l}(t,x,D) \phi= P_{\calC} e^{-i \psi}_{<k+2l}(t,x,D) \tilde{P}_{\calC} \phi.  \label{loczcomuting} \ee
Then one can can discard $ P_{\calC} P_l^{\omega} \bar{Q}_{<k+2l} $ and prove
$$ 2^{-\frac{k}{2}-\frac{d-2}{2}k'-\frac{d-3}{2}l' } \vn{e^{-i \psi}_{<k+2l}(t,x,D) \tilde{P}_{\calC} \phi }_{L^2 L^{\infty}}\ls \vn{\tilde{P}_{\calC} \nabla_{t,x} \phi(0)}_{L^2_x}+ \vn{\tilde{P}_{\calC} \Box_m \phi}_{\bar{N}_k} $$
By Lemma \ref{waves}, this reduces to 
\be  e^{-i \psi}_{<k+2l}(t,x,D) e^{\pm i t \jb{D}} \tilde{P}_{\calC} :L^2_x \to 2^{\frac{k}{2}+\frac{d-2}{2}k'+\frac{d-3}{2}l' } L^2 L^{\infty}	\label{red:L2Linf}
\ee
which follows by Lemma \ref{lemma:locz:symb} from Corollary \ref{Cor:L2Linf}.

The same argument applies to $ \bar{S}_k^{Str} $ except that one uses \eqref{reducedStr} and Lemma \ref{lemma:locz:symb} instead of \eqref{red:L2Linf}.

\subsubsection{\bf The PW norms ($ d=4, k \geq 1$)} We fix $ l, -k \leq l', k', \omega, \calC=\calC_{k'}(l') $ as before and use \eqref{loczcomuting}. We discard $ P_{\calC} P_l^{\omega} \bar{Q}^{\pm}_{<k+2l} $ and prove 
$$ 2^{-\frac{3}{2}(k'+l')+k} \vn{e^{-i \psi}_{<k+2l}(t,x,D) \bar{Q}^{\pm}_{<k+2l} \tilde{P}_{\calC} \phi}_{PW^{\pm}_{\calC}} \ls \vn{\tilde{P}_{\calC} \nabla_{t,x} \phi(0)}_{L^2_x}+ \vn{\tilde{P}_{\calC} \Box_m \phi}_{\bar{N}_k}
$$
Let's assume $ \pm=+ $. By Lemma \ref{waves}, we reduce to
\be \label{PWwaves}  \vn{e^{-i \psi}_{<k+2l}(t,x,D) e^{it \jb{D}} \tilde{P}_{\calC} u_k}_{PW^{\pm}_{\calC} } \ls 2^{\frac{3}{2}(k'+l')} \vn{\tilde{P}_{\calC} u_k}_{L^2_x} \ee
From Corollary \ref{corPW} and Lemma \ref{lemma:locz:symb} we deduce that
$$ 2^{-\frac{3}{2}(k'-k)} e^{- i \psi}_{<k+2l} (t,x,D) e^{i t \jb{D}} P_k P_{C_{k'}^i(-k)} : L^2_x \to L^2_{t_{\omega_i, \lmd}} L^{\infty}_{x_{\omega_i,\lmd}} $$ 
holds for $ C_{k'}^i(-k) \subset \calC  $ where $ \omega_i $ is the direction of the center of $ C_{k'}^i(-k) $.

We can cover $ \calC=\calC_{k'}(l') $ by roughly $ 2^{3(l'+k)} $ boxes of size $ 2^{k'} \times ( 2^{k'-k} )^3 $:
$$ \calC=\cup_{i=1}^{O(2^{3(l'+k)})} C_{k'}^i(-k).  $$

Notice that $ \lmd $ can be chosen the same for all $ i $. By the definition of $ PW_{\calC}^{\pm} $ \eqref{PW:norm}
$$ \text{LHS } \eqref{PWwaves} \leq  \sum_i \vn{ e^{-i \psi}_{<k+2l}(t,x,D) e^{it \jb{D}} P_{C_{k'}^i(-k)} u_k}_{L^2_{t_{\omega_i,\lmd}} L^{\infty}_{x_{\omega_i,\lmd}} }  \ls $$
$$ \ls 2^{\frac{3}{2}(k'-k)} \sum_i \vn{\tilde{P}_{C_{k'}^i(-k)} u_k}_{L^2_x} \ls 2^{\frac{3}{2}(k'-k)} 2^{\frac{3}{2}(l'+k)} ( \sum_i  \vn{\tilde{P}_{C_{k'}^i(-k)} u_k}_{L^2_x}^2)^{\frac{1}{2}} \ls 2^{\frac{3}{2}(k'+l')} \vn{\tilde{P}_{\calC} u_k}_{L^2_x}   $$ 
where we have used Cauchy-Schwarz and orthogonality. This proves \eqref{PWwaves}.

\subsubsection{\bf The NE norms ($ d=4, k \geq 1$)} We fix $ l, -k \leq l', k', \calC=\calC_{k'}(l') $ as before and use \eqref{loczcomuting}. We prove 
$$ 2^k \vn{P_{\calC} P_l^{\omega} \bar{Q}^{\pm}_{<k+2l} e^{-i \psi}_{<k+2l} \bar{Q}^{\pm}_{<k+2l} \tilde{P}_{\calC} \phi}_{NE^{\pm}_{\calC}} \ls \vn{\tilde{P}_{\calC} \nabla_{t,x} \phi(0)}_{L^2_x}+ \vn{\tilde{P}_{\calC} \Box_m \phi}_{\bar{N}_k}
$$
Now we split again $ e^{-i \psi}_{<k+2l}=(e^{-i \psi}_{<k+2l}- e^{-i \psi}_{<k})+    e^{-i \psi}_{<k} $. The first term 
$$ P_{\calC} P_l^{\omega} \bar{Q}^{\pm}_{<k+2l} (e^{-i \psi}_{<k+2l}- e^{-i \psi}_{<k}) \bar{Q}_{<k+2l}^{\pm}\tilde{P}_{\calC} \phi $$
  is estimated by appropriately applying Prop. \ref{Xembedding} and Cor. \ref{coremb}.

For the second term we may discard $ P_{\calC} P_l^{\omega} \bar{Q}^{\pm}_{<k+2l} $ and prove
$$ 2^k \vn{e^{-i \psi}_{<k} \bar{Q}^{\pm}_{<k+2l} \tilde{P}_{\calC} \phi }_{NE^{\pm}_{\calC} }\ls  \vn{\tilde{P}_{\calC} \nabla_{t,x} \phi(0)}_{L^2_x}+ \vn{\tilde{P}_{\calC} \Box_m \phi}_{\bar{N}_k}. $$
This is reduced by Lemma \ref{waves} to
$$ e^{-i \psi}_{<k}(t,x,D) e^{\pm it \jb{D}} \tilde{P}_{\calC} : L^2_x \to NE^{\pm}_{\calC},  $$
which follows from Corollary \ref{CorNE}.

\subsection{Proof of Lemma \ref{Nspacelemma}} \ 

\

We follow the method from \cite{KST} based on Moser-type estimates. The more difficult case is $ d=4 $ and the argument can be simplified for $ d \geq 5 $. In the proof we will need the following lemmas.

\begin{lemma}
Let $ 1 \leq q \leq p \leq \infty $ and $ k \geq 0 $. Then for all $ j-O(1) \leq k' \leq k $ we have 
\begin{align}
\label{oneestim1} \vn{e^{\pm i \psi^{k}_{<j,\pm}}_{k'} (t,x,D) \bar{P}_k}_{L^p L^2 \to L^q L^2} & \ls \ep 2^{(\frac{1}{p}-\frac{1}{q})j} 2^{2(j-k')} \\
 \label{oneestim2} \vn{e^{\pm i \psi^{k}_{\pm}}_{k'} (t,x,D) \bar{P}_k}_{L^2 L^2 \to L^{\frac{10}{7}} L^2} & \ls \ep 2^{-\frac{1}{5}k'} 
\end{align} 

By duality, the same bounds holds for the right quantization. 
\end{lemma}
\begin{remark} \label{rk:heuristic}
To motivate the proof, we note that applying the $ k'(>j ) $ localization in the power series
$$ e^{i \psi_{<j}(t,x,\xi)} =1+i \psi_{<j}(t,x,\xi)+O\big(\psi_{<j}^2(t,x,\xi) \big) $$
makes the linear term $ 1+i\psi $ vanish. For the higher order terms H\" older inequality applies (in the form of decomposable lemma \ref{decomp_lemma}).
\end{remark}
\begin{proof}
To prove \eqref{oneestim1}, let $ L_{k'} $ be a disposable multiplier in the $ (t,x) $-frequencies such that 
$$  e^{\pm i \psi^{k}_{<j,\pm}}_{k'} = 2^{-2k'} L_{k'} \Delta_{t,x} e^{\pm i \psi^{k}_{<j,\pm}}=2^{-2k'} L_{k'}(- \vm{\pt_{t,x}  \psi^{k}_{<j,\pm}}^2\pm i \Delta_{t,x} \psi^{k}_{<j,\pm}   )e^{\pm i \psi^{k}_{<j,\pm}}.    $$
We may dispose of $ L_{k'} $ by translation invariance. Then \eqref{oneestim1} follows from \eqref{decomp2}.

To  prove \eqref{oneestim2} we write 
$$ e^{\pm i \psi^{k}_{\pm}}_{k'}=e^{\pm i \psi^{k}_{<k'-C,\pm}}_{k'} \pm i \int_{[k'-C,k-C]} \lpr  \psi^{k}_{\pm,l} e^{\pm i \psi^{k}_{<l,\pm}}  \rpr_{k'} \dd l  $$
For the first term we use \eqref{oneestim1}. For the second term, we have
$$ \vn{  \psi^{k}_{\pm,l} e^{\pm i \psi^{k}_{<l,\pm}}(t,x,D)}_{L^2 L^2 \to L^{\frac{10}{7}} L^2}  \ls \ep 2^{-\frac{l}{5}} $$
by Lemma \ref{decomp_lemma}, \eqref{decomp2} and \eqref{fixedtime1b}, from which \eqref{oneestim2} follows.
\end{proof}

\begin{lemma} For $k \geq 0,\ k \geq k' \geq j-O(1) $, $ j-C\leq l' \leq l \leq k-C $ and for both quantizations, we have:
 \begin{align} \label{lemastep3} 
\vn{\bar{Q}_j [ (\psi_{k'}^k e^{\pm i \psi^{k}_{< j,\pm}}_{< j})_{k'}  \bar{Q}_{\prec j} G_k]}_{L^2_{t,x}} 2^{\frac{j}{2}} & \ls \ep 2^{\frac{1}{4}(j-k')} \vn{G}_{L^{\infty} L^2} \\
 \label{lemalstep}
\vn{\bar{Q}_j [ (\psi_{l}^k \psi_{l'}^k e^{\pm i \psi^{k}_{< j,\pm}}_{< j})_{k'}  \bar{Q}_{\prec j} G_k]  }_{L^2_{t,x}} 2^{\frac{j}{2}} & \ls \ep^2  2^{\frac{1}{12}(j-l)} 2^{\frac{1}{6}(j-l')}  \vn{G}_{L^{\infty} L^2} . 
\end{align} 
\end{lemma}

\begin{proof} \pfstep{Step~1} By translation invariance we may discard the outer $ k' $ localization. By Lemma \ref{geom:cone} we deduce that in \eqref{lemastep3} the contribution of $ \psi^{k,\pm}_{k',\tht} $ (which define $ \psi^{k}_{k',\pm} $ in \eqref{phase_piece}) is zero unless $ \tht \gtrsim 2^{\frac{1}{2}(j-k')} $ and $ j \geq k'-2k-C $. For these terms, from \eqref{decomp1} we get
$$ 2^{\frac{j}{2}} \sum_{\tht \gtrsim 2^{\frac{1}{2}(j-k')}}  \vn{\psi^{k,\pm}_{k',\tht}}_{D_{k}^{\tht}(L^2 L^{\infty})} \ls  \ep 2^{\frac{1}{4}(j-k')} $$ 
from which \eqref{lemastep3} follows by Lemma \ref{decomp_lemma}. When $ k=0 $ no angular localization are needed.

\pfstep{Step~2} Now we prove \eqref{lemalstep}. First we consider the case $ l'+c \leq l=k'+O(1) $ and define $ \tht_0 \defeq 2^{\frac{1}{2}(j-l)} $.  By appropriately applying Lemma \ref{geom:cone} we deduce that the terms that contribute to \eqref{lemalstep} are $ \psi_{l}^k \psi_{l',\tht'}^k  $ for $ \tht' \gtrsim \tht_0  $ and $ \psi_{l,\tht}^k \psi_{l',\tht'}^k $ for $ \tht' \ll \tht_0, \ \tht \gtrsim \tht_0 $. We use \eqref{decomp1} with $ q=3 $ for the large angle terms and with $ q=6 $ for the other term, obtaining 
\be \label{psi:angl}  \sum_{ \tht' \gtrsim \tht_0  } \vn{\psi_{l}^k \psi_{l',\tht'}^k  e^{i\psi}   }_{L^{\infty} L^2 \to L^2_{t,x}}+ \sum_{\substack{ \tht' \ll \tht_0 \\ \tht \gtrsim \tht_0   }} \vn{\psi_{l,\tht}^k \psi_{l',\tht'}^k e^{i\psi}}_{L^{\infty} L^2 \to L^2_{t,x}} \ls \ep^2 2^{-\frac{j}{2}} 2^{\frac{1}{12}(j-l)} 2^{\frac{1}{6}(j-l')} 
\ee
In the high-high case $ l=l'+O(1) \geq k' $ the same argument applies with $ \tht_0 \defeq 2^{\frac{1}{2}(j-k')}2^{k'-l} $ and \eqref{psi:angl} also follows in this case.
\end{proof}

\begin{proof}[Proof of Lemma \ref{Nspacelemma}]
For brevity, we supress the $ k $ superscript and write $ \psi $ to denote $ \psi_{\pm}^{k} $.

\pfstep{Step~1} [The contribution of $ \bar{Q}_{>j-c}G_k $]

We use Lemma \ref{Sobolev_lemma} and \eqref{oneestim2}
$$ \vn{\bar{Q}_{j} e^{\pm i \psi}_{k'}  \bar{Q}_{>j-c} G_k}_{L^2_{t,x}} \ls 2^{\frac{j}{5}} \vn{e^{\pm i \psi}_{k'}  \bar{Q}_{>j-c} G_k}_{ L^{\frac{10}{7}} L^2} \ls \ep 2^{\frac{j-k'}{5}} \vn{\bar{Q}_{>j-c} G_k}_{L^2_{t,x}}. $$
and the last norm is $ \ls 2^{-j/2} \vn{G}_{\bar{N}_k^*} $.

\pfstep{Step~2} [The contribution of $ \bar{Q}_{\prec j}G_k $]
Motivated by remark \ref{rk:heuristic}, by iterating the fundamental theorem of calculus, we decompose the symbol
$$ e^{\pm i \psi}(t,x,\xi)=\mathcal{T}_0 \pm i \mathcal{T}_1 - \mathcal{T}_2 \mp  i\mathcal{T}_3 $$
where $ \mathcal{T}_0=e^{ i \psi_{<j}} $ and
\begin{align*}
 \mathcal{T}_1 &=\int_{[j-C,k-C]}   \psi_{l} e^{ i \psi_{<j}}  \dd l ,\qquad   \mathcal{T}_2 =\iint_{j-C \leq l' \leq l \leq k-C} \psi_{l} \psi_{l'} e^{ i \psi_{<j}}  \dd l \dd l'  \\
  \mathcal{T}_3 &=\iiint_{j-C \leq l'' \leq l' \leq l \leq k-C}  \psi_{l} \psi_{l'}  \psi_{l''}  e^{ i \psi_{<l''}}    \dd l \dd l' \dd l''   
\end{align*}
The term $ \mathcal{T}_0 $ is estimated by \eqref{oneestim1}:
$$ \vn{(\mathcal{T}_0 )_{k'}(t,x,D)  \bar{Q}_{\leq k}  G_k}_{L^2_{t,x}} \ls \ep 2^{\frac{-j}{2}} 2^{2 (j-k')} \vn{\bar{Q}_{\leq k} G_k}_{L^{\infty} L^2} $$ 
Next, we split $  \mathcal{T}_1= \mathcal{T}_1^1+\mathcal{T}_1^2 $ where
$$  \mathcal{T}_1^1=\int_{[j-C,k-C]}   \psi_{l} e^{ i \psi_{<j}}_{<j}  \dd l, \qquad \mathcal{T}_1^2= \int_{[j-C,k-C]}   \psi_{l} e^{ i \psi_{<j}}_{>j-c/2}  \dd l 
$$
By appying the $ k' $ localization, the integral defining $ \mathcal{T}_1^1 $ vanishes unless $ l=k'+O(1) $ for which we may apply \eqref{lemastep3}.
To estimate $ \mathcal{T}_1^2  $ we use Lemma \ref{decomp_lemma} with \eqref{decomp2} for $ q=6 $ and \eqref{oneestim1} with $ L^{\infty} L^2 \to L^{3} L^2 $.

We turn to  $ \mathcal{T}_2 $ and separate $ e^{ i \psi_{<j}}=e^{ i \psi_{<j}}_{<j}+e^{ i \psi_{<j}}_{>j-c/2} $ as before. For the first component we use \eqref{lemalstep}. For the second, we use \eqref{decomp2} with $ q=6 $ for $  \psi_{l}, \psi_{l'} $ and \eqref{oneestim1} with $ L^{\infty} L^2 \to L^{6} L^2 $ obtaining: 
\begin{align*}
& 2^{\frac{1}{2}j} \vn{\psi_{l} \psi_{l'} e^{ i \psi_{<j}}_{>j-c/2}}_{L^{\infty}L^2 \to L^2 L^2} \ls \ep 2^{\frac{1}{6} (j-l)} 2^{\frac{1}{6} (j-l')}  \qquad \qquad \text{for} \  l>k'-c \\  
& 2^{\frac{1}{2}j} \vn{\psi_{l} \psi_{l'} e^{ i \psi_{<j}}_{k'} }_{L^{\infty}L^2 \to L^2 L^2} \ls \ep 2^{\frac{1}{6} (j-l)} 2^{\frac{1}{6} (j-l')} 2^{2(j-k')} \qquad \text{for} \  l<k'-c. 
\end{align*}
For $ \mathcal{T}_3 $ we use \eqref{decomp2} for $ q=6 $. When $ l<k'-C $ we use \eqref{oneestim1} with $ p=q=\infty $ and it remains to integrate
$$ 2^{\frac{1}{2}j}  \vn{ \psi_{l} \psi_{l'}  \psi_{l''}  e^{ i \psi_{<l''}}_{k'} }_{L^{\infty}L^2 \to L^2 L^2} \ls \ep 2^{\frac{1}{2}j} 2^{-\frac{1}{6}l} 2^{-\frac{1}{6}l'} 2^{-\frac{1}{6}l''} 2^{2(l''-k')}.
$$
On $ l \geq k'-C $ it suffices to integrate
$$ 2^{\frac{1}{2}j}  \vn{ \psi_{l} \psi_{l'}  \psi_{l''}  e^{ i \psi_{<l''}} }_{L^{\infty}L^2 \to L^2 L^2} \ls \ep 2^{\frac{1}{2}j} 2^{-\frac{1}{6}l} 2^{-\frac{1}{6}l'} 2^{-\frac{1}{6}l''}.
$$
\end{proof}

% ----------------------------------------------------------------

\nocite{*}
\bibliographystyle{abbrv}
\bibliography{4dmMKG}

%\bibliography
% ----------------------------------------------------------------

\end{document}